\documentclass[letter,10pt]{report}
\usepackage{amssymb,amsmath,amsfonts,amsthm}
\usepackage{mathrsfs}
\usepackage{subfigure}
\usepackage{float}
\usepackage{graphicx}
\usepackage[left=1 in, right=1 in,top=1 in, bottom=1 in]{geometry}

\usepackage{hyperref}

\newtheorem{theorem}{Theorem}[section]
\newtheorem{lemma}[theorem]{Lemma}

\newtheorem{corollary}[theorem]{Corollary}

\newtheorem{remark}[theorem]{Remark}
\newtheorem{definition}[theorem]{Definition}

\makeatletter
\renewcommand \theequation {%
\ifnum \c@chapter>\z@ \@arabic\c@chapter.%
\fi\@arabic\c@equation} \@addtoreset{equation}{chapter}
\makeatother

\DeclareMathOperator*{\esssup}{ess\,sup}

\providecommand{\abs}[1]{\left\vert#1\right\vert}
\providecommand{\nm}[1]{\left\Vert#1\right\Vert}
\providecommand{\nnm}[1]{\left\Vert\left\vert#1\right\vert\right\Vert}
\providecommand{\br}[1]{\left\langle #1 \right\rangle}

\providecommand{\tnm}[1]{\left\vert#1\right\vert_{2}}
\providecommand{\tnnm}[1]{{\left\Vert#1\right\Vert}_{2}}
\providecommand{\lnm}[2]{\left\vert#1\right\vert_{\infty,{#2}}}
\providecommand{\lnnm}[2]{{\left\Vert#1\right\Vert}_{\infty,{#2}}}
\providecommand{\ltnm}[2]{{\left\Vert#1\right\Vert}_{m,{#2}}}

\providecommand{\snnm}[1]{{\left\Vert#1\right\Vert}_{\infty}}

\providecommand{\tnmh}[1]{\left\vert#1\right\vert_{2}}
\providecommand{\lnmh}[1]{\left\vert#1\right\vert_{\infty,{\vth,\varrho}}}

\providecommand{\lnmv}[1]{\left\vert#1\right\vert_{\infty,\vth,\varrho}}
\providecommand{\lnnmv}[1]{{\left\Vert#1\right\Vert}_{\infty,\vth,\varrho}}

\providecommand{\bw}[2]{\br{\vvv}^{#2}\ue^{{#1}\abs{\vvv}^2}}

\providecommand{\unm}[1]{\left\vert#1\right\vert_{\nu}}
\providecommand{\unnm}[1]{\left\Vert#1\right\Vert_{\nu}}

\providecommand{\tsm}[2]{\left\Vert#1\right\Vert_{{#2},2}}
\providecommand{\lsm}[2]{\left\Vert#1\right\Vert_{{#2},\infty,\vrh,\vth}}
\providecommand{\ssm}[2]{\left\Vert#1\right\Vert_{{#2},\infty}}

\providecommand{\pnm}[3]{\left\Vert#1\right\Vert_{{#2},{#3}}}
\providecommand{\pnnm}[2]{{\left\Vert#1\right\Vert}_{#2}}

\providecommand{\onm}[2]{{\left\vert#1\right\vert}_{{#2}(\p\Omega)}}
\providecommand{\onnm}[2]{{\left\Vert#1\right\Vert}_{{#2}(\Omega)}}

\providecommand{\tnnmt}[1]{{\left\Vert\left\vert#1\right\vert\right\Vert}_{2}}

\providecommand{\snnmt}[1]{{\left\Vert\left\vert#1\right\vert\right\Vert}_{\infty}}
\providecommand{\lnnmvt}[1]{{\left\Vert\left\vert#1\right\vert\right\Vert}_{\infty,\vth,\varrho}}
\providecommand{\tsmt}[2]{\left\Vert\left\vert#1\right\vert\right\Vert_{{#2},2}}
\providecommand{\lsmt}[2]{\left\Vert\left\vert#1\right\vert\right\Vert_{{#2},\infty,\vrh,\vth}}
\providecommand{\ssmt}[2]{\left\Vert\left\vert#1\right\vert\right\Vert_{{#2},\infty}}
\providecommand{\pnmt}[3]{\left\Vert\left\vert#1\right\vert\right\Vert_{{#2},{#3}}}
\providecommand{\pnnmt}[2]{{\left\Vert\left\vert#1\right\vert\right\Vert}_{#2}}
\providecommand{\unnmt}[1]{\left\Vert\left\vert#1\right\vert\right\Vert_{\nu}}

\providecommand{\onnmt}[2]{{\left\Vert#1\right\Vert}_{{#2}([0,t]\times\Omega)}}

\def\ud{\mathrm{d}}
\def\dt{\partial_t}
\def\p{\partial}
\def\ls{\lesssim}
\def\gs{\gtrsim}
\def\half{\frac{1}{2}}
\def\rt{\rightarrow}
\def\r{\mathbb{R}}
\def\no{\nonumber}
\def\ue{\mathrm{e}}

\def\ds{\displaystyle}

\def\e{\epsilon}
\def\d{\delta}
\def\kk{\kappa}
\def\s{\mathbb{S}}
\def\ee{\mathbf{e}}
\def\ne{\varphi}
\def\id{{\bf{1}}}

\def\vx{x}
\def\vv{v}
\def\va{v_{\eta}}
\def\vb{v_{\phi}}
\def\vc{v_{\psi}}
\def\vn{n}

\def\fs{\mathfrak{F}}
\def\f{F}
\def\fb{\mathscr{F}}
\def\fl{\mathcal{F}}

\def\m{\mu}
\def\mb{\m_{\bb}^{\e}}
\def\bb{b}
\def\mh{\m^{\frac{1}{2}}}
\def\mhh{\m^{-\frac{1}{2}}}

\def\th{\theta}
\def\rh{\rho}
\def\vu{u}
\def\vo{\omega}

\def\vuu{{\mathfrak{u}}}
\def\vvv{{\mathfrak{v}}}
\def\vsi{\sigma}
\def\ua{\mathfrak{u}_{\eta}}
\def\ub{\mathfrak{u}_{\phi}}
\def\uc{\mathfrak{u}_{\psi}}

\def\tu{\tilde u}
\def\tv{\tilde v}

\def\dx{\Delta_x}
\def\nx{\nabla_{x}}
\def\rr{\mathscr{R}}

\def\ll{\mathcal{L}}
\def\lll{\mathscr{L}}
\def\nk{\mathcal{N}}
\def\sn{\nu^{\frac{1}{2}}}
\def\snn{\nu^{-\frac{1}{2}}}

\def\vth{\vartheta}
\def\vrh{\varrho}
\def\ze{\zeta}
\def\bv{\br{\vv}^{\vth}\ue^{\varrho\abs{\vv}^2}}
\def\bvv{\br{\vvv}^{\vth}\ue^{\varrho\abs{\vvv}^2}}
\def\bvvp{\br{\vvv'}^{\vth}\ue^{\varrho\abs{\vvv'}^2}}

\def\g{g}
\def\gg{\mathcal{G}}

\def\ss{S}

\def\pk{{\mathbb{P}}}
\def\ik{{\mathbb{I}}}
\def\pp{\mathcal{P}}
\def\pe{\mathcal{P}^{\e}}
\def\q{Q}
\def\qb{\mathscr{Q}}
\def\ql{\mathcal{Q}}

\def\t{\mathcal{T}}
\def\k{\mathcal{K}}
\def\a{\mathscr{A}}
\def\b{\mathscr{B}}
\def\c{\mathscr{C}}
\def\dd{\mathscr{D}}
\def\w{\mathscr{W}}

\def\vr{r}
\def\vt{\varsigma}
\def\vbb{b}
\def\nn{\mathcal{V}}
\def\vh{w}
\def\tvh{\tilde{w}}

\def\rp{\r^+}
\def\v{\mathcal{V}}

\begin{document}

\title{Asymptotic Analysis of Boltzmann Equation in Bounded Domains}


\author{Lei Wu\footnote{
Email: lew218@lehigh.edu}\\
Department of Mathematics, Lehigh University
\and\\
Zhimeng Ouyang\footnote{Email: Zhimeng\_Ouyang@brown.edu}\\
Department of Mathematics, Brown University}
\date{}


\maketitle

\begin{abstract}
Consider 3D Boltzmann equation in convex domains with diffusive-reflection boundary condition. We study the hydrodynamic limits as the Knudsen number and Strouhal number $\e\rt 0^+$. Using the Hilbert expansion, we rigorously justify that the solution of stationary/evolutionary problem converges to that of the steady/unsteady Navier-Stokes-Fourier system.

This is the first paper to justify the hydrodynamic limits of nonlinear Boltzmann equations with hard-sphere collision kernel in bounded domain in the $L^{\infty}$ sense. The proof relies on a novel analysis on the boundary layer effect with geometric correction.

The difficulty mainly comes from three sources: 3D domain, boundary layer regularity, and time dependence. To fully solve this problem, we introduce several techniques: (1) boundary layer with geometric correction; (2) remainder estimates with $L^2-L^{6}-L^{\infty}$ framework.
\end{abstract}

\tableofcontents

\newpage


\pagestyle{myheadings} \thispagestyle{plain} \markboth{LEI WU AND ZHIMENG OUYANG}{ASYMPTOTIC ANALYSIS OF BOLTZMANN EQUATION}

\newpage

\chapter{Introduction}

\section{Stationary Boltzmann Equation}

\subsection{Problem Presentation}

We consider the stationary Boltzmann equation in a three-dimensional smooth convex domain $\Omega\ni\vx=(x_1,x_2,x_3)$
with velocity $\vv=(v_1,v_2,v_3)\in\r^3$. The density function $\fs^{\e}(\vx,\vv)$ satisfies
\begin{align}\label{large system}
\left\{
\begin{array}{l}
\e\vv\cdot\nx \fs^{\e}=Q[\fs^{\e},\fs^{\e}]\ \ \text{in}\ \
\Omega\times\r^3,\\\rule{0ex}{1.5em} \fs^{\e}(\vx_0,\vv)=P^{\e}
[\fs^{\e}](\vx_0,\vv) \ \ \text{for}\ \ \vx_0\in\p\Omega\ \ \text{and}\ \ \vv\cdot\vn(\vx_0)<0,
\end{array}
\right.
\end{align}
where $\vn(\vx_0)$ is the unit outward normal vector at $\vx_0$. 

The Knudsen number $\e$ characterizes the average distance a particle might travel between two collisions, and we assume $0<\e<<1$. Intuitively, as $\e\rt0$, the collisions occur more and more frequently and the overall behaviors of this particle system get closer and closer to that of the fluids. 

In this paper, we assume that $Q$ is the hard-sphere collision operator (see \cite[Chapter 1]{Glassey1996} and the following subsections), and in the diffusive-reflection boundary condition
\begin{align}
P^{\e}[\fs^{\e}](\vx_0,\vv):=\mb(\vx_0,\vv)\displaystyle\int_{\vuu\cdot\vn(\vx_0)>0}\fs^{\e}(\vx_0,\vuu)\abs{\vuu\cdot\vn(\vx_0)}\ud{\vuu}.
\end{align} 
It describes that the particles are absorbed by the boundary and then re-emitted based on a boundary Maxwellian
\begin{align}
\mb(\vx_0,\vv):=\frac{\rh^{\e}_{\bb}(\vx_0)}{2\pi\Big(\th^{\e}_{\bb}(\vx_0)\Big)^2}\exp\left(-\frac{\abs{\vv-\vu^{\e}_{\bb}(\vx_0)}^2}{2\th^{\e}_{\bb}(\vx_0)}\right),
\end{align}
where density, velocity and temperature $(\rh^{\e}_{\bb},\vu^{\e}_{\bb},\th^{\e}_{\bb})$ is an $\e$-perturbation of the standard Maxwellian
\begin{align}
\m(\vv):=\frac{1}{2\pi}\exp\left(-\frac{\abs{\vv}^2}{2}\right).
\end{align}
We assume that both $\mb$ and $\m$ satisfies the normalization condition
\begin{align}\label{att 04}
\int_{\vv\cdot\vn(\vx_0)>0}\mb(\vx_0,\vv)\abs{\vv\cdot\vn(\vx_0)}\ud{\vv}=\int_{\vv\cdot\vn(\vx_0)>0}\m(\vv)\abs{\vv\cdot\vn(\vx_0)}\ud{\vv}=1.
\end{align}
In addition, we require that the particles are only reflected on $\p\Omega$ without in-flow or out-flow, i.e.
\begin{align}\label{att 04'}
\int_{\r^3}\mb(\vx_0,\vv)\big(\vv\cdot\vn(\vx_0)\big)\ud{\vv}=\int_{\r^3}\m(\vv)\big(\vv\cdot\vn(\vx_0)\big)\ud{\vv}=0.
\end{align}
We also assume that $\rh^{\e}_{\bb}$, $\vu^{\e}_{\bb}$ and $\th^{\e}_{\bb}$ can be expanded into a power series with respect to $\e$,
\begin{align}
\rh^{\e}_{\bb}(\vx_0):=1+\sum_{k=1}^{\infty}\e^k\rh_{\bb,k}(\vx_0),\quad
\vu^{\e}_{\bb}(\vx_0):=0+\sum_{k=1}^{\infty}\e^k\vu_{\bb,k}(\vx_0),\quad
\th^{\e}_{\bb}(\vx_0):=1+\sum_{k=1}^{\infty}\e^k\th_{\bb,k}(\vx_0),
\end{align}
i.e. $\Big(\rh^{\e}_{\bb},\vu^{\e}_{\bb},\th^{\e}_{\bb}\Big)$ is an $\e$-perturbation of $(1,0,1)$. Hence, we may also expand the boundary Maxwellian $\mb$ into power series with respect to $\e$,
\begin{align}\label{expansion assumption}
\mb(\vx_0,\vv):=\m(\vv)+\m^{\frac{1}{2}}(\vv)\left(\sum_{k=1}^{\infty}\e^k\m_{k}(\vx_0,\vv)\right).
\end{align}
In particular, the first-order perturbation has the form
\begin{align}\label{att 06}
\m_1(\vx_0,\vv):=&\m^{\frac{1}{2}}(\vv)\bigg(\rh_{\bb,1}(\vx_0)+\vu_{\bb,1}(\vx_0)\cdot\vv+\th_{\bb,1}(\vx_0)\frac{\abs{\vv}^2-3}{2}\bigg).
\end{align}
We further assume that the perturbation is small, i.e.
\begin{align}\label{smallness assumption}
\abs{\bv\frac{\mb-\m}{\m^{\frac{1}{2}}}}\leq C_0\e,
\end{align}
for any $0\leq\varrho<\dfrac{1}{4}$ and $3<\vth\leq\vth_0$ with some given large $\vth_0$, and constant $C_0>0$ is sufficiently small. Based on \eqref{att 04}, \eqref{att 04'} and \eqref{expansion assumption}, comparing the order of $\e$, we know
\begin{align}\label{boundary compatibility}
\int_{\r^3}\m_k(\vx_0,\vv)\m^{\frac{1}{2}}(\vv)\abs{\vv\cdot\vn(\vx_0)}\ud{\vv}=&0\ \ \text{for}\ \ k\geq1,\\
\int_{\vv\cdot\vn(\vx_0)\lessgtr0}\m_k(\vx_0,\vv)\m^{\frac{1}{2}}(\vv)\abs{\vv\cdot\vn(\vx_0)}\ud{\vv}=&0\ \ \text{for}\ \ k\geq1.\label{boundary compatibility'}
\end{align}
\begin{remark}
In particular for $k=1$, we know $u_{b,1}\cdot\vn=0$. In fluid mechanics, this corresponds to non-penetration boundary condition.
\end{remark}
Note that if $\fs^{\e}$ is a solution to \eqref{large system}, then for any constant $M\in\r$, $\fs^{\e}+M\mb$ is also a solution. To guarantee uniqueness, we require the normalization condition
\begin{align}\label{large normalization}
\iint_{\Omega\times\r^3}\fs^{\e}(\vx,\vv)\ud\vv\ud\vx=\iint_{\Omega\times\r^3}\m(\vv)\ud\vv\ud\vx=\sqrt{2\pi}\abs{\Omega}.
\end{align}
We intend to study the behavior of $\fs^{\e}$ as $\e\rt0$.

\subsection{Perturbation Equation}

Considering \eqref{large normalization}, the solution $\fs^{\e}$ to \eqref{large system} can be expressed as a perturbation of the standard Maxwellian
\begin{align}
\fs^{\e}(\vx,\vv)=&\m(\vv)+\m^{\frac{1}{2}}(\vv)f^{\e}(\vx,\vv),
\end{align}
with the normalization condition
\begin{align}\label{small normalization}
\iint_{\Omega\times\r^3}f^{\e}(\vx,\vv)\m^{\frac{1}{2}}(\vv)\ud{\vv}\ud{\vx}=0.
\end{align}
Here $f^{\e}(\vx,\vv)$ satisfies the perturbation equation
\begin{align}\label{small system}
\left\{
\begin{array}{l}
\e\vv\cdot\nx
f^{\e}+\ll[f^{\e}]=\Gamma[f^{\e},f^{\e}]\ \ \text{in}\ \
\Omega\times\r^3,\\\rule{0ex}{1.5em}
f^{\e}(\vx_0,\vv)=\pe[f^{\e}](\vx_0,\vv) \ \ \text{for}\ \ \vx_0\in\p\Omega\ \ \text{and}\ \ \vv\cdot\vn(\vx_0)<0,
\end{array}
\right.
\end{align}
where
\begin{align}\label{att 31}
\ll[f^{\e}]:=-2\m^{-\frac{1}{2}}Q\Big[\m,\m^{\frac{1}{2}}f^{\e}\Big],\quad
\Gamma[f^{\e},f^{\e}]:=&\m^{-\frac{1}{2}}Q\Big[\m^{\frac{1}{2}}f^{\e},\m^{\frac{1}{2}}f^{\e}\Big],
\end{align}
and
\begin{align}\label{att 05}
\\
\pe[f^{\e}](\vx_0,\vv):=\mb(\vx_0,\vv)\m^{-\frac{1}{2}}(\vv)
\displaystyle\int_{\vuu\cdot\vn(\vx_0)>0}\m^{\frac{1}{2}}(\vuu)
f^{\e}(\vx_0,\vuu)\abs{\vuu\cdot\vn(\vx_0)}\ud{\vuu}+\m^{-\frac{1}{2}}(\vv)\bigg(\mb(\vx_0,\vv)-\m(\vv)\bigg).\no
\end{align}
Hence, in order to study $\fs^{\e}$, it suffices to consider $f^{\e}$.

\subsection{Linearized Boltzmann Operator}

To clarify, we specify the hard-sphere collision operator $Q$ in \eqref{large system} and \eqref{att 31}
\begin{align}
Q[F,G]:=&\int_{\r^3}\int_{\s^2}q(\vo,\abs{\vuu-\vv})\bigg(F(\vuu_{\ast})G(\vv_{\ast})-F(\vuu)G(\vv)\bigg)\ud{\vo}\ud{\vuu},
\end{align}
with
\begin{align}
\vuu_{\ast}:=\vuu+\vo\bigg((\vv-\vuu)\cdot\vo\bigg),\qquad \vv_{\ast}:=\vv-\vo\bigg((\vv-\vuu)\cdot\vo\bigg),
\end{align}
and the hard-sphere collision kernel
\begin{align}
q(\vo,\abs{\vuu-\vv}):=q_0\vo\cdot(\vv-\vuu),
\end{align}
for a positive constant $q_0$. Based on \cite[Chapter 3]{Glassey1996}, the linearized Boltzmann operator $\ll$ is
\begin{align}\label{att 11}
\ll[f]=&-2\m^{-\frac{1}{2}}Q\big[\m,\m^{\frac{1}{2}}f\big]:=\nu(\vv)f-K[f],
\end{align}
where
\begin{align}
\nu(\vv)=&\int_{\r^3}\int_{\s^2}q(\vo,\abs{\vuu-\vv})\m(\vuu)\ud{\vo}\ud{\vuu}
=\pi^2q_0\Bigg(\bigg(2\abs{\vv}+\frac{1}{\abs{\vv}}\bigg)\int_0^{\abs{\vv}}\ue^{-z^2}\ud{z}+\ue^{-\abs{\vv}^2}\Bigg),\\
K[f](\vv)=&K_2[f](\vv)-K_1[f](\vv)=\int_{\r^3}k(\vuu,\vv)f(\vuu)\ud{\vuu},\\
K_1[f](\vv)=&\m^{\frac{1}{2}}(\vv)\int_{\r^3}\int_{\s^1}q(\vo,\abs{\vuu-\vv})\m^{\frac{1}{2}}(\vuu)f(\vuu)\ud{\vo}\ud{\vuu}=\int_{\r^3}k_1(\vuu,\vv)f(\vuu)\ud{\vuu},\\
K_2[f](\vv)=&\int_{\r^3}\int_{\s^2}q(\vo,\abs{\vuu-\vv})\m^{\frac{1}{2}}(\vuu)\bigg(\m^{\frac{1}{2}}(\vv_{\ast})f(\vuu_{\ast})
+\m^{\frac{1}{2}}(\vuu_{\ast})f(\vv_{\ast})\bigg)\ud{\vo}\ud{\vuu}=\int_{\r^3}k_2(\vuu,\vv)f(\vuu)\ud{\vuu},
\end{align}
for some kernels
\begin{align}
k(\vuu,\vv)&=k_2(\vuu,\vv)-k_1(\vuu,\vv),\\
k_1(\vuu,\vv)&=\pi q_0\abs{\vuu-\vv}\exp\bigg(-\half\abs{\vuu}^2-\half\abs{\vv}^2\bigg),\\
k_2(\vuu,\vv)&=\frac{2\pi q_0}{\abs{\vuu-\vv}}\exp\bigg(-\frac{1}{4}\abs{\vuu-\vv}^2-\frac{1}{4}\frac{(\abs{\vuu}^2-\abs{\vv}^2)^2}{\abs{\vuu-\vv}^2}\bigg).
\end{align}
In particular, $\ll$ is self-adjoint in $L^2(\r^3)$ and the null space $\nk$ is a five-dimensional space spanned by the orthonormal basis
\begin{align}\label{att 32}
\mh\bigg\{1,\vv,\frac{\abs{\vv}^2-3}{2}\bigg\}.
\end{align}
We denote $\nk^{\perp}$ the orthogonal complement of $\nk$ in $L^2(\r^3)$.

\subsection{Main Theorem}

Let $\br{\cdot,\cdot}$ be the standard $L^2$ inner product for $\vv\in\r^3$. Define the $L^p$ and $L^{\infty}$ norms in $\r^3$:
\begin{align}
\abs{f(\vx)}_{p}:=\bigg(\int_{\r^3}\abs{f(\vx,\vv)}^p\ud{\vv}\bigg)^{\frac{1}{p}},\quad
\abs{f(\vx)}_{\infty}:=\esssup_{\vv\in\r^3}\abs{f(\vx,\vv)}.
\end{align}
Furthermore, we define the $L^p$ and $L^{\infty}$ norms in $\Omega\times\r^3$:
\begin{align}
\nm{f}_{p}:=\bigg(\iint_{\Omega\times\r^3}\abs{f(\vx,\vv)}^p\ud{\vv}\ud{\vx}\bigg)^{\frac{1}{p}},\quad
\nm{f}_{\infty}:=\esssup_{(\vx,\vv)\in\Omega\times\r^3}\abs{f(\vx,\vv)}.
\end{align}
Define the weighted $L^{2}$ norms:
\begin{align}
\unm{f(\vx)}:=\tnm{\nu^{\frac{1}{2}}f(\vx)},\quad\unnm{f}:=\tnnm{\nu^{\frac{1}{2}}f}.
\end{align}
Denote the Japanese bracket:
\begin{align}
\br{\vv}=\left(1+\abs{\vv}^2\right)^{\frac{1}{2}}
\end{align}
Define the weighted $L^{\infty}$ norm for $\vrh,\vth\geq0$:
\begin{align}
\lnmv{f(\vx)}=&\esssup_{\vv\in\r^3}\bigg(\bv\abs{f(\vx,\vv)}\bigg),\quad\lnnmv{f}=\esssup_{(\vx,\vv)\in\Omega\times\r^3}\bigg(\bv\abs{f(\vx,\vv)}\bigg).
\end{align}
In \eqref{large system} and \eqref{small system}, based on the flow direction, we can divide the boundary $\gamma:=\{(\vx_0,\vv):\ \vx_0\in\p\Omega,\vv\in\r^3\}$ into the in-flow boundary $\gamma_-$, the out-flow boundary $\gamma_+$, and the grazing set $\gamma_0$:
\begin{align}
\gamma_{-}:=&\{(\vx_0,\vv):\ \vx_0\in\p\Omega,\ \vv\cdot\vn(\vx_0)<0\},\\
\gamma_{+}:=&\{(\vx_0,\vv):\ \vx_0\in\p\Omega,\ \vv\cdot\vn(\vx_0)>0\},\\
\gamma_{0}:=&\{(\vx_0,\vv):\ \vx_0\in\p\Omega,\ \vv\cdot\vn(\vx_0)=0\}.
\end{align}
It is easy to see $\gamma=\gamma_+\cup\gamma_-\cup\gamma_0$. In particular, the boundary condition is only given on $\gamma_{-}$.

Define $\ud{\gamma}=\abs{\vv\cdot\vn}\ud{\varpi}\ud{\vv}$ on $\gamma$ for the surface measure $\varpi$. Define the $L^p$ and
$L^{\infty}$ norms on the boundary:
\begin{align}
&\nm{f}_{\gamma,p}=\bigg(\iint_{\gamma}\abs{f(\vx,\vv)}^p\ud{\gamma}\bigg)^{\frac{1}{p}},\quad\nm{f}_{\gamma,\infty}=\esssup_{(\vx,\vv)\in\gamma}\abs{f(\vx,\vv)}.
\end{align}
Also, define the weighted $L^{\infty}$ norm for $\vrh,\vth\geq0$:
\begin{align}
\nm{f}_{\gamma,\infty,\vrh,\vth}=&\esssup_{(\vx,\vv)\in\gamma}\bigg(\bv\abs{f(\vx,\vv)}\bigg).
\end{align}
The similar notation also applies to $\gamma_{\pm}$.

\begin{theorem}\label{main}
For given $\mb$ satisfying \eqref{att 04}, \eqref{att 04'}, \eqref{expansion assumption} and \eqref{smallness assumption}, there exists a unique
solution $\fs^{\e}=\mb+\m^{\frac{1}{2}}f^{\e}$ to the stationary Boltzmann equation \eqref{large system} with \eqref{large normalization}. In particular, $f^{\e}$ satisfies the equation \eqref{small system} with \eqref{small normalization}, and fulfils that for $0\leq\varrho<\dfrac{1}{4}$ and $3<\vth\leq\vth_0$
\begin{align}
\lnnmv{f^{\e}-\e\f}\ls_{\d} \e^{\frac{4}{3}-\d},
\end{align}
for any $0<\d<<1$. Here
\begin{align}
\f=&\m^{\frac{1}{2}}\left(\rh+\vu\cdot\vv+\th\frac{\abs{\vv}^2-3}{2}\right),
\end{align}
in which $(\rh,\vu,\th)$ satisfies the steady Navier-Stokes-Fourier system
\begin{align}\label{interior 0}
\left\{
\begin{array}{l}
\vu\cdot\nx\vu -\gamma_1\dx\vu +\nx p =0,\\\rule{0ex}{1.5em}
\nx\cdot\vu =0,\\\rule{0ex}{1.5em}
\vu \cdot\nx\th -\gamma_2\dx\th =0,
\end{array}
\right.
\end{align}
with boundary data
\begin{align}
\rh (\vx_0)=\rh_{\bb,1}(\vx_0)+M(\vx_0),\quad
\vu (\vx_0)=\vu_{\bb,1}(\vx_0),\quad
\th (\vx_0)=\th_{\bb,1}(\vx_0).
\end{align}
Here $\gamma_1>0$ and $\gamma_2>0$ are constants. $M(\vx_0)$ is a function chosen such that the Boussinesq relation
\begin{align}
\nx(\rh +\th )=0,
\end{align}
and the normalization condition \eqref{small normalization} hold.
\end{theorem}

\begin{remark}
The Boussinesq relation implies that $\rh(\vx) +\th(\vx)=C$ for some constant $C$ in the whole domain $\Omega$. However, the boundary data $\rh_{\bb,1}(\vx_0)$ and $\th_{\bb,1}(\vx_0)$ do not necessarily have the same sum at different $\vx_0$. Hence, $M(\vx_0)$ is designed to fill this gap. Note that we are still free to choose the constant $C$ (i.e. $M$ still has one dimension of freedom) and it is eventually determined by the normalization condition \eqref{small normalization}.
\end{remark}

\begin{remark}
From the above theorem, we know $f^{\e}\sim \e\f$ is of order $O(\e)$. The difference $f^{\e}-\e\f=o(\e)$ as $\e\rt0$.
\end{remark}

\begin{remark}
The case $\rh_{\bb,1}(\vx_0)=0$, $\vu_{\bb,1}(\vx_0)=0$ and $\th_{\bb,1}(\vx_0)\neq0$ is called the non-isothermal model, which represents a system that only has heat transfer through the boundary but has no particle exchange and no work done between the environment and the system. Based on the above theorem, the hydrodynamic limit is a steady Navier-Stokes-Fourier system with non-slip boundary condition. This provides a rigorous derivation of this important fluid model.
\end{remark}

\newpage

\section{Evolutionary Boltzmann Equation}

\subsection{Problem Presentation}\label{ott 02.}

We consider the evolutionary Boltzmann equation in a three-dimensional smooth convex domain $\Omega\ni\vx=(x_1,x_2,x_3)$
with velocity $\vv=(v_1,v_2,v_3)\in\r^3$. The density function $\fs^{\e}(t,\vx,\vv)$ satisfies
\begin{align}\label{large system.}
\left\{
\begin{array}{l}
\e^2\dt\fs^{\e}+\e\vv\cdot\nx \fs^{\e}=Q[\fs^{\e},\fs^{\e}]\ \ \text{in}\ \
\rp\times\Omega\times\r^3,\\\rule{0ex}{1.5em}
\fs^{\e}(0,\vx,\vv)=\fs^{\e}_0(\vx,\vv)\ \ \text{in}\ \
\Omega\times\r^3,\\\rule{0ex}{1.5em}
\fs^{\e}(t,\vx_0,\vv)=P^{\e}[\fs^{\e}](t,\vx_0,\vv) \ \ \text{for}\ \ \vx_0\in\p\Omega\ \ \text{and}\ \ \vv\cdot\vn(\vx_0)<0,
\end{array}
\right.
\end{align}
where $\vn(\vx_0)$ is the unit outward normal vector at $\vx_0$, the Knudsen number $\e$ satisfies $0<\e<<1$, the diffusive-reflection boundary
\begin{align}
P^{\e}[\fs^{\e}](t,\vx_0,\vv):=\mb(t,\vx_0,\vv)\displaystyle\int_{\vuu\cdot\vn(\vx_0)>0}
\fs^{\e}(t,\vx_0,\vuu)\abs{\vuu\cdot\vn(\vx_0)}\ud{\vuu}.
\end{align}
\ \\
\textbf{Boundary Assumption:}\\
The boundary Maxwellian
\begin{align}\label{ott 01.}
\mb(t,\vx_0,\vv):=\frac{\rh^{\e}_{\bb}(t,\vx_0)}{2\pi\Big(\th^{\e}_{\bb}(t,\vx_0)\Big)^2}
\exp\left(-\frac{\abs{\vv-\vu^{\e}_{\bb}(t,\vx_0)}^2}{2\th^{\e}_{\bb}(t,\vx_0)}\right),
\end{align}
is an $\e$-perturbation of the standard Maxwellian
\begin{align}
\m(\vv)=\frac{1}{2\pi}\exp\left(-\frac{\abs{\vv}^2}{2}\right).
\end{align}
We assume that both $\mb$ and $\m$ satisfies the normalization condition
\begin{align}\label{itt 01}
\int_{\vv\cdot\vn(\vx_0)>0}\mb(t,\vx_0,\vv)\abs{\vv\cdot\vn(\vx_0)}\ud{\vv}=\int_{\vv\cdot\vn(\vx_0)>0}\m(\vv)\abs{\vv\cdot\vn(\vx_0)}\ud{\vv}=1.
\end{align}
In addition, we require that the particles are only reflected on $\p\Omega$ without in-flow or out-flow, i.e.
\begin{align}\label{itt 01'}
\int_{\r^3}\mb(t,\vx_0,\vv)\big(\vv\cdot\vn(\vx_0)\big)\ud{\vv}=\int_{\r^3}\m(\vv)\big(\vv\cdot\vn(\vx_0)\big)\ud{\vv}=0.
\end{align}
We also assume that $\rh^{\e}_{\bb}$, $\vu^{\e}_{\bb}$ and $\th^{\e}_{\bb}$ can be expanded into a power series with respect to $\e$,
\begin{align}
\\
\rh^{\e}_{\bb}(t,\vx_0):=1+\sum_{k=1}^{\infty}\e^k\rh_{\bb,k}(t,\vx_0),\quad
\vu^{\e}_{\bb}(t,\vx_0):=0+\sum_{k=1}^{\infty}\e^k\vu_{\bb,k}(t,\vx_0),\quad
\th^{\e}_{\bb}(t,\vx_0):=1+\sum_{k=1}^{\infty}\e^k\th_{\bb,k}(t,\vx_0),\no
\end{align}
i.e. $\Big(\rh^{\e}_{\bb},\vu^{\e}_{\bb},\th^{\e}_{\bb}\Big)$ is an $\e$-perturbation of $(1,0,1)$. Hence, we may also expand the boundary Maxwellian $\mb$ into power series with respect to $\e$,
\begin{align}\label{expansion assumption.}
\mb(t,\vx_0,\vv)=\m(\vv)+\m^{\frac{1}{2}}(\vv)\left(\sum_{k=1}^{\infty}\e^k\m_{k}(t,\vx_0,\vv)\right).
\end{align}
In particular, we have
\begin{align}\label{att 06}
\m_1(t,\vx_0,\vv):=&\m^{\frac{1}{2}}(\vv)\bigg(\rh_{\bb,1}(t,\vx_0)+\vu_{\bb,1}(t,\vx_0)\cdot\vv+\th_{\bb,1}(t,\vx_0)\frac{\abs{\vv}^2-3}{2}\bigg).
\end{align}
We further assume that
\begin{align}\label{smallness assumption.}
\abs{\ue^{K_0t}\bv\frac{\mb-\m}{\m^{\frac{1}{2}}}}+\abs{\ue^{K_0t}\bv\frac{\dt(\mb-\m)}{\mh}}\leq C_0\e,
\end{align}
for any $0\leq\varrho<\dfrac{1}{4}$ and $3<\vth\leq\vth_0$ with some given large $\vth_0$. Here $C_0,K_0>0$ are constants and $C_0>0$ is sufficiently small. This indicates that the boundary Maxwellian $\mb$ is very close to the global Maxwellian $\m$ and its time derivative is also very small.

Based on \eqref{itt 01}, \eqref{itt 01'} and \eqref{expansion assumption.}, we know
\begin{align}\label{boundary compatibility.}
\int_{\r^3}\m_k(t,\vx_0,\vv)\m^{\frac{1}{2}}(\vv)\abs{\vv\cdot\vn(\vx_0)}\ud{\vv}=&0\ \ \text{for}\ \ k\geq1,\\
\int_{\vv\cdot\vn(\vx_0)\lessgtr0}\m_k(t,\vx_0,\vv)\m^{\frac{1}{2}}(\vv)\abs{\vv\cdot\vn(\vx_0)}\ud{\vv}=&0\ \ \text{for}\ \ k\geq1.\no
\end{align}
\begin{remark}
In particular for $k=1$, we know $u_{b,1}\cdot\vn=0$. In fluid mechanics, this corresponds to non-penetration boundary condition.
\end{remark}
\ \\
\textbf{Initial Assumption:}\\
We assume that the initial data $\fs_0$ is a perturbation of the standard Maxwellian
\begin{align}
\fs^{\e}_0(\vx,\vv):=\m(\vv)+\mh(\vv)f_0(\vx,\vv):=\m(\vv)+\mh(\vv)\sum_{k=1}^{\infty}\e^kf_{0,k}(\vx,\vv),
\end{align}
satisfying
\begin{align}\label{itt 02}
\iint_{\Omega\times\r^3}\mh(\vv)f_0(\vx,\vv)\ud\vv\ud\vx=0,
\end{align}
which means that
\begin{align}\label{initial compatibility.}
\iint_{\Omega\times\r^3}\mh(\vv)f_{0,k}(\vx,\vv)\ud\vv\ud\vx=0\ \ \text{for}\ \ k\geq1.
\end{align}
In particular, we assume that the initial data $f_{0,1}\in\nk$, i.e.
\begin{align}
f_{0,1}(\vx,\vv):=\mh(\vv)\bigg(\rh_{0,1}(\vx)+\vu_{0,1}(\vx)\cdot\vv+\th_{0,1}(\vx)\frac{\abs{\vv}^2-3}{2}\bigg).
\end{align}
for some smooth function $(\rh_{0,1},\vu_{0,1},\th_{0,1})$ satisfying the Boussinesq relation $\rh_{0,1}+\th_{0,1}=\text{constant}$.
\begin{remark}
The assumption on $f_{0,1}$ is designed to simplify the discussion of initial layer and highlight the boundary effects. For example, if $\fs_0^{\e}$ is a local Maxwellian like $\mb$ in \eqref{ott 01.}, then this requirement is naturally verified.
\end{remark}
Also, we assume the smallness of initial perturbation
\begin{align}\label{smallness assumption..}
\abs{\bv f_{0}}\leq C_0\e,
\end{align}
for any $0\leq\varrho<\dfrac{1}{4}$ and $3<\vth\leq\vth_0$. Here the constant $C_0>0$ is sufficiently small.\\
\ \\
\textbf{Compatibility Assumption:}\\
Also, the initial and boundary data satisfy the compatibility conditions at $t=0$ and $\vx_0\in\p\Omega$
\begin{align}\label{compatibility condition.}
&\m_k(0,\vx_0,\vv)=0,\quad\dt\m_k(0,\vx_0,\vv)=0\ \ \text{for}\ \ k\geq1,\\
&f_{0,k}(\vx_0,\vv)=\rh_{0,k}(\vx_0)\mh,\quad\nx f_{0,k}(\vx_0,\vv)=0,\quad\nx^2 f_{0,k}(\vx_0,\vv)=0\ \ \text{for}\ \ k\geq1.\no
\end{align}
\begin{remark}
Roughly speaking, the compatibility conditions requires that $\mb\sim\m$ and $\fs^{\e}_0\sim C\mu$ at $(0,\vx_0,\vv)$, i.e. the initial data and boundary data do not have severe variations at the intersection point. They are designed to simplify the interaction of initial layer and boundary layer. 
\end{remark}

We may directly check that the solution $\fs^{\e}$ satisfies the mass conservation
\begin{align}\label{large normalization.}
\iint_{\Omega\times\r^3}\fs^{\e}(t,\vx,\vv)\ud\vv\ud\vx=\iint_{\Omega\times\r^3}\fs_0(\vx,\vv)\ud\vv\ud\vx=\iint_{\Omega\times\r^3}\m(\vv)\ud\vv\ud\vx=\sqrt{2\pi}\abs{\Omega}.
\end{align}
We intend to study the behavior of $\fs^{\e}$ as $\e\rt0$.

\subsection{Linearization}

We rewrite the solution $\fs^{\e}$ as a perturbation of the standard Maxwellian
\begin{align}
\fs^{\e}(t,\vx,\vv)=&\m(\vv)+\m^{\frac{1}{2}}(\vv)f^{\e}(t,\vx,\vv).
\end{align}
\eqref{large normalization.} implies that $f^{\e}$ satisfies the conservation law
\begin{align}\label{small normalization.}
\iint_{\Omega\times\r^3}f^{\e}(t,\vx,\vv)\m^{\frac{1}{2}}(\vv)\ud{\vv}\ud{\vx}=0,
\end{align}
and the equation
\begin{align}\label{small system.}
\left\{
\begin{array}{l}
\e^2\dt f^{\e}+\e\vv\cdot\nx f^{\e}+\ll[f^{\e}]=\Gamma[f^{\e},f^{\e}]\ \ \text{in}\ \ \rp\times\Omega\times\r^3,\\\rule{0ex}{1.5em}
f^{\e}(0,\vx,\vv)=f_0(\vx,\vv)\ \ \text{in}\ \ \Omega\times\r^3,\\\rule{0ex}{1.5em}
f^{\e}(t,\vx_0,\vv)=\pe[f^{\e}](t,\vx_0,\vv) \ \ \text{for}\ \ t\in\rp,\ \ \vx_0\in\p\Omega\ \ \text{and}\ \ \vv\cdot\vn(\vx_0)<0,
\end{array}
\right.
\end{align}
where
\begin{align}\label{att 31.}
\ll[f^{\e}]:=-2\m^{-\frac{1}{2}}Q\Big[\m,\m^{\frac{1}{2}}f^{\e}\Big],\quad
\Gamma[f^{\e},f^{\e}]:=&\m^{-\frac{1}{2}}Q\Big[\m^{\frac{1}{2}}f^{\e},\m^{\frac{1}{2}}f^{\e}\Big],
\end{align}
and
\begin{align}\label{att 05.}
\\
\pe[f^{\e}](t,\vx_0,\vv):=\mb(t,\vx_0,\vv)\m^{-\frac{1}{2}}(\vv)
\displaystyle\int_{\vuu\cdot\vn(\vx_0)>0}\m^{\frac{1}{2}}(\vuu)
f^{\e}(t,\vx_0,\vuu)\abs{\vuu\cdot\vn(\vx_0)}\ud{\vuu}+\m^{-\frac{1}{2}}(\vv)\bigg(\mb(t,\vx_0,\vv)-\m(\vv)\bigg).\no
\end{align}
Hence, in order to study $\fs^{\e}$, it suffices to consider $f^{\e}$.

\subsection{Linearized Boltzmann Operator}

To clarify, we specify the hard-sphere collision operator $Q$ in \eqref{large system.} and \eqref{small system.}
\begin{align}
Q[F,G]:=&\int_{\r^3}\int_{\s^2}q(\vo,\abs{\vuu-\vv})\bigg(F(\vuu_{\ast})G(\vv_{\ast})-F(\vuu)G(\vv)\bigg)\ud{\vo}\ud{\vuu},
\end{align}
with
\begin{align}
\vuu_{\ast}:=\vuu+\vo\bigg((\vv-\vuu)\cdot\vo\bigg),\qquad \vv_{\ast}:=\vv-\vo\bigg((\vv-\vuu)\cdot\vo\bigg),
\end{align}
and the hard-sphere collision kernel
\begin{align}
q(\vo,\abs{\vuu-\vv}):=q_0\vo\cdot(\vv-\vuu),
\end{align}
for a positive constant $q_0$.

Based on \cite[Chapter 3]{Glassey1996}, the linearized Boltzmann operator $\ll$ is
\begin{align}\label{att 11}
\ll[f]=&-2\m^{-\frac{1}{2}}Q\big[\m,\m^{\frac{1}{2}}f\big]:=\nu(\vv)f-K[f],
\end{align}
where
\begin{align}
\nu(\vv)=&\int_{\r^3}\int_{\s^2}q(\vo,\abs{\vuu-\vv})\m(\vuu)\ud{\vo}\ud{\vuu}
=\pi^2q_0\Bigg(\bigg(2\abs{\vv}+\frac{1}{\abs{\vv}}\bigg)\int_0^{\abs{\vv}}\ue^{-z^2}\ud{z}+\ue^{-\abs{\vv}^2}\Bigg),\\
K[f](\vv)=&K_2[f](\vv)-K_1[f](\vv)=\int_{\r^3}k(\vuu,\vv)f(\vuu)\ud{\vuu},\\
K_1[f](\vv)=&\m^{\frac{1}{2}}(\vv)\int_{\r^3}\int_{\s^1}q(\vo,\abs{\vuu-\vv})\m^{\frac{1}{2}}(\vuu)f(\vuu)\ud{\vo}\ud{\vuu}=\int_{\r^3}k_1(\vuu,\vv)f(\vuu)\ud{\vuu},\\
K_2[f](\vv)=&\int_{\r^3}\int_{\s^2}q(\vo,\abs{\vuu-\vv})\m^{\frac{1}{2}}(\vuu)\bigg(\m^{\frac{1}{2}}(\vv_{\ast})f(\vuu_{\ast})
+\m^{\frac{1}{2}}(\vuu_{\ast})f(\vv_{\ast})\bigg)\ud{\vo}\ud{\vuu}=\int_{\r^3}k_2(\vuu,\vv)f(\vuu)\ud{\vuu},
\end{align}
for some kernels
\begin{align}
k(\vuu,\vv)&=k_2(\vuu,\vv)-k_1(\vuu,\vv),\\
k_1(\vuu,\vv)&=\pi q_0\abs{\vuu-\vv}\exp\bigg(-\half\abs{\vuu}^2-\half\abs{\vv}^2\bigg),\\
k_2(\vuu,\vv)&=\frac{2\pi q_0}{\abs{\vuu-\vv}}\exp\bigg(-\frac{1}{4}\abs{\vuu-\vv}^2-\frac{1}{4}\frac{(\abs{\vuu}^2-\abs{\vv}^2)^2}{\abs{\vuu-\vv}^2}\bigg).
\end{align}
In particular, $\ll$ is self-adjoint in $L^2(\r^3)$ and the null space $\nk$ is a five-dimensional space spanned by the orthonormal basis
\begin{align}\label{att 32}
\mh\bigg\{1,\vv,\frac{\abs{\vv}^2-3}{2}\bigg\}.
\end{align}
We denote $\nk^{\perp}$ the orthogonal complement of $\nk$ in $L^2(\r^3)$.

\subsection{Main Theorem}

Let $\br{\cdot,\cdot}$ be the standard $L^2$ inner product for $\vv\in\r^3$. Define the $L^p$ and $L^{\infty}$ norms in $\r^3$:
\begin{align}
\abs{f(t,\vx)}_{p}:=\bigg(\int_{\r^3}\abs{f(t,\vx,\vv)}^p\ud{\vv}\bigg)^{\frac{1}{p}},\quad
\abs{f(t,\vx)}_{\infty}:=\esssup_{(\vv)\in\r^3}\abs{f(t,\vx,\vv)}.
\end{align}
Furthermore, we define the $L^p$ and $L^{\infty}$ norms in $\Omega\times\r^3$:
\begin{align}
\nm{f(t)}_{p}:=\bigg(\iint_{\Omega\times\r^3}\abs{f(t,\vx,\vv)}^p\ud{\vv}\ud{\vx}\bigg)^{\frac{1}{p}},\quad
\nm{f(t)}_{\infty}:=\esssup_{(\vx,\vv)\in\Omega\times\r^3}\abs{f(t,\vx,\vv)}.
\end{align}
Moreover, we define the $L^p$ and $L^{\infty}$ norms in $\rp\times\Omega\times\r^3$:
\begin{align}
\nnm{f}_{p}:=\bigg(\int_{\rp}\iint_{\Omega\times\r^3}\abs{f(\vx,\vv)}^p\ud{\vv}\ud{\vx}\bigg)^{\frac{1}{p}},\quad
\nnm{f}_{\infty}:=\esssup_{(t,\vx,\vv)\in\rp\times\Omega\times\r^3}\abs{f(t,\vx,\vv)}.
\end{align}
Define the weighted $L^{2}$ norms:
\begin{align}
\unm{f(t,\vx)}:=\tnm{\nu^{\frac{1}{2}}f(t,\vx)},\quad\unnm{f(t)}:=\tnnm{\nu^{\frac{1}{2}}f(t)},\quad\unnmt{f}:=\tnnmt{\nu^{\frac{1}{2}}f}.
\end{align}
Denote the Japanese bracket:
\begin{align}
\br{\vv}=\left(1+\abs{\vv}^2\right)^{\frac{1}{2}}
\end{align}
Define the weighted $L^{\infty}$ norm for $\vrh,\vth\geq0$:
\begin{align}
&\lnmv{f(t,\vx)}=\esssup_{\vv\in\r^3}\bigg(\bv\abs{f(t,\vx,\vv)}\bigg),\quad\\
&\lnnmv{f(t)}=\esssup_{(\vx,\vv)\in\Omega\times\r^3}\bigg(\bv\abs{f(t,\vx,\vv)}\bigg),\no\\
&\lnnmvt{f}=\esssup_{(t,\vx,\vv)\in\rp\times\Omega\times\r^3}\bigg(\bv\abs{f(t,\vx,\vv)}\bigg).\no
\end{align}
In \eqref{large system.} and \eqref{small system.}, based on the flow direction, we can divide the boundary $\gamma:=\{(\vx_0,\vv):\ \vx_0\in\p\Omega,\vv\in\r^3\}$ into the in-flow boundary $\gamma_-$, the out-flow boundary $\gamma_+$, and the grazing set $\gamma_0$:
\begin{align}
\gamma_{-}:=&\{(\vx_0,\vv):\ \vx_0\in\p\Omega,\ \vv\cdot\vn(\vx_0)<0\},\\
\gamma_{+}:=&\{(\vx_0,\vv):\ \vx_0\in\p\Omega,\ \vv\cdot\vn(\vx_0)>0\},\\
\gamma_{0}:=&\{(\vx_0,\vv):\ \vx_0\in\p\Omega,\ \vv\cdot\vn(\vx_0)=0\}.
\end{align}
It is easy to see $\gamma=\gamma_+\cup\gamma_-\cup\gamma_0$. In particular, the boundary condition is only given on $\rp\times\gamma_{-}$.

Define $\ud{\gamma}=\abs{\vv\cdot\vn}\ud{\varpi}\ud{\vv}$ on $\gamma$ for the surface measure $\varpi$ the surface measure. Define the $L^p$ and
$L^{\infty}$ norms on the boundary:
\begin{align}
&\nm{f(t)}_{\gamma,p}=\bigg(\iint_{\gamma}\abs{f(t,\vx,\vv)}^p\ud{\gamma}\bigg)^{\frac{1}{p}},\quad\nm{f(t)}_{\gamma,\infty}=\esssup_{(\vx,\vv)\in\gamma}\abs{f(t,\vx,\vv)}.
\end{align}
Also, Define the $L^p$ and $L^{\infty}$ norms on the boundary with time:
\begin{align}
&\nnm{f}_{\gamma,p}=\bigg(\int_{\rp}\iint_{\gamma}\abs{f(t,\vx,\vv)}^p\ud{\gamma}\bigg)^{\frac{1}{p}},
\quad\nnm{f}_{\gamma,\infty}=\esssup_{(t,\vx,\vv)\in\rp\times\gamma}\abs{f(\vx,\vv)}.
\end{align}
Also, define the weighted $L^{\infty}$ norm for $\vrh,\vth\geq0$:
\begin{align}
&\nm{f(t)}_{\gamma,\infty,\vrh,\vth}=\esssup_{(\vx,\vv)\in\gamma}\bigg(\bv\abs{f(t,\vx,\vv)}\bigg),\\\quad &\nnm{f}_{\gamma,\infty,\vrh,\vth}=\esssup_{(t,\vx,\vv)\in\rp\times\gamma}\bigg(\bv\abs{f(t,\vx,\vv)}\bigg).\no
\end{align}
The similar notation also applies to $\gamma_{\pm}$. In all above notation, we can replace $\rp$ by $[0,t]$ or even $[s,t]$, and it can be understood from the context without confusion.

\begin{theorem}\label{main.}
For given $\mb$ and $\fs_0^{\e}$ satisfying the assumptions in Section \ref{ott 02.},
there exists a unique
solution $\fs^{\e}=\m+\m^{\frac{1}{2}}f^{\e}$ to the evolutionary Boltzmann equation \eqref{large system.}. In particular, $f^{\e}$ satisfies the equation \eqref{small system.} with \eqref{small normalization.}, and fulfils that for $0\leq\varrho<\dfrac{1}{4}$ and $3<\vth\leq\vth_0$, there exists $K>0$, such that
\begin{align}
\lnnmvt{\ue^{Kt}\Big(f^{\e}-\e\f\Big)}\ls_{\d} \e^{\frac{4}{3}-\d},
\end{align}
for any $0<\d<<1$, where
\begin{align}
\f=&\m^{\frac{1}{2}}\left(\rh+\vu\cdot\vv+\th\frac{\abs{\vv}^2-3}{2}\right),
\end{align}
in which $(\rh,\vu,\th)$ satisfies the unsteady Navier-Stokes-Fourier system
\begin{align}\label{interior 0}
\left\{
\begin{array}{l}
\dt\vu+\vu\cdot\nx\vu -\gamma_1\dx\vu +\nx p =0,\\\rule{0ex}{1.5em}
\nx\cdot\vu =0,\\\rule{0ex}{1.5em}
\dt\th+\vu \cdot\nx\th -\gamma_2\dx\th =0,
\end{array}
\right.
\end{align}
with initial and boundary data
\begin{align}
&\rh(0,\vx)=\rh_{0,1},\quad \vu(0,\vx)=\vu_{0,1},\quad \th(0,\vx)=\th_{0,1},\\
&\rh(t,\vx_0)=\rh_{\bb,1}(t,\vx_0)+M(t,\vx_0),\quad \vu(t,\vx_0)=\vu_{\bb,1}(t,\vx_0),\quad \th(t,\vx_0)=\th_{\bb,1}(t,\vx_0).
\end{align}
Here $\gamma_1>0$ and $\gamma_2>0$ are some constants, $M(t,\vx_0)$ is a function chosen such that the Boussinesq relation
\begin{align}
\nx(\rh +\th )=0,
\end{align}
and the conservation law \eqref{small normalization.} hold for all time $t$.
\end{theorem}

\begin{remark}
The Boussinesq relation implies that $\rh(t,\vx) +\th(t,\vx)=C(t)$ for some time-dependent function $C(t)$ in the whole domain $\Omega$. However, at each $t$, the boundary data $\rh_{\bb,1}(t,\vx_0)$ and $\th_{\bb,1}(t,\vx_0)$ do not necessarily have the same sum at different $\vx_0$. Hence, $M(t,\vx_0)$ is designed to fill this gap. Note that we are still free to choose the $C(t)$ (i.e. $M$ still has one dimension of freedom at each $t$) and it is eventually determined by the conservation law \eqref{small normalization.}.
\end{remark}

\begin{remark}
From the above theorem, we know $f^{\e}\sim \e\f$ is of order $O(\e)$. The difference $f^{\e}-\e\f=o(\e)$ as $\e\rt0$.
\end{remark}


\begin{remark}
In the smallness assumption \eqref{smallness assumption.}, if $K_0=0$, then the main theorem still holds with $K=0$. Exponential decay in time is not necessary.
\end{remark}

\begin{remark}
Our proof of the main theorem relies on the assumptions in Section \ref{ott 02.}. To remove these technical requirements will be the main topics of our future research.
\end{remark}

\newpage

\section{History and Motivation}

\subsection{Previous Results}

Hydrodynamic limits are central to connecting the kinetic theory and fluid mechanics. It provides rigorous derivation of fluid equations (like Euler equations or Navier-Stokes equations, etc.) from the kinetic equations (like Boltzmann equations, Landau equations, etc.). As an integrated step to tackle the well-known Hilbert's Sixth Problem, since early 20th century, this type of problems have been extensively studied in many different settings: stationary or evolutionary, linear or nonlinear, strong solution or weak solution, etc.

The early result \cite{Hilbert1916} dates back to 1916 by Hilbert himself, using the so-called Hilbert's expansion, i.e. an expansion of the density function $\fs^{\e}$ as a power series of the Knudsen number $\e$.

On a large time scale $\sim\e^{-1}$, the diffusion effects dominate and the formal derivation reveals that the Boltzmann solution is close to that of the incompressible Navier-Stokes-Fourier system.
A lot of works for $\r^n$ or $\mathbb{T}^n$ domains have been presented, e.g. \cite{Golse.Saint-Raymond2004}, \cite{Masi.Esposito.Lebowitz1989}, \cite{Bardos.Golse.Levermore1991}, \cite{Bardos.Golse.Levermore1993}, \cite{Bardos.Golse.Levermore1998}, \cite{Bardos.Golse.Levermore2000}, \cite{Guo2006}, for either smooth solutions or renormalized solutions.

The general theory of initial-boundary-value problems was first developed in \cite{Grad1963}, and then extended by \cite{Sone1969}, \cite{Sone1971}, \cite{Sone1991}, \cite{Sone.Aoki1987}, for both the evolutionary and stationary equations. The classical books \cite{Sone2002} and \cite{Sone2007} provide a comprehensive summary of previous results and give a complete analysis of such approaches. However, the results in \cite{Sone2002} and \cite{Sone2007} are only formal and lack rigorous justifications.

It is worth noting that the approach of renormalized solution does not work for stationary hydrodynamic limit problems due to the lack of $L^1$ and entropy estimates. Hence, it is necessary to develop a different theory based on strong solutions and energy estimates.

\subsection{Asymptotic Analysis}

For stationary Boltzmann equation where the state of gas is close to a uniform state at rest, the expansion of the perturbation $f^{\e}=O(\e)$ consists of two parts: the interior solution $f^{\e}_{\text{in}}$, which is based on a hierarchy of linearized Boltzmann equations and satisfies a steady Navier-Stokes-Fourier system, and the boundary layer $f^{\e}_{\text{bl}}$, which is based on a half-space kinetic equation and decays rapidly when it is away from the boundary.

The justification of hydrodynamic limits for stationary problems usually involves two steps: well-posedness of expansion and remainder estimates:
\begin{itemize}
\item
Step 1: Expanding $f^{\e}_{\text{in}}=\ds\sum_{k=1}^{\infty}\e^k\f_k$ and $f^{\e}_{\text{bl}}=\ds\sum_{k=1}^{\infty}\e^k\fb_k$ as power series of $\e$ and proving the coefficients $\f_k$ and $\fb_k$ are well-defined. This is doable by inserting above expansion ansatz into the Boltzmann equation to compare the order of $\e$ and get a hierarchy of equations for $\f_k$ and $\fb_k$. Traditionally, the estimates of interior solutions $\f_k$ are relatively straightforward. On the other hand, boundary layers $\fb_k$ satisfy one-dimensional half-space problems which lose some key structures of the original equations. The well-posedness of boundary layer equations are sometimes extremely difficult and it is possible that they are actually ill-posed (e.g. certain type of Prandtl layers).
\item
Step 2: Proving that $R=f^{\e}-\e\f_1-\e\fb_1=o(\e)$ as $\e\rt0$. Ideally, this should be done just by expanding to the leading-order level $\f_1$ and $\fb_1$. However, in singular perturbation problems, the estimates of the remainder $R$ usually involve negative powers of $\e$, which require expansion to higher order terms $\f_N$ and $\fb_N$ for $N\geq2$ such that we have sufficient power of $\e$. In other words, we define $R=f^{\e}-\ds\sum_{k=1}^{N}\e^k\f_k-\ds\sum_{k=1}^{N}\e^k\fb_k$ for $N\geq2$ instead of $R=f^{\e}-\e\f_1-\e\fb_1$ to get better estimate of $R$.
\end{itemize}

For evolutionary Boltzmann equation, besides the interior solution and boundary layer, there are also rapid variation of solutions near $t=0$, which is called the initial layer $f^{\e}_{\text{il}}$. Similar to the above analysis, we may expand $f^{\e}_{\text{bl}}= \ds\sum_{k=1}^{\infty}\e^k\fl_k$ and try to justify that $R=f^{\e}-\e\f_1-\e\fb_1-\e\fl_1=o(\e)$ as $\e\rt0$. In particular, the boundary layer and initial layer have complicated interaction at $t=0$ and $\vx\in\p\Omega$, which possibly generates the so-called initial-boundary layer. We know very little about this mixed effect and its well-posedness is unknown even in 1D.

Note that boundary layer plays a significant role in proving the asymptotic convergence in the $L^{\infty}$ sense. If instead we consider $L^p$ convergence for $1\leq p<\infty$, then the boundary layer $\fb_1$ is of order $\e^{\frac{1}{p}}$ due to rescaling, which is negligible compared with $\f_1$ as $\e\rt0$. \cite{Esposito.Guo.Kim.Marra2015} justifies the $L^p$ convergence for 3D stationary and evolutionary Boltzmann without boundary layer analysis. In this monograph, we will focus on the $L^{\infty}$ convergence.

We start this whole project from a simple kinetic model -- neutron transport equations and investigate the effects of boundary layers in details. We developed a whole new theory about the construction of boundary layers with geometric corrections and justified the diffusive limit with either in-flow or diffusive-reflection boundary. We refer to our papers \cite{AA003}, \cite{AA006}, \cite{AA007}, \cite{AA009}, \cite{AA014} for steady problems, and \cite{AA005}, \cite{AA012} for unsteady problems. In particular, our recent \cite{AA016} and monograph \cite{BB001} provide a comprehensive description of the motivation of this theory and its application in 3D steady/unsteady neutron transport equations.

As for the Boltzmann equation, we refer to \cite{AA004} and \cite{AA013} for the asymptotic convergence in 2D stationary problems. In particular, our recent work \cite{AA013} provides a careful discussion about the implementation and difficulties of the boundary layer analysis in 2D.

\newpage

\section{Difficulties and Methodology}

\subsection{Formulation and Difficulties in Stationary Problems}

In this monograph, our formulation is similar to that of \cite{Esposito.Guo.Kim.Marra2015} and \cite{AA013}.
As \cite[Section 1]{Esposito.Guo.Kim.Marra2015} and \cite[Section 2.2]{AA013} reveal, compared with 2D model and $L^p$ convergence, 3D problems and $L^{\infty}$ convergence have several key difficulties:
\begin{itemize}

\item
Remainder estimates:\\
\cite{Esposito.Guo.Kim.Marra2015} pointed out that $L^{\infty}$ convergence requires improved remainder estimates. However, as one of the key steps, the embedding theorem is much worse in 3D than in 2D. For example, the result as \cite[(4.14)]{AA013} is only true when $1\leq m< 3$. This restricts our choice of $m$ and makes it hard to further improve the remainder estimates. \cite{AA009}, \cite{BB001} and \cite{AA016} provide the estimates in 3D neutron transport equation, and we can clearly see the lose of powers in $\e$.

\item
Boundary layer formulation:\\
If $\Omega\subset\r^3$ is a smooth convex domain, then $\p\Omega$ is a 2D manifold. For each point on $\p\Omega$, there are two orthogonal principal directions corresponding to two principal curvatures. The boundary layer equation actually depends on the detailed form of curvature. As \cite{AA009} reveals, even the well-posedness is highly non-trivial, let alone the regularity estimates.

\item
Singular kernel:\\
A more serious issue is that 3D collision kernel $k(v,v')$ contains the singularity $\dfrac{1}{\abs{v-v'}}$ which is absent in 2D. Then the preliminary lemmas in \cite[Section 2.1]{AA013} may not hold any more. Hence, the key argument \cite[(5.81),(5.82),(5.95)]{AA013} cannot be adapted to the singular kernel $k$.

\end{itemize}
In summary, though the structure of our proof is similar to that of \cite{Esposito.Guo.Kim.Marra2015} and \cite{AA013}, we need fresh new techniques to handle the above difficulties.

For the stationary problem, our major upshots focus on the boundary layer analysis:
\begin{itemize}
    \item
    We use a non-standard energy method to justify the well-posedness and decay of the boundary layer. The main idea is to first utilize the well-known macro-micro decomposition to separate the kernel and non-kernel contribution; then by delicately choosing test functions and using $L^2-L^{\infty}$ framework, we arrive at the desired bounds. Different from \cite{AA004} and \cite{AA013}, we need to justify the $L^{\infty}$ estimate and exponential decay simultaneously to close the proof.
    \item
    We use an intricate characteristic analysis to analyze the weighted non-local operator and justify the regularity of the boundary layer. This is the most important contribution of this monograph. We first proved the important preliminary results: Lemma \ref{Regularity lemma 0} -- Lemma \ref{Regularity lemma 3}. Then we implemented them in the different regions of the characteristics, which is a major improvement from the argument as in \cite[Section 5]{AA013}.
\end{itemize}
As a minor upshot, we also improve and simplify the remainder estimates in \cite[Section 4]{AA013}. Similar to \cite{Esposito.Guo.Kim.Marra2015}, here we use a modified $L^2-L^6-L^{\infty}$ framework. We list some recent development using this path \cite{Esposito.Guo.Kim.Marra2013}, \cite{Esposito.Guo.Marra2018}, \cite{Guo2010}, \cite{Guo.Kim.Tonon.Trescases2013}, \cite{Guo.Kim.Tonon.Trescases2016}, \cite{Kim2011}.

\subsection{Formulation and Difficulties in Evolutionary Problems}

For evolutionary problems, the major difficulty is that the remainder estimates are much worse than that of the stationary problem. The main theorem requires at least $\e^{-1}$ convergence for the remainder $R$. Stationary remainder estimate reads
\begin{align}\label{eqn 8}
\nm{R}_{L^{\infty}(\Omega\times\r^3)}\ls \e^{-\frac{5}{2}}\nm{S}_{L^{\frac{6}{5}}(\Omega\times\r^3)}+\text{good terms},
\end{align}
where $S=(v\cdot\nx+\e^{-1}\ll)[R]\sim \e^3$ and the boundary layer rescaling $\eta\sim \e^{-1}\mu$ provides additional $\e^{\frac{5}{6}}$ under $L^{\frac{6}{5}}$ norm. Eventually, we get $\nm{R}_{L^{\infty}}\ls\e^{-\frac{5}{2}}\times\e^{3+\frac{5}{6}}=\e^{\frac{4}{3}}$ convergence. However, for evolutionary problem, so far the best estimate is
\begin{align}
\nm{R}_{L^{\infty}(\r^+\times\Omega\times\r^3)}\ls \e^{-\frac{7}{2}}\nm{S(t)}_{L^{2}(\Omega\times\r^3)}+\text{good terms},
\end{align}
where $S(t)=(\e\dt+v\cdot\nx+\e^{-1}\ll)[R(t)]\sim \e^3$ and we only obtain extra $\e^{\frac{1}{2}}$ from the boundary layer rescaling. Hence, we have $\nm{R}_{L^{\infty}}\ls\e^{-\frac{7}{2}}\times\e^{3+\frac{1}{2}}=\e^{0}\sim 1$ which is far from closing the proof.

Our strategy includes two steps:

First, we introduce a novel $L^2-L^6-L^{\infty}$ boostrapping framework. This is rooted in the nonlinear energy method and we use an intricate energy-dissipation structure to bound both the instantaneous and accumulative $R$ with mutual dependence. In detail, we justify $L^2$ bounds of $R$ and $\dt R$ with non-standard energy method and prove $L^6$ bound of $R(t)$ with a fresh kernel estimates with interpolation argument. Ideally, we should be able to show that
\begin{align*}
\nm{R}_{L^{\infty}(\r^+\times\Omega\times\r^3)}\ls \e^{-\frac{5}{2}}\Big(\nm{S(t)}_{L^{\frac{6}{5}}(\Omega\times\r^3)}+\nm{\dt S}_{L^{2}(\r^+\times\Omega\times\r^3)}\Big)+\text{good terms}.
\end{align*}

Next, our central idea is to smartly utilize the ``good'' stationary remainder estimates. We design the highest-order boundary layer in a rather unusual way. Specifically, we reformulate the $\e$-Milne problem with geometric correction, such that it recovers the stationary equation. This allows us to use stationary remainder estimates \eqref{eqn 8} to control boundary layers, leaving several non-trivial remainder terms which can be delicately handled. In this fashion, we get
\begin{align}
\nm{S(t)}_{L^{\frac{6}{5}}(\Omega\times\r^3)}+\nm{\dt S}_{L^{2}(\r^+\times\Omega\times\r^3)}\ls \e^{3+\frac{5}{6}}.
\end{align}
Hence, we get the same $\e^{\frac{4}{3}}$ convergence as in the stationary case.

\subsection{Notation and Convention}

Throughout this paper, $C>0$ denotes a constant that only depends on
the domain $\Omega$, but does not depend on the data or $\e$. It is
referred as universal and can change from one inequality to another.
When we write $C(z)$, it means a certain positive constant depending
on the quantity $z$. We write $a\ls b$ to denote $a\leq Cb$.

This paper is organized as follows: in Chapter 2, we study the stationary problem, and in Chapter 3, we study the evolutionary problem. Chapter 4 focuses on the analysis of boundary layer equation, i.e. the $\e$-Milne problem with geometric correction.

\newpage

\chapter{Stationary Boltzmann Equation}

\section{Asymptotic Expansion}

\subsection{Interior Expansion}\label{att section 1}

We define the interior expansion
\begin{align}\label{interior expansion}
f^{\e}_{\text{in}}(\vx,\vv)=\sum_{k=1}^{3}\e^k\f_k(\vx,\vv).
\end{align}
Plugging it into the equation \eqref{small system} and comparing the order of $\e$, we obtain
\begin{align}
\ll[\f_1]=&0,\label{ott 01}\\
\ll[\f_2]=&-\vv\cdot\nx\f_1+\Gamma[\f_1,\f_1],\label{ott 02}\\
\ll[\f_3]=&-\vv\cdot\nx\f_2+2\Gamma[\f_1,\f_2].\label{ott 03}
\end{align}
The analysis of $\f_k$ solvability is standard and well-known. Note that the null space $\nk$ of the operator $\ll$ is spanned by
\begin{align}
\m^{\frac{1}{2}}\bigg\{1,v_1,v_2,v_3,\dfrac{\abs{\vv}^2-3}{2}\bigg\}=\{\ne_0,\ne_1,\ne_2,\ne_3,\ne_4\}.
\end{align}
Then $\ll[f]=S$ is solvable if and only if $S\in\nk^{\perp}$ the orthogonal complement of $\nk$ in $L^2(\r^3)$. As \cite[Chapter 4]{Sone2002} and \cite[Chapter 3]{Sone2007} reveal, each $\f_k$ consists of three parts:
\begin{align}
\f_k(\vx,\vv):=A_k (\vx,\vv)+B_k (\vx,\vv)+C_k (\vx,\vv).
\end{align}
\begin{itemize}
\item
Principal contribution $\ds A_k:=\sum_{i=0}^4A_{k,i}\ne_i\in\nk$, where the coefficients $A_{k,i}$ must be determined at each order $k$ independently.
\item
Connecting contribution $\ds B_k:=\sum_{i=0}^4B_{k,i}\ne_i\in\nk$, where the coefficients $B_{k,i}$ depends on $A_s$ for $1\leq s\leq k-1$. In other words, $B_k$ is accumulative information from previous orders and thus is not independent. This term is present due to the nonlinearity in $\Gamma$. In detail,
\begin{align}\label{at 12}
B_{k,0}=&0,\quad
B_{k,1}=\sum_{i=1}^{k-1}A_{i,0} A_{k-i,1} ,\quad
B_{k,2}=\sum_{i=1}^{k-1}A_{i,0} A_{k-i,2} ,\quad
B_{k,3}=\sum_{i=1}^{k-1}A_{i,0} A_{k-i,3} ,\\
B_{k,4}=&\sum_{i=1}^{k-1}\bigg(A_{i,0} A_{k-i,4} +A_{i,1} A_{k-i,1} +A_{i,2} A_{k-i,2}+A_{i,3} A_{k-i,3}\no\\
&+\sum_{j=1}^{k-1-i}A_{i,0} (A_{j,1} A_{k-i-j,1} +A_{j,2} A_{k-i-j,2}+A_{j,3} A_{k-i-j,3} )\bigg).\no
\end{align}
\item
Orthogonal contribution $C_k\in\nk^{\perp}$ satisfying
\begin{align}
\ll[C_k]=&-\vv\cdot\nx\f_{k-1}+\sum_{i=1}^{k-1}\Gamma[\f_i,\f_{k-i}],
\end{align}
which can be uniquely determined. Similar to $B_k$, here $C_k$ is also accumulative information from previous orders and thus is not independent.
\end{itemize}
All in all, we will focus on how to determine $A_k$. Traditionally, we write
\begin{align}
A_k=\m^{\frac{1}{2}}\left(\rh_k +\vu_k\cdot\vv+\th_k \left(\frac{\abs{\vv}^2-3}{2}\right)\right),
\end{align}
where the coefficients $\rh_k$, $\vu_k$ and $\th_k$ represent density, velocity and temperature in the macroscopic scale. \cite[Chapter 4]{Sone2002} and \cite[Chapter 3]{Sone2007} states that $(\rh_k,\vu_k,\th_k)$ satisfies the equations as follows:\\
\ \\
\eqref{ott 02} implies:
\begin{align}
p_1 -(\rh_1 +\th_1 )=&0,\label{att 07}\\
\nx p_1 =&0,\\
\nx\cdot\vu_1 =&0,
\end{align}
\eqref{ott 03} implies:
\begin{align}
p_2 -(\rh_2 +\th_2 +\rh_1 \th_1 )=&0,\\
\vu_1 \cdot\nx\vu_1 -\gamma_1\dx\vu_1 +\nx p_2 =&0,\\
\vu_1 \cdot\nx\th_1 -\gamma_2\dx\th_1 =&0,\\
\nx\cdot\vu_2 +\vu_1\cdot\nx\rh_1=&0.\label{att 08}
\end{align}
Here $p_1$ and $p_2$ represent the pressure, $\gamma_1>0$ and $\gamma_2>0$ are constants. (In particular, for different collision kernel, these constants may be different.) The higher-order expansion produces more complicated fluid equations, which can be found in \cite[Chapter 4]{Sone2002}. If the interior solution $\f_k$ cannot satisfy the boundary condition in \eqref{small system}, then we have to introduce boundary layer $\fb_k$ to handle the gap.

\subsection{Quasi-Spherical Coordinate System}\label{att section 2}

In order to define boundary layer, we need to design a coordinate system based on the normal and tangential directions on the boundary surface. This is physically more reasonable and mathematically more convenient.

Our main goal is to rewrite the three-dimensional transport operator $\vv\cdot\nx$ with this new coordinate system (we call it quasi-spherical coordinate system). This is basically textbook-level differential geometry, so we omit the details.\\
\ \\
Substitution 1: Spacial Substitution:\\
We choose the simplest coordinate system to parameterize the surface $\p\Omega$. For smooth manifold $\p\Omega$, there exists an orthogonal curvilinear coordinates system $(\iota_1,\iota_2)$ such that the coordinate lines coincide with the principal directions at $\vx_0$ (at least locally).

Assume $\p\Omega$ is parameterized by $\vr=\vr(\iota_1,\iota_2)$. Let $\abs{\cdot}$ denote the length and $\p_i$ denote the derivative with respect to $\iota_i$ for $i=1,2$. Hence, $\p_1\vr$ and $\p_2\vr$ represent two orthogonal tangential vectors. Denote $P_i=\abs{\p_i\vr}$ for $i=1,2$. Then define the two orthogonal unit tangential vectors
\begin{align}\label{coordinate 14}
\vt_1:=\frac{\p_1\vr}{P_1},\ \ \vt_2:=\frac{\p_2\vr}{P_2}.
\end{align}
Also, the outward unit normal vector is
\begin{align}\label{coordinate 1}
\vn:=\frac{\p_1\vr\times\p_2\vr}{\abs{\p_1\vr\times\p_2\vr}}=\vt_1\times\vt_2.
\end{align}
Obviously, $(\vt_1,\vt_2,\vn)$ forms a new orthogonal frame. Hence, consider the corresponding new coordinate system $(\iota_1,\iota_2,\mu)$, where $\mu$ denotes the normal distance to boundary surface $\p\Omega$, i.e.
\begin{align}
\vx=\vr-\mu\vn.
\end{align}
Note that $\mu=0$ means $\vx\in\p\Omega$ and $\mu>0$ means $\vx\in\Omega$ (before reaching the other side of $\p\Omega$). Using this new coordinate system, the transport operator becomes
\begin{align}\label{coordinate 8}
\vv\cdot\nx=&-\frac{\Big(\left(\p_1\vr-\mu\p_1\vn\right)\times\left(\p_2\vr-\mu\p_2\vn\right)\Big)\cdot\vv}
{\Big(\left(\p_1\vr-\mu\p_1\vn\right)\times\left(\p_2\vr-\mu\p_2\vn\right)\Big)\cdot\vn}\frac{\p f}{\p\mu}\\
&+\frac{\Big(\left(\p_2\vr-\mu\p_2\vn\right)\times\vn\Big)\cdot\vv}{\Big(\left(\p_1\vr-\mu\p_1\vn\right)\times\left(\p_2\vr-\mu\p_2\vn\right)\Big)\cdot\vn}\frac{\p f}{\p\iota_1}-\frac{\Big(\left(\p_1\vr-\mu\p_1\vn\right)\times\vn\Big)\cdot\vv}{\Big(\left(\p_1\vr-\mu\p_1\vn\right)\times\left(\p_2\vr-\mu\p_2\vn\right)\Big)\cdot\vn}\frac{\p f}{\p\iota_2}.\no
\end{align}
We may further simplify this expression utilizing the orthogonality. Denote the first fundamental form
\begin{align}
(E,F,G):=\Big(\p_1\vr\cdot\p_1\vr,\ \p_1\vr\cdot\p_2\vr,\ \p_2\vr\cdot\p_2\vr\Big),
\end{align}
and second fundamental form
\begin{align}
(L,M,N):=\Big(\p_{11}\vr\cdot\vn,\ \p_{12}\vr\cdot\vn,\ \p_{22}\vr\cdot\vn\Big).
\end{align}
Then we have $F=M=0$ due to the orthogonality. Two principal curvatures are given by
\begin{align}
\kk_1:=\frac{L}{E},\ \ \kk_2:=\frac{N}{G}.
\end{align}
Also, we know the relation
\begin{align}\label{coordinate 3}
\p_1\vn=\kk_1\p_1\vr,\ \ \p_2\vn=\kk_2\p_2\vr.
\end{align}
Hence, direct computation using \eqref{coordinate 1} and \eqref{coordinate 3} reveals that
\begin{align}
\Big(\left(\p_1\vr-\mu\p_1\vn\right)\times\left(\p_2\vr-\mu\p_2\vn\right)\Big)\cdot\vn=&(1-\kk_1\mu)(1-\kk_2\mu)(\p_1\vr\times\p_2\vr)\cdot\vn\label{coordinate 4}\\
=&(1-\kk_1\mu)(1-\kk_2\mu)P_1P_2,\no\\
\Big(\left(\p_1\vr-\mu\p_1\vn\right)\times\left(\p_2\vr-\mu\p_2\vn\right)\Big)\cdot\vv=&(1-\kk_1\mu)(1-\kk_2\mu)(\p_1\vr\times\p_2\vr)\cdot\vv\label{coordinate 5}\\
=&(1-\kk_1\mu)(1-\kk_2\mu)P_1P_2(\vv\cdot\vn),\no
\end{align}
and
\begin{align}
\Big(\left(\p_2\vr-\mu\p_2\vn\right)\times\vn\Big)\cdot\vv=&(1-\kk_2\mu)(\p_2\vr\times\vn)\cdot\vv=(1-\kk_2\mu)P_2(\vv\cdot\vt_1),\label{coordinate 6}\\
\Big(\left(\p_1\vr-\mu\p_1\vn\right)\times\vn\Big)\cdot\vv=&(1-\kk_1\mu)(\p_1\vr\times\vn)\cdot\vv=-(1-\kk_1\mu)P_1(\vv\cdot\vt_2).\label{coordinate 7}
\end{align}
Hence, plugging \eqref{coordinate 4}, \eqref{coordinate 5}, \eqref{coordinate 6} and \eqref{coordinate 7} into \eqref{coordinate 8}, we have the transport operator
\begin{align}\label{coordinate 11}
\vv\cdot\nx=-(\vv\cdot\vn)\frac{\p}{\p\mu}-\frac{\vv\cdot\vt_1}{P_1(\kk_1\mu-1)}\frac{\p}{\p\iota_1}-\frac{\vv\cdot\vt_2}{P_2(\kk_2\mu-1)}\frac{\p}{\p\iota_2}.
\end{align}
Therefore, under substitution $(x_1,x_2,x_3)\rt(\mu,\iota_1,\iota_2)$, the equation \eqref{small system} is transformed into
\begin{align}
\left\{
\begin{array}{l}\displaystyle
\e\bigg(-(\vv\cdot\vn)\frac{\p f^{\e}}{\p\mu}-\frac{\vv\cdot\vt_1}{P_1(\kk_1\mu-1)}\frac{\p f^{\e}}{\p\iota_1}-\frac{\vv\cdot\vt_2}{P_2(\kk_2\mu-1)}\frac{\p f^{\e}}{\p\iota_2}\bigg)+f^{\e}+\ll[f^{\e}]=\Gamma[f^{\e},f^{\e}]\ \ \text{in}\ \ \Omega\times\r^3,\\\rule{0ex}{2.0em}
f^{\e}(0,\iota_1,\iota_2,\vv)=\pe[f^{\e}](0,\iota_1,\iota_2,\vv)\ \ \text{for}\ \ \vv\cdot\vn<0.
\end{array}
\right.
\end{align}
\ \\
Substitution 2: Velocity Substitution.\\
Define the orthogonal velocity substitution for $\vvv:=(\va,\vb,\vc)$ as
\begin{align}
\left\{
\begin{array}{l}
-\vv\cdot\vn:=\va,\\
-\vv\cdot\vt_1:=\vb,\\
-\vv\cdot\vt_2:=\vc.
\end{array}
\right.
\end{align}
Then using chain rule, fundamental forms and \eqref{coordinate 3}, we have
\begin{align}
\frac{\p}{\p\iota_1}\rt&\frac{\p}{\p\iota_1}+\frac{\p}{\p\va}\frac{\p\va}{\p\iota_1}+\frac{\p}{\p\vb}\frac{\p\vb}{\p\iota_1}+\frac{\p}{\p\vc}\frac{\p\vc}{\p\iota_1}\\
=&\frac{\p}{\p\iota_1}-\kk_1P_1\vb\frac{\p}{\p\va}
+\bigg((\p_{11}\vr\cdot\vn)\frac{1}{P_1}\va+(\p_{11}\vr\cdot\p_2\vr)\frac{1}{P_1P_2}\vc\bigg)\frac{\p}{\p\vb}\no\\
&+\bigg((\p_{12}\vr\cdot\vn)\frac{1}{P_2}\va+(\p_{12}\vr\cdot\p_1\vr)\frac{1}{P_1P_2}\vb\bigg)\frac{\p}{\p\vc},\no\\
=&\frac{\p}{\p\iota_1}-\kk_1P_1\bigg(\vb\frac{\p}{\p\va}-\va\frac{\p}{\p\vb}\bigg)\no\\
&+(\p_{11}\vr\cdot\p_2\vr)\frac{1}{P_1P_2}\vc\frac{\p}{\p\vb}+(\p_{12}\vr\cdot\p_1\vr)\frac{1}{P_1P_2}\vb\frac{\p}{\p\vc},\no
\end{align}
and
\begin{align}
\frac{\p}{\p\iota_2}\rt&\frac{\p}{\p\iota_2}+\frac{\p}{\p\va}\frac{\p\va}{\p\iota_2}+\frac{\p}{\p\vb}\frac{\p\vb}{\p\iota_2}+\frac{\p}{\p\vc}\frac{\p\vc}{\p\iota_2}\\
=&\frac{\p}{\p\iota_2}-\kk_2P_2\vc\frac{\p}{\p\va}
+\bigg((\p_{12}\vr\cdot\vn)\frac{1}{P_1}\va+(\p_{12}\vr\cdot\p_2\vr)\frac{1}{P_1P_2}\vc\bigg)\frac{\p}{\p\vb}\no\\
&+\bigg((\p_{22}\vr\cdot\vn)\frac{1}{P_2}\va+(\p_{22}\vr\cdot\p_1\vr)\frac{1}{P_1P_2}\vb\bigg)\frac{\p}{\p\vc}\no\\
=&\frac{\p}{\p\iota_2}-\kk_2P_2\bigg(\vc\frac{\p}{\p\va}-\va\frac{\p}{\p\vc}\bigg)\no\\
&+(\p_{12}\vr\cdot\p_2\vr)\frac{1}{P_1P_2}\vc\frac{\p}{\p\vb}+(\p_{22}\vr\cdot\p_1\vr)\frac{1}{P_1P_2}\vb\frac{\p}{\p\vc}.\no
\end{align}
Here, we utilize $\p_{12}\vr\cdot\vn=M=0$ in the second fundamental form and $\p_{ii}\vr\cdot\vn=-\p_i\vr\cdot\p_i\vn=-\kk_i\abs{\p_i\vr}^2$ for $i=1,2$. Then the transport operator in \eqref{coordinate 11} becomes
\begin{align}
\vv\cdot\nx=&\va\frac{\p}{\p\mu}-\frac{1}{R_1-\mu}\bigg(\vb^2\frac{\p}{\p\va}-\va\vb\frac{\p}{\p\vb}\bigg)
-\frac{1}{R_2-\mu}\bigg(\vc^2\frac{\p}{\p\va}-\va\vc\frac{\p}{\p\vc}\bigg)\\
&-\frac{1}{P_1P_2}\Bigg(\frac{\p_{11}\vr\cdot\p_2\vr}{P_1(\kk_1\mu-1)}\vb\vc
+\frac{\p_{12}\vr\cdot\p_2\vr}{P_2(\kk_2\mu-1)}\vc^2\Bigg)\frac{\p}{\p\vb}\no\\
&-\frac{1}{P_1P_2}\Bigg(\frac{\p_{22}\vr\cdot\p_1\vr}{P_2(\kk_2\mu-1)}\vb\vc
+\frac{\p_{12}\vr\cdot\p_1\vr}{P_1(\kk_1\mu-1)}\vb^2\Bigg)\frac{\p}{\p\vc}\no\\
&-\bigg(\frac{\vb}{P_1(\kk_1-1\mu)}\frac{\p}{\p\iota_1}+\frac{\vc}{P_2(\kk_2\mu-1)}\frac{\p}{\p\iota_2}\bigg)\no,
\end{align}
where $R_1=\dfrac{1}{\kk_1}$ and $R_2=\dfrac{1}{\kk_2}$ represent the radius of principal curvature. Hence, under substitution $\vv\rt\vvv$,
the equation \eqref{small system} is transformed into
\begin{align}
\left\{
\begin{array}{l}\displaystyle
\e\va\dfrac{\p f^{\e}}{\p\mu}-\dfrac{\e}{R_1-\mu}\bigg(\vb^2\dfrac{\p f^{\e}}{\p\va}-\va\vb\dfrac{\p f^{\e}}{\p\vb}\bigg)
-\dfrac{\e}{R_2-\mu}\bigg(\vc^2\dfrac{\p f^{\e}}{\p\va}-\va\vc\dfrac{\p f^{\e}}{\p\vc}\bigg)\\\rule{0ex}{2.0em}
-\dfrac{\e}{P_1P_2}\Bigg(\dfrac{\p_{11}\vr\cdot\p_2\vr}{P_1(\kk_1\mu-1)}\vb\vc
+\dfrac{\p_{12}\vr\cdot\p_2\vr}{P_2(\kk_2\mu-1)}\vc^2\Bigg)\dfrac{\p f^{\e}}{\p\vb}\\\rule{0ex}{2.0em}
-\dfrac{\e}{P_1P_2}\Bigg(\dfrac{\p_{22}\vr\cdot\p_1\vr}{P_2(\kk_2\mu-1)}\vb\vc
+\dfrac{\p_{12}\vr\cdot\p_1\vr}{P_1(\kk_1\mu-1)}\vb^2\Bigg)\dfrac{\p f^{\e}}{\p\vc}\\\rule{0ex}{2.0em}
-\e\bigg(\dfrac{\vb}{P_1(\kk_1\mu-1)}\dfrac{\p f^{\e}}{\p\iota_1}+\dfrac{\vc}{P_2(\kk_2\mu-1)}\dfrac{\p f^{\e}}{\p\iota_2}\bigg)
+\ll[f^{\e}]=\Gamma[f^{\e},f^{\e}]\ \ \text{in}\ \ \Omega\times\r^3,\\\rule{0ex}{2.0em}
f^{\e}(0,\iota_1,\iota_2,\vvv)=\pe[f^{\e}](0,\iota_1,\iota_2,\vvv)\ \ \text{for}\
\ \va>0.
\end{array}
\right.
\end{align}
\ \\
Substitution 3: Scaling Substitution.\\
Finally, we define the scaled variable $\eta=\dfrac{\mu}{\e}$, which implies $\dfrac{\p}{\p\mu}=\dfrac{1}{\e}\dfrac{\p}{\p\eta}$. Then, under the substitution $\mu\rt\eta$, the equation \eqref{small system} is transformed into
\begin{align}\label{small system'}
\left\{
\begin{array}{l}\displaystyle
\va\dfrac{\p f^{\e}}{\p\eta}-\dfrac{\e}{R_1-\e\eta}\bigg(\vb^2\dfrac{\p f^{\e}}{\p\va}-\va\vb\dfrac{\p f^{\e}}{\p\vb}\bigg)
-\dfrac{\e}{R_2-\e\eta}\bigg(\vc^2\dfrac{\p f^{\e}}{\p\va}-\va\vc\dfrac{\p f^{\e}}{\p\vc}\bigg)\\\rule{0ex}{2.0em}
-\dfrac{\e}{P_1P_2}\Bigg(\dfrac{\p_{11}\vr\cdot\p_2\vr}{P_1(\e\kk_1\eta-1)}\vb\vc
+\dfrac{\p_{12}\vr\cdot\p_2\vr}{P_2(\e\kk_2\eta-1)}\vc^2\Bigg)\dfrac{\p f^{\e}}{\p\vb}\\\rule{0ex}{2.0em}
-\dfrac{\e}{P_1P_2}\Bigg(\dfrac{\p_{22}\vr\cdot\p_1\vr}{P_2(\e\kk_2\eta-1)}\vb\vc
+\dfrac{\p_{12}\vr\cdot\p_1\vr}{P_1(\e\kk_1\eta-1)}\vb^2\Bigg)\dfrac{\p f^{\e}}{\p\vc}\\\rule{0ex}{2.0em}
-\e\bigg(\dfrac{\vb}{P_1(\e\kk_1\eta-1)}\dfrac{\p f^{\e}}{\p\iota_1}+\dfrac{\vc}{P_2(\e\kk_2\eta-1)}\dfrac{\p f^{\e}}{\p\iota_2}\bigg)
+\ll[f^{\e}]=\Gamma[f^{\e},f^{\e}]\ \ \text{in}\ \ \Omega\times\r^3,\\\rule{0ex}{2.0em}
f^{\e}(0,\iota_1,\iota_2,\vvv)=\pe[f^{\e}](0,\iota_1,\iota_2,\vvv)\ \ \text{for}\
\ \va>0.
\end{array}
\right.
\end{align}

\subsection{Boundary Layer Expansion}

We define the boundary layer expansion:
\begin{align}\label{boundary layer expansion}
f^{\e}_{\text{bl}}(\eta,\iota_1,\iota_2,\vvv)=\sum_{k=1}^{2}\e^k\fb_k(\eta,\iota_1,\iota_2,\vvv),
\end{align}
where $\fb_k$ can be defined by comparing the order of $\e$ via
plugging \eqref{boundary layer expansion} into the equation
\eqref{small system'}. Thus, in a neighborhood of the boundary, we have
\begin{align}
\va\dfrac{\p\fb_1}{\p\eta}-\dfrac{\e}{R_1-\e\eta}\bigg(\vb^2\dfrac{\p\fb_1}{\p\va}-\va\vb\dfrac{\p\fb_1}{\p\vb}\bigg)
-\dfrac{\e}{R_2-\e\eta}\bigg(\vc^2\dfrac{\p\fb_1}{\p\va}-\va\vc\dfrac{\p\fb_1}{\p\vc}\bigg)+\ll[\fb_1]=&0,\label{expansion temp 6}\\
\va\dfrac{\p\fb_2}{\p\eta}-\dfrac{\e}{R_1-\e\eta}\bigg(\vb^2\dfrac{\p\fb_2}{\p\va}-\va\vb\dfrac{\p\fb_2}{\p\vb}\bigg)
-\dfrac{\e}{R_2-\e\eta}\bigg(\vc^2\dfrac{\p\fb_2}{\p\va}-\va\vc\dfrac{\p\fb_2}{\p\vc}\bigg)+\ll[\fb_2]=&Z,
\end{align}
where $Z=Z\left[\f_1,\fb_1,\dfrac{\p\fb_1}{\p\vb},\dfrac{\p\fb_1}{\p\vc},\dfrac{\p\fb_1}{\p\iota_1},\dfrac{\p\fb_1}{\p\iota_2}\right]$ as
\begin{align}
Z:=&2\Gamma[\f_1,\fb_1]+\Gamma[\fb_1,\fb_1]
+\dfrac{1}{P_1P_2}\Bigg(\dfrac{\p_{11}\vr\cdot\p_2\vr}{P_1(\e\kk_1\eta-1)}\vb\vc
+\dfrac{\p_{12}\vr\cdot\p_2\vr}{P_2(\e\kk_2\eta-1)}\vc^2\Bigg)\dfrac{\p\fb_1}{\p\vb}\\\rule{0ex}{2.0em}
&+\dfrac{1}{P_1P_2}\Bigg(\dfrac{\p_{22}\vr\cdot\p_1\vr}{P_2(\e\kk_2\eta-1)}\vb\vc
+\dfrac{\p_{12}\vr\cdot\p_1\vr}{P_1(\e\kk_1\eta-1)}\vb^2\Bigg)\dfrac{\p\fb_1}{\p\vc}
+\dfrac{\vb}{P_1(\e\kk_1\eta-1)}\dfrac{\p\fb_1}{\p\iota_1}+\dfrac{\vc}{P_2(\e\kk_2\eta-1)}\dfrac{\p\fb_1}{\p\iota_2}.\no
\end{align}

\subsection{Boundary Condition Expansion}\label{att section 04}

The bridge between the interior solution and boundary layer is the boundary condition. Define
\begin{align}
\pp[f](\vx_0,\vv):=\m^{\frac{1}{2}}(\vv)
\int_{\vuu\cdot\vn(\vx_0)>0}\m^{\frac{1}{2}}(\vuu)f(\vx_0,\vuu)\abs{\vuu\cdot\vn(\vx_0)}\ud{\vuu}.
\end{align}
Plugging the combined expansion from \eqref{interior expansion} and \eqref{boundary layer expansion}
\begin{align}
f^{\e}\sim\sum_{k=1}^{3}\e^k\f_k+\sum_{k=1}^{2}\e^k\fb_k
\end{align}
into the boundary condition \eqref{small system} and \eqref{att 05}, and comparing the order of $\e$, we obtain
\begin{align}
\f_1+\fb_1=&\pp[\f_1+\fb_1]+\m_1(\vx_0,\vv),\label{btt 1}\\
\f_2+\fb_2=&\pp[\f_2+\fb_2]+\m_1(\vx_0,\vv)
\int_{\vuu\cdot\vn(\vx_0)>0}\m^{\frac{1}{2}}(\vuu)(\f_1+\fb_1)\abs{\vuu\cdot\vn(\vx_0)}
\ud{\vuu}+\m_2(\vx_0,\vv).\label{btt 2}
\end{align}
In particular, we do not further expand the boundary layer, so we directly require
\begin{align}
\f_3=&\pp[\f_3]+\m_2(\vx_0,\vv)
\int_{\vuu\cdot\vn(\vx_0)>0}\m^{\frac{1}{2}}(\vuu)(\f_1+\fb_1)\abs{\vuu\cdot\vn(\vx_0)}\ud{\vuu}\\
&+\m_1(\vx_0,\vv)
\int_{\vuu\cdot\vn(\vx_0)>0}\m^{\frac{1}{2}}(\vuu)(\f_2+\fb_2)\abs{\vuu\cdot\vn(\vx_0)}\ud{\vuu}
+\m_3(\vx_0,\vv).\no
\end{align}
These are the boundary conditions $\f_k$ and $\fb_k$ need to satisfy.

\subsection{Matching Procedure}\label{att section 03}

Define the length of boundary layer in rescaled variable $L:=\e^{-\frac{1}{2}}$.
Also, denote $\rr[\va,\vb,\vc]=(-\va,\vb,\vc)$. \\
\ \\
Step 1: Construction of $\f_1$ and $\fb_1$.\\
Based on Section \ref{att section 1}, we know $F_1=A_1$ since there is no contribution of $B_1$ and $C_1$. Considering the first-order boundary Maxwellian expansion \eqref{att 06} and reorganizing \eqref{att 07}--\eqref{att 08}, we define
\begin{align}
\f_1=&\m^{\frac{1}{2}}\left(\rh_{1}+\vu_{1}\cdot\vv+\th_{1}\frac{\abs{\vv}^2-3}{2}\right),
\end{align}
where $(\rh_1,\vu_1,\th_1)$ satisfies the Navier-Stokes-Fourier system
\begin{align}\label{interior 1}
\left\{
\begin{array}{l}
\vu_1\cdot\nx\vu_1 -\gamma_1\dx\vu_1 +\nx p_2 =0,\\\rule{0ex}{1.0em}
\nx\cdot\vu_1 =0,\\\rule{0ex}{1.0em}
\vu_1 \cdot\nx\th_1 -\gamma_2\dx\th_1 =0,
\end{array}
\right.
\end{align}
with the boundary condition
\begin{align}
\rh_1(\vx_0)=\rh_{\bb,1}(\vx_0)+M_1(\vx_0),\quad \vu_1(\vx_0)=\vu_{\bb,1}(\vx_0),\quad\th_1(\vx_0)=\th_{\bb,1}(\vx_0).
\end{align}
Here $M_1(\vx_0)$ is chosen such that the Boussinesq relation
\begin{align}
\nx(\rh_1 +\th_1 )=0
\end{align}
is satisfied (which is part of \eqref{att 07}--\eqref{att 08}). Note that now $M_1$ still has one dimension of freedom and is finally fully determined by the normalization condition
\begin{align}
\iint_{\Omega\times\r^3}\f_1(\vx,\vv)\mh(\vv)\ud{\vv}\ud{\vx}=0.
\end{align}
which is a requirement from \eqref{small normalization}.

On the other hand, based on \eqref{boundary compatibility'}, we naturally obtain
\begin{align}\label{att 09}
\pp[\f_1]=M_1\mu^{\frac{1}{2}},
\end{align}
which means
\begin{align}
\f_1=\pp[\f_1]+\m_1\ \ \text{on}\ \ \p\Omega.
\end{align}
Therefore, compared with \eqref{btt 1}, since $\f_1$ already satisfies the boundary condition, it is not necessary to introduce the boundary layer at this order and we simply take $\fb_1=0$.\\
\ \\
Step 2: Construction of $\f_2$ and $\fb_2$.\\
Define $\f_2=A_2+B_2+C_2$, where $B_2$ and $C_2$ can be uniquely determined following Section \ref{att section 1}, and
\begin{align}
A_2=&\m^{\frac{1}{2}}\left(\rh_{2}+\vu_{2}\cdot\vv+\th_{2}\frac{\abs{\vv}^2-3}{2}\right),
\end{align}
satisfying a fluid-type equation (see \cite[Page 92]{Sone2002})
\begin{align}\label{interior 1'}
\\
\left\{
\begin{array}{l}
\nx\Big(p_2-(\rh_2 +\th_2 +\rh_1 \th_1 )\Big)=0,\\\rule{0ex}{1.0em}
\vu_1\cdot\nx\vu_2+(\rh_1\vu_1+\vu_2)\cdot\nx\vu_1-\gamma_1\dx\vu_2 +\nx p_3 =-\gamma_2\nx\cdot\dx\th_1-\gamma_4\nx\cdot\bigg(\th_1\Big(\nx\vu_1+(\nx\vu)^T\Big)\bigg),\\\rule{0ex}{1.0em}
\nx\cdot\vu =-\vu_1\cdot\nx\rh_1,\\\rule{0ex}{1.0em}
\vu_1 \cdot\nx\th_2+(\rh_1\vu_1+\vu_2)\cdot\nx\th_1-\vu_1\cdot\nx p_2 =\gamma_1\Big(\nx\vu_1+(\nx\vu)^T\Big)^2+\dx\Big(\gamma_2\th_2+\gamma_5\th_1^2\Big),\no
\end{array}
\right.
\end{align}
where $\gamma_3,\gamma_4,\gamma_5$ are constants. Now $\f_2$ does not satisfy \eqref{btt 2} alone, so we have to introduce boundary layer. 
Let $\fb_2$ satisfy the $\e$-Milne problem with geometric correction
\begin{align}\label{att 3}
\left\{
\begin{array}{l}\displaystyle
\va\dfrac{\p\fb_2}{\p\eta}-\dfrac{\e}{R_1-\e\eta}\bigg(\vb^2\dfrac{\p\fb_2}{\p\va}-\va\vb\dfrac{\p\fb_2}{\p\vb}\bigg)
-\dfrac{\e}{R_2-\e\eta}\bigg(\vc^2\dfrac{\p\fb_2}{\p\va}-\va\vc\dfrac{\p\fb_2}{\p\vc}\bigg)+\ll[\fb_2]
=0,\\\rule{0ex}{2.0em}
\fb_2(0,\iota_1,\iota_2,\vvv)=h(\iota_1,\iota_2,,\vvv)-\tilde h(\iota_1,\iota_2,\vvv)\ \
\text{for}\ \ \va>0,\\\rule{0ex}{2.0em}
\displaystyle\fb_2(L,\iota_1,\iota_2,,\vvv)
=\fb_2(L,\iota_1,\iota_2,,\rr[\vvv]),
\end{array}
\right.
\end{align}
with the in-flow boundary data
\begin{align}\label{btt 3}
h(\iota_1,\iota_2,\vvv)=&M_1\m_1(\vx_0,\vv)+
\m_2(\vx_0,\vv)-\bigg((B_2+C_2)-\pp[B_2+C_2]\bigg).
\end{align}
Using \eqref{boundary compatibility}, considering $B_2$ and $C_2$ given in Section \ref{att section 1} and using symmetry, we may directly check that
\begin{align}\label{att 10}
&\int_{\va>0}\mh(\vvv)h(\iota_1,\iota_2,\vvv)\abs{\va}\ud\vvv\\
=&-\int_{\vv\cdot\vn(\vx_0)<0}(B_2+C_2)(\vx_0)\Big(\vv\cdot\vn(\vx_0)\Big)\ud\vv+\int_{\vv\cdot\vn(\vx_0)<0}\pp[B_2+C_2](\vx_0)\Big(\vv\cdot\vn(\vx_0)\Big)\ud\vv\no\\
=&-\int_{\r^3}(B_2+C_2)(\vx_0)\Big(\vv\cdot\vn(\vx_0)\Big)\ud\vv=0.\no
\end{align}
Here the last equality holds based on the construction in Section \ref{att section 1}. The $C_2$ integral vanishes due to orthogonality. The $B_2$ integral vanishes due to symmetry and $\vu_{\bb,1}\cdot\vn=0$. Based on Theorem \ref{Milne theorem 3} and Theorem \ref{Milne theorem 4}, there exists a unique
\begin{align}
\tilde h(\iota_1,\iota_2,\vvv)=\m^{\frac{1}{2}}\sum_{k=0}^4\tilde D_k(\iota_1,\iota_2)\ee_k,
\end{align}
such that \eqref{att 3} is well-posed and the solution decays exponentially fast (here $\ee_k$ with $k=0,1,2,3,4$ form a basis of null space $\nk$ of $\ll$). In particular, $\tilde D_1=0$.
Then we further require that $A_2$ satisfies the boundary condition
\begin{align}\label{att 1}
A_2(\vx_0,\vv)=\tilde h(\iota_1,\iota_2,\vvv)+M_2(\vx_0)\m^{\frac{1}{2}}(\vv).
\end{align}
Here $\vx_0$ corresponds to $(\iota_1,\iota_2)$ and $\vv$ corresponds to $\vvv$, based on substitution in Section \ref{att section 2}. Here, the constant $M_2(\vx_0)$ is chosen to enforce the Boussinesq relation
\begin{align}
p_2 -(\rh_2 +\th_2 +\rh_1 \th_1 )=&\text{constant},
\end{align}
where $p_2$ is the pressure solved from \eqref{interior 1}. Similar to the construction of $\f_1$, due to \eqref{small normalization}, we can choose this constant to satisfy the normalization condition
\begin{align}
\iint_{\Omega\times\r^3}(\f_2+\fb_2)(\vx,\vv)\m^{\frac{1}{2}}(\vv)\ud{\vv}\ud{\vx}=0.
\end{align}
Also, based on \eqref{att 09}, $\fb_1=0$, \eqref{att 1} and \eqref{att 3}, we have
\begin{align}
&A_2+\fb_2=M_2\m^{\frac{1}{2}}+h\\
=&M_2\m^{\frac{1}{2}}+\m_1
\int_{\vuu\cdot\vn(\vx_0)>0}\m^{\frac{1}{2}}(\vuu)(\f_1+\fb_1)\abs{\vuu\cdot\vn(\vx_0)}
\ud{\vuu}+\m_2-\bigg((B_2+C_2)-\pp[B_2+C_2]\bigg).\no
\end{align}
Comparing this with the desired boundary expansion \eqref{btt 2}, i.e.
\begin{align}
A_2+B_2+C_2+\fb_2=\pp[A_2+B_2+C_2+\fb_2]+\m_1
\int_{\vuu\cdot\vn(\vx_0)>0}\m^{\frac{1}{2}}(\vuu)(\f_1+\fb_1)\abs{\vuu\cdot\vn(\vx_0)}
\ud{\vuu}+\m_2,
\end{align}
we only need to verify that
\begin{align}
\pp[A_2+\fb_2]=M_2\m^{\frac{1}{2}}.
\end{align}
Based on Theorem \ref{Milne theorem 3}, the equation \eqref{att 3} implies the zero mass-flux condition of $\fb_2$ as
\begin{align}\label{att 2}
\int_{\r^3}\m^{\frac{1}{2}}(\vuu)\fb_2(\vx,\vuu)(\vuu\cdot\vn)\ud{\vuu}=0.
\end{align}
Since $\m_1$ and $\m_2$ satisfy \eqref{boundary compatibility}, based on \eqref{att 1}, we have
\begin{align}
\pp[A_2+\fb_2]
=&
\m^{\frac{1}{2}}\int_{\vuu\cdot\vn>0}\m^{\frac{1}{2}}(\vuu)A_2(\vx,\vuu)(\vuu\cdot\vn)\ud{\vuu}
+\m^{\frac{1}{2}}\int_{\vuu\cdot\vn>0}\m^{\frac{1}{2}}(\vuu)\fb_2(\vx,\vuu)(\vuu\cdot\vn)\ud{\vuu}\\
=&
M_2\m^{\frac{1}{2}}+\m^{\frac{1}{2}}\int_{\vuu\cdot\vn>0}\m^{\frac{1}{2}}(\vuu)\tilde h(\vx,\vuu)(\vuu\cdot\vn)\ud{\vuu}+
\m^{\frac{1}{2}}\int_{\vuu\cdot\vn>0}\m^{\frac{1}{2}}(\vuu)\fb_2(\vx,\vuu)(\vuu\cdot\vn)\ud{\vuu}.\no
\end{align}
Using \eqref{att 2} and \eqref{att 3}, noting $\va>0$ represents in-flow boundary, we know
\begin{align}
\pp[A_2+\fb_2]
=&M_2\m^{\frac{1}{2}}+\m^{\frac{1}{2}}\int_{\vuu\cdot\vn>0}\m^{\frac{1}{2}}(\vuu)\tilde h(\vx,\vuu)(\vuu\cdot\vn)\ud{\vuu}-
\m^{\frac{1}{2}}\int_{\vuu\cdot\vn<0}\m^{\frac{1}{2}}(\vuu)\fb_2(\vx,\vuu)(\vuu\cdot\vn)\ud{\vuu}\\
=&
M_2\m^{\frac{1}{2}}+\m^{\frac{1}{2}}\int_{\vuu\cdot\vn>0}\m^{\frac{1}{2}}(\vuu)\tilde h(\vx,\vuu)(\vuu\cdot\vn)\ud{\vuu}-
\m^{\frac{1}{2}}\int_{\vuu\cdot\vn<0}\m^{\frac{1}{2}}(\vuu)(h-\tilde h)(\vx,\vuu)(\vuu\cdot\vn)\ud{\vuu}\no.
\end{align}
Then direct computation reveals that
\begin{align}
\pp[A_2+\fb_2]
=&M_2\m^{\frac{1}{2}}+\m^{\frac{1}{2}}\int_{\vuu\cdot\vn>0}\m^{\frac{1}{2}}(\vuu)\tilde h(\vx,\vuu)(\vuu\cdot\vn)\ud{\vuu}-
\m^{\frac{1}{2}}\int_{\vuu\cdot\vn<0}\m^{\frac{1}{2}}(\vuu)h(\vx,\vuu)(\vuu\cdot\vn)\ud{\vuu}\\
&+
\m^{\frac{1}{2}}\int_{\vuu\cdot\vn<0}\m^{\frac{1}{2}}(\vuu)\tilde h(\vx,\vuu)(\vuu\cdot\vn)\ud{\vuu}\no\\
=&
M_2\m^{\frac{1}{2}}+\m^{\frac{1}{2}}\int_{\r^3}\m^{\frac{1}{2}}(\vuu)\tilde h(\vx,\vuu)(\vuu\cdot\vn)\ud{\vuu}-
\m^{\frac{1}{2}}\int_{\vuu\cdot\vn<0}\m^{\frac{1}{2}}(\vuu)h(\vx,\vuu)(\vuu\cdot\vn)\ud{\vuu}\no.
\end{align}
Finally, using \eqref{btt 3}, \eqref{boundary compatibility}, \eqref{att 10} and $\tilde D_1=0$, we obtain
\begin{align}
\pp[A_2+\fb_2]
=&M_2\m^{\frac{1}{2}}+\tilde D_1-0=M_2\m^{\frac{1}{2}}.
\end{align}
$F_3$ can be defined in a similar fashion which satisfies an even more complicated fluid-type system (see \cite[Page 92]{Sone2002}). We skip the details and simply note that the well-posedness is always guaranteed.


\newpage

\section{Remainder Estimates}

We consider the linearized stationary Boltzmann equation
\begin{align}\label{linear steady}
\left\{
\begin{array}{l}
\e\vv\cdot\nx f+\ll[f]=S(\vx,\vv)\ \ \text{in}\ \ \Omega\times\r^3,\\\rule{0ex}{1.5em}
f(\vx_0,\vv)=\pp[f](\vx_0,\vv)+h(\vx_0,\vv)\ \ \text{for}\ \ \vx_0\in\p\Omega\ \
\text{and}\ \ \vv\cdot\vn<0,
\end{array}
\right.
\end{align}
where
\begin{align}
\pp[f](\vx_0,\vv)=\m^{\frac{1}{2}}(\vv)
\int_{\gamma_+}f(\vx_0,\vv)\mh(\vv)\ud{\gamma}.
\end{align}
The data $S$ and $h$ satisfy the compatibility condition
\begin{align}\label{linear steady compatibility}
\iint_{\Omega\times\r^3}S(\vx,\vv)\m^{\frac{1}{2}}(\vv)\ud{\vv}\ud{\vx}+\int_{\gamma_-}h(\vx,\vv)\m^{\frac{1}{2}}(\vv)\ud{\gamma}=0.
\end{align}
It is easy to see if $f$ is a solution to \eqref{linear steady}, then $f+C\m^{\frac{1}{2}}$ is also a solution for arbitrary $C\in\r$. Hence, to guarantee uniqueness, the solution should satisfy the normalization condition
\begin{align}\label{linear steady normalization}
\iint_{\Omega\times\r^3}f(\vx,\vv)\m^{\frac{1}{2}}(\vv)\ud{\vv}\ud{\vx}=&0.
\end{align}
Our analysis is based on the ideas in \cite{Esposito.Guo.Kim.Marra2013}, \cite{Guo2010}, \cite{AA004} and \cite{AA013}. Since proof of the well-posedness of \eqref{linear steady} is standard, we will focus on the a priori estimates here.

\subsection{Preliminaries}

We first introduce the well-known micro-macro decomposition. Define $\pk$ as the orthogonal projection onto the null space of $\ll$:
\begin{align}\label{ktt 04}
\pk[f]:=\m^{\frac{1}{2}}(\vv)\bigg(a_f(\vx)+\vv\cdot\vbb_f(\vx)+\frac{\abs{\vv}^2-3}{2}c_f(\vx)\bigg)\in\nk,
\end{align}
where $a_f$, $\vbb_f$ and $c_f$ are coefficients. When there is no confusion, we will simply write $a,\vbb, c$. Definitely, $\ll\Big[\pk[f]\Big]=0$. Then the operator $\ik-\pk$ is naturally
\begin{align}
(\ik-\pk)[f]:=f-\pk[f],
\end{align}
which satisfies $(\ik-\pk)[f]\in\nk^{\perp}$, i.e. $\ll[f]=\ll\Big[(\ik-\pk)[f]\Big]$.

\begin{lemma}\label{ktt 01}
The linearized collision operator $\ll=\nu I-K$ defined in \eqref{att 11} is self-adjoint in $L^2$. It satisfies
\begin{align}
&\br{\vv}\ls\nu(\vv)\ls\br{\vv},\\
&\br{f,\ll[f]}(\vx)=\br{(\ik-\pk)[f],\ll\Big[(\ik-\pk)[f]\Big]}(\vx),\\
&\unm{(\ik-\pk)[f(\vx)]}^2\ls \br{f,\ll[f]}(\vx)\ls \unm{(\ik-\pk)[f(\vx)]}^2.
\end{align}
\end{lemma}
\begin{proof}
These are standard properties of $\ll$. See \cite[Chapter 3]{Glassey1996} and \cite[Lemma 3]{Guo2010}.
\end{proof}

\begin{lemma}\label{ktt 02}
For $0<\d<<1$, define the near-grazing set of $\gamma_{\pm}$:
\begin{align}
\gamma_{\pm}^{\d}:=\left\{(\vx,\vv)\in\gamma_{\pm}:
\abs{\vn(\vx)\cdot\vv}\leq\d\ \text{or}\ \abs{\vv}\geq\frac{1}{\d}\ \text{or}\
\abs{\vv}\leq\d\right\}.
\end{align}
Then
\begin{align}
\pnm{f\id_{\gamma_{\pm}\backslash\gamma_{\pm}^{\d}}}{\gamma}{1}\leq
C(\delta)\bigg(\pnnm{f}{1}+\pnnm{\vv\cdot\nx f}{1}\bigg).
\end{align}
Here $\id$ denotes the indicator function.
\end{lemma}
\begin{proof}
See \cite[Lemma 2.1]{Esposito.Guo.Kim.Marra2013}.
\end{proof}

\begin{lemma}[Time-Independent Green's Identity]\label{ktt 03}
Assume $f(\vx,\vv),\ g(\vx,\vv)\in L^2(\Omega\times\r^2)$ and
$\vv\cdot\nx f,\ \vv\cdot\nx g\in L^2(\Omega\times\r^2)$ with $f,\
g\in L^2(\gamma)$. Then
\begin{align}
\iint_{\Omega\times\r^2}\bigg((\vv\cdot\nx f)g+(\vv\cdot\nx
g)f\bigg)\ud{\vx}\ud{\vv}=\iint_{\gamma_+}fg\ud{\gamma}-\iint_{\gamma_-}fg\ud{\gamma}.
\end{align}
\end{lemma}
\begin{proof}
See \cite[Lemma 2.2]{Esposito.Guo.Kim.Marra2013}.
\end{proof}

\begin{lemma}\label{Regularity lemma 1}
For Boltzmann collision operator $k$, we have
\begin{align}
\abs{k(\vuu,\vvv)}\ls \left(\abs{\vuu-\vvv}+\frac{1}{\abs{\vuu-\vvv}}\right)\ue^{-\frac{1}{8}\abs{\vuu-\vvv}^2-\frac{1}{8}\frac{\abs{\abs{\vuu}^2-\abs{\vvv}^2}^2}{\abs{\vuu-\vvv}^2}}.
\end{align}
\end{lemma}
\begin{proof}
See \cite[Lemma 3]{Guo2010}.
\end{proof}

\begin{lemma}\label{Regularity lemma 1'}
Let $0\leq\varrho< \dfrac{1}{4}$ and $\vth\geq0$. Then for $\d>0$ sufficiently small and any $\vvv\in\r^3$,
\begin{align}
\int_{\r^3}\ue^{\d\abs{\vuu-\vvv}^2}\abs{k(\vuu,\vvv)}
\frac{\br{\vvv}^{\vth}\ue^{\vrh\abs{\vvv}^2}}{\br{\vuu}^{\vth}\ue^{\vrh\abs{\vuu}^2}}\ud{\vuu}
\ls \frac{1}{\br{\vvv}}.
\end{align}
\end{lemma}
\begin{proof}
See \cite[Lemma 3]{Guo2010}.
\end{proof}

\subsection{$L^{2m}$ Estimates}

Throughout this section, we consider $\dfrac{3}{2}< m<3$. Let $o(1)$ denote a sufficiently small constant.

\begin{lemma}\label{ktt lemma 1}
The solution $f(\vx,\vv)$ to the equation \eqref{linear steady} satisfies
\begin{align}\label{ktt 39}
\e\pnnm{\pk[f]}{2m}\ls\e\pnm{(1-\pp)[f]}{\gamma_+}{\frac{4m}{3}}+\tnnm{(\ik-\pk)[f]}+\e\pnnm{(\ik-\pk)[f]}{2m}+\tnnm{\snn S}+\e\pnm{h}{\gamma_-}{\frac{4m}{3}}.
\end{align}
\end{lemma}
\begin{proof}
Apply Green's identity in Lemma \ref{ktt 03} to the equation \eqref{linear steady}. Then for any $\psi\in L^2(\Omega\times\r^3)$ satisfying $\vv\cdot\nx\psi\in L^2(\Omega\times\r^3)$ and $\psi\in L^2(\gamma)$, we have
\begin{align}\label{ktt 05}
&\e\iint_{\gamma_+}f\psi\ud{\gamma}-\e\iint_{\gamma_-}f\psi\ud{\gamma}
-\e\iint_{\Omega\times\r^3}(\vv\cdot\nx\psi)f=-\iint_{\Omega\times\r^3}\psi\ll\Big[(\ik-\pk)[f]\Big]+\iint_{\Omega\times\r^3}S\psi.
\end{align}
Consider \eqref{ktt 04}, our goal is to choose a particular test function $\psi$ to estimate $a$, $\vbb$ and $c$.\\
\ \\
Step 1: Estimates of $c$.\\
We choose the test function
\begin{align}\label{ktt 06}
\psi=\psi_c=\m^{\frac{1}{2}}(\vv)\left(\abs{\vv}^2-\beta_c\right)\Big(\vv\cdot\nx\phi_c(\vx)\Big),
\end{align}
where
\begin{align}
\left\{
\begin{array}{l}
-\dx\phi_c=c\abs{c}^{2m-2}(\vx)\ \ \text{in}\ \ \Omega,\\\rule{0ex}{1.5em}
\phi_c=0\ \ \text{on}\ \ \p\Omega,
\end{array}
\right.
\end{align}
and $\beta_c\in\r$ will be determined later. Based on the standard elliptic estimates in \cite{Krylov2008}, we have
\begin{align}
\onnm{\phi_c}{W^{2,\frac{2m}{2m-1}}}\ls \onnm{\abs{c}^{2m-1}}{L^{\frac{2m}{2m-1}}}\ls \onnm{c}{L^{2m}}^{2m-1}.
\end{align}
Hence, by Sobolev embedding theorem, we know
\begin{align}
&\tnnm{\psi_c}\ls\onnm{\phi_c}{H^1}\ls\onnm{\phi_c}{W^{2,\frac{2m}{2m-1}}}\ls\onnm{c}{L^{2m}}^{2m-1},\label{ktt 07}\\
&\onnm{\phi_c}{W^{1,\frac{2m}{2m-1}}}\ls\onnm{\phi_c}{W^{2,\frac{2m}{2m-1}}}\ls\onnm{c}{L^{2m}}^{2m-1}.\label{ktt 08}
\end{align}
Also, for $1\leq m\leq 3$, using Sobolev embedding theorem and trace estimates, we have
\begin{align}\label{ktt 13}
\onm{\nx\phi_c}{L^{\frac{4m}{4m-3}}}\ls\onm{\nx\phi_c}{W^{\frac{1}{2m},\frac{2m}{2m-1}}}\ls\onnm{\nx\phi_c}{W^{1,\frac{2m}{2m-1}}}\ls \onnm{\phi_c}{W^{2,\frac{2m}{2m-1}}}\ls\onnm{c}{L^{2m}}^{2m-1}.
\end{align}
We first consider the right-hand side (RHS) of \eqref{ktt 05}. With the choice of \eqref{ktt 06} and H\"older's inequality, using \eqref{ktt 07} and Lemma \ref{ktt 01}, we have
\begin{align}
\abs{\iint_{\Omega\times\r^3}\psi_c\ll\Big[(\ik-\pk)[f]\Big]}&=\abs{\iint_{\Omega\times\r^3}\ll[\psi_c](\ik-\pk)[f]}\ls \tnnm{\ll[\psi_c]}\tnnm{(\ik-\pk)[f]}\\
&\ls \tnnm{\psi_c}\tnnm{(\ik-\pk)[f]}\ls \onnm{c}{L^{2m}}^{2m-1}\tnnm{(\ik-\pk)[f]},\no
\end{align}
and
\begin{align}
\abs{\iint_{\Omega\times\r^3}S\psi_c}\ls \tnnm{\psi_c}\tnnm{\snn S}\ls \onnm{c}{L^{2m}}^{2m-1}\tnnm{\snn S}.
\end{align}
Therefore, we know
\begin{align}\label{ktt 09}
\text{RHS}\ls&\bigg(\tnnm{(\ik-\pk)[f]}+\tnnm{\snn S}\bigg)\onnm{c}{L^{2m}}^{2m-1}.
\end{align}
Then we turn to the left-hand side (LHS) of \eqref{ktt 05}. Based on \eqref{linear steady}, note the decomposition
\begin{align}
f\big|_{\gamma}&=\id_{\gamma_+}f+\id_{\gamma_-}\pp[f]+\id_{\gamma_-}h=\id_{\gamma}\pp[f]+\id_{\gamma_+}(1-\pp)[f]+\id_{\gamma_-}h.
\end{align}
Then, we will choose $\beta_c$ such that
\begin{align}\label{ktt 10}
\int_{\r^3}\m^{\frac{1}{2}}(\vv)\left(\abs{\vv}^2-\beta_c\right)v_i^2\ud{\vv}=0\ \ \text{for}\ \ i=1,2,3,
\end{align}
which, combined with oddness, implies
\begin{align}
\iint_{\gamma_+}\mh\psi_c\ud{\gamma}-\iint_{\gamma_-}\mh\psi_c\ud{\gamma}=0.
\end{align}
Hence, the boundary term on the LHS of \eqref{ktt 05} can be simplified as
\begin{align}\label{ktt 12}
&\e\iint_{\gamma_+}f\psi_c\ud{\gamma}-\e\iint_{\gamma_-}f\psi_c\ud{\gamma}\\
=&\bigg(\iint_{\gamma_+}\pp[f]\psi_c\ud{\gamma}-\iint_{\gamma_-}\pp[f]\psi_c\ud{\gamma}\bigg)
+\e\iint_{\gamma_+}(1-\pp)[f]\psi_c\ud{\gamma}-\e\iint_{\gamma_-}h\psi_c\ud{\gamma}\no\\
=&\e\iint_{\gamma_+}(1-\pp)[f]\psi_c\ud{\gamma}-\e\iint_{\gamma_-}h\psi_c\ud{\gamma}.\no
\end{align}
Applying H\"older's inequality and \eqref{ktt 13} to \eqref{ktt 12}, we have
\begin{align}\label{ktt 14}
\abs{\e\iint_{\gamma_+}f\psi_c\ud{\gamma}-\e\iint_{\gamma_-}f\psi_c\ud{\gamma}}
&\ls\e\bigg(\pnm{(1-\pp)[f]}{\gamma_+}{\frac{4m}{3}}+\pnm{h}{\gamma_-}{\frac{4m}{3}}\bigg)\pnm{\psi_c}{\gamma}{\frac{4m}{4m-3}}\\
&\ls\e\bigg(\pnm{(1-\pp)[f]}{\gamma_+}{\frac{4m}{3}}+\pnm{h}{\gamma_-}{\frac{4m}{3}}\bigg)\onm{\nx\phi_c}{L^{\frac{4m}{4m-3}}}\no\\
&\ls\e\bigg(\pnm{(1-\pp)[f]}{\gamma_+}{\frac{4m}{3}}+\pnm{h}{\gamma_-}{\frac{4m}{3}}\bigg)\onnm{c}{L^{2m}}^{2m-1}.\no
\end{align}
For the bulk term on the LHS of \eqref{ktt 05}, noting the decomposition
\begin{align}
f=\pk[f]+(\ik-\pk)[f]=\mh\bigg(a+\vv\cdot\vbb+\frac{\abs{\vv}^2-3}{2}c\bigg)+(\ik-\pk)[f],
\end{align}
we have
\begin{align}\label{ktt 15}
-\e\iint_{\Omega\times\r^3}(\vv\cdot\nx\psi_c)f=&-\e\iint_{\Omega\times\r^3}(\vv\cdot\nx\psi_c)\mh(\vv)\bigg(a+\vv\cdot\vbb+\frac{\abs{\vv}^2-3}{2}c\bigg)\\
&-\e\iint_{\Omega\times\r^3}(\vv\cdot\nx\psi_c)(\ik-\pk)[f].\no
\end{align}
Considering \eqref{ktt 06}, we directly compute
\begin{align}\label{ktt 16}
\vv\cdot\nx\psi_c=\m^{\frac{1}{2}}(\vv)\left(\abs{\vv}^2-\beta_c\right)\bigg(\sum_{i,j=1}^3v_iv_j\p_i\p_j\phi_c\bigg).
\end{align}
Due to oddness, the $\vbb$ contribution in $\eqref{ktt 15}$ vanishes. \eqref{ktt 10} implies that the $a$ contribution in $\eqref{ktt 15}$ also vanishes. For the $c$ contribution, using \eqref{ktt 16} and oddness, with
\begin{align}
\int_{\r^3}\m(\vv)\abs{v_i}^2\left(\abs{\vv}^2-\beta_c\right)\frac{\abs{\vv}^2-3}{2}\ud{\vv}\neq0\ \ \text{for}\ \ i=1,2,3,
\end{align}
we know
\begin{align}\label{ktt 17}
-\e\iint_{\Omega\times\r^3}(\vv\cdot\nx\psi_c)\mh(\vv)\frac{\abs{\vv}^2-3}{2}c=-\int_{\Omega}c\dx\phi_c=\onnm{c}{L^{2m}}^{2m}.
\end{align}
Also, H\"older's inequality and \eqref{ktt 08} yield
\begin{align}\label{ktt 18}
\abs{\e\iint_{\Omega\times\r^3}(\vv\cdot\nx\psi_c)(\ik-\pk)[f]}&\ls\e\pnnm{\vv\cdot\nx\psi_c}{\frac{2m}{2m-1}}\pnnm{(\ik-\pk)[f]}{2m}\\
&\ls\e\onnm{\phi_c}{W^{2,\frac{2m}{2m-1}}}\pnnm{(\ik-\pk)[f]}{2m}\ls \e\onnm{c}{L^{2m}}^{2m-1}\pnnm{(\ik-\pk)[f]}{2m}.\no
\end{align}
Collecting \eqref{ktt 09}, \eqref{ktt 14}, \eqref{ktt 17}, \eqref{ktt 18}, and cancelling $\onnm{c}{L^{2m}}^{2m-1}$, we have
\begin{align}\label{ktt 19}
&\e\onnm{c}{L^{2m}}\ls
\e\pnm{(1-\pp)[f]}{\gamma_+}{\frac{4m}{3}}+\tnnm{(\ik-\pk)[f]}+\e\pnnm{(\ik-\pk)[f]}{2m}+\tnnm{\snn S}+\e\pnm{h}{\gamma_-}{\frac{4m}{3}}.
\end{align}
\ \\
Step 2: Estimates of $\vbb$.\\
We further divide this step into several sub-steps:\\
\ \\
Sub-Step 2.1: Estimates of $\bigg(\p_{i}\p_j\dx^{-1}\Big(b_j\abs{b_j}^{2m-2}\Big)\bigg)b_i$ for $i,j=1,2,3$.\\
Let $\vbb=(b_1,b_2,b_3)$. We choose the test functions for $i,j=1,2,3$,
\begin{align}\label{ktt 20}
\psi=\psi_{b,i,j}=\m^{\frac{1}{2}}(\vv)\left(v_i^2-\beta_{b,i,j}\right)\p_j\phi_{b,j},
\end{align}
where
\begin{align}
\left\{
\begin{array}{l}
-\dx\phi_{b,j}=b_j\abs{b_j}^{2m-2}(\vx)\ \ \text{in}\ \ \Omega,\\\rule{0ex}{1.5em}
\phi_{b,j}=0\ \ \text{on}\ \ \p\Omega,
\end{array}
\right.
\end{align}
and $\beta_{b,i,j}\in\r$ will be determined later. This is very similar to Step 1. We can recover the elliptic estimates and trace estimates as in \eqref{ktt 07}, \eqref{ktt 08} and \eqref{ktt 13}. With the choice of \eqref{ktt 20}, the right-hand side (RHS) of \eqref{ktt 05} is bounded by
\begin{align}\label{ktt 21}
\text{RHS}\ls \bigg(\tnnm{(\ik-\pk)[f]}+\tnnm{\snn S}\bigg)\onnm{\vbb}{L^{2m}}^{2m-1}.
\end{align}
We will choose $\beta_b$ such that
\begin{align}\label{ktt 24}
\int_{\r^3}\m(\vv)\left(\abs{v_i}^2-\beta_{b,i,j}\right)\ud{\vv}=0,
\end{align}
which means for the boundary term on the left-hand side of \eqref{ktt 05}, there is no $\pp[f]$ contribution. We may recover estimates as \eqref{ktt 14}
\begin{align}\label{ktt 22}
\abs{\e\iint_{\gamma_+}f\psi_{b,i,j}\ud{\gamma}-\e\iint_{\gamma_-}f\psi_{b,i,j}\ud{\gamma}}
&\ls\e\bigg(\pnm{(1-\pp)[f]}{\gamma_+}{\frac{4m}{3}}+\pnm{h}{\gamma_-}{\frac{4m}{3}}\bigg)\onnm{\vbb}{L^{2m}}^{2m-1}.
\end{align}
For the bulk term on the LHS of \eqref{ktt 05}, the $a$ and $c$ contribution vanish due to oddness of \eqref{ktt 20}. Then we focus on the $\vbb$ contribution
\begin{align}\label{ktt 25}
-\e\iint_{\Omega\times\r^3}(\vv\cdot\nx\psi_{b,i,j})\mh(\vv)(\vbb\cdot\vv)
=-\e\iint_{\Omega\times\r^3}\sum_{k,s=1}^3\m(\vv)\left(v_i^2-\beta_{b,i,j}\right)v_kv_sb_s\p_k\p_j\phi_{b,j}.
\end{align}
Due to oddness, for $k\neq s$, the terms in \eqref{ktt 25} vanish. Hence, we have
\begin{align}\label{ktt 26}
-\e\iint_{\Omega\times\r^3}(\vv\cdot\nx\psi_{b,i,j})\mh(\vv)(\vbb\cdot\vv)
=-\e\iint_{\Omega\times\r^3}\sum_{k=1}^3\m(\vv)\left(v_i^2-\beta_{b,i,j}\right)v_k^2b_k\p_k\p_j\phi_{b,j}.
\end{align}
Based on our choice of $\beta_{b,i,j}$ in \eqref{ktt 24}, we directly compute
\begin{align}
\int_{\r^3}\m(\vv)\left(\abs{v_i}^2-\beta_b\right)v_k^2\ud{\vv}=&0\ \ \text{for}\ \ k\neq i,\\
\int_{\r^3}\m(\vv)\left(\abs{v_i}^2-\beta_b\right)v_i^2\ud{\vv}\neq&0.
\end{align}
Thus, for $k\neq i$, the terms in \eqref{ktt 26} vanish. Hence, we have
\begin{align}\label{ktt 27}
-\e\iint_{\Omega\times\r^3}(\vv\cdot\nx\psi_{b,i,j})\mh(\vv)(\vbb\cdot\vv)\ud{\vv}
&=-\e\iint_{\Omega\times\r^3}\m(\vv)\left(v_i^2-\beta_{b,i,j}\right)v_i^2b_i\p_i\p_j\phi_{b,j}\\
&=-\e\int_{\Omega}b_i\p_i\p_j\phi_{b,j}=-\int_{\Omega}\bigg(\p_{i}\p_j\dx^{-1}\Big(b_j\abs{b_j}^{2m-2}\Big)\bigg)b_i.
\end{align}
Finally, the $(\ik-\pk)[f]$ contribution on the LHS of \eqref{ktt 05} can be estimated as
\begin{align}\label{ktt 23}
\abs{\e\iint_{\Omega\times\r^3}(\vv\cdot\nx\psi_{b,i,j})(\ik-\pk)[f]}&\ls\e\onnm{\vbb}{L^{2m}}^{2m-1}\pnnm{(\ik-\pk)[f]}{2m}.
\end{align}
Collecting \eqref{ktt 21}, \eqref{ktt 22}, \eqref{ktt 27} and \eqref{ktt 23}, we obtain
\begin{align}\label{ktt 28}
&\e\abs{\int_{\Omega}\bigg(\p_{i}\p_j\dx^{-1}\Big(b_j\abs{b_j}^{2m-2}\Big)\bigg)b_i}\\
\ls&\onnm{\vbb}{L^{2m}}^{2m-1}\bigg(\e\pnm{(1-\pp)[f]}{\gamma_+}{\frac{4m}{3}}+\tnnm{(\ik-\pk)[f]}+\e\pnnm{(\ik-\pk)[f]}{2m}+\tnnm{\snn S}+\e\pnm{h}{\gamma_-}{\frac{4m}{3}}\bigg).\no
\end{align}
Note that we cannot further simplify the LHS of \eqref{ktt 28} at this stage since we do not include all derivative terms in $\dx b_j$. For example, $\ds\p_{1}\p_1\dx^{-1}\Big(b_2\abs{b_2}^{2m-2}\Big)b_2$ is not controlled here.\\
\ \\
Sub-Step 2.2: Estimates of $\bigg(\p_{i}\p_i\dx^{-1}\Big(b_j\abs{b_j}^{2m-2}\Big)\bigg)b_j$ for $i\neq j$.\\
Notice that the $i=j$ case is included in Sub-Step 2.1. We choose the test function
\begin{align}
\psi=\tilde\psi_{b,i,j}=\m^{\frac{1}{2}}(\vv)\abs{\vv}^2v_iv_j\p_i\phi_{b,j}\ \ \text{for}\ \ i\neq j.
\end{align}
Similar to Sub-Step 2.1, we focus on the $\vbb$ contribution on the LHS of \eqref{ktt 05}
\begin{align}\label{ktt 29}
-\e\iint_{\Omega\times\r^3}\Big(\vv\cdot\nx\tilde\psi_{b,i,j}\Big)\mh(\vv)(\vbb\cdot\vv)
=-\e\iint_{\Omega\times\r^3}\sum_{k,s=1}^3\m(\vv)\abs{\vv}^2v_kv_sv_iv_jb_s\p_k\p_i\phi_{b,j}.
\end{align}
Due to oddness, the terms in \eqref{ktt 29} do not vanish only if $k=i,s=j$ or $k=j,s=i$. Hence, we are left with
\begin{align}\label{ktt 30}
-\e\iint_{\Omega\times\r^3}\Big(\vv\cdot\nx\tilde\psi_{b,i,j}\Big)\mh(\vv)(\vbb\cdot\vv)
&=-\e\iint_{\Omega\times\r^3}\m(\vv)\abs{\vv}^2v_i^2v_j^2\bigg(b_j\p_i\p_i\phi_{b,j}+b_i\p_j\p_i\phi_{b,j}\bigg)\\
&=-\e\iint_{\Omega}\bigg(b_j\p_i\p_i\phi_{b,j}+b_i\p_j\p_i\phi_{b,j}\bigg).\no
\end{align}
Note that $\ds\e\iint_{\Omega}b_i\p_j\p_i\phi_{b,j}$ has been controlled by Sub-Step 2.2. Hence, we obtain
\begin{align}\label{ktt 31}
&\e\abs{\int_{\Omega}\bigg(\p_{i}\p_i\dx^{-1}\Big(b_j\abs{b_j}^{2m-2}\Big)\bigg)b_j}\\
\ls&\onnm{\vbb}{L^{2m}}^{2m-1}\bigg(\e\pnm{(1-\pp)[f]}{\gamma_+}{\frac{4m}{3}}+\tnnm{(\ik-\pk)[f]}+\e\pnnm{(\ik-\pk)[f]}{2m}+\tnnm{\snn S}+\e\pnm{h}{\gamma_-}{\frac{4m}{3}}\bigg).\no
\end{align}
\ \\
Sub-Step 2.3: Synthesis.\\
Summarizing \eqref{ktt 28} and \eqref{ktt 31}, we may sum up over $j=1,2,3$ to obtain, for any $i=1,2,3$,
\begin{align}
\\
\e\onnm{b_i}{L^{2m}}^{2m}
\ls&\onnm{\vbb}{L^{2m}}^{2m-1}\bigg(\e\pnm{(1-\pp)[f]}{\gamma_+}{\frac{4m}{3}}+\tnnm{(\ik-\pk)[f]}+\e\pnnm{(\ik-\pk)[f]}{2m}+\tnnm{\snn S}+\e\pnm{h}{\gamma_-}{\frac{4m}{3}}\bigg).\no
\end{align}
which further implies
\begin{align}\label{ktt 32}
\e\onnm{\vbb}{L^{2m}}
\ls&\e\pnm{(1-\pp)[f]}{\gamma_+}{\frac{4m}{3}}+\tnnm{(\ik-\pk)[f]}+\e\pnnm{(\ik-\pk)[f]}{2m}+\tnnm{\snn S}+\e\pnm{h}{\gamma_-}{\frac{4m}{3}}.
\end{align}
\ \\
Step 3: Estimates of $a$.\\
We choose the test function
\begin{align}
\psi=\psi_a=\m^{\frac{1}{2}}(\vv)\left(\abs{\vv}^2-\beta_a\right)\Big(\vv\cdot\nx\phi_a(\vx)\Big),
\end{align}
where
\begin{align}\label{ktt 36}
\left\{
\begin{array}{l}
-\dx\phi_a=a\abs{a}^{2m-2}(\vx)-\dfrac{1}{\abs{\Omega}}\ds\int_{\Omega}a\abs{a}^{2m-2}(\vx)\ud{\vx}\ \ \text{in}\ \ \Omega,\\\rule{0ex}{1.5em}
\dfrac{\p\phi_a}{\p\vn}=0\ \ \text{on}\ \ \p\Omega,
\end{array}
\right.
\end{align}
and $\beta_a\in\r$ will be determined later. Since
\begin{align}
\int_{\Omega}\bigg(a^{2m-1}(\vx)-\dfrac{1}{\abs{\Omega}}\ds\int_{\Omega}a^{2m-1}(\vx)\ud{\vx}\bigg)\ud{\vx}=0,
\end{align}
based on standard elliptic estimates, we may recover the estimates as \eqref{ktt 07}, \eqref{ktt 08} and \eqref{ktt 13}. Then similar to Step 1, we obtain that the right-hand side (RHS) of \eqref{ktt 05} is bounded as
\begin{align}\label{ktt 33}
\text{RHS}\ls \bigg(\tnnm{(\ik-\pk)[f]}+\tnnm{\snn S}\bigg)\onnm{a}{L^{2m}}^{2m-1}.
\end{align}
For the left-hand side (LHS) of \eqref{ktt 05}, the bulk term can be estimated as Step 1. There is no $\vbb$ contribution due to oddness. We will choose $\beta_a$ such that
\begin{align}
\int_{\r^3}\m^{\frac{1}{2}}(\vv)\left(\abs{\vv}^2-\beta_a\right)\frac{\abs{\vv}^2-3}{2}v_i^2\ud{\vv}=0\
\ \text{for}\ \ i=1,2,3,
\end{align}
which will eliminate $c$ contribution. Hence, the remaining $a$ contribution will be
\begin{align}\label{ktt 34}
-\e\iint_{\Omega\times\r^3}(\vv\cdot\nx\psi_a)\mh(\vv)a=-\int_{\Omega}a\dx\phi_a=\onnm{a}{L^{2m}}^{2m}.
\end{align}
Here, we use the fact that $\ds\int_{\Omega}a=\int_{\Omega\times\r^3}f(\vx,\vv)=0$ due to \eqref{linear steady normalization}. Similarly, $(\ik-\pk)[f]$ contribution is
\begin{align}\label{ktt 35}
\abs{\e\iint_{\Omega\times\r^3}(\vv\cdot\nx\psi_a)(\ik-\pk)[f]}\ls \e\onnm{a}{L^{2m}}^{2m-1}\pnnm{(\ik-\pk)[f]}{2m}.
\end{align}
Now the only difficulty is the boundary term in \eqref{ktt 05}. In particular, as in \eqref{ktt 12}, we are concerned with
\begin{align}
\iint_{\gamma_+}\pp[f]\psi_a\ud{\gamma}-\iint_{\gamma_-}\pp[f]\psi_a\ud{\gamma}
=\int_{\p\Omega\times\r^3}(\vv\cdot\vn)\mh(\vv)\left(\abs{\vv}^2-\beta_a\right)\Big(\vv\cdot\nx\phi_a(\vx_0)\Big)\pp[f](\vx_0,\vv).
\end{align}
This cannot be directly killed with the previous techniques (oddness and choice of $\beta_a$ cannot do it). Notice that $\pp[f](\vx_0,\vv)=z(\vx_0)\mh(\vv)$ for $z(\vx_0)=\ds\int_{\gamma_+}f\ud\gamma$. We decompose the velocity into normal and tangential directions:
\begin{align}
\vv=\vn(\vv\cdot\vn)+\vn^{\perp},
\end{align}
where $\vn^{\perp}$ is the tangential part. Then
\begin{align}
&\int_{\p\Omega\times\r^3}(\vv\cdot\vn)\mh(\vv)\left(\abs{\vv}^2-\beta_a\right)\Big(\vv\cdot\nx\phi_a(\vx_0)\Big)\pp[f](\vx_0,\vv)\\
=&\int_{\p\Omega\times\r^3}(\vv\cdot\vn)^2\mh(\vv)\left(\abs{\vv}^2-\beta_a\right)\frac{\p\phi_a(\vx_0)}{\p\vn}\pp[f](\vx_0,\vv)
+\int_{\p\Omega\times\r^3}(\vv\cdot\vn)\m(\vv)\left(\abs{\vv}^2-\beta_a\right)\Big(\vn^{\perp}\cdot\nx\phi_a(\vx_0)\Big)z(\vx_0).\no
\end{align}
Here, in the RHS, the first term vanishes due to the Neumann boundary condition in \eqref{ktt 36}, and the second term vanishes due to oddness. Then in total, we have
\begin{align}
\iint_{\gamma_+}\pp[f]\psi_a\ud{\gamma}-\iint_{\gamma_-}\pp[f]\psi_a\ud{\gamma}=0.
\end{align}
With this in hand, we can bound as \eqref{ktt 14} to get the boundary contribution
\begin{align}\label{ktt 37}
\abs{\e\iint_{\gamma_+}f\psi_a\ud{\gamma}-\e\iint_{\gamma_-}f\psi_a\ud{\gamma}}
&\ls\e\bigg(\pnm{(1-\pp)[f]}{\gamma_+}{\frac{4m}{3}}+\pnm{h}{\gamma_-}{\frac{4m}{3}}\bigg)\onnm{a}{L^{2m}}^{2m-1}.
\end{align}
Collecting \eqref{ktt 33}, \eqref{ktt 34}, \eqref{ktt 35}, \eqref{ktt 37}, and cancelling $\onnm{a}{L^{2m}}^{2m-1}$, we have
\begin{align}\label{ktt 38}
&\e\onnm{a}{L^{2m}}\ls\e\pnm{(1-\pp)[f]}{\gamma_+}{\frac{4m}{3}}+\tnnm{(\ik-\pk)[f]}+\e\pnnm{(\ik-\pk)[f]}{2m}+\tnnm{\snn S}+\e\pnm{h}{\gamma_-}{\frac{4m}{3}}.
\end{align}
\ \\
Step 4: Synthesis.\\
Collecting \eqref{ktt 19}, \eqref{ktt 32} and \eqref{ktt 38}, we deduce
\begin{align}
\e\pnnm{\pk[f]}{2m}\ls\e\pnm{(1-\pp)[f]}{\gamma_+}{\frac{4m}{3}}+\tnnm{(\ik-\pk)[f]}+\e\pnnm{(\ik-\pk)[f]}{2m}+\tnnm{\snn S}+\e\pnm{h}{\gamma_-}{\frac{4m}{3}}.
\end{align}
This completes our proof.
\end{proof}

\begin{theorem}\label{LN estimate}
The solution $f(\vx,\vv)$ to the equation \eqref{linear steady} satisfies the estimate
\begin{align}
&\frac{1}{\e^{\frac{1}{2}}}\tsm{(1-\pp)[f]}{\gamma_+}+\frac{1}{\e}\unnm{(\ik-\pk)[f]}+\pnnm{\pk[f]}{2m}\\
\ls& o(1)\e^{\frac{3}{2m}}\Big(\ssm{f}{\gamma_+}+\snnm{f}\Big)
+\frac{1}{\e^{2}}\pnnm{\pk[S]}{\frac{2m}{2m-1}}+\frac{1}{\e}\tnnm{\snn(\ik-\pk)[S]}+\pnm{h}{\gamma_-}{\frac{4m}{3}}+\frac{1}{\e}\tsm{h}{\gamma_-}.\no
\end{align}
\end{theorem}
\begin{proof}
\ \\
Step 1: Energy Estimate.\\
Multiplying $f$ on both sides of \eqref{linear steady} and applying Green's identity in Lemma \ref{ktt 03} imply
\begin{align}\label{ktt 40}
\frac{\e}{2}\tsm{f}{\gamma_+}^2-\frac{\e}{2}\tsm{\pp[f]+h}{\gamma_-}^2+\int_{\Omega\times\r^3}f\ll[f]=&\iint_{\Omega\times\r^3}fS.
\end{align}
A direct computation shows that
\begin{align}
\tsm{(1-\pp)[f]}{\gamma_+}^2&=\int_{\gamma_+}\Big(f-\pp[f]\Big)^2\ud\gamma=\int_{\gamma_+}f^2\ud\gamma+\int_{\gamma_+}\Big(\pp[f]\Big)^2\ud\gamma-2\int_{\gamma_+}f\pp[f]\ud\gamma\\
&=\tsm{f}{\gamma_+}^2+\tsm{\pp[f]}{\gamma_+}^2-2\tsm{\pp[f]}{\gamma_+}^2=\tsm{f}{\gamma_+}^2-\tsm{\pp[f]}{\gamma_+}^2.\no
\end{align}
Obviously, $\tsm{\pp[f]}{\gamma_+}^2=\tsm{\pp[f]}{\gamma_-}^2$. Hence, we have
\begin{align}\label{ktt 41}
\frac{\e}{2}\tsm{f}{\gamma_+}^2-\frac{\e}{2}\tsm{\pp[f]+h}{\gamma_-}^2
&=\frac{\e}{2}\tsm{f}{\gamma_+}^2-\frac{\e}{2}\tsm{\pp[f]}{\gamma_-}^2-\frac{\e}{2}\tsm{h}{\gamma_-}^2+\e\int_{\gamma_-}h\pp[f]\ud\gamma\\
&=\frac{\e}{2}\tsm{(1-\pp)[f]}{\gamma_+}^2-\frac{\e}{2}\tsm{h}{\gamma_-}^2+\e\int_{\gamma_-}h\pp[f]\ud\gamma\no\\
&\gs \e\tsm{(1-\pp)[f]}{\gamma_+}^2-\frac{1}{\eta}\tsm{h}{\gamma_-}^2-\e^2\eta\tsm{\pp[f]}{\gamma_+},\no
\end{align}
where $0<\eta<<1$ will be determined later. On the other hand, based on Lemma \ref{ktt 01}, we know
\begin{align}\label{ktt 42}
\int_{\Omega\times\r^3}f\ll[f]\gs \unnm{(\ik-\pk)[f]}^2.
\end{align}
Inserting \eqref{ktt 41} and \eqref{ktt 42} into \eqref{ktt 40}, we have
\begin{align}\label{ktt 43}
\e\tsm{(1-\pp)[f]}{\gamma_+}^2+\unnm{(\ik-\pk)[f]}^2
\ls \eta\e^2\tsm{\pp[f]}{\gamma_+}+\frac{1}{\eta}\tsm{h}{\gamma_-}^2+\int_{\Omega\times\r^3}fS.
\end{align}
\ \\
Step 2: Estimate of $\tsm{\pp[f]}{\gamma_+}$.\\
Multiplying $f$ on both sides of the equation \eqref{linear steady}, we have
\begin{align}\label{ktt 44}
\vv\cdot\nx(f^2)=\frac{2}{\e}\Big(-f\ll[f]+fS\Big).
\end{align}
Taking absolute value and integrating \eqref{ktt 44} over $\Omega\times\r^3$, using Lemma \ref{ktt 01}, we deduce
\begin{align}
\pnnm{\vv\cdot\nx(f^2)}{1}\ls&\frac{1}{\e}\bigg(\tnnm{(\ik-\pk)[f]}^2+\abs{\int_{\Omega\times\r^3}fS}\bigg).
\end{align}
On the other hand, applying Lemma \ref{ktt 02} to $f^2$, for near grazing set $\gamma^{\d}$, we have
\begin{align}\label{ktt 46}
\pnm{\id_{\gamma\backslash\gamma^{\d}}f}{\gamma}{2}^2&=\pnm{\id_{\gamma\backslash\gamma^{\d}}f^2}{\gamma}{1}
\leq
C(\d)\left(\pnnm{f^2}{1}+\pnnm{\vv\cdot\nx(f^2)}{1}\right)=C(\d)\left(\tnnm{f}^2+\pnnm{\vv\cdot\nx(f^2)}{1}\right)\\
&\ls C(\d)\left(\tnnm{f}^2+\frac{1}{\e}\tnnm{(\ik-\pk)[f]}^2+\frac{1}{\e}\abs{\int_{\Omega\times\r^3}fS}\right).\no
\end{align}
As in Step 3 of proof to Lemma \ref{ktt lemma 1}, we can rewrite $\pp[f](\vx_0,\vv)=z(\vx)\mh(\vv)$. Then for $\d$ small, we deduce
\begin{align}\label{ktt 45}
\tsm{\pp[\id_{\gamma\backslash\gamma^{\d}}f]}{\gamma}^2
=&\int_{\p\Omega}\abs{z(\vx)}^2
\left(\int_{\vv\cdot\vn(\vx)\geq\d,\d\leq\abs{\vv}\leq\d^{-1}}\m(\vv)\abs{\vv\cdot\vn(\vx)}\ud{\vv}\right)\ud{\vx}\\
\geq&\half\left(\int_{\p\Omega}\abs{z(\vx)}^2\ud{\vx}\right)\left(\int_{\gamma_+}\m(\vv)\abs{\vv\cdot\vn(\vx)}\ud{\vv}\right)=\half\tsm{\pp[f]}{\gamma_+}^2,\no
\end{align}
where we utilize the bounds that
\begin{align}
\int_{\vv\cdot\vn(\vx)\leq\d}\m(\vv)\abs{\vv\cdot\vn(\vx)}\ud{\vv}\ls&\d,\\
\int_{\abs{\vv}\leq\d\ \text{or}\ \abs{\vv}\geq\d^{-1}}\m(\vv)\abs{\vv\cdot\vn(\vx)}\ud{\vv}\ls&\d.
\end{align}
Therefore, from \eqref{ktt 45} and the fact
\begin{align}
\tsm{\pp[\id_{\gamma\backslash\gamma^{\d}}f]}{\gamma_+}\ls
\tsm{\id_{\gamma\backslash\gamma^{\d}}f}{\gamma_+}\ls \tsm{\id_{\gamma\backslash\gamma^{\d}}f}{\gamma},
\end{align}
we conclude
\begin{align}
\tsm{\pp[f]}{\gamma_+}^2\ls \tsm{\pp[\id_{\gamma\backslash\gamma^{\d}}f]}{\gamma_+}\ls\tsm{\id_{\gamma\backslash\gamma^{\d}}f}{\gamma}.
\end{align}
Considering \eqref{ktt 46}, we have
\begin{align}
\tsm{\pp[f]}{\gamma_+}^2\ls&C(\d)\left(\tnnm{f}^2+\frac{1}{\e}\tnnm{(\ik-\pk)[f]}^2+\frac{1}{\e}\abs{\int_{\Omega\times\r^3}fS}\right).
\end{align}
For fixed $0<\d<<1$ and using $f=\pk[f]+(\ik-\pk)[f]$, we obtain
\begin{align}\label{ktt 47}
\tsm{\pp[f]}{\gamma_+}^2\ls&\tnnm{\pk[f]}^2+\frac{1}{\e}\tnnm{(\ik-\pk)[f]}^2+\frac{1}{\e}\abs{\int_{\Omega\times\r^3}fS}.
\end{align}
\ \\
Step 3: Interpolation Estimates.\\
Plugging \eqref{ktt 47} into \eqref{ktt 43} with $\e$ sufficiently small to absorb $\unnm{(\ik-\pk)[f]}^2$ into the left-hand side, we obtain
\begin{align}\label{ktt 48}
\e\tsm{(1-\pp)[f]}{\gamma_+}^2+\unnm{(\ik-\pk)[f]}^2
\ls \eta\e^2\tnnm{\pk[f]}^2+\frac{1}{\eta}\tsm{h}{\gamma_-}^2+\abs{\int_{\Omega\times\r^3}fS}.
\end{align}
We square on both sides of \eqref{ktt 39} to obtain
\begin{align}\label{ktt 49}
\e^2\pnnm{\pk[f]}{2m}^2\ls\e^2\pnm{(1-\pp)[f]}{\gamma_+}{\frac{4m}{3}}^2+\tnnm{(\ik-\pk)[f]}^2+\e^2\pnnm{(\ik-\pk)[f]}{2m}^2+\tnnm{\snn S}^2+\e^2\pnm{h}{\gamma_-}{\frac{4m}{3}}^2.
\end{align}
H\"older's inequality implies
\begin{align}
\tnnm{\pk[f]}\ls \pnnm{\pk[f]}{2m}.
\end{align}
Multiplying a small constant on both sides of \eqref{ktt 49} and
adding to \eqref{ktt 48} with $\eta>0$ sufficiently small to absorb $\eta\e^2\tnnm{\pk[f]}^2$ and $\tnnm{(\ik-\pk)[f]}^2$ into the left-hand side, we obtain
\begin{align}\label{ktt 50}
&\e\tsm{(1-\pp)[f]}{\gamma_+}^2+\unnm{(\ik-\pk)[f]}^2+\e^2\pnnm{\pk[f]}{2m}^2\\
\ls& \e^2\pnm{(1-\pp)[f]}{\gamma_+}{\frac{4m}{3}}^2+\e^2\pnnm{(\ik-\pk)[f]}{2m}^2+\tnnm{\snn S}^2+\e^2\pnm{h}{\gamma_-}{\frac{4m}{3}}^2+\tsm{h}{\gamma_-}^2+\abs{\int_{\Omega\times\r^3}fS}.\no
\end{align}
Now we need to handle the extra term $\e^2\pnm{(1-\pp)[f]}{\gamma_+}{\frac{4m}{3}}^2$ and $\e^2\pnnm{(\ik-\pk)[f]}{2m}^2$ on the right-hand side of \eqref{ktt 50}. By interpolation estimate and Young's inequality, we have
\begin{align}
\pnm{(1-\pp)[f]}{\gamma_+}{\frac{4m}{3}}\leq&\tsm{(1-\pp)[f]}{\gamma_+}^{\frac{3}{2m}}\ssm{(1-\pp)[f]}{\gamma_+}^{\frac{2m-3}{2m}}\\
=&\bigg(\frac{1}{\e^{\frac{6m-9}{4m^2}}}\tsm{(1-\pp)[f]}{\gamma_+}^{\frac{3}{2m}}\bigg)
\bigg(\e^{\frac{6m-9}{4m^2}}\ssm{(1-\pp)[f]}{\gamma_+}^{\frac{2m-3}{2m}}\bigg)\no\\
\ls&\bigg(\frac{1}{\e^{\frac{6m-9}{4m^2}}}\tsm{(1-\pp)[f]}{\gamma_+}^{\frac{3}{2m}}\bigg)^{\frac{2m}{3}}+o(1)
\bigg(\e^{\frac{6m-9}{4m^2}}\ssm{(1-\pp)[f]}{\gamma_+}^{\frac{2m-3}{2m}}\bigg)^{\frac{2m}{2m-3}}\no\\
\leq&\frac{1}{\e^{\frac{2m-3}{2m}}}\tsm{(1-\pp)[f]}{\gamma_+}+o(1)\e^{\frac{3}{2m}}\ssm{(1-\pp)[f]}{\gamma_+}\no\\
\leq&\frac{1}{\e^{\frac{2m-3}{2m}}}\tsm{(1-\pp)[f]}{\gamma_+}+o(1)\e^{\frac{3}{2m}}\ssm{(1-\pp)[f]}{\gamma_+}.\no
\end{align}
Similarly, we have
\begin{align}
\pnnm{(\ik-\pk)[f]}{2m}\leq&\tnnm{(\ik-\pk)[f]}^{\frac{1}{m}}\snnm{(\ik-\pk)[f]}^{\frac{m-1}{m}}\\
=&\bigg(\frac{1}{\e^{\frac{3m-3}{2m^2}}}\tnnm{(\ik-\pk)[f]}^{\frac{1}{m}}\bigg)
\bigg(\e^{\frac{3m-3}{2m^2}}\snnm{(\ik-\pk)[f]}^{\frac{m-1}{m}}\bigg)\no\\
\ls&\bigg(\frac{1}{\e^{\frac{3m-3}{2m^2}}}\tnnm{(\ik-\pk)[f]}^{\frac{1}{m}}\bigg)^{m}+o(1)
\bigg(\e^{\frac{3m-3}{2m^2}}\snnm{(\ik-\pk)[f]}^{\frac{m-1}{m}}\bigg)^{\frac{m}{m-1}}\no\\
\leq&\frac{1}{\e^{\frac{3m-3}{2m}}}\tnnm{(\ik-\pk)[f]}+o(1)\e^{\frac{3}{2m}}\snnm{(\ik-\pk)[f]}.\no
\end{align}
We need this extra $\e^{\frac{3}{2m}}$ for the convenience of $L^{\infty}$ estimate. Then we know for sufficiently small $\e$ and $\dfrac{3}{2}< m<3$,
\begin{align}\label{ktt 51}
\e^2\pnm{(1-\pp)[f]}{\gamma_+}{\frac{4m}{3}}^2
\ls&\e^{2-\frac{2m-3}{m}}\tsm{(1-\pp)[f]}{\gamma_+}^2+o(1)\e^{2+\frac{3}{m}}\ssm{(1-\pp)[f]}{\gamma_+}^2\\
\ls&o(1)\e\tsm{(1-\pp)[f]}{\gamma_+}^2+o(1)\e^{2+\frac{3}{m}}\ssm{f}{\gamma_+}^2.\no
\end{align}
Similarly, we have
\begin{align}\label{ktt 52}
\e^2\pnnm{(\ik-\pk)[f]}{2m}^2\ls&\e^{2-\frac{3m-3}{m}}\tnnm{(\ik-\pk)[f]}^2+o(1)\e^{2+\frac{3}{m}}\snnm{(\ik-\pk)[f]}^2\\
\ls& o(1)\tnnm{(\ik-\pk)[f]}^2+o(1)\e^{2+\frac{3}{m}}\snnm{f}^2.\no
\end{align}
Inserting \eqref{ktt 51} and \eqref{ktt 52} into \eqref{ktt 50}, we can absorb $o(1)\e\tsm{(1-\pp)[f]}{\gamma_+}^2$ and $o(1)\tnnm{(\ik-\pk)[f]}^2$ into the left-hand side to obtain
\begin{align}\label{ktt 52}
&\e\tsm{(1-\pp)[f]}{\gamma_+}^2+\unnm{(\ik-\pk)[f]}^2+\e^2\pnnm{\pk[f]}{2m}^2\\
\ls& o(1)\e^{2+\frac{3}{m}}\Big(\ssm{f}{\gamma_+}^2+\snnm{f}^2\Big)+
\tnnm{\snn S}^2+\e^2\pnm{h}{\gamma_-}{\frac{4m}{3}}^2+\tsm{h}{\gamma_-}^2+\abs{\int_{\Omega\times\r^3}fS}.\no
\end{align}
\ \\
Step 4: Synthesis.\\
We can decompose
\begin{align}\label{ktt 55}
\int_{\Omega\times\r^3}fS=\iint_{\Omega\times\r^3}\pk[f]\pk[S]+\iint_{\Omega\times\r^3}(\ik-\pk)[f](\ik-\pk)[S].
\end{align}
H\"older's inequality and Cauchy's inequality imply
\begin{align}\label{ktt 53}
\iint_{\Omega\times\r^3}\pk[f]\pk[S]\leq\pnnm{\pk[f]}{2m}\pnnm{\pk[S]}{\frac{2m}{2m-1}}
\ls o(1)\e^2\pnnm{\pk[f]}{2m}^2+\frac{1}{\e^{2}}\pnnm{\pk[S]}{\frac{2m}{2m-1}}^2,
\end{align}
and
\begin{align}\label{ktt 54}
\iint_{\Omega\times\r^3}(\ik-\pk)[f](\ik-\pk)[S]\ls o(1)\unnm{(\ik-\pk)[f]}^2+\tnnm{\snn(\ik-\pk)[S]}^2.
\end{align}
Inserting \eqref{ktt 53} and \eqref{ktt 54} into \eqref{ktt 55} and further \eqref{ktt 52}, absorbing $o(1)\e^2\pnnm{\pk[f]}{2m}^2$ and $o(1)\unnm{(\ik-\pk)[f]}^2$ into the left-hand side, we get
\begin{align}
&\e\tsm{(1-\pp)[f]}{\gamma_+}^2+\unnm{(\ik-\pk)[f]}^2+\e^2\pnnm{\pk[f]}{2m}^2\\
\ls& o(1)\e^{2+\frac{3}{m}}\Big(\ssm{f}{\gamma_+}^2+\snnm{f}^2\Big)+
+\frac{1}{\e^{2}}\pnnm{\pk[S]}{\frac{2m}{2m-1}}^2+\tnnm{\snn(\ik-\pk)[S]}^2+\e^2\pnm{h}{\gamma_-}{\frac{4m}{3}}^2+\tsm{h}{\gamma_-}^2.\no
\end{align}
Therefore, we have
\begin{align}
&\frac{1}{\e^{\frac{1}{2}}}\tsm{(1-\pp)[f]}{\gamma_+}+\frac{1}{\e}\unnm{(\ik-\pk)[f]}+\pnnm{\pk[f]}{2m}\\
\leq& o(1)\e^{\frac{3}{2m}}\Big(\ssm{f}{\gamma_+}+\snnm{f}\Big)+
+\frac{1}{\e^{2}}\pnnm{\pk[S]}{\frac{2m}{2m-1}}+\frac{1}{\e}\tnnm{\snn(\ik-\pk)[S]}+\pnm{h}{\gamma_-}{\frac{4m}{3}}+\frac{1}{\e}\tsm{h}{\gamma_-}\bigg).\no
\end{align}
\end{proof}

\subsection{$L^{\infty}$ Estimates}

Now we begin to consider mild formulation. When tracking the solution backward along the characteristics, once it hits the in-flow boundary, due to diffusive reflection boundary, actually the information comes from the integral of characteristics hitting the out-flow boundary. Following this idea, we may define the backward stochastic cycles, with multiple hitting times and out-flow integrals.

\begin{definition}[Hitting Time and Position]
For any $(\vx,\vv)\in\Omega\times\r^3$ with $(\vx,\vv)\notin\gamma_0$, define the backward the hitting time
\begin{align}\label{ktt 58}
t_b(\vx,\vv):=&\inf\{t>0:\vx-\e t\vv\notin\Omega\}.
\end{align}
Also, define the hitting position
\begin{align}
\vx_b:=\vx-\e t_b(\vx,\vv)\vv\notin\Omega.
\end{align}
\end{definition}

\begin{definition}[Stochastic Cycle]
For any $(\vx,\vv)\in\Omega\times\r^3$ with $(\vx,\vv)\notin\gamma_0$, let $(t_0,\vx_0,\vv_0)=(0,\vx,\vv)$. Define the first stochastic triple
\begin{align}
(t_1,\vx_1,\vv_1):=\Big(t_b(\vx_0,\vv_0),\vx_b(\vx_0,\vv_0),\vv_1\Big),
\end{align}
for some $\vv_1$ satisfying $\vv_1\cdot\vn(\vx_1)>0$.

Inductively, assume we know the $k^{th}$ stochastic triple $(t_k,\vx_k,\vv_k)$. Define the $(k+1)^{th}$ stochastic triple
\begin{align}
(t_{k+1},\vx_{k+1},\vv_{k+1}):=\Big(t_k+t_b(\vx_k,\vv_k),\vx_k(\vx_k,\vv_k),\vv_{k+1}\Big),
\end{align}
for some $\vv_{k+1}$ satisfying $\vv_{k+1}\cdot\vn(\vx_{k+1})>0$.
\end{definition}

\begin{remark}
Roughly speaking, this definition describes one characteristic line with reflection (alternatively so-called stochastic cycle), starting from $(\vx_k,\vv_k)\in\gamma_+$, tracking back to $(\vx_{k+1},\vv_k)\in\gamma_-$, diffusively reflected to $(\vx_{k+1},\vv_{k+1})\in\gamma_+$, and beginning a new cycle. $t_k$ the accumulative time the characteristic moves backward. Note that we are free to choose any $\vv_k\cdot\vn(\vx_k)>0$, so different sequence $\ds\{\vv_k\}_{k=1}^{\infty}$ represents different stochastic cycles.
\end{remark}

\begin{definition}[Diffusive Reflection Integral]
Define $\nn_{k}=\{\vv\in \r^3:\vv\cdot\vn(\vx_{k})>0\}$, so the stochastic cycle must satisfy $\vv_k\in\nn_k$. Let the iterated integral for $k\geq2$ be defined as
\begin{align}
\int_{\prod_{j=1}^{k-1}\nn_j}\prod_{j=1}^{k-1}\ud{\sigma_j}:=\int_{\nn_1}\ldots\bigg(\int_{\nn_{k-1}}\ud{\sigma_{k-1}}\bigg)\ldots\ud{\sigma_1}
\end{align}
where $\ud{\sigma_j}:=\m(\vv_j)\abs{\vv_j\cdot\vn(\vx_j)}\ud{\vv_j}$ is a probability measure.
\end{definition}

We define a weight function scaled with parameter $\xi$, for $0\leq\vrh<\dfrac{1}{4}$ and $\vth\geq0$,
\begin{align}\label{ktt 56}
\vh(\vv):=\bv,
\end{align}
and
\begin{align}\label{ktt 57}
\tvh(\vv):=\frac{1}{\m^{\frac{1}{2}}(\vv)\vh(\vv)}=\sqrt{2\pi}\frac{\ue^{\left(\frac{1}{4}-\varrho\right)\abs{\vv}^2}}
{\left(1+\abs{\vv}^2\right)^{\frac{\vth}{2}}}.
\end{align}

%

\begin{lemma}\label{ktt lemma 2}
For $T_0>0$ sufficiently large, there exists constants $C_1,C_2>0$ independent of $T_0$, such that for $k=C_1T_0^{\frac{5}{4}}$, and
$(\vx,\vv)\in\times\bar\Omega\times\r^3$,
\begin{align}\label{ktt 59}
\int_{\Pi_{j=1}^{k-1}\nn_j}\id_{\{t_k(\vx,\vv,\vv_1,\ldots,\vv_{k-1})<\frac{T_0}{\e}\}}\prod_{j=1}^{k-1}\ud{\sigma_j}\leq
\left(\frac{1}{2}\right)^{C_2T_0^{\frac{5}{4}}}.
\end{align}
\end{lemma}
\begin{proof}
This is a rescaled version of \cite[Lemma 4.1]{Esposito.Guo.Kim.Marra2013}. Since our hitting time in \eqref{ktt 58} is rescaled with $\e$, we should rescale back in the statement of lemma.
\end{proof}

\begin{remark}
Roughly speaking, Lemma \ref{ktt lemma 2} states that even though we have the freedom to choose $\vv_k$ in each stochastic cycle, in the long run, the accumulative time will not be too small. After enough reflections $\sim k$, most characteristics has the accumulative time that will exceed any set threshold $T_0$.
\end{remark}

\begin{theorem}\label{LI estimate}
Assume \eqref{linear steady compatibility} and \eqref{linear steady normalization} hold. The solution $f(\vx,\vv)$ to the equation \eqref{linear steady} satisfies for $\vth\geq0$ and $0\leq\varrho<\dfrac{1}{4}$,
\begin{align}
&\lnnmv{f}+\lsm{f}{\gamma_+}\\
\ls& \frac{1}{\e^{2+\frac{3}{2m}}}\pnnm{\pk[S]}{\frac{2m}{2m-1}}+\frac{1}{\e^{1+\frac{3}{2m}}}\tnnm{\snn(\ik-\pk)[S]}+\lnnmv{\nu^{-1}S}\no\\
&+\frac{1}{\e^{\frac{3}{2m}}}\pnm{h}{\gamma_-}{\frac{4m}{3}}+\frac{1}{\e^{1+\frac{3}{2m}}}\tsm{h}{\gamma_-}+\lsm{h}{\gamma_-}.\no
\end{align}
\end{theorem}
\begin{proof}
\ \\
Step 1: Mild formulation.\\
Denote the weighted solution
\begin{align}
g(\vx,\vv):=&\vh(\vv) f(\vx,\vv),
\end{align}
and the weighted non-local operator
\begin{align}
K_{\vh(\vv)}[g](\vv):=&\vh(\vv)K\left[\frac{g}{\vh}\right](\vv)=\int_{\r^3}k_{\vh(\vv)}(\vv,\vuu)g(\vuu)\ud{\vuu},
\end{align}
where
\begin{align}
k_{\vh(\vv)}(\vv,\vuu):=k(\vv,\vuu)\frac{\vh(\vv)}{\vh(\vuu)}.
\end{align}
Multiplying $\vh$ on both sides of \eqref{linear steady}, we have
\begin{align}\label{ktt 64}
\left\{
\begin{array}{l}
\e\vv\cdot\nx g+\nu g=K_{\vh}(\vx,\vv)+\vh(\vv) S(\vx,\vv)\ \ \text{in}\ \ \Omega\times\r^3,\\\rule{0ex}{2em}
g(\vx_0,\vv)=\ds\vh(\vv)\mh(\vv)\int_{\vuu\cdot\vn>0}\tvh(\vuu)g(\vx_0,\vuu)\ud\vuu+\vh h(\vx_0,\vv)\ \ \text{for}\ \ \vx_0\in\p\Omega\ \
\text{and}\ \ \vv\cdot\vn<0,
\end{array}
\right.
\end{align}
We can rewrite the solution of the equation \eqref{ktt 64} along the characteristics by Duhamel's principle as
\begin{align}
g(\vx,\vv)=& \vh(\vv)h(\vx_1,\vv)\ue^{-\nu(\vv)
t_{1}}+\int_{0}^{t_{1}}\vh(\vv) S\Big(\vx-\e(t_1-s)\vv,\vv\Big)\ue^{-\nu(\vv)
(t_{1}-s)}\ud{s}\\
&+\int_{0}^{t_{1}}K_{\vh(\vv)}[g]\Big(\vx-\e(t_1-s)\vv,\vv\Big)\ue^{-\nu(\vv)
(t_{1}-s)}\ud{s}+\frac{\ue^{-\nu(\vv)
t_{1}}}{\tvh(\vv)}\int_{\nn_1}g(\vx_1,\vv_1)\tvh(\vv_1)\ud{\sigma_1},\no
\end{align}
where the last term refers to $\pp[f]$. We may further rewrite the last term using \eqref{ktt 64} along the stochastic cycle by applying Duhamel's principle $k$ times as
\begin{align}\label{ktt 65}
g(\vx,\vv)=& \vh(\vv)h(\vx_1,\vv)\ue^{-\nu(\vv)t_{1}}+\int_{0}^{t_{1}}\vh(\vv) S\Big(\vx-\e(t_1-s)\vv,\vv\Big)\ue^{-\nu(\vv)(t_{1}-s)}\ud{s}\\
&+\int_{0}^{t_{1}}K_{\vh(\vv)}[g]\Big(\vx-\e(t_1-s)\vv,\vv\Big)\ue^{-\nu(\vv)(t_{1}-s)}\ud{s}\no\\
&+\frac{\ue^{-\nu(\vv)t_{1}}}{\tvh(\vv)}\sum_{\ell=1}^{k-1}\int_{\prod_{j=1}^{\ell}\nn_j}\Big(G_{\ell}[\vx,\vv]+H_{\ell}[\vx,\vv]\Big)\tvh(\vv_\ell)
\bigg(\prod_{j=1}^{\ell}\ue^{-\nu(\vv_j)(t_{j+1}-t_j)}\ud{\sigma_j}\bigg)\no\\
&+\frac{\ue^{-\nu(\vv)t_{1}}}{\tvh(\vv)}\int_{\prod_{j=1}^{k}\nn_j}g(\vx_k,\vv_k)\tvh(\vv_{k})
\bigg(\prod_{j=1}^{k}\ue^{-\nu(\vv_{j})(t_{j+1}-t_{j})}\ud{\sigma_j}\bigg),\no
\end{align}
where
\begin{align}\label{ktt 66}
G_{\ell}[\vx,\vv]:=&\vh(\vv_\ell)h(\vx_{\ell+1},\vv_{\ell})+\int_{t_{\ell}}^{t_{\ell+1}}\bigg(
\vh(\vv_\ell)S\Big(\vx_{\ell}-\e(t_{\ell+1}-s)\vv_{\ell},\vv_{\ell}\Big)\ue^{\nu(\vv_\ell)s}\bigg)\ud{s}\\
H_{\ell}[\vx,\vv]:=&\int_{t_{\ell}}^{t_{\ell+1}}\bigg(K_{\vh(\vv_{\ell})}[g]\Big(\vx_{\ell}-\e(t_{\ell+1}-s)\vv_{\ell},\vv_{\ell}\Big)\ue^{\nu(\vv_{\ell})s}\bigg)\ud{s}.\label{ktt 67}
\end{align}
\ \\
Step 2: Estimates of source terms and boundary terms.\\
We set $k=CT_0^{\frac{5}{4}}$ for $T_0$ defined in Lemma \ref{ktt lemma 2}. Consider all terms in \eqref{ktt 65} related to $h$ and $S$.

Since $t_1\geq0$, we have
\begin{align}\label{ktt 68}
\abs{\vh(\vv)h(\vx_1,\vv)\ue^{-\nu(\vv)t_{1}}}\leq \ssm{\vh h}{\gamma_-}.
\end{align}
Also,
\begin{align}\label{ktt 69}
\abs{\int_{0}^{t_{1}}\vh(\vv) S\Big(\vx-\e(t_1-s)\vv,\vv\Big)\ue^{-\nu(\vv)(t_{1}-s)}\ud{s}}&\leq \snnm{\nu^{-1}\vh S}\abs{\int_{0}^{t_{1}}\nu(\vv)\ue^{-\nu(\vv)(t_{1}-s)}\ud{s}}
\leq \snnm{\nu^{-1}\vh S}.
\end{align}
Then we turn to terms defined in $G_{\ell}$ of \eqref{ktt 66}. Noting that $\dfrac{1}{\tvh}\ls 1$, we know
\begin{align}\label{ktt 70}
&\abs{\frac{\ue^{-\nu(\vv)t_{1}}}{\tvh(\vv)}\sum_{\ell=1}^{k-1}\int_{\prod_{j=1}^{\ell}\nn_j}\Big(\vh(\vv_\ell)h(\vx_{\ell+1},\vv_{\ell})\Big)\tvh(\vv_\ell)
\bigg(\prod_{j=1}^{\ell}\ue^{-\nu(\vv_j)(t_{j+1}-t_j)}\ud{\sigma_j}\bigg)}\\
\ls&\ssm{\vh h}{\gamma_-}\abs{\sum_{\ell=1}^{k-1}\int_{\prod_{j=1}^{\ell}\nn_j}\tvh(\vv_\ell)\prod_{j=1}^{\ell}\ud{\sigma_j}}
\ls\ssm{\vh h}{\gamma_-}\abs{\sum_{\ell=1}^{k-1}\int_{\nn_{\ell}}\tvh(\vv_\ell)\ud{\sigma_\ell}}\ls CT_0^{\frac{5}{4}}\ssm{\vh h}{\gamma_-}.\no
\end{align}
Similarly,
\begin{align}\label{ktt 71}
\\
&\abs{\frac{\ue^{-\nu(\vv)t_{1}}}{\tvh(\vv)}\sum_{\ell=1}^{k-1}\int_{\prod_{j=1}^{\ell}\nn_j}\bigg(\int_{t_{\ell}}^{t_{\ell+1}}\bigg(
\vh(\vv_\ell)S\Big(\vx_{\ell}-\e(t_{\ell+1}-s)\vv_{\ell},\vv_{\ell}\Big)\ue^{\nu(\vv_\ell)s}\bigg)\ud{s}\bigg)\tvh(\vv_\ell)
\bigg(\prod_{j=1}^{\ell}\ue^{-\nu(\vv_j)(t_{j+1}-t_j)}\ud{\sigma_j}\bigg)}\no\\
\ls&\snnm{\nu^{-1}\vh S}\sum_{\ell=1}^{k-1}\int_{\prod_{j=1}^{\ell}\nn_j}\bigg(\int_{t_{\ell}}^{t_{\ell+1}}\abs{
\nu(\vv_\ell)\ue^{\nu(\vv_\ell)(s-(t_{\ell+1}-t_{\ell}))}\ud{s}}\tvh(\vv_\ell)
\prod_{j=1}^{\ell}\ud{\sigma_j}\bigg)\ls CT_0^{\frac{5}{4}}\snnm{\nu^{-1}\vh S}.\no
\end{align}
Collecting all terms in \eqref{ktt 68}, \eqref{ktt 69}, \eqref{ktt 70} and \eqref{ktt 71}, we have
\begin{align}\label{ktt 74}
\text{Boundary Term Contribution}\ls CT_0^{\frac{5}{4}}\ssm{\vh h}{\gamma_-}\ls \ssm{\vh h}{\gamma_-},
\end{align}
and
\begin{align}\label{ktt 75}
\text{Source Term Contribution}\ls CT_0^{\frac{5}{4}}\snnm{\nu^{-1}\vh S}\ls \snnm{\nu^{-1}\vh S}.
\end{align}
\ \\
Step 3: Estimates of Multiple Reflection.\\
We focus on the last term in \eqref{ktt 65}, which can be decomposed based on accumulative time $t_{k+1}$:
\begin{align}
&\abs{\frac{\ue^{-\nu(\vv)t_{1}}}{\tvh(\vv)}\int_{\prod_{j=1}^{k}\nn_j}g(\vx_k,\vv_k)\tvh(\vv_{k})
\bigg(\prod_{j=1}^{k}\ue^{-\nu(\vv_{j})(t_{j+1}-t_{j})}\ud{\sigma_j}\bigg)}\\
\leq&\abs{\frac{\ue^{-\nu(\vv)t_{1}}}{\tvh(\vv)}\int_{\Pi_{j=1}^{k}\nn_j}\id_{\left\{t_{k}\leq
\frac{T_0}{\e}\right\}}g(\vx_k,\vv_k)\tvh(\vv_{k})
\bigg(\prod_{j=1}^{k}\ue^{-\nu(\vv_{j})(t_{j+1}-t_{j})}\ud{\sigma_j}\bigg)}\no\\
&+\abs{\frac{\ue^{-\nu(\vv)t_{1}}}{\tvh(\vv)}\int_{\Pi_{j=1}^{k}\nn_j}\id_{\left\{t_{k}\geq
\frac{T_0}{\e}\right\}}g(\vx_k,\vv_k)\tvh(\vv_{k})
\bigg(\prod_{j=1}^{k}\ue^{-\nu(\vv_{j})(t_{j+1}-t_{j})}\ud{\sigma_j}\bigg)}:=J_1+J_2.\no
\end{align}
Based on Lemma \ref{ktt lemma 2}, we have
\begin{align}\label{ktt 72}
J_1\ls&\snnm{g}\abs{\int_{\Pi_{j=1}^{k-1}\nn_j}\id_{\left\{t_{k+1}\leq
\frac{T_0}{\e}\right\}}\bigg(\int_{\nn_k}\tvh(\vv_{k})\ud\sigma_k\bigg)
\bigg(\prod_{j=1}^{k-1}\ud{\sigma_j}\bigg)}\\
\ls&\snnm{g}\abs{\int_{\Pi_{j=1}^{k-1}\nn_j}\id_{\left\{t_{k+1}\leq
\frac{T_0}{\e}\right\}}
\bigg(\prod_{j=1}^{k-1}\ud{\sigma_j}\bigg)}\ls \bigg(\frac{1}{2}\bigg)^{C_2T_0^{\frac{5}{4}}}\snnm{g}.\no
\end{align}
On the other hand, when $t_k$ is large, the exponential terms become extremely small, so we obtain
\begin{align}\label{ktt 73}
J_2\ls&\snnm{g}\abs{\ue^{-\nu(\vv)t_{1}}\int_{\Pi_{j=1}^{k-1}\nn_j}\id_{\left\{t_{k+1}\geq
\frac{T_0}{\e}\right\}}\bigg(\int_{\nn_k}\tvh(\vv_{k})\ud\sigma_k\bigg)
\bigg(\prod_{j=1}^{k-1}\ue^{-\nu(\vv_{j})(t_{j+1}-t_{j})}\ud{\sigma_j}\bigg)}\\
\ls&\snnm{g}\abs{\ue^{-\nu(\vv)t_{1}}\int_{\Pi_{j=1}^{k-1}\nn_j}\id_{\left\{t_{k+1}\geq
\frac{T_0}{\e}\right\}}
\bigg(\prod_{j=1}^{k-1}\ue^{-\nu(\vv_{j})(t_{j+1}-t_{j})}\ud{\sigma_j}\bigg)}\ls \ue^{-\frac{T_0}{\e}}\snnm{g}.\no
\end{align}
Summarizing \eqref{ktt 72} and \eqref{ktt 73}, we get for $\d$ arbitrarily small
\begin{align}\label{ktt 76}
\text{Multiple Reflection Term Contribution}\ls \d \snnm{g}.
\end{align}
\ \\
Step 4: Estimates of $K_{\vh}$ terms.\\
So far, the only remaining terms in \eqref{ktt 65} are related to $K_{\vh}$. We focus on
\begin{align}
\abs{\int_{0}^{t_{1}}K_{\vh(\vv)}[g]\Big(\vx-\e(t_1-s)\vv,\vv\Big)\ue^{-\nu(\vv)(t_{1}-s)}\ud{s}}&\ls \snnm{K_{\vh(\vv)}[g]\Big(\vx-\e(t_1-s)\vv,\vv\Big)}.
\end{align}
Denote $X(s;\vx,\vv):=\vx-\e(t_1-s)\vv$. Define the back-time stochastic cycle from $(s,X,\vv')$ as $(t_i',\vx_i',\vv_i')$ with $(t_0',\vx_0',\vv_0')=(s,X,\vv')$. Then we can rewrite $K_{\vh}$ along the stochastic cycle as \eqref{ktt 65}
\begin{align}\label{ktt 77}
&\abs{K_{\vh(\vv)}[g]\Big(\vx-\e(t_1-s)\vv,\vv\Big)}=\abs{K_{\vh(\vv)}[g](X,\vv)}=\abs{\int_{\r^3}k_{\vh(\vv)}(\vv,\vv')g(X,\vv')\ud{\vv'}}\\
\leq&\abs{\int_{\r^3}\int_{0}^{t_1'}k_{\vh(\vv)}(\vv,\vv')K_{\vh(\vv')}[g]\Big(X-\e(t_1'-r)\vv',\vv'\Big)\ue^{-\nu(\vv')
(t_1'-r)}\ud{r}\ud{\vv'}}\no\\
&+\abs{\int_{\r^3}\frac{\ue^{-\nu(\vv')t_{1}'}}{\tvh(\vv')}\sum_{\ell=1}^{k-1}\int_{\prod_{j=1}^{\ell}\nn_j'}k_{\vh(\vv)}(\vv,\vv')H_{\ell}[X,\vv']\tvh(\vv_{\ell}')
\bigg(\prod_{j=1}^{\ell}\ue^{-\nu(\vv_j')(t_{j+1}'-t_j')}\ud{\sigma_j'}\bigg)\ud{\vv'}}\no\\
&+\abs{\int_{\r^3}k_{\vh(\vv)}(\vv,\vv')\Big(\text{boundary terms + source terms + multiple reflection terms}\Big)\ud{\vv'}}\no\\
:=&I+II+III.\no
\end{align}
Using estimates \eqref{ktt 74}, \eqref{ktt 75}, \eqref{ktt 76} from Step 2 and Step 3, and Lemma \ref{Regularity lemma 1'}, we can bound $III$ directly
\begin{align}\label{ktt 78}
III\ls \ssm{\vh h}{\gamma_-}+\snnm{\nu^{-1}\vh S}+\d \snnm{g}.
\end{align}
$I$ and $II$ are much more complicated. We may further rewrite $I$ as
\begin{align}
I=&\abs{\int_{\r^3}\int_{\r^3}\int_{0}^{t_1'}k_{\vh(\vv)}(\vv,\vv')k_{\vh(\vv')}(\vv',\vv'')g\Big(X-\e(t_1'-r)\vv',\vv''\Big)\ue^{-\nu(\vv')
(t_1'-r)}\ud{r}\ud{\vv'}\ud{\vv''}},
\end{align}
which will estimated in four cases:
\begin{align}
I:=I_1+I_2+I_3+I_4.
\end{align}
\ \\
Case I: $I_1:$ $\abs{\vv}\geq N$.\\
Based on Lemma \ref{Regularity lemma 1'}, we have
\begin{align}
\abs{\int_{\r^3}\int_{\r^3}k_{\vh(\vv)}(\vv,\vv')k_{\vh(\vv')}(\vv',\vv'')\ud{\vv'}\ud{\vv''}}\ls\frac{1}{1+\abs{\vv}}\ls\frac{1}{N}.
\end{align}
Hence, we get
\begin{align}\label{ktt 87}
I_1\ls\frac{1}{N}\snnm{g}.
\end{align}
\ \\
Case II: $I_2:$ $\abs{\vv}\leq N$, $\abs{\vv'}\geq2N$, or $\abs{\vv'}\leq
2N$, $\abs{\vv''}\geq3N$.\\
Notice this implies either $\abs{\vv'-\vv}\geq N$ or
$\abs{\vv'-\vv''}\geq N$. Hence, either of the following is valid
correspondingly:
\begin{align}
\abs{k_{\vh(\vv)}(\vv,\vv')}\leq& C\ue^{-\d N^2}\abs{k_{\vh(\vv)}(\vv,\vv')}\ue^{\d\abs{\vv-\vv'}^2},\\
\abs{k_{\vh(\vv')}(\vv',\vv'')}\leq& C\ue^{-\d N^2}\abs{k_{\vh(\vv')}(\vv',\vv'')}\ue^{\d\abs{\vv'-\vv''}^2}.
\end{align}
Based on Lemma \ref{Regularity lemma 1'}, we know
\begin{align}
\int_{\r^3}\abs{k_{\vh(\vv)}(\vv,\vv')}\ue^{\d\abs{\vv-\vv'}^2}\ud{\vv'}<&\infty,\\
\int_{\r^3}\abs{k_{\vh(\vv')}(\vv',\vv'')}\ue^{\d\abs{\vv'-\vv''}^2}\ud{\vv''}<&\infty.
\end{align}
Hence, we have
\begin{align}\label{ktt 88}
I_2\ls \ue^{-\d N^2}\snnm{g}.
\end{align}
\ \\
Case III: $I_3:$ $t_1'-r\leq\d$ and $\abs{\vv}\leq N$, $\abs{\vv'}\leq 2N$, $\abs{\vv''}\leq 3N$.\\
In this case, since the integral with respect to $r$ is restricted in a very short interval, there is a small contribution as
\begin{align}\label{ktt 89}
I_3\ls\abs{\int_{t_1'-\d}^{t_1'}\ue^{-(t_1'-r)}\ud{r}}\snnm{g}\ls \d\snnm{g}.
\end{align}
\ \\
Case IV: $I_4:$ $t_1'-r\geq\d$ and $\abs{\vv}\leq N$, $\abs{\vv'}\leq 2N$, $\abs{\vv''}\leq 3N$.\\
This is the most complicated case. Since $k_{\vh(\vv)}(\vv,\vv')$ has
possible integrable singularity of $\dfrac{1}{\abs{\vv-\vv'}}$, we can
introduce the truncated kernel $k_N(\vv,\vv')$ which is smooth and has compactly supported range such that
\begin{align}\label{ktt 79}
\sup_{\abs{\vv}\leq 3N}\int_{\abs{\vv'}\leq
3N}\abs{k_N(\vv,\vv')-k_{\vh(\vv)}(\vv,\vv')}\ud{\vv'}\leq\frac{1}{N}.
\end{align}
Then we can split
\begin{align}\label{ktt 80}
k_{\vh(\vv)}(\vv,\vv')k_{\vh(\vv')}(\vv',\vv'')=&k_N(\vv,\vv')k_N(\vv',\vv'')
+\bigg(k_{\vh(\vv)}(\vv,\vv')-k_N(\vv,\vv')\bigg)k_{\vh(\vv')}(\vv',\vv'')\\
&+\bigg(k_{\vh(\vv')}(\vv',\vv'')-k_N(\vv',\vv'')\bigg)k_N(\vv,\vv').\no
\end{align}
This means that we further split $I_4$ into
\begin{align}
I_4:=I_{4,1}+I_{4,2}+I_{4,3}.
\end{align}
Based on \eqref{ktt 79}, we have
\begin{align}\label{ktt 82}
I_{4,2}\ls&\frac{1}{N}\snnm{g},\quad I_{4,3}\ls\frac{1}{N}\snnm{g}.
\end{align}
Therefore, the only remaining term is $I_{4,1}$. Note that we always have $X-\e(t_1'-r)\vv'\in\Omega$. Hence, we define the change of variable $\vv'\rt y$ as
$y=(y_1,y_2,y_3)=X-\e(t_1'-r)\vv'$. Then the Jacobian
\begin{align}\label{ktt 81}
\abs{\frac{\ud{y}}{\ud{\vv'}}}=\abs{\left\vert\begin{array}{ccc}
\e(t_1'-r)&0&0\\
0&\e(t_1'-r)&0\\
0&0&\e(t_1'-r)
\end{array}\right\vert}=\e^3(t_1'-r)^3\geq \e^3\d^3.
\end{align}
Considering $\abs{\vv},\abs{\vv'},\abs{\vv''}\leq 3N$, we know $\abs{g}\ls\abs{f}$. Also, since $k_N$ is bounded, we estimate
\begin{align}\label{ktt 85}
I_{4,1}\ls&
\int_{\abs{\vv'}\leq2N}\int_{\abs{\vv''}\leq3N}\int_{0}^{t_1'}
\id_{\{X-\e(t_1'-r)\vv'\in\Omega\}}\abs{f(X-\e(t_1'-r)\vv',\vv'')}\ue^{-\nu(\vv')
(t_1'-r)}\ud{r}\ud{\vv'}\ud{\vv''}.
\end{align}
Using the decomposition $f=\pk[f]+(\ik-\pk)[f]$, \eqref{ktt 81} and H\"older's inequality, we estimate them separately,
\begin{align}\label{ktt 83}
&\int_{\abs{\vv'}\leq2N}\int_{\abs{\vv''}\leq3N}\int_{0}^{t_1'}
\id_{\{X-\e(t_1'-r)\vv'\in\Omega\}}\abs{\pk[f](X-\e(t_1'-r)\vv',\vv'')}\ue^{-\nu(\vv')
(t_1'-r)}\ud{r}\ud{\vv'}\ud{\vv''}\\
\leq&\bigg(\int_{\abs{\vv'}\leq2N}\int_{\abs{\vv''}\leq3N}\int_{0}^{t_1'}
\id_{\{X-\e(t_1'-r)\vv'\in\Omega\}}\ue^{-\nu(\vv')
(t_1'-r)}\ud{r}\ud{\vv'}\ud{\vv''}\bigg)^{\frac{2m-1}{2m}}\no\\
&\times\bigg(\int_{\abs{\vv'}\leq2N}\int_{\abs{\vv''}\leq3N}\int_{0}^{t_1'}
\id_{\{X-\e(t_1'-r)\vv'\in\Omega\}}\Big(\pk[f]\Big)^{2m}(X-\e(t_1'-r)\vv',\vv'')\ue^{-\nu(\vv')
(t_1'-r)}\ud{r}\ud{\vv'}\ud{\vv''}\bigg)^{\frac{1}{2m}}\no\\
\ls&\abs{\int_{0}^{t_1'}\frac{1}{\e^3\d^3}\int_{\abs{\vv''}\leq3N}\int_{\Omega}\id_{\{
y\in\Omega\}}\Big(\pk[f]\Big)^{2m}(y,\vv'')\ue^{-(t_1'-r)}\ud{y}\ud{\vv''}\ud{r}}^{\frac{1}{2m}}\ls \frac{1}{\e^{\frac{3}{2m}}\d^{\frac{3}{2m}}}\pnnm{\pk[f]}{2m},\no
\end{align}
and
\begin{align}\label{ktt 84}
&\int_{\abs{\vv'}\leq2N}\int_{\abs{\vv''}\leq3N}\int_{0}^{t_1'}
\id_{\{X-\e(t_1'-r)\vv'\in\Omega\}}\abs{(\ik-\pk)[f](X-\e(t_1'-r)\vv',\vv'')}\ue^{-\nu(\vv')
(t_1'-r)}\ud{r}\ud{\vv'}\ud{\vv''}\\
\leq&\bigg(\int_{\abs{\vv'}\leq2N}\int_{\abs{\vv''}\leq3N}\int_{0}^{t_1'}
\id_{\{X-\e(t_1'-r)\vv'\in\Omega\}}\ue^{-\nu(\vv')
(t_1'-r)}\ud{r}\ud{\vv'}\ud{\vv''}\bigg)^{\frac{1}{2}}\no\\
&\times\bigg(\int_{\abs{\vv'}\leq2N}\int_{\abs{\vv''}\leq3N}\int_{0}^{t_1'}
\id_{\{X-\e(t_1'-r)\vv'\in\Omega\}}\Big((\ik-\pk)[f]\Big)^{2}(X-\e(t_1'-r)\vv',\vv'')\ue^{-\nu(\vv')
(t_1'-r)}\ud{r}\ud{\vv'}\ud{\vv''}\bigg)^{\frac{1}{2}}\no\\
\ls&\abs{\int_{0}^{t_1'}\frac{1}{\e^3\d^3}\int_{\abs{\vv''}\leq3N}
\int_{\Omega}\id_{\{y\in\Omega\}}\Big((\ik-\pk)[f]\Big)^{2}(y,\vv'')\ue^{-(t_1'-r)}\ud{y}\ud{\vv''}\ud{r}}^{\frac{1}{2}}
\ls \frac{1}{\e^{\frac{3}{2}}\d^{\frac{3}{2}}}\tnnm{(\ik-\pk)[f]}.\no
\end{align}
Inserting \eqref{ktt 83} and \eqref{ktt 84} into \eqref{ktt 85}, we obtain
\begin{align}\label{ktt 86}
I_{4,1}\ls \frac{1}{\e^{\frac{3}{2m}}\d^{\frac{3}{2m}}}\pnnm{\pk[f]}{2m}+\frac{1}{\e^{\frac{3}{2}}\d^{\frac{3}{2}}}\tnnm{(\ik-\pk)[f]}.
\end{align}
Combined with \eqref{ktt 82}, we know
\begin{align}\label{ktt 90}
I_4\ls \frac{1}{N}\snnm{g}+\frac{1}{\e^{\frac{3}{2m}}\d^{\frac{3}{2m}}}\pnnm{\pk[f]}{2m}+\frac{1}{\e^{\frac{3}{2}}\d^{\frac{3}{2}}}\tnnm{(\ik-\pk)[f]}.
\end{align}
Summarizing all four cases in \eqref{ktt 87}, \eqref{ktt 88}, \eqref{ktt 89} and \eqref{ktt 90}, we obtain
\begin{align}
I\ls \bigg(\frac{1}{N}+\ue^{-\d N^2}+\d\bigg)\snnm{g}+\frac{1}{\e^{\frac{3}{2m}}\d^{\frac{3}{2m}}}\pnnm{\pk[f]}{2m}+\frac{1}{\e^{\frac{3}{2}}\d^{\frac{3}{2}}}\tnnm{(\ik-\pk)[f]}.
\end{align}
Choosing $\d$ sufficiently small and then taking $N$ sufficiently large, we have
\begin{align}\label{ktt 91}
I\ls \d\snnm{g}+\frac{1}{\e^{\frac{3}{2m}}\d^{\frac{3}{2m}}}\pnnm{\pk[f]}{2m}+\frac{1}{\e^{\frac{3}{2}}\d^{\frac{3}{2}}}\tnnm{(\ik-\pk)[f]}.
\end{align}
By a similar but tedious computation, we arrive at
\begin{align}\label{ktt 92}
II\ls \d\snnm{g}+\frac{1}{\e^{\frac{3}{2m}}\d^{\frac{3}{2m}}}\pnnm{\pk[f]}{2m}+\frac{1}{\e^{\frac{3}{2}}\d^{\frac{3}{2}}}\tnnm{(\ik-\pk)[f]}.
\end{align}
Combined with \eqref{ktt 78}, we have
\begin{align}
\abs{\int_{0}^{t_{1}}K_{\vh(\vv)}[g]\Big(\vx-\e(t_1-s)\vv,\vv\Big)\ue^{-\nu(\vv)(t_{1}-s)}\ud{s}}
&\ls\d\snnm{g}+\frac{1}{\e^{\frac{3}{2m}}\d^{\frac{3}{2m}}}\pnnm{\pk[f]}{2m}+\frac{1}{\e^{\frac{3}{2}}\d^{\frac{3}{2}}}\tnnm{(\ik-\pk)[f]}\\
&+\ssm{\vh h}{\gamma_-}+\snnm{\nu^{-1}\vh S}.\no
\end{align}
All the other terms in \eqref{ktt 65} related to $K_{\vh}$ can be estimated in a similar fashion. At the end of the day, we have
\begin{align}\label{ktt 93}
\\
K_{\vh}\ \text{term contribution}\ls \d\snnm{g}+\frac{1}{\e^{\frac{3}{2m}}\d^{\frac{3}{2m}}}\pnnm{\pk[f]}{2m}+\frac{1}{\e^{\frac{3}{2}}\d^{\frac{3}{2}}}\tnnm{(\ik-\pk)[f]}+\ssm{\vh h}{\gamma_-}+\snnm{\nu^{-1}\vh S}.\no
\end{align}
\ \\
Step 5: Synthesis.\\
Summarizing all above and inserting \eqref{ktt 74}, \eqref{ktt 75}, \eqref{ktt 76} and \eqref{ktt 93} into \eqref{ktt 65}, we obtain for any $(\vx,\vv)\in\bar\Omega\times\r^3$,
\begin{align}\label{ctt 2}
\abs{g(\vx,\vv)}\ls\d\snnm{g}+\frac{1}{\e^{\frac{3}{2m}}\d^{\frac{3}{2m}}}\pnnm{\pk[f]}{2m}+\frac{1}{\e^{\frac{3}{2}}\d^{\frac{3}{2}}}\tnnm{(\ik-\pk)[f]}
+\ssm{\vh h}{\gamma_-}+\snnm{\nu^{-1}\vh S}.
\end{align}
Taking supremum over $(\vx,\vv)\in\gamma_+$ in \eqref{ctt 2}, we have
\begin{align}
\ssm{g}{\gamma_+}\ls \d\snnm{g}+\frac{1}{\e^{\frac{3}{2m}}\d^{\frac{3}{2m}}}\pnnm{\pk[f]}{2m}+\frac{1}{\e^{\frac{3}{2}}\d^{\frac{3}{2}}}\tnnm{(\ik-\pk)[f]}
+\ssm{\vh h}{\gamma_-}+\snnm{\nu^{-1}\vh S}.
\end{align}
Based on Theorem \ref{LN estimate}, for $\dfrac{3}{2}< m<3$, we obtain
\begin{align}
\ssm{g}{\gamma_+}\ls&\d\snnm{g}+o(1)\Big(\ssm{f}{\gamma_+}+\snnm{f}\Big)+E
\ls\d\snnm{g}+o(1)\Big(\ssm{g}{\gamma_+}+\snnm{g}\Big)+E,
\end{align}
where
\begin{align}
E:=&\frac{1}{\e^{2+\frac{3}{2m}}}\pnnm{\pk[S]}{\frac{2m}{2m-1}}+\frac{1}{\e^{1+\frac{3}{2m}}}\tnnm{\snn(\ik-\pk)[S]}+\snnm{\nu^{-1}\vh S}\\
&+\frac{1}{\e^{\frac{3}{2m}}}\pnm{h}{\gamma_-}{\frac{4m}{3}}+\frac{1}{\e^{1+\frac{3}{2m}}}\tsm{h}{\gamma_-}+\ssm{\vh h}{\gamma_-}.\no
\end{align}
Absorbing $o(1)\ssm{g}{\gamma_+}$ into the left-hand side, we have
\begin{align}\label{ctt 1}
\ssm{g}{\gamma_+}\ls \d\snnm{g}+o(1)\snnm{g}+E.
\end{align}
On the other hand, taking supremum over $(\vx,\vv)\in\Omega\times\r^3$ in \eqref{ctt 2}, we have
\begin{align}
\snnm{g}\ls \d\snnm{g}+\frac{1}{\e^{\frac{3}{2m}}\d^{\frac{3}{2m}}}\pnnm{\pk[f]}{2m}+\frac{1}{\e^{\frac{3}{2}}\d^{\frac{3}{2}}}\tnnm{(\ik-\pk)[f]}
+\ssm{\vh h}{\gamma_-}+\snnm{\nu^{-1}\vh S}.
\end{align}
Based on Theorem \ref{LN estimate}, we obtain
\begin{align}
\snnm{g}\ls&\d\snnm{g}+o(1)\Big(\ssm{g}{\gamma_+}+\snnm{g}\Big)+E.
\end{align}
Absorbing $\d\snnm{g}$ and $o(1)\snnm{g}$ into the left-hand side, we have
\begin{align}\label{ctt 3}
\snnm{g}\ls&o(1)\ssm{g}{\gamma_+}+E.
\end{align}
Inserting \eqref{ctt 1} into \eqref{ctt 3}, and absorbing $\d\snnm{g}$ and $o(1)\snnm{g}$ into the left-hand side, we get
\begin{align}
\snnm{g}\ls E.
\end{align}
Then \eqref{ctt 1} implies
\begin{align}
\ssm{g}{\gamma_+}\ls E.
\end{align}
In summary, we have
\begin{align}
\snnm{g}+\ssm{g}{\gamma_+}\ls& \frac{1}{\e^{2+\frac{3}{2m}}}\pnnm{\pk[S]}{\frac{2m}{2m-1}}+\frac{1}{\e^{1+\frac{3}{2m}}}\tnnm{\snn(\ik-\pk)[S]}+\snnm{\nu^{-1}\vh S}\\
&+\frac{1}{\e^{\frac{3}{2m}}}\pnm{h}{\gamma_-}{\frac{4m}{3}}+\frac{1}{\e^{1+\frac{3}{2m}}}\tsm{h}{\gamma_-}+\ssm{\vh h}{\gamma_-}.\no
\end{align}
Then our result naturally follows.
\end{proof}

\begin{remark}\label{LN remark}
Inserting Theorem \ref{LI estimate} into Theorem \ref{LN estimate}, we actually have
\begin{align}
&\frac{1}{\e^{\frac{1}{2}}}\tsm{(1-\pp)[f]}{\gamma_+}+\frac{1}{\e}\unnm{(\ik-\pk)[f]}+\pnnm{\pk[f]}{2m}\\
\ls& \frac{1}{\e^{2}}\pnnm{\pk[S]}{\frac{2m}{2m-1}}+\frac{1}{\e}\tnnm{\snn(\ik-\pk)[S]}+\lnnmv{\nu^{-1}S}\no\\
&+\pnm{h}{\gamma_-}{\frac{4m}{3}}+\frac{1}{\e}\tsm{h}{\gamma_-}+\lsm{h}{\gamma_-}.\no
\end{align}
\end{remark}

\newpage

\section{Hydrodynamic Limit}

\subsection{Nonlinear Estimates}

\begin{lemma}\label{nonlinear lemma}
The nonlinear term $\Gamma$ defined in \eqref{att 31} satisfies $\Gamma[f,g]\in\nk^{\perp}$. Also, for $0\leq\varrho<\dfrac{1}{4}$ and $\vth\geq0$,
\begin{align}
&\tnnm{\Gamma[f,g]}\ls \Big(\sup_{\vx\in\Omega}\tnm{\nu g(x)}\Big)\tnnm{\nu f},\label{ftt 02}\\
&\lnnmv{\nu^{-1}\Gamma[f,g]}\ls\lnnmv{f}\lnnmv{g},\label{ftt 01}.
\end{align}
\end{lemma}
\begin{proof}
The orthogonality is shown in \cite[Section 3.8]{Glassey1996}. \eqref{ftt 02} can be shown following the idea in \cite[Lemma 2.3]{Guo2002}. From \eqref{att 31},
\begin{align}
\Gamma[f,g]:=&\m^{-\frac{1}{2}}Q\Big[\m^{\frac{1}{2}}f,\m^{\frac{1}{2}}g\Big]=\Gamma_{\text{gain}}[f,g]-\Gamma_{\text{loss}}[f,g],
\end{align}
where using the energy conservation $\abs{\vuu}^2+\abs{\vv}^2=\abs{\vuu_{\ast}}^2+\abs{\vv_{\ast}}^2$,
\begin{align}
\Gamma_{\text{gain}}[f,g]&:=q_0\int_{\r^3}\int_{\s^2}\ue^{-\frac{\abs{\vuu}^2}{2}}\Big(\vo\cdot(\vv-\vuu)\Big)f(\vuu_{\ast})g(\vv_{\ast})\ud{\vo}\ud{\vuu},\\
\Gamma_{\text{loss}}[f,g]&:=q_0\int_{\r^3}\int_{\s^2}\ue^{-\frac{\abs{\vuu}^2}{2}}\Big(\vo\cdot(\vv-\vuu)\Big)f(\vuu)g(\vv)\ud{\vo}\ud{\vuu},
\end{align}
with
\begin{align}
\vuu_{\ast}:=\vuu+\vo\bigg((\vv-\vuu)\cdot\vo\bigg),\qquad \vv_{\ast}:=\vv-\vo\bigg((\vv-\vuu)\cdot\vo\bigg).
\end{align}
For the loss term, we substitute $u=\vv-\vuu$, so we know
\begin{align}
\Gamma_{\text{loss}}[f,g]&:=q_0g(\vv)\int_{\r^3}\int_{\s^2}\ue^{-\frac{\abs{v-u}^2}{2}}(\vo\cdot u)f(v-u)\ud{\vo}\ud{u}.
\end{align}
Hence, using H\"older's inequality, we have
\begin{align}\label{ftt 27}
\int_{\r^3}\bigg(\Gamma_{\text{loss}}[f,g](\vx)\bigg)^2\ud\vv=&q_0^2\int_{\r^3}g^2(\vx,\vv)\bigg(\int_{\r^3}\int_{\s^2}\ue^{-\frac{\abs{v-u}^2}{2}}(\vo\cdot u)f(\vx,v-u)\ud{\vo}\ud{u}\bigg)^2\ud\vv\\
\ls&\int_{\r^3}g^2(\vx,\vv)\bigg(\int_{\r^3}\ue^{-\abs{v-u}^2}\abs{u}^2\ud{u}\bigg)\bigg(\int_{\r^3}f^2(\vx,v-u)\ud{u}\bigg)\ud\vv\no\\
\ls&\tnm{f(x)}^2\tnm{\nu g(x)}^2,\no
\end{align}
where we utilize the fact that
\begin{align}
\int_{\r^3}\ue^{-\abs{v-u}^2}\abs{u}^2\ud{u}\ls \nu^2(\vv).
\end{align}
On the other hand, for the gain term, after substituting $u=\vv-\vuu$, we know
\begin{align}
\Gamma_{\text{gain}}[f,g]&:=q_0\int_{\r^3}\int_{\s^2}\ue^{-\frac{\abs{v-u}^2}{2}}(\vo\cdot u)f(\vv-u_{\perp})g(\vv-u_{\parallel})\ud{\vo}\ud{u},
\end{align}
where
\begin{align}
u_{\perp}=u-\vo(u\cdot\vo),\quad u_{\parallel}=\vo(u\cdot\vo).
\end{align}
Hence, using H\"older's inequality, we have
\begin{align}
\int_{\r^3}\bigg(\Gamma_{\text{gain}}[f,g](\vx)\bigg)^2\ud\vv=&q_0^2\int_{\r^3}\bigg(\int_{\r^3}\int_{\s^2}\ue^{-\frac{\abs{v-u}^2}{2}}(\vo\cdot u)f(\vx,v-u_{\perp})g(\vx,\vv-u_{\parallel})\ud{\vo}\ud{u}\bigg)^2\ud\vv\\
\ls&\int_{\r^3}\bigg(\int_{\r^3}\ue^{-\abs{v-u}^2}\abs{u}^2\ud{u}\bigg)\bigg(\int_{\r^3}f^2(\vx,\vv-u_{\perp})g^2(\vx,\vv-u_{\parallel})\ud{u}\bigg)\ud\vv\no\\
\ls&\int_{\r^3}\int_{\r^3}\nu^2(\vv)f^2(\vx,\vv-u_{\perp})g^2(\vx,\vv-u_{\parallel})\ud{u}\ud\vv\no
\end{align}
Denote $u'=\vv-u_{\perp}$ and $v'=\vv-u_{\parallel}$. Consider substitution $(u,v)\rt(u',v')$. It is well-known (see the proof of \cite[Lemma 2.3]{Guo2002}) that $\ud u\ud v=\ud u'\ud v'$ and $\abs{v}\ls\abs{u'}+\abs{v'}$. Hence, we have
\begin{align}\label{ftt 28}
\int_{\r^3}\bigg(\Gamma_{\text{gain}}[f,g](\vx)\bigg)^2\ud\vv\ls&\int_{\r^3}\int_{\r^3}\Big(\nu^2(u')+\nu^2(v')\Big)f^2(\vx,u')g^2(\vx,v')\ud{u'}\ud v'\\
\ls&\bigg(\int_{\r^3}\nu^2(u')f^2(\vx,u')\ud u'\bigg)\bigg(\int_{\r^3}\nu^2(v')g^2(\vx,v')\ud v'\bigg)\no\\
\ls&\tnm{\nu f(x)}^2\tnm{\nu g(x)}^2.\no
\end{align}
Combining \eqref{ftt 27} and \eqref{ftt 28}, we know
\begin{align}
\int_{\r^3}\bigg(\Gamma[f,g](\vx)\bigg)^2\ud\vv\ls \tnm{\nu f(x)}^2\tnm{\nu g(x)}^2,
\end{align}
which further implies
\begin{align}
\int_{\Omega}\int_{\r^3}\bigg(\Gamma[f,g]\bigg)^2\ud\vv\ud\vx\ls \Big(\sup_{\vx\in\Omega}\tnm{\nu g(x)}^2\Big)\tnnm{\nu f}^2.
\end{align}
Therefore, \eqref{ftt 02} naturally follows. Also, \eqref{ftt 01} is proved in \cite[Lemma 5]{Guo2010}.
\end{proof}

\subsection{Perturbed Remainder Estimates}

We consider the perturbed linearized stationary Boltzmann equation
\begin{align}\label{nonlinear steady}
\left\{
\begin{array}{l}
\e\vv\cdot\nx f+\ll[f]=\Gamma[f,g]+S(\vx,\vv)\ \ \text{in}\ \ \Omega\times\r^3,\\\rule{0ex}{1.5em}
f(\vx_0,\vv)=\pp[f](\vx_0,\vv)+(\mb-\m)\m^{-1}\pp[f]+h(\vx_0,\vv)\ \ \text{for}\ \ \vx_0\in\p\Omega\ \
\text{and}\ \ \vv\cdot\vn<0.
\end{array}
\right.
\end{align}
Assume that a priori
\begin{align}\label{perturbed normalization}
\iint_{\Omega\times\r^3}f(\vx,\vv)\m^{\frac{1}{2}}(\vv)\ud{\vv}\ud{\vx}=&0.
\end{align}
and
\begin{align}\label{perturbed smallness}
\lnnmv{g}=o(1)\e.
\end{align}
The data $S$ and $h$ satisfy the compatibility condition
\begin{align}\label{perturbed compatibility}
\iint_{\Omega\times\r^3}S(\vx,\vv)\m^{\frac{1}{2}}(\vv)\ud{\vv}\ud{\vx}+\int_{\gamma_-}h(\vx,\vv)\m^{\frac{1}{2}}(\vv)\ud{\gamma}=0.
\end{align}

\begin{theorem}\label{LN estimate.}
Assume \eqref{linear steady compatibility} and \eqref{linear steady normalization} hold. The solution $f(\vx,\vv)$ to the equation \eqref{nonlinear steady} satisfies
\begin{align}
&\frac{1}{\e^{\frac{1}{2}}}\tsm{(1-\pp)[f]}{\gamma_+}+\frac{1}{\e}\unnm{(\ik-\pk)[f]}+\pnnm{\pk[f]}{2m}\\
\leq& o(1)\e^{\frac{3}{2m}}\Big(\ssm{f}{\gamma_+}+\snnm{f}\Big)
+\frac{1}{\e^{2}}\pnnm{\pk[S]}{\frac{2m}{2m-1}}+\frac{1}{\e}\tnnm{\snn(\ik-\pk)[S]}+\pnm{h}{\gamma_-}{\frac{4m}{3}}+\frac{1}{\e}\tsm{h}{\gamma_-}.\no
\end{align}
\end{theorem}
\begin{proof}
Since the perturbed term $\Gamma[f,g]\in\nk^{\perp}$, we apply Theorem \ref{LN estimate} to \eqref{nonlinear steady} to obtain
\begin{align}\label{ftt 04}
&\frac{1}{\e^{\frac{1}{2}}}\tsm{(1-\pp)[f]}{\gamma_+}+\frac{1}{\e}\unnm{(\ik-\pk)[f]}+\pnnm{\pk[f]}{2m}\\
\leq& o(1)\e^{\frac{3}{2m}}\Big(\ssm{f}{\gamma_+}+\snnm{f}\Big)
+\frac{1}{\e^{2}}\pnnm{\pk[S]}{\frac{2m}{2m-1}}+\frac{1}{\e}\tnnm{\snn(\ik-\pk)[S]}+\pnm{h}{\gamma_-}{\frac{4m}{3}}+\frac{1}{\e}\tsm{h}{\gamma_-}\no\\
&+\frac{1}{\e}\tnnm{\snn\Gamma[f,g]}+\pnm{(\mb-\m)\m^{-1}\pp[f]}{\gamma_-}{\frac{4m}{3}}+\frac{1}{\e}\tsm{(\mb-\m)\m^{-1}\pp[f]}{\gamma_-}.\no
\end{align}
Using Lemma \ref{nonlinear lemma} and \eqref{perturbed smallness}, we have
\begin{align}\label{ftt 03}
\frac{1}{\e}\tnnm{\snn\Gamma[f,g]}\ls o(1)\tnnm{\sn f}\ls o(1)\unnm{\pk[f]}+o(1)\unnm{(\ik-\pk)[f]}.
\end{align}
Note that direct computation reveals that
\begin{align}
\pnnm{\pk[f]}{2m}\gs\unnm{\pk[f]},
\end{align}
so inserting \eqref{ftt 03} into \eqref{ftt 04}, we can absorb $o(1)\unnm{\pk[f]}$ and $o(1)\unnm{(\ik-\pk)[f]}$ into the left-hand side. On the other hand, due to \eqref{smallness assumption}, we know
\begin{align}\label{ftt 18}
\pnm{(\mb-\m)\m^{-1}\pp[f]}{\gamma_-}{\frac{4m}{3}}+\frac{1}{\e}\tsm{(\mb-\m)\m^{-1}\pp[f]}{\gamma_-}&\ls o(1)\e\pnm{\pp[f]}{\gamma_-}{\frac{4m}{3}}+o(1)\tsm{\pp[f]}{\gamma_-}\\
&\ls o(1)\e\ssm{f}{\gamma_+}+o(1)\tsm{\pp[f]}{\gamma_+}.\no
\end{align}
Here, $o(1)\ssm{f}{\gamma_+}$ can be combined with the corresponding term on the right-hand side of \eqref{ftt 04}. Also, the bound of $\tsm{\pp[f]}{\gamma_+}$ has been achieved in the proof of Theorem \ref{LN estimate}. Inserting \eqref{ktt 47} into \eqref{ftt 04}, using \eqref{ftt 18}, H\"older's inequality and Theorem \ref{LN estimate}, we know
\begin{align}
\tsm{\pp[f]}{\gamma_+}\ls&\tnnm{\pk[f]}+\frac{1}{\e}\tnnm{(\ik-\pk)[f]}+\frac{1}{\e^{\frac{1}{2}}}\bigg(\abs{\int_{\Omega\times\r^3}fS}\bigg)^{\frac{1}{2}}\\
\ls&o(1)\e^{\frac{3}{2m}}\Big(\ssm{f}{\gamma_+}+\snnm{f}\Big)\no\\
&+\frac{1}{\e^{2}}\pnnm{\pk[S]}{\frac{2m}{2m-1}}+\frac{1}{\e}\tnnm{\snn(\ik-\pk)[S]}+\pnm{h}{\gamma_-}{\frac{4m}{3}}+\frac{1}{\e}\tsm{h}{\gamma_-}+o(1)\tsm{\pp[f]}{\gamma_+}.\no
\end{align}
Then absorbing $o(1)\tsm{\pp[f]}{\gamma_+}$ into the left-hand side, we get control of $\tsm{\pp[f]}{\gamma_+}$. Then inserting it into \eqref{ftt 18} and further \eqref{ftt 04}, we get the desired result.
\end{proof}

\begin{theorem}\label{LI estimate.}
Assume \eqref{linear steady compatibility} and \eqref{linear steady normalization} hold. The solution $f(\vx,\vv)$ to the equation \eqref{nonlinear steady} satisfies for $\vth\geq0$ and $0\leq\varrho<\dfrac{1}{4}$,
\begin{align}
\lnnmv{f}+\lsm{f}{\gamma_+}
\ls& \frac{1}{\e^{2+\frac{3}{2m}}}\pnnm{\pk[S]}{\frac{2m}{2m-1}}+\frac{1}{\e^{1+\frac{3}{2m}}}\tnnm{\snn(\ik-\pk)[S]}+\lnnmv{\nu^{-1} S}\\
&+\frac{1}{\e^{\frac{3}{2m}}}\pnm{h}{\gamma_-}{\frac{4m}{3}}+\frac{1}{\e^{1+\frac{3}{2m}}}\tsm{h}{\gamma_-}+\lsm{h}{\gamma_-}.\no
\end{align}
\end{theorem}
\begin{proof}
Since we already have bounds for $f$ in $L^{2m}$ as Theorem \ref{LN estimate.}, following the proof of Theorem \ref{LI estimate}, we obtain
\begin{align}\label{ftt 06}
\lnnmv{f}+\lsm{f}{\gamma_+}
\ls& \frac{1}{\e^{2+\frac{3}{2m}}}\pnnm{\pk[S]}{\frac{2m}{2m-1}}+\frac{1}{\e^{1+\frac{3}{2m}}}\tnnm{\snn(\ik-\pk)[S]}+\lnnmv{\nu^{-1}S}\\
&+\frac{1}{\e^{\frac{3}{2m}}}\pnm{h}{\gamma_-}{\frac{4m}{3}}+\frac{1}{\e^{1+\frac{3}{2m}}}\tsm{h}{\gamma_-}+\lsm{h}{\gamma_-}\no\\
&+\lnnmv{\nu^{-1}\Gamma[f,g]}+\lsm{(\mb-\m)\m^{-1}\pp[f]}{\gamma_-}.\no
\end{align}
Using Lemma \ref{nonlinear lemma} and \eqref{perturbed smallness}, we have
\begin{align}\label{ftt 05}
\lnnmv{\nu^{-1}\Gamma[f,g]}\ls \lnnmv{f}\lnnmv{g}\ls o(1)\lnnmv{f}.
\end{align}
Inserting \eqref{ftt 05} into \eqref{ftt 06}, we can absorb $o(1)\lnnmv{f}$ into the left-hand side. Also, using \eqref{smallness assumption}, we have
\begin{align}\label{ftt 19}
\lsm{(\mb-\m)\m^{-1}\pp[f]}{\gamma_-}\ls o(1)\lsm{f}{\gamma_+}.
\end{align}
Inserting \eqref{ftt 19} into \eqref{ftt 06} and absorbing $o(1)\lsm{f}{\gamma_+}$ into the left-hand side, we obtain the desired result.
\end{proof}

\subsection{Analysis of Asymptotic Expansion}

Based on the construction of interior solutions in Section \ref{att section 03}, we know $\f_1$, $\f_2$ and $\f_3$ satisfy certain fluid equations. For small data, the well-posedness and regularity of these equations are well-know, so we omit the proof and only present the main results.
\begin{theorem}\label{limit theorem 2}
For $K_0>0$ sufficiently small, the boundary layer satisfies
\begin{align}
\nm{\bv\f_1}_{H^3_xL^{\infty}_v}\ls 1,\quad\nm{\bv\f_2}_{H^3_xL^{\infty}_v}\ls 1,\quad\nm{\bv\f_3}_{H^3_xL^{\infty}_v}\ls 1.
\end{align}
\end{theorem}
On the other hand, based on the construction of boundary layers in Section \ref{att section 03}, we know $\fb_1=0$ and $\fb_2$ is well-defined. Using Theorem \ref{Milne theorem 4}, Theorem \ref{Regularity theorem 2}, Theorem \ref{Milne tangential} and Theorem \ref{Milne velocity}, we have for $0\leq\vrh<\dfrac{1}{4}$ and $\vth>3$,
\begin{theorem}\label{limit theorem 1}
For $K_0>0$ sufficiently small, the boundary layer $\fb_2$ satisfies
\begin{align}
\lnnmv{\ue^{K_0\eta}\fb_2}\ls 1,
\end{align}
and
\begin{align}
\begin{array}{ll}
\lnnmv{\ue^{K_0\eta}\va\dfrac{\p\fb_2}{\p\eta}}+\lnnmv{\ue^{K_0\eta}\dfrac{\p\fb_2}{\p\iota_1}}+\lnnmv{\ue^{K_0\eta}\dfrac{\p\fb_2}{\p\iota_2}}\ls \abs{\ln(\e)}^8,\\\rule{0ex}{2.0em}
\lnnmv{\ue^{K_0\eta}\nu\dfrac{\p\fb_2}{\p\va}}+\lnnmv{\ue^{K_0\eta}\nu\dfrac{\p\fb_2}{\p\vb}}+\lnnmv{\ue^{K_0\eta}\nu\dfrac{\p\fb_2}{\p\vc}}\ls\abs{\ln(\e)}^8.
\end{array}
\end{align}
\end{theorem}

\begin{remark}
Note that the norms defined in studying $\e$-Milne problem with geometric correction can naturally be extended to include $(\iota_1,\iota_2)$ dependence and are consistent with the current format.
\end{remark}

\subsection{Proof of Main Theorem}

Now we turn to the proof of the main result, Theorem \ref{main}.\\
\ \\
Step 1: Remainder definitions.\\
Define the remainder as
\begin{align}\label{ftt 07}
\e^3 R:=&f^{\e}-\Big(\e\f_1+\e^2\f_2+\e^3\f_3\Big)-\Big(\e\fb_1+\e^2\fb_2\Big)=f^{\e}-\q -\qb,
\end{align}
where
\begin{align}
\q:=&\e\f_1+\e^2\f_2+\e^3\f_3,\quad \qb:=\e\fb_1+\e^2\fb_2.
\end{align}
We write $\lll$ to denote the linearized Boltzmann operator
\begin{align}
\lll[f]=&\e\vv\cdot\nx f+\ll[f].
\end{align}
In studying boundary layers, we use substitutions to rewrite $\lll$ into normal and tangential component as in \eqref{small system'}:
\begin{align}
\lll[f]=&\va\dfrac{\p f}{\p\eta}-\dfrac{\e}{R_1-\e\eta}\bigg(\vb^2\dfrac{\p f}{\p\va}-\va\vb\dfrac{\p f}{\p\vb}\bigg)
-\dfrac{\e}{R_2-\e\eta}\bigg(\vc^2\dfrac{\p f}{\p\va}-\va\vc\dfrac{\p f}{\p\vc}\bigg)\\\rule{0ex}{2.0em}
&-\dfrac{\e}{P_1P_2}\Bigg(\dfrac{\p_{11}\vr\cdot\p_2\vr}{P_1(\e\kk_1\eta-1)}\vb\vc
+\dfrac{\p_{12}\vr\cdot\p_2\vr}{P_2(\e\kk_2\eta-1)}\vc^2\Bigg)\dfrac{\p f}{\p\vb}\no\\\rule{0ex}{2.0em}
&-\dfrac{\e}{P_1P_2}\Bigg(\dfrac{\p_{22}\vr\cdot\p_1\vr}{P_2(\e\kk_2\eta-1)}\vb\vc
+\dfrac{\p_{12}\vr\cdot\p_1\vr}{P_1(\e\kk_1\eta-1)}\vb^2\Bigg)\dfrac{\p f}{\p\vc}\no\\\rule{0ex}{2.0em}
&-\e\bigg(\dfrac{\vb}{P_1(\e\kk_1\eta-1)}\dfrac{\p f}{\p\iota_1}+\dfrac{\vc}{P_2(\e\kk_2\eta-1)}\dfrac{\p f}{\p\iota_2}\bigg)
+\ll[f].\no
\end{align}
\ \\
Step 2: Representation of $\lll[R]$.\\
The equation \eqref{small system} is actually
\begin{align}
\lll[f^{\e}]=\Gamma[f^{\e},f^{\e}],
\end{align}
which means
\begin{align}\label{ftt 08}
\lll[\q+\qb+\e^3R]=\Gamma[\q+\qb+\e^3R,\q+\qb+\e^3R].
\end{align}
In \eqref{ftt 08}, the nonlinear terms on the right-hand side can be decomposed as
\begin{align}\label{ftt 09}
\Gamma[\q+\qb+\e^3R,\q+\qb+\e^3R]=&\e^6\Gamma[R,R]+2\e^3\Gamma[R,\q+\qb]+\Gamma[\q+\qb,\q+\qb].
\end{align}
For the left-hand side of \eqref{ftt 08}, based on the construction of interior solutions in Section \ref{att section 03}, the interior contribution
\begin{align}\label{ftt 10}
\lll[\q]=&\e\vv\cdot\nx\Big(\e\f_1+\e^2\f_2+\e^3\f_3\Big) +\ll[\e\f_1+\e^2\f_2+\e^3\f_3]\\
=&\e^4\vv\cdot\nx\f_3+\e^2\Gamma[\f_1,\f_1]+2\e^3\Gamma[\f_1,\f_2].\no
\end{align}
On the other hand, based on the construction of boundary layers in Section \ref{att section 03}, we know the boundary layer contribution with $\qb=\e^2\fb_2$
\begin{align}\label{ftt 11}
\lll[\qb]=
&-\e^3\dfrac{1}{P_1P_2}\Bigg(\dfrac{\p_{11}\vr\cdot\p_2\vr}{P_1(\e\kk_1\eta-1)}\vb\vc
+\dfrac{\p_{12}\vr\cdot\p_2\vr}{P_2(\e\kk_2\eta-1)}\vc^2\Bigg)\dfrac{\p\fb_2}{\p\vb}\\\rule{0ex}{2.0em}
&-\e^3\dfrac{1}{P_1P_2}\Bigg(\dfrac{\p_{22}\vr\cdot\p_1\vr}{P_2(\e\kk_2\eta-1)}\vb\vc
+\dfrac{\p_{12}\vr\cdot\p_1\vr}{P_1(\e\kk_1\eta-1)}\vb^2\Bigg)\dfrac{\p\fb_2}{\p\vc}\no\\\rule{0ex}{2.0em}
&-\e^3\bigg(\dfrac{\vb}{P_1(\e\kk_1\eta-1)}\dfrac{\p\fb_2}{\p\iota_1}+\e^3\dfrac{\vc}{P_2(\e\kk_2\eta-1)}\dfrac{\p\fb_2}{\p\iota_2}\bigg).\no
\end{align}
Therefore, inserting \eqref{ftt 09}, \eqref{ftt 10} and \eqref{ftt 11} into \eqref{ftt 08}, we have
\begin{align}\label{ftt 12}
\lll[R]=&\e^3\Gamma[R,R]+2\Gamma[R,\q+\qb]+S_1+S_2,
\end{align}
where
\begin{align}\label{ftt 13}
S_1=&-\e\vv\cdot\nx\f_3+\dfrac{1}{P_1P_2}\Bigg(\dfrac{\p_{11}\vr\cdot\p_2\vr}{P_1(\e\kk_1\eta-1)}\vb\vc
+\dfrac{\p_{12}\vr\cdot\p_2\vr}{P_2(\e\kk_2\eta-1)}\vc^2\Bigg)\dfrac{\p\fb_2}{\p\vb}\\\rule{0ex}{2.0em}
&+\dfrac{1}{P_1P_2}\Bigg(\dfrac{\p_{22}\vr\cdot\p_1\vr}{P_2(\e\kk_2\eta-1)}\vb\vc
+\dfrac{\p_{12}\vr\cdot\p_1\vr}{P_1(\e\kk_1\eta-1)}\vb^2\Bigg)\dfrac{\p\fb_2}{\p\vc}\no\\\rule{0ex}{2.0em}
&+\bigg(\dfrac{\vb}{P_1(\e\kk_1\eta-1)}\dfrac{\p\fb_2}{\p\iota_1}+\e^3\dfrac{\vc}{P_2(\e\kk_2\eta-1)}\dfrac{\p\fb_2}{\p\iota_2}\bigg),\no\\
S_2=&2\Gamma[\f_1,\fb_2]+2\e\Gamma[\f_1,\f_3]+\text{higher-order $\Gamma$ terms up to $\e^3$}..\label{ftt 14}
\end{align}
\ \\
Step 3: Representation of $R-\pp[R]$.\\
The boundary condition of \eqref{small system} is essentially
\begin{align}
f^{\e}=\mb\m^{-1}\pp[f^{\e}]+\mhh(\mb-\m).
\end{align}
which means
\begin{align}
\q+\qb+\e^3R=\pp[\q+\qb+\e^3R]+(\mb-\m)\m^{-1}\pp[\q+\qb+\e^3R]+\mhh(\mb-\m).
\end{align}
Based on the boundary condition expansion in Section \ref{att section 04}, we have
\begin{align}\label{ftt 15}
R-\pp[R]=&H[R]+h,
\end{align}
where
\begin{align}\label{ftt 16}
H[R](\vx_0,\vv)=(\mb-\m)\m^{-1}\pp[R],
\end{align}
and
\begin{align}\label{ftt 17}
h&=\e^{-2}\Big(\mb-\m-\e\mh\m_1-\e^2\mh\m_2\Big)\m^{-1}\pp[\f_1+\fb_1]+\e^{-1}\Big(\mb-\m-\e\mh\m_1\Big)\m^{-1}\pp[\f_2+\fb_2]\\
&+(\mb-\m)\m^{-1}\pp[\f_3]+\e^{-3}\mhh\Big(\mb-\m-\e\mh\m_1-\e^2\mh\m_2-\e^3\mh\m_3\Big).\no
\end{align}
\ \\
Step 4: Remainder Estimate.\\
The equation \eqref{ftt 12} and boundary condition \eqref{ftt 15} forms a system that fits into \eqref{nonlinear steady}:
\begin{align}\label{remainder equation}
\left\{
\begin{array}{l}
\e\vv\cdot\nx R+\ll[R]=\Gamma\Big[R,2(\q+\qb)+\e^3R\Big]+S_1(\vx,\vv)+S_2(\vx,\vv)\ \ \text{in}\ \ \Omega\times\r^3,\\\rule{0ex}{1.5em}
R(\vx_0,\vv)=\pp[R](\vx_0,\vv)+H[R](\vx_0,\vv)+h(\vx_0,\vv)\ \ \text{for}\ \ \vx_0\in\p\Omega\ \
\text{and}\ \ \vv\cdot\vn<0.
\end{array}
\right.
\end{align}
We assume (this will be verified later) that
\begin{align}\label{ftt 26}
\lnnmv{\e^3R}\ls o(1)\e.
\end{align}
Then we can verify \eqref{remainder equation} satisfies the assumptions \eqref{perturbed smallness} (since $\q$ and $\qb$ are small), \eqref{perturbed compatibility}. Also, the construction in Section \ref{att section 03} implies that the solution satisfies \eqref{perturbed normalization}. Applying Theorem \ref{LI estimate.} to \eqref{remainder equation}, we obtain
\begin{align}\label{ftt 24}
\lnnmv{R}+\lsm{R}{\gamma_+}
\ls&\frac{1}{\e^{2+\frac{3}{2m}}}\pnnm{\pk[S_1]}{\frac{2m}{2m-1}}+\frac{1}{\e^{1+\frac{3}{2m}}}\tnnm{\snn(\ik-\pk)[S_1]}+\lnnmv{\nu^{-1} S_1}\\
&+\frac{1}{\e^{2+\frac{3}{2m}}}\pnnm{\pk[S_2]}{\frac{2m}{2m-1}}+\frac{1}{\e^{1+\frac{3}{2m}}}\tnnm{\snn(\ik-\pk)[S_2]}+\lnnmv{\nu^{-1} S_2}\no\\
&+\frac{1}{\e^{\frac{3}{2m}}}\pnm{h}{\gamma_-}{\frac{4m}{3}}+\frac{1}{\e^{1+\frac{3}{2m}}}\tsm{h}{\gamma_-}+\lsm{h}{\gamma_-}.\no
\end{align}
\ \\
Step 5: Estimates of $S_1$ Terms.\\
Based on Theorem \ref{limit theorem 2}, we know
\begin{align}
\pnnm{\e\vv\cdot\nx\f_3}{\frac{2m}{2m-1}}+\tnnm{\snn\Big(\e\vv\cdot\nx\f_3\Big)}+\lnnmv{\nu^{-1}\Big(\e\vv\cdot\nx\f_3\Big)}\ls \e.
\end{align}
On the other hand, based on Theorem \ref{limit theorem 1}, using the rescaling $\eta=\dfrac{\mu}{\e}$, we have
\begin{align}
\pnnm{\nu\frac{\p\fb_2}{\p\iota_1}}{\frac{2m}{2m-1}}+\pnnm{\nu\frac{\p\fb_2}{\p\iota_2}}{\frac{2m}{2m-1}}+
\pnnm{\nu^2\frac{\p\fb_2}{\p\vb}}{\frac{2m}{2m-1}}+\pnnm{\nu^2\frac{\p\fb_2}{\p\vc}}{\frac{2m}{2m-1}}&\ls \e^{1-\frac{1}{2m}}\abs{\ln(\e)}^8,\\
\tnnm{\sn\frac{\p\fb_2}{\p\iota_1}}+\tnnm{\sn\frac{\p\fb_2}{\p\iota_2}}
+\tnnm{\nu^{\frac{3}{2}}\frac{\p\fb_2}{\p\vb}}+\tnnm{\nu^{\frac{3}{2}}\frac{\p\fb_2}{\p\vc}}&\ls \e^{\frac{1}{2}}\abs{\ln(\e)}^8,\\
\lnnmv{\frac{\p\fb_2}{\p\iota_1}}+\lnnmv{\frac{\p\fb_2}{\p\iota_2}}+\lnnmv{\nu\frac{\p\fb_2}{\p\vb}}+\lnnmv{\nu\frac{\p\fb_2}{\p\vc}}&\ls \abs{\ln(\e)}^8.
\end{align}
Collecting all terms, we have
\begin{align}\label{ftt 21}
\pnnm{\pk[S_1]}{\frac{2m}{2m-1}}\ls \e^{1-\frac{1}{2m}}\abs{\ln(\e)}^8,\quad \tnnm{\snn(\ik-\pk)[S_1]}\ls \e^{\frac{1}{2}}\abs{\ln(\e)}^8,\quad
\lnnmv{\nu^{-1} S_1}\ls \abs{\ln(\e)}^8.
\end{align}
\ \\
Step 6: Estimates of $S_2$ Terms.\\
Since $S_2$ are all nonlinear terms, Lemma \ref{nonlinear lemma} implies that $\pk[S_2]=0$. Then the leading-order term is $\Gamma[\f_1,\fb_2]$. Hence, using Theorem \ref{limit theorem 2}, Theorem \ref{limit theorem 1} and Lemma \ref{nonlinear lemma}, we have
\begin{align}
\tnnm{\snn\Gamma[\f_1,\fb_2]}&\ls \sup_{\vx\in\Omega}\Big(\tnm{\nu \f_1(x)}\Big)\tnnm{\nu \fb_2}\ls \e^{\frac{1}{2}}\abs{\ln(\e)}^8,\\
\lnnmv{\nu^{-1}\Gamma[\f_1,\fb_2]}&\ls\lnnmv{\f_1}\lnnmv{\fb_2}\ls \abs{\ln(\e)}^8.
\end{align}
Hence, we have
\begin{align}\label{ftt 22}
\pnnm{\pk[S_2]}{\frac{2m}{2m-1}}=0,\quad \tnnm{\snn(\ik-\pk)[S_2]}\ls \e^{\frac{1}{2}}\abs{\ln(\e)}^8,\quad
\lnnmv{\nu^{-1} S_2}\ls \abs{\ln(\e)}^8.
\end{align}
\ \\
Step 7: Estimates of $h$ Terms.\\
Note that all terms in $h$ are at least at order of $\e$. Hence, we directly bound
\begin{align}\label{ftt 23}
\pnm{h}{\gamma_-}{\frac{4m}{3}}\ls \e,\quad \tsm{h}{\gamma_-}\ls\e,\quad\lsm{h}{\gamma_-}\ls\e.
\end{align}
\ \\
Step 8: Synthesis.\\
Inserting \eqref{ftt 21}, \eqref{ftt 22} and \eqref{ftt 23} into \eqref{ftt 24}, we have
\begin{align}
\lnnmv{R}+\lsm{R}{\gamma_+}
\ls& \frac{1}{\e^{2+\frac{3}{2m}}}\bigg(\e^{1-\frac{1}{2m}}\abs{\ln(\e)}^8\bigg)+\frac{1}{\e^{1+\frac{3}{2m}}}\bigg(\e^{\frac{1}{2}}\abs{\ln(\e)}^8\bigg)
+\bigg(\abs{\ln(\e)}^8\bigg)\\
&+\frac{1}{\e^{\frac{3}{2m}}}(\e)+\frac{1}{\e^{1+\frac{3}{2m}}}(\e)+(\e)+\e^{1-\frac{3}{2m}}\lnnmv{R}^2\no\\
\ls& \e^{-1-\frac{2}{m}}\abs{\ln(\e)}^8+\e^{1-\frac{3}{2m}}\lnnmv{R}^2.\no
\end{align}
In particular, we can verify the validity of assumption \eqref{ftt 26}. Considering \eqref{ftt 07}, this means that we have shown
\begin{align}
\frac{1}{\e^3}\lnnmv{f^{\e}-\Big(\e\f_1+\e^2\f_2+\e^3\f_3\Big)-\Big(\e\fb_1+\e^2\fb_2\Big)}\ls \e^{-1-\frac{2}{m}}\abs{\ln(\e)}^8.
\end{align}
Therefore, we know
\begin{align}
\lnnmv{f^{\e}-\e\f_1-\e\fb_1}\ls\e^{2-\frac{2}{m}}\abs{\ln(\e)}^8.
\end{align}
Since $\fb_1=0$, then we naturally have for $\f=\f_1$.
\begin{align}
\lnnmv{f^{\e}-\e\f}\ls\e^{2-\frac{2}{m}}\abs{\ln(\e)}.
\end{align}
Here $\dfrac{3}{2}< m<3$, so we may further bound
\begin{align}
\lnnmv{f^{\e}-\e\f}\ls C(\d)\e^{\frac{4}{3}-\d},
\end{align}
for any $0<\d<<1$.

\chapter{Evolutionary Boltzmann Equation}

\section{Asymptotic Expansion}

\subsection{Interior Expansion}\label{att section 6.}

We define the interior expansion
\begin{align}\label{interior expansion.}
f^{\e}_{\text{in}}(t,\vx,\vv)=\sum_{k=1}^{3}\e^k\f_k(t,\vx,\vv).
\end{align}
Plugging it into the equation \eqref{small system.} and comparing the order of $\e$, we obtain
\begin{align}
\ll[\f_1]=&0,\label{interior expansion 1.}\\
\ll[\f_2]=&-\vv\cdot\nx\f_1+\Gamma[\f_1,\f_1],\label{interior expansion 2.}\\
\ll[\f_3]=&-\dt\f_1-\vv\cdot\nx\f_2+2\Gamma[\f_1,\f_2].\label{interior expansion 3.}
\end{align}
The analysis of $\f_k$ solvability is standard and well-known. Note that the null space $\nk$ of the operator $\ll$ is spanned by
\begin{align}
\m^{\frac{1}{2}}\bigg\{1,v_1,v_2,v_3,\dfrac{\abs{\vv}^2-3}{2}\bigg\}=\{\ne_0,\ne_1,\ne_2,\ne_3,\ne_4\}.
\end{align}
Then $\ll[f]=S$ is solvable if and only if $S\in\nk^{\perp}$ the orthogonal complement of $\nk$ in $L^2(\r^3)$. As \cite[Chapter 4]{Sone2002} and \cite[Chapter 3]{Sone2007} reveal, similar to stationary problems in Section \ref{att section 1}, each $\f_k$ consists of three parts:
\begin{align}
\f_k(t,\vx,\vv):=A_k (t,\vx,\vv)+B_k (t,\vx,\vv)+C_k (t,\vx,\vv).
\end{align}
\begin{itemize}
\item
Principal contribution $\ds A_k:=\sum_{i=0}^4A_{k,i}\ne_i\in\nk$, where the coefficients $A_{k,i}$ must be determined at each order $k$ independently.
\item
Connecting contribution $\ds B_k:=\sum_{i=0}^4B_{k,i}\ne_i\in\nk$, where the coefficients $B_{k,i}$ depends on $A_s$ for $1\leq s\leq k-1$. In other words, $B_k$ is accumulative information from previous orders and thus is not independent. This term is present due to the nonlinearity in $\Gamma$.
\item
Orthogonal contribution $C_k\in\nk^{\perp}$ satisfying
\begin{align}
\ll[C_k]=&\dt\f_{k-2}-\vv\cdot\nx\f_{k-1}+\sum_{i=1}^{k-1}\Gamma[\f_i,\f_{k-i}],
\end{align}
which can be uniquely determined. Similar to $B_k$, here $C_k$ is also accumulative information from previous orders and thus is not independent.
\end{itemize}
All in all, we will focus on how to determine $A_k$. Traditionally, we write
\begin{align}
A_k=\m^{\frac{1}{2}}\left(\rh_k +\vu_k\cdot\vv+\th_k \left(\frac{\abs{\vv}^2-3}{2}\right)\right),
\end{align}
where the coefficients $\rh_k$, $\vu_k$ and $\th_k$ represent density, velocity and temperature in the macroscopic scale. \cite[Chapter 4]{Sone2002} and \cite[Chapter 3]{Sone2007} state that $(\rh_k,\vu_k,\th_k)$ satisfies 
as follows:\\
\ \\
\eqref{interior expansion 2.} implies
\begin{align}
p_1 -(\rh_1 +\th_1 )=&0,\\
\nabla_{x} p_1 =&0,\\
\nx\cdot\vu_1 =&0,
\end{align}
\eqref{interior expansion 3.} implies
\begin{align}
p_2 -(\rh_2 +\th_2 +\rh_1 \th_1 )=&0,\\
\dt\vu_1+\vu_1 \cdot\nx\vu_1 -\gamma_1\dx\vu_1 +\nx p_2 =&0,\\
\dt\th_1+\vu_1 \cdot\nx\th_1 -\gamma_2\dx\th_1 =&0,\\
\nx\cdot\vu_2 +\vu_1\cdot\nx\rh_1=&0.
\end{align}
Here $p_1$ and $p_2$ represent the pressure, $\gamma_1$ and $\gamma_2$ are constants. The higher-order expansion produces more complicated fluid equations, which can be found in \cite[Chapter 4]{Sone2002}. 
If the interior solution $\f_k$ cannot satisfy the initial and boundary condition, then we have to introduce initial layer $\fl_k$ and boundary layer $\fb_k$ to handle the gap.

\subsection{Initial-Layer Expansion}\label{att section 3.}

Temporal Substitution: \\
We define the rescaled time variable $\tau$ by making the scaling transform $\tau=\dfrac{t}{\e^2}$, which implies $\dfrac{\p }{\p t}=\dfrac{1}{\e^2}\dfrac{\p }{\p\tau}$.
Then, under the substitution $t\rt\tau$, the equation \eqref{small system.} is transformed into
\begin{align}\label{initial.}
\left\{
\begin{array}{l}
\p_{\tau}f^{\e}+\e\vv\cdot\nx
f^{\e}+\ll[f^{\e}]=\Gamma[f^{\e},f^{\e}],\ \ \text{in}\ \ \rp\times\Omega\times\r^3\\\rule{0ex}{1.5em}
f^{\e}(0,\vx,\vv)=f_0(\vx,\vv),\ \ \text{in}\ \ \Omega\times\r^3\\\rule{0ex}{1.5em}
f^{\e}(\tau,\vx_0,\vv)=\pe[f^{\e}](\tau,\vx_0,\vv) \ \ \text{for}\ \ \tau\in\rp,\ \ \vx_0\in\p\Omega\ \ \text{and}\ \ \vv\cdot\vn(\vx_0)<0,
\end{array}
\right.
\end{align}
We define the initial layer expansion:
\begin{align}\label{initial layer expansion.}
f^{\e}_{\text{il}}(\tau,\vx,\vv)=\sum_{k=1}^{4}\e^k\fl_k(\tau,\vx,\vv),
\end{align}
where $\fl_{k}$ can be determined by comparing the order of $\e$ via plugging \eqref{initial layer expansion.} into the equation
\eqref{initial.}. Thus, we have
\begin{align}
\p_{\tau}\fl_{1}+\ll[\fl_1]=&0,\label{initial expansion 1}\\
\p_{\tau}\fl_{2}+\ll[\fl_2]=&-\vv\cdot\nx\fl_{1}+\Gamma[\fl_1,\fl_1]+2\Gamma[\f_1,\fl_1],\label{initial expansion 2}\\
\p_{\tau}\fl_{3}+\ll[\fl_3]=&-\vv\cdot\nx\fl_{2}+2\Gamma[\fl_1,\fl_2]+2\Gamma[\f_1,\fl_2]+2\Gamma[\f_2,\fl_1],\label{initial expansion 3}\\
\p_{\tau}\fl_{4}+\ll[\fl_4]=&-\vv\cdot\nx\fl_{3}+2\Gamma[\fl_1,\fl_3]+\Gamma[\fl_2,\fl_2]+2\Gamma[\f_1,\fl_3]+2\Gamma[\f_3,\fl_1]+2\Gamma[\f_2,\fl_2].\label{initial expansion 3}
\end{align}

\subsection{Boundary Layer Expansion}\label{att section 2.}

This is very similar to the stationary problem in Section \ref{att section 2}. We need to introduce several geometric substitutions.
\begin{enumerate}
\item
In a neighborhood of $\vx_0\in\p\Omega$ define an orthogonal curvilinear coordinates system $(\iota_1,\iota_2)$ such that at $\vx_0$ the coordinate lines coincide with the principal directions. Let $\mu$ be the normal distance to the boundary. Then $(\mu,\iota_1,\iota_2)$ forms a local orthogonal coordinate system.

Assume $\p\Omega$ is parameterized by $\vr=\vr(\iota_1,\iota_2)$. Denote $P_i=\abs{\p_i\vr}$ for $i=1,2$. Then define the two orthogonal unit tangential vectors
\begin{align}\label{coordinate 14}
\vt_1:=\frac{\p_1\vr}{P_1},\ \ \vt_2:=\frac{\p_2\vr}{P_2}.
\end{align}
Also, the outward unit normal vector is
\begin{align}\label{coordinate 1}
\vn:=\frac{\p_1\vr\times\p_2\vr}{\abs{\p_1\vr\times\p_2\vr}}=\vt_1\times\vt_2.
\end{align}
Let $\kk_1$ and $\kk_2$ denote two principal curvatures and $R_1$ and $R_2$ two radii of principal curvatures.
\item
We also decompose the velocity into normal and tangential directions
\begin{align}
\left\{
\begin{array}{rcl}
-\vv\cdot\vn&=&\va,\\
-\vv\cdot\vt_1&=&\vb,\\
-\vv\cdot\vt_2&=&\vc.
\end{array}
\right.
\end{align}
Denote $\vvv=(\va,\vb,\vc)$.
\item
Define the scaled variable $\eta=\dfrac{\mu}{\e}$, which implies $\dfrac{\p}{\p\mu}=\dfrac{1}{\e}\dfrac{\p}{\p\eta}$.
\end{enumerate}
Under these substitutions, the equation \eqref{small system.} is transformed into
\begin{align}\label{small system'.}
\left\{
\begin{array}{l}\displaystyle
\e^2\dt f^{\e}+\va\dfrac{\p f^{\e}}{\p\eta}-\dfrac{\e}{R_1-\e\eta}\bigg(\vb^2\dfrac{\p f^{\e}}{\p\va}-\va\vb\dfrac{\p f^{\e}}{\p\vb}\bigg)
-\dfrac{\e}{R_2-\e\eta}\bigg(\vc^2\dfrac{\p f^{\e}}{\p\va}-\va\vc\dfrac{\p f^{\e}}{\p\vc}\bigg)\\\rule{0ex}{2.0em}
-\dfrac{\e}{P_1P_2}\Bigg(\dfrac{\p_{11}\vr\cdot\p_2\vr}{P_1(\e\kk_1\eta-1)}\vb\vc
+\dfrac{\p_{12}\vr\cdot\p_2\vr}{P_2(\e\kk_2\eta-1)}\vc^2\Bigg)\dfrac{\p f^{\e}}{\p\vb}\\\rule{0ex}{2.0em}
-\dfrac{\e}{P_1P_2}\Bigg(\dfrac{\p_{22}\vr\cdot\p_1\vr}{P_2(\e\kk_2\eta-1)}\vb\vc
+\dfrac{\p_{12}\vr\cdot\p_1\vr}{P_1(\e\kk_1\eta-1)}\vb^2\Bigg)\dfrac{\p f^{\e}}{\p\vc}\\\rule{0ex}{2.0em}
-\e\bigg(\dfrac{\vb}{P_1(\e\kk_1\eta-1)}\dfrac{\p f^{\e}}{\p\iota_1}+\dfrac{\vc}{P_2(\e\kk_2\eta-1)}\dfrac{\p f^{\e}}{\p\iota_2}\bigg)
+\ll[f^{\e}]=\Gamma[f^{\e},f^{\e}]\ \ \text{in}\ \ \rp\times\Omega\times\r^3,\\\rule{0ex}{2.0em}
f^{\e}(0,\eta,\iota_1,\iota_2,\vvv)=f_0(\eta,\iota_1,\iota_2,\vvv)\ \ \text{in}\ \ \Omega\times\r^3,\\\rule{0ex}{2.0em}
f^{\e}(t,0,\iota_1,\iota_2,\vvv)=\pe[f^{\e}](t,0,\iota_1,\iota_2,\vvv)\ \ \text{for}\
\ \va>0.
\end{array}
\right.
\end{align}
We define the boundary layer expansion as follows:
\begin{align}\label{boundary layer expansion.}
f^{\e}_{\text{bl}}(t,\eta,\iota_1,\iota_2,\vvv)=\sum_{k=1}^{3}\e^k\fb_k(t,\eta,\iota_1,\iota_2,\vvv),
\end{align}
where $\fb_k$ can be defined by comparing the order of $\e$ via plugging \eqref{boundary layer expansion.} into the equation
\eqref{small system'.}. Thus, in a neighborhood of the boundary, we have
\begin{align}
\va\dfrac{\p\fb_1}{\p\eta}-\dfrac{\e}{R_1-\e\eta}\bigg(\vb^2\dfrac{\p\fb_1}{\p\va}-\va\vb\dfrac{\p\fb_1}{\p\vb}\bigg)
-\dfrac{\e}{R_2-\e\eta}\bigg(\vc^2\dfrac{\p\fb_1}{\p\va}-\va\vc\dfrac{\p\fb_1}{\p\vc}\bigg)+\ll[\fb_1]=&0,\\
\va\dfrac{\p\fb_2}{\p\eta}-\dfrac{\e}{R_1-\e\eta}\bigg(\vb^2\dfrac{\p\fb_2}{\p\va}-\va\vb\dfrac{\p\fb_2}{\p\vb}\bigg)
-\dfrac{\e}{R_2-\e\eta}\bigg(\vc^2\dfrac{\p\fb_2}{\p\va}-\va\vc\dfrac{\p\fb_2}{\p\vc}\bigg)+\ll[\fb_2]=&Z_1,
\end{align}
where $Z_1=Z_1\left[\f_1,\fb_1,\dfrac{\p\fb_1}{\p\vb},\dfrac{\p\fb_1}{\p\vc},\dfrac{\p\fb_1}{\p\iota_1},\dfrac{\p\fb_1}{\p\iota_2}\right]$ as
\begin{align}
Z_1:=&2\Gamma[\f_1,\fb_1]+\Gamma[\fb_1,\fb_1]
+\dfrac{1}{P_1P_2}\Bigg(\dfrac{\p_{11}\vr\cdot\p_2\vr}{P_1(\e\kk_1\eta-1)}\vb\vc
+\dfrac{\p_{12}\vr\cdot\p_2\vr}{P_2(\e\kk_2\eta-1)}\vc^2\Bigg)\dfrac{\p\fb_1}{\p\vb}\\\rule{0ex}{2.0em}
&+\dfrac{1}{P_1P_2}\Bigg(\dfrac{\p_{22}\vr\cdot\p_1\vr}{P_2(\e\kk_2\eta-1)}\vb\vc
+\dfrac{\p_{12}\vr\cdot\p_1\vr}{P_1(\e\kk_1\eta-1)}\vb^2\Bigg)\dfrac{\p\fb_1}{\p\vc}
+\dfrac{\vb}{P_1(\e\kk_1\eta-1)}\dfrac{\p\fb_1}{\p\iota_1}+\dfrac{\vc}{P_2(\e\kk_2\eta-1)}\dfrac{\p\fb_1}{\p\iota_2}.\no
\end{align}
However, we define $\fb_3$ in a completely different fashion. Let $\fb_3$ satisfy
\begin{align}\label{extra boundary layer}
&\va\dfrac{\p \fb_3}{\p\eta}-\dfrac{\e}{R_1-\e\eta}\bigg(\vb^2\dfrac{\p \fb_3}{\p\va}-\va\vb\dfrac{\p \fb_3}{\p\vb}\bigg)
-\dfrac{\e}{R_2-\e\eta}\bigg(\vc^2\dfrac{\p \fb_3}{\p\va}-\va\vc\dfrac{\p \fb_3}{\p\vc}\bigg)\\\rule{0ex}{2.0em}
&-\dfrac{\e}{P_1P_2}\Bigg(\dfrac{\p_{11}\vr\cdot\p_2\vr}{P_1(\e\kk_1\eta-1)}\vb\vc
+\dfrac{\p_{12}\vr\cdot\p_2\vr}{P_2(\e\kk_2\eta-1)}\vc^2\Bigg)\dfrac{\p \fb_3}{\p\vb}\no\\\rule{0ex}{2.0em}
&-\dfrac{\e}{P_1P_2}\Bigg(\dfrac{\p_{22}\vr\cdot\p_1\vr}{P_2(\e\kk_2\eta-1)}\vb\vc
+\dfrac{\p_{12}\vr\cdot\p_1\vr}{P_1(\e\kk_1\eta-1)}\vb^2\Bigg)\dfrac{\p \fb_3}{\p\vc}\no\\\rule{0ex}{2.0em}
&-\e\bigg(\dfrac{\vb}{P_1(\e\kk_1\eta-1)}\dfrac{\p \fb_3}{\p\iota_1}+\dfrac{\vc}{P_2(\e\kk_2\eta-1)}\dfrac{\p \fb_3}{\p\iota_2}\bigg)
+\ll[\fb_3]=Z_2,\no
\end{align}
where
\begin{align}
Z_2:=&2\Gamma[\fb_1,\fb_2]+2\Gamma[\f_1,\fb_2]+2\Gamma[\f_2,\fb_1]
+\dfrac{1}{P_1P_2}\Bigg(\dfrac{\p_{11}\vr\cdot\p_2\vr}{P_1(\e\kk_1\eta-1)}\vb\vc
+\dfrac{\p_{12}\vr\cdot\p_2\vr}{P_2(\e\kk_2\eta-1)}\vc^2\Bigg)\dfrac{\p\fb_2}{\p\vb}\\\rule{0ex}{2.0em}
&+\dfrac{1}{P_1P_2}\Bigg(\dfrac{\p_{22}\vr\cdot\p_1\vr}{P_2(\e\kk_2\eta-1)}\vb\vc
+\dfrac{\p_{12}\vr\cdot\p_1\vr}{P_1(\e\kk_1\eta-1)}\vb^2\Bigg)\dfrac{\p\fb_2}{\p\vc}
+\dfrac{\vb}{P_1(\e\kk_1\eta-1)}\dfrac{\p\fb_2}{\p\iota_1}+\dfrac{\vc}{P_2(\e\kk_2\eta-1)}\dfrac{\p\fb_2}{\p\iota_2}.\no
\end{align}
Obviously, \eqref{extra boundary layer} actually contains all terms in \eqref{small system'.} except the time derivative, so it is essentially
\begin{align}
\e\vv\cdot\nx\fb_3+\ll[\fb_3]=Z_2.
\end{align}
Hence, we will resort to the well-posedness and decay theory of linearized stationary problem instead of that of the half-space boundary layer equation (the so-called $\e$-Milne problem with geometric correction).

\subsection{Initial Condition Expansion}

The bridge between the interior solution and initial layer is the initial condition. Plugging the combined expansion from \eqref{interior expansion.} and \eqref{initial layer expansion.}
\begin{align}
f^{\e}\sim\sum_{k=1}^{3}\e^k\f_k+\sum_{k=1}^{4}\e^k\fl_k
\end{align}
into the initial condition \eqref{small system.}, and comparing the order of $\e$, we obtain
\begin{align}
\f_1+\fl_1=&f_{0,1},\label{ltt 1.}\\
\f_2+\fl_2=&f_{0,2},\label{ltt 2.}\\
\f_3+\fl_3=&f_{0,3}.\label{ltt 3.}
\end{align}
Since we do not expand the interior solution $f^{\e}_{\text{in}}$ to higher order, we simply require the initial condition such that $\fl_4$ decays to zero as $\tau\rt\infty$.

\subsection{Boundary Condition Expansion}

Similar to the stationary problem in Section \ref{att section 04}, the bridge between the interior solution and boundary layer is the boundary condition. Define
\begin{align}
\pp[f](t,\vx_0,\vv):=\m^{\frac{1}{2}}(\vv)
\int_{\vuu\cdot\vn(\vx_0)>0}\m^{\frac{1}{2}}(\vuu)f(t,\vx_0,\vuu)\abs{\vuu\cdot\vn(\vx_0)}\ud{\vuu}.
\end{align}
Plugging the combined expansion from \eqref{interior expansion.} and \eqref{boundary layer expansion.}
\begin{align}
f^{\e}\sim\sum_{k=1}^{3}\e^k\f_k+\sum_{k=1}^{2}\e^k\fb_k
\end{align}
into the boundary condition \eqref{small system.} and \eqref{att 05.}, and comparing the order of $\e$, we obtain
\begin{align}
\f_1+\fb_1=&\pp[\f_1+\fb_1]+\m_1(\vx_0,\vv),\label{btt 1.}\\
\f_2+\fb_2=&\pp[\f_2+\fb_2]+\m_1(\vx_0,\vv)
\int_{\vuu\cdot\vn(\vx_0)>0}\m^{\frac{1}{2}}(\vuu)(\f_1+\fb_1)\abs{\vuu\cdot\vn(\vx_0)}
\ud{\vuu}+\m_2(\vx_0,\vv).\label{btt 2.}
\end{align}
For $\fb_3$ and $\fb_3$, since the boundary layer $\fb_3$ is defined differently, we can assign stronger version
\begin{align}
\f_3+\fb_3=&\pp[\f_3+\fb_3]+\e^{-2}\Big(\mb-\m-\e\mh\m_1\Big)\m^{-1}\pp[\f_1+\fb_1]+\e^{-1}\Big(\mb-\m\Big)\m^{-1}\pp[\f_2+\fb_2]\\
&+\e^{-3}\mhh\Big(\mb-\m-\e\mh\m_1-\e^2\mh\m_2\Big).\no
\end{align}

\subsection{Matching Procedure}\label{att section 1.}

Define the length of boundary layer $L=\e^{-\frac{1}{2}}$. Also, denote $\rr[\va,\vb,\vc]=(-\va,\vb,\vc)$. \\
\ \\
Step 1: Construction of $\f_1$, $\fl_1$ and $\fb_1$.\\
A direct computation reveals that $\f_1=A_1+B_1+C_1$, where $B_1=C_1=0$.
Define
\begin{align}
\f_1=&\m^{\frac{1}{2}}\left(\rh_{1}+\vu_{1}\cdot\vv+\th_{1}\frac{\abs{\vv}^2-3}{2}\right),
\end{align}
where $(\rh_1,\vu_1,\th_1)$ satisfies the Navier-Stokes-Fourier system
\begin{align}\label{interior 1.}
\left\{
\begin{array}{l}
\dt\vu_1+\vu_1\cdot\nx\vu_1-\gamma_1\dx\vu_1+\nx p_2 =0,\\\rule{0ex}{1.0em}
\nx\cdot\vu_1=0,\\\rule{0ex}{1.0em}
\dt\th_1+\vu_1\cdot\nx\th_1-\gamma_2\dx\th_1=0,
\end{array}
\right.
\end{align}
with the initial condition
\begin{align}
\rh_1(0,\vx)=\rh_{0,1}(\vx),\quad \vu_1(0,\vx)=\vu_{0,1}(\vx),\quad \th_1(0,\vx)=\th_{0,1}(\vx),
\end{align}
and the boundary condition
\begin{align}
\rh_1(t,\vx_0)=\rh_{\bb,1}(t,\vx_0)+M_1(t,\vx_0),\quad \vu_1(t,\vx_0)=\vu_{\bb,1}(t,\vx_0),\quad\th_1(t,\vx_0)=\th_{\bb,1}(t,\vx_0).
\end{align}
Here $M_1(t,\vx_0)$ is such that the Boussinesq relation
\begin{align}
\nx(\rh_1+\th_1)=0
\end{align}
is satisfied.
Note that the above requirement means that for fixed $t$, $M_1$ still has one dimension of freedom. It is eventually fully determined by enforcing the conservation law
\begin{align}
\iint_{\Omega\times\r^3}\f_1(t,\vx,\vv)\mh(\vv)\ud{\vv}\ud{\vx}=0.
\end{align}
\ \\
Then based on the compatibility condition of $\mu_1$ which is
\begin{align}
\int_{\vuu\cdot\vn(\vx_0)>0}\m^{\frac{1}{2}}(\vuu)\mu_1(t,\vx_0,\vuu)\abs{\vuu\cdot\vn(\vx_0)}\ud{\vuu}=0,
\end{align}
we naturally obtain $\pp[\f_1]=M_1\mu^{\frac{1}{2}}$, which means
\begin{align}
\f_1=\pp[\f_1]+\m_1\ \ \text{on}\ \ \p\Omega.
\end{align}
Therefore, compared with \eqref{btt 1.}, it is not necessary to introduce the boundary layer at this order and we simply take $\fb_1=0$.
Also, the interior solution can already satisfy the initial data, so it is not necessary to introduce the initial layer at this order and we simply take $\fl_1=0$.\\
\ \\
Step 2: Construction of $\f_2$, $\fl_2$ and $\fb_2$.\\
Define $\f_2=A_2+B_2+C_2$, where $B_2$ and $C_2$ can be uniquely determined following previous analysis in Section \ref{att section 6.}, and
\begin{align}
A_2=&\m^{\frac{1}{2}}\left(\rh_{2}+\vu_{2}\cdot\vv+\th_{2}\frac{\abs{\vv}^2-3}{2}\right),
\end{align}
satisfying a linear fluid-type equation provided $\f_1$ is known.
Now $\f_2$ does not satisfy \eqref{btt 2.} alone, so we have to introduce boundary layer. 
Let $\fb_2$ satisfy the $\e$-Milne problem with geometric correction
\begin{align}\label{att 3.}
\left\{
\begin{array}{l}\displaystyle
\va\dfrac{\p\fb_2}{\p\eta}-\dfrac{\e}{R_1-\e\eta}\bigg(\vb^2\dfrac{\p\fb_2}{\p\va}-\va\vb\dfrac{\p\fb_2}{\p\vb}\bigg)
-\dfrac{\e}{R_2-\e\eta}\bigg(\vc^2\dfrac{\p\fb_2}{\p\va}-\va\vc\dfrac{\p\fb_2}{\p\vc}\bigg)+\ll[\fb_2]
=0,\\\rule{0ex}{2.0em}
\fb_2(t,0,\iota_1,\iota_2,\vvv)=h(t,\iota_1,\iota_2,,\vvv)-\tilde h(t,\iota_1,\iota_2,,\vvv)\ \
\text{for}\ \ \va>0,\\\rule{0ex}{2.0em}
\displaystyle\fb_2(t,L,\iota_1,\iota_2,,\vvv)
=\fb_2(t,L,\iota_1,\iota_2,,\rr[\vvv]),
\end{array}
\right.
\end{align}
with the in-flow boundary data
\begin{align}\label{btt 3.}
h(t,\iota_1,\iota_2,\vvv)=&M_1\m_1(t,\vx_0,\vv)+\m_2(t,\vx_0,\vv)-\bigg((B_2+C_2)-\pp[B_2+C_2]\bigg).
\end{align}
Based on Theorem \ref{Milne theorem 3} and \ref{Milne theorem 4}, there exists a unique
\begin{align}
\tilde h(t,\iota_1,\iota_2,\vvv)=\m^{\frac{1}{2}}\sum_{k=0}^4\tilde D_k(t,\iota_1,\iota_2)\ee_k,
\end{align}
such that \eqref{att 3.} is well-posed and the solution decays exponentially fast to zero (here $\ee_k$ with $k=0,1,2,3,4$ form a basis of null space $\nk$ of $\ll$). In particular, $\tilde D_1=0$.
Then we further require that $A_2$ satisfies the boundary condition
\begin{align}\label{att 1}
A_2(t,\vx_0,\vv)=\tilde h(t,\iota_1,\iota_2,\vvv)+M_2(t,\vx_0)\m^{\frac{1}{2}}(\vv).
\end{align}
Here $\vx_0$ corresponds to $(\iota_1,\iota_2)$ and $\vv$ corresponds to $\vvv$, based on substitution in Section \ref{att section 2.}. Here, the constant $M_2(t,\vx_0)$ is chosen to enforce the Boussinesq relation
\begin{align}
p_2 -(\rh_2 +\th_2 +\rh_1 \th_1 )=&0,
\end{align}
where $p_2$ is the pressure solved from \eqref{interior 1.}.
Similar to the construction of $\f_1$, due to \eqref{small normalization.}, we can choose $M_2$ to satisfy the conservation law
\begin{align}
\iint_{\Omega\times\r^3}(\f_2+\fb_2+\fl_2)(t,\vx,\vv)\m^{\frac{1}{2}}(\vv)\ud{\vv}\ud{\vx}=0.
\end{align}
We can verify that such construction satisfies the boundary condition \eqref{btt 2.}.

Also, the initial layer is no longer zero at this order. It satisfies
\begin{align}
\left\{
\begin{array}{l}
\p_{\sigma}\fl_{2}+\ll[\fl_2]=0,\\\rule{0ex}{1.5em}
\fl_2(0,\vx,\vv)=(B_2+C_2)(0,\vx,\vv)-\fl_{2,\infty},
\end{array}
\right.
\end{align}
where $\fl_{2,\infty}(\vx,\vv)\in\nk$ is chosen based on Theorem \ref{initial theorem} such that
\begin{align}
\lim_{\tau\rt\infty}\fl_2(\tau,\vx,\vv)=0.
\end{align}
Then we further require that $A_2$ satisfies the initial condition
\begin{align}\label{att 1}
A_2(0,\vx,\vv)=\fl_{2,\infty}(\vx,\vv).
\end{align}
\ \\
Step 3: Construction of $\f_3$, $\fl_3$ and $\fb_3$.\\
This is almost the same as Step 2. Define $\f_3=A_3+B_3+C_3$, where $B_3$ and $C_3$ can be uniquely determined following previous analysis, and
\begin{align}
A_3=&\m^{\frac{1}{2}}\left(\rh_{3}+\vu_{3}\cdot\vv+\th_{3}\frac{\abs{\vv}^2-3}{2}\right),
\end{align}
satisfying a linear fluid-type equation provided $\f_1$ and $\f_2$ are known. In particular, since the boundary layer at this order is defined in a more tricky way, we simply define the boundary condition
\begin{align}
A_3(t,\vx_0,\vv)=0.
\end{align}
On the other hand, define the boundary layer $\fb_3$
\begin{align}
&\va\dfrac{\p \fb_3}{\p\eta}-\dfrac{\e}{R_1-\e\eta}\bigg(\vb^2\dfrac{\p \fb_3}{\p\va}-\va\vb\dfrac{\p \fb_3}{\p\vb}\bigg)
-\dfrac{\e}{R_2-\e\eta}\bigg(\vc^2\dfrac{\p \fb_3}{\p\va}-\va\vc\dfrac{\p \fb_3}{\p\vc}\bigg)\\\rule{0ex}{2.0em}
&-\dfrac{\e}{P_1P_2}\Bigg(\dfrac{\p_{11}\vr\cdot\p_2\vr}{P_1(\e\kk_1\eta-1)}\vb\vc
+\dfrac{\p_{12}\vr\cdot\p_2\vr}{P_2(\e\kk_2\eta-1)}\vc^2\Bigg)\dfrac{\p \fb_3}{\p\vb}\\\rule{0ex}{2.0em}
&-\dfrac{\e}{P_1P_2}\Bigg(\dfrac{\p_{22}\vr\cdot\p_1\vr}{P_2(\e\kk_2\eta-1)}\vb\vc
+\dfrac{\p_{12}\vr\cdot\p_1\vr}{P_1(\e\kk_1\eta-1)}\vb^2\Bigg)\dfrac{\p \fb_3}{\p\vc}\\\rule{0ex}{2.0em}
&-\e\bigg(\dfrac{\vb}{P_1(\e\kk_1\eta-1)}\dfrac{\p \fb_3}{\p\iota_1}+\dfrac{\vc}{P_2(\e\kk_2\eta-1)}\dfrac{\p \fb_3}{\p\iota_2}\bigg)
+\ll[\fb_3]=Z,
\end{align}
where
\begin{align}
Z:=&2\Gamma[\f_1,\fb_2]
+\dfrac{1}{P_1P_2}\Bigg(\dfrac{\p_{11}\vr\cdot\p_2\vr}{P_1(\e\kk_1\eta-1)}\vb\vc
+\dfrac{\p_{12}\vr\cdot\p_2\vr}{P_2(\e\kk_2\eta-1)}\vc^2\Bigg)\dfrac{\p\fb_2}{\p\vb}\\\rule{0ex}{2.0em}
&+\dfrac{1}{P_1P_2}\Bigg(\dfrac{\p_{22}\vr\cdot\p_1\vr}{P_2(\e\kk_2\eta-1)}\vb\vc
+\dfrac{\p_{12}\vr\cdot\p_1\vr}{P_1(\e\kk_1\eta-1)}\vb^2\Bigg)\dfrac{\p\fb_2}{\p\vc}
+\dfrac{\vb}{P_1(\e\kk_1\eta-1)}\dfrac{\p\fb_2}{\p\iota_1}+\dfrac{\vc}{P_2(\e\kk_2\eta-1)}\dfrac{\p\fb_2}{\p\iota_2}.\no
\end{align}
This is essentially,
\begin{align}
    \e\vv\cdot\nx\fb_3+\ll[\fb_3]=Z.
\end{align}
The boundary condition is taken as
\begin{align}
\fb_3=&\pp[\fb_3]+\e^{-2}\Big(\mb-\m-\e\mh\m_1\Big)\m^{-1}\pp[\f_1+\fb_1]+\e^{-1}\Big(\mb-\m\Big)\m^{-1}\pp[\f_2+\fb_2]\\
&+\e^{-3}\mhh\Big(\mb-\m-\e\mh\m_1-\e^2\mh\m_2\Big)-\bigg((B_3+C_3)-\pp[B_3+C_3]\bigg).\no
\end{align}
Also, the initial layer satisfies
\begin{align}
\left\{
\begin{array}{l}
\p_{\sigma}\fl_{3}+\ll[\fl_3]=-\vv\cdot\nx\fl_{2}+2\Gamma[\f_1,\fl_2],\\\rule{0ex}{1.5em}
\fl_3(0,\vx,\vv)=(B_3+C_3)(0,\vx,\vv)-\fl_{3,\infty},
\end{array}
\right.
\end{align}
where $\fl_{3,\infty}(\vx,\vv)\in\nk$ is chosen based on Theorem \ref{initial theorem} such that
\begin{align}
\lim_{\tau\rt\infty}\fl_3(\tau,\vx,\vv)=0.
\end{align}
Then we further require that $A_3$ satisfies the initial condition
\begin{align}\label{att 1}
A_3(0,\vx,\vv)=\fl_{3,\infty}(\vx,\vv).
\end{align}
In a similar fashion, we can define $\fl_4$
\begin{align}
\left\{
\begin{array}{l}
\p_{\sigma}\fl_{4}+\ll[\fl_4]=-\vv\cdot\nx\fl_{3}+2\Gamma[\f_1,\fl_3]+2\Gamma[\f_3,\fl_1]+2\Gamma[\f_2,\fl_2],\\\rule{0ex}{1.5em}
\fl_4(0,\vx,\vv)=-\fl_{4,\infty},
\end{array}
\right.
\end{align}
where $\fl_{4,\infty}(\vx,\vv)\in\nk$ is chosen based on Theorem \ref{initial theorem} such that
\begin{align}
\lim_{\tau\rt\infty}\fl_4(\tau,\vx,\vv)=0.
\end{align}

\newpage

\section{Remainder Estimates}

We consider the linearized evolutionary Boltzmann equation
\begin{align}\label{linear unsteady}
\left\{
\begin{array}{l}
\e^2\dt f+\e\vv\cdot\nx f+\ll[f]=S(t,\vx,\vv)\ \ \text{in}\ \ \rp\times\Omega\times\r^3,\\\rule{0ex}{1.5em}
f(0,\vx,\vv)=z(\vx,\vv)\ \ \text{in}\ \ \Omega\times\r^3,\\\rule{0ex}{1.5em}
f(t,\vx_0,\vv)=\pp[f](t,\vx_0,\vv)+h(t,\vx_0,\vv)\ \ \text{on}\ \ \rp\times\gamma_-,
\end{array}
\right.
\end{align}
where
\begin{align}
\pp[f](t,\vx_0,\vv)=\mh(\vv)
\int_{\vuu\cdot\vn(\vx_0)>0}\m^{\frac{1}{2}}(\vuu)f(t,\vx_0,\vuu)\abs{\vuu\cdot\vn(\vx_0)}\ud{\vuu}.
\end{align}
The data $z$, $S$ and $h$ satisfy the compatibility condition
\begin{align}\label{linear unsteady compatibility}
\iint_{\Omega\times\r^3}\mh z=0,\quad \iint_{\Omega\times\r^3}S(\vx,\vv)\m^{\frac{1}{2}}(\vv)\ud{\vv}\ud{\vx}+\int_{\gamma_-}h(\vx,\vv)\m^{\frac{1}{2}}(\vv)\ud{\gamma}=0.
\end{align}
Then we can easily derive
\begin{align}\label{linear unsteady normalization}
\iint_{\Omega\times\r^3}\mh f(t)=0.
\end{align}
Our analysis is based on the ideas in \cite{Esposito.Guo.Kim.Marra2013}, \cite{Guo2010}, \cite{AA004} and \cite{AA013}. In particular, we will invoke the results of stationary problem. Since proof of the well-posedness of \eqref{linear unsteady} is standard, we will focus on the a priori estimates here.

\subsection{Preliminaries}

We first introduce the well-known micro-macro decomposition. Define $\pk$ as the orthogonal projection onto the null space of $\ll$:
\begin{align}\label{htt lemma 06}
\pk[f]:=\m^{\frac{1}{2}}(\vv)\bigg(a_f(t,\vx)+\vv\cdot\vbb_f(t,\vx)+\frac{\abs{\vv}^2-3}{2}c_f(t,\vx)\bigg)\in\nk,
\end{align}
where $a_f$, $\vbb_f$ and $c_f$ are coefficients. When there is no confusion, we will simply write $a,\vbb, c$. Definitely, $\ll\Big[\pk[f]\Big]=0$. Then the operator $\ik-\pk$ is naturally
\begin{align}
(\ik-\pk)[f]:=f-\pk[f],
\end{align}
which satisfies $(\ik-\pk)[f]\in\nk^{\perp}$, i.e. $\ll[f]=\ll\Big[(\ik-\pk)[f]\Big]$.

\begin{lemma}\label{htt lemma 05}
The linearized collision operator $\ll=\nu I-K$ defined in \eqref{att 11} is self-adjoint in $L^2$. It satisfies
\begin{align}
&\br{\vv}\ls\nu(\vv)\ls\br{\vv},\\
&\br{f,\ll[f]}(t,\vx)=\br{(\ik-\pk)[f],\ll\Big[(\ik-\pk)[f]\Big]}(t,\vx),\\
&\unm{(\ik-\pk)[f(t,\vx)]}^2\ls \br{f,\ll[f]}(t,\vx)\ls \unm{(\ik-\pk)[f(t,\vx)]}^2.
\end{align}
\end{lemma}
\begin{proof}
These are standard properties of $\ll$. See \cite[Chapter 3]{Glassey1996} and \cite[Lemma 3]{Guo2010}.
\end{proof}

\begin{lemma}\label{htt lemma 02}
For $0<\d<<1$, define the near-grazing set of $\gamma_{\pm}$
\begin{align}
\gamma_{\pm}^{\delta}:=\left\{(\vx,\vv)\in\gamma_{\pm}: \abs{\vn(\vx)\cdot\vv}\leq\delta\ \ \text{or}\ \ \abs{\vv}\geq\frac{1}{\d}\ \ \text{or}\ \ \abs{\vv}\leq\d\right\}.
\end{align}
Then
\begin{align}
&\int_s^t\pnm{f{\bf{1}}_{\gamma_{\pm}\backslash\gamma_{\pm}^{\delta}}}{\gamma}{1}
\leq C(\delta)\bigg(\e\pnnm{f(s)}{1}+\int_s^t\Big(\pnnm{f}{1}+\pnnm{\e\dt f+\vv\cdot\nx f}{1}\Big)\bigg).
\end{align}
\end{lemma}
\begin{proof}
See \cite[Lemma 2.1]{Esposito.Guo.Kim.Marra2013} with a standard time rescaling argument.
\end{proof}


\begin{lemma}[Time-Dependent Green's Identity]\label{htt lemma 03}
Assume $f(t,\vx,\vv),\ g(t,\vx,\vv)\in L^{2}(\rp\times\Omega\times\r^3)$ and $\dt f+\vv\cdot\nx f,\ \dt g+\vv\cdot\nx g\in L^2(\rp\times\Omega\times\r^3)$
with $f,\ g\in L^2(\rp\times\gamma)$. Then for almost all $s,t\in\rp$,
\begin{align}
&\int_s^t\iint_{\Omega\times\r^3}\bigg(\dt f+\vv\cdot\nx f)g+(\dt g+\vv\cdot\nx g)f\bigg)\\
=&\int_s^t\iint_{\gamma_+}fg\ud{\gamma}-\int_s^t\iint_{\gamma_-}fg\ud{\gamma}+\iint_{\Omega\times\r^3}f(t)g(t)-\iint_{\Omega\times\r^3}f(s)g(s).\no
\end{align}
\end{lemma}
\begin{proof}
See \cite[Lemma 2.2]{Esposito.Guo.Kim.Marra2013}.
\end{proof}

\subsection{$L^2$ Estimates}

\begin{lemma}\label{htt lemma 04}
Assume \eqref{linear unsteady compatibility} and \eqref{linear unsteady normalization} hold. The solution $f(t,\vx,\vv)$ to the equation \eqref{linear unsteady} satisfies
\begin{align}\label{htt 12}
\e\tnnmt{\pk[f]}\ls&\e^{\frac{3}{2}}\tnnm{f(t)}+\e\pnmt{(1-\pp)[f]}{\gamma_+}{2}+\tnnmt{(\ik-\pk)[f]}\\
&+\tnnmt{\snn S}+\e^{\frac{3}{2}}\tnnm{z}+\e\pnmt{h}{\gamma_-}{2}.\no
\end{align}
\end{lemma}
\begin{proof}
Apply Green's identity in Lemma \ref{htt lemma 03} to the equation \eqref{linear unsteady}. Then for any $\psi\in L^2(\rp\times\Omega\times\r^3)$
satisfying $\e\dt\psi+\vv\cdot\nx\psi\in L^2(\rp\times\Omega\times\r^3)$ and $\psi\in L^2(\rp\times\gamma)$, we have
\begin{align}\label{htt 01}
&\e\int_0^t\iint_{\gamma_+}f\psi\ud{\gamma}-\e\int_0^t\iint_{\gamma_-}f\psi\ud{\gamma}
-\e\int_0^t\iint_{\Omega\times\r^3}(\vv\cdot\nx\psi)f\\
=&\e^2\int_0^t\iint_{\Omega\times\r^3}f\dt\psi
-\e^2\iint_{\Omega\times\r^3}f(t)\psi(t)+\e^2\iint_{\Omega\times\r^3}f(0)\psi(0)
-\int_0^t\iint_{\Omega\times\r^3}\psi\ll\Big[(\ik-\pk)[f]\Big]+\int_0^t\iint_{\Omega\times\r^3}S\psi.\no
\end{align}
The proof follows the same idea as in stationary version of Lemma \ref{ktt lemma 1} with $m=1$. Actually, we use almost the same test function $\psi\sim \mh\vv\cdot\nx\phi$ to estimate $a$, $\vbb$ and $c$, where $\phi$ satisfies proper elliptic equations. Hence, we will omit the details and only present the main result. Compared with stationary estimate, the new terms only show up on the right-hand side of \eqref{htt 01}. Using H\"older's inequality, we know
\begin{align}
\abs{\e^2\int_0^t\iint_{\Omega\times\r^3}f\dt\psi}\ls \e^2\tnnmt{f}\tnnmt{\dt\psi}\ls\e^2\tnnmt{\dt f}\tnnmt{\dt\nx\phi}.
\end{align}
In a similar fashion, we have
\begin{align}
\abs{\e^2\iint_{\Omega\times\r^3}f(t)\psi(t)}&\ls \e^2\tnnm{f(t)}\tnnm{\psi(t)}\ls \e^2\tnnm{f(t)}\tnnm{\nx\phi(t)},\\
\abs{\e^2\iint_{\Omega\times\r^3}f(0)\psi(0)}&\ls \e^2\tnnm{f(0)}\tnnm{\psi(0)}\ls \e^2\tnnm{z}\tnnm{\nx\phi(0)}.
\end{align}
\ \\
Step 1: Estimates of $c$.\\
We choose the test function
\begin{align}
\psi=\psi_c=\m^{\frac{1}{2}}(\vv)\left(\abs{\vv}^2-\beta_c\right)\Big(\vv\cdot\nx\phi_c(t,\vx)\Big),
\end{align}
where for fixed $t$,
\begin{align}
\left\{
\begin{array}{l}
-\dx\phi_c=c(t,\vx)\ \ \text{in}\ \
\Omega,\\\rule{0ex}{1.5em}
\phi_c=0\ \ \text{on}\ \ \p\Omega,
\end{array}
\right.
\end{align}
and $\beta_c\in\r$ will be determined as in stationary problem. Based on the standard elliptic estimates (see \cite{Krylov2008}), we have
\begin{align}
\onnm{\phi_c(t)}{H^2}\ls \onnm{c(t)}{L^2}.
\end{align}
Eventually, we have
\begin{align}\label{htt 02}
\e\onnmt{c}{L^{2}}^2\ls&\bigg(\tnnmt{(\ik-\pk)[f]}+\tnnmt{\snn S}+\e\pnmt{(1-\pp)[f]}{\gamma_+}{2}+\e\pnmt{h}{\gamma_-}{2}\bigg)\onnmt{c}{L^2}\\
&+\e^2\bigg(\tnnmt{f}\tnnmt{\dt\nx\phi_c}+\tnnm{f(t)}\onnm{c(t)}{L^2}+\tnnm{z}\onnm{c(0)}{L^2}\bigg).\no
\end{align}
\ \\
Step 2: Estimates of $\vbb$.\\
We further divide this step into several sub-steps:\\
\ \\
Sub-Step 2.1: Estimates of $\bigg(\p_{i}\p_j\dx^{-1}b_j\bigg)b_i$ for $i,j=1,2,3$.\\
Let $\vbb=(b_1,b_2,b_3)$. We choose the test functions for $i,j=1,2,3$,
\begin{align}
\psi=\psi_{b,i,j}=\m^{\frac{1}{2}}(\vv)\left(v_i^2-\beta_{b,i,j}\right)\p_j\phi_{b,j},
\end{align}
where
\begin{align}
\left\{
\begin{array}{l}
-\dx\phi_{b,j}=b_j(t,\vx)\ \ \text{in}\ \ \Omega,\\\rule{0ex}{1.5em}
\phi_{b,j}=0\ \ \text{on}\ \ \p\Omega,
\end{array}
\right.
\end{align}
and $\beta_{b,i,j}\in\r$ will be determined as in stationary problem. Eventually, we obtain
\begin{align}\label{htt 03}
\\
\e\abs{\int_0^t\int_{\Omega}\bigg(\p_{i}\p_j\dx^{-1}b_j\bigg)b_i}
\ls&\bigg(\tnnmt{(\ik-\pk)[f]}+\tnnmt{\snn S}+\e\pnmt{(1-\pp)[f]}{\gamma_+}{2}+\e\pnmt{h}{\gamma_-}{2}\bigg)\onnmt{\vbb}{L^2}\no\\
&+\e^2\bigg(\tnnmt{f}\tnnmt{\dt\nx\phi_{b,j}}+\tnnm{f(t)}\onnm{\vbb(t)}{L^2}+\tnnm{z}\onnm{\vbb(0)}{L^2}\bigg).\no
\end{align}
\ \\
Sub-Step 2.2: Estimates of $\bigg(\p_{i}\p_i\dx^{-1}b_j\bigg)b_j$ for $i\neq j$.\\
Notice that the $i=j$ case is included in Sub-Step 2.1. We choose the test function
\begin{align}
\psi=\tilde\psi_{b,i,j}=\m^{\frac{1}{2}}(\vv)\abs{\vv}^2v_iv_j\p_i\phi_{b,j}\ \ \text{for}\ \ i\neq j.
\end{align}
Eventually, we obtain
\begin{align}\label{htt 04}
\\
\e\abs{\int_0^t\int_{\Omega}\bigg(\p_{i}\p_i\dx^{-1}b_j\bigg)b_j}
\ls&\bigg(\tnnmt{(\ik-\pk)[f]}+\tnnmt{\snn S}+\e\pnmt{(1-\pp)[f]}{\gamma_+}{2}+\e\pnmt{h}{\gamma_-}{2}\bigg)\onnmt{\vbb}{L^2}\no\\
&+\e^2\bigg(\tnnmt{f}\tnnmt{\dt\nx\phi_{b,j}}+\tnnm{f(t)}\onnm{\vbb(t)}{L^2}+\tnnm{z}\onnm{\vbb(0)}{L^2}\bigg).\no
\end{align}
\ \\
Sub-Step 2.3: Synthesis.\\
Summarizing \eqref{htt 03} and \eqref{htt 04}, we may sum up over $j=1,2,3$ to obtain, for any $i=1,2,3$,
\begin{align}
\e\onnmt{b_i}{L^{2}}^{2}
\ls&\bigg(\tnnmt{(\ik-\pk)[f]}+\tnnmt{\snn S}+\e\pnmt{(1-\pp)[f]}{\gamma_+}{2}+\e\pnmt{h}{\gamma_-}{2}\bigg)\onnmt{\vbb}{L^2}\\
&+\e^2\bigg(\tnnmt{f}\sum_{j=1}^3\tnnmt{\dt\nx\phi_{b,j}}+\tnnm{f(t)}\onnm{\vbb(t)}{L^2}+\tnnm{z}\onnm{\vbb(0)}{L^2}\bigg),\no
\end{align}
which further implies
\begin{align}\label{htt 05}
\e\onnmt{\vbb}{L^{2}}^2\ls&\bigg(\tnnmt{(\ik-\pk)[f]}+\tnnmt{\snn S}+\e\pnmt{(1-\pp)[f]}{\gamma_+}{2}+\e\pnmt{h}{\gamma_-}{2}\bigg)\onnmt{\vbb}{L^2}\\
&+\e^2\bigg(\tnnmt{f}\sum_{j=1}^3\tnnmt{\dt\nx\phi_{b,j}}+\tnnm{f(t)}\onnm{\vbb(t)}{L^2}+\tnnm{z}\onnm{\vbb(0)}{L^2}\bigg).\no
\end{align}
\ \\
Step 3: Estimates of $a$.\\
We choose the test function
\begin{align}
\psi=\psi_a=\m^{\frac{1}{2}}(\vv)\left(\abs{\vv}^2-\beta_a\right)\Big(\vv\cdot\nx\phi_a(t,\vx)\Big),
\end{align}
where
\begin{align}
\left\{
\begin{array}{l}
-\dx\phi_a=a(t,\vx)\ \ \text{in}\ \ \Omega,\\\rule{0ex}{1.5em}
\dfrac{\p\phi_a}{\p\vn}=0\ \ \text{on}\ \ \p\Omega,
\end{array}
\right.
\end{align}
and $\beta_a$ is a real number to be determined as in stationary problem. Eventually, we get
\begin{align}\label{htt 06}
\e\onnmt{a}{L^{2}}^2\ls&\bigg(\tnnmt{(\ik-\pk)[f]}+\tnnmt{\snn S}+\e\pnmt{(1-\pp)[f]}{\gamma_+}{2}+\e\pnmt{h}{\gamma_-}{2}\bigg)\onnmt{a}{L^2}\\
&+\e^2\bigg(\tnnmt{f}\tnnmt{\dt\nx\phi_a}+\tnnm{f(t)}\onnm{a(t)}{L^2}+\tnnm{z}\onnm{a(0)}{L^2}\bigg).\no
\end{align}
\ \\
Step 4: First Synthesis.\\
Collecting \eqref{htt 02}, \eqref{htt 06} and \eqref{htt 06}, we deduce
\begin{align}\label{htt 07}
\e\tnnmt{\pk[f]}^2\ls&\bigg(\tnnmt{(\ik-\pk)[f]}+\tnnmt{\snn S}+\e\pnmt{(1-\pp)[f]}{\gamma_+}{2}+\e\pnmt{h}{\gamma_-}{2}\bigg)\tnnmt{\pk[f]}\\
&+\e^2\tnnmt{f}\bigg(\tnnmt{\dt\nx\phi_a}+\sum_{j=1}^3\tnnmt{\dt\nx\phi_{b,j}}+\tnnmt{\dt\nx\phi_c}\bigg)+\e^2\tnnm{f(t)}^2+\e^2\tnnm{z}^2.\no
\end{align}
In order to close the proof, we must bound $\tnnmt{\dt\nx\phi_a}$, $\tnnmt{\dt\nx\phi_{b,j}}$ and $\tnnmt{\dt\nx\phi_c}$.

Apply Green's identity in Lemma \ref{ktt 03} to the equation \eqref{linear unsteady}. Then for any $\psi\in L^2(\Omega\times\r^3)$ independent of time $t$
satisfying $\vv\cdot\nx\psi\in L^2(\Omega\times\r^3)$ and $\psi\in L^2(\gamma)$, we have
\begin{align}\label{htt 08}
\e^2\iint_{\Omega\times\r^3}\dt f(t)\psi=&-\e\iint_{\gamma_+}f(t)\psi\ud{\gamma}+\e\iint_{\gamma_-}f(t)\psi\ud{\gamma}
+\e\iint_{\Omega\times\r^3}(\vv\cdot\nx\psi)f(t)\\
&-\iint_{\Omega\times\r^3}\psi\ll\Big[(\ik-\pk)[f](t)\Big]+\iint_{\Omega\times\r^3}S(t)\psi.\no
\end{align}
\ \\
Step 5: Estimate of $\dt\nx\phi_c$.\\
For fixed $t$, taking $\psi=-\m^{\frac{1}{2}}\dfrac{\abs{\vv}^2-3}{2}\dt\phi_c(t)$, using integration by parts, we have
\begin{align}
\e^2\iint_{\Omega\times\r^3}\dt f(t)\psi&=-\e^2\iint_{\Omega\times\r^3}\dt f(t)\m^{\frac{1}{2}}\frac{\abs{\vv}^2-3}{2}\dt\phi_c(t)=-\e^2\int_{\Omega}\dt c(t)\dt\phi_c(t)\\
&=-\e^2\int_{\Omega}\Delta_x\dt\phi_c(t)\dt\phi_c(t)=\e^2\int_{\Omega}\abs{\dt\nx\phi_c(t)}^2=\onnm{\dt\nx\phi_c(t)}{L^2}^2.\no
\end{align}
Following a similar argument as in Step 1 - Step 3, we have
\begin{align}\label{htt 09}
\e^2\tnnmt{\dt\nx\phi_c}\ls\e\onnmt{\vbb}{L^2}+\e\tnnmt{(\ik-\pk)[f]}+\tnnmt{\snn S}.
\end{align}
\ \\
Step 6: Estimate of $\dt\nx\phi_{b,j}$.\\
For fixed $t$, taking $\psi=-\m^{\frac{1}{2}}v_j\dt\phi_{b,j}(t)$, using integration by parts, we have
\begin{align}
\e^2\iint_{\Omega\times\r^3}\dt f(t)\psi&=-\e^2\iint_{\Omega\times\r^3}\dt f(t)\m^{\frac{1}{2}}v_j\dt\phi_{b,j}(t)=-\e^2\int_{\Omega}\dt \phi_{b,j}(t)\dt\phi_{b,j}(t)\\
&=-\e^2\int_{\Omega}\Delta_x\dt\phi_{b,j}(t)\dt\phi_b^i(t)=\e^2\int_{\Omega}\abs{\dt\nx\phi_{b,j}(t)}^2=\onnm{\dt\nx\phi_{b,j}}{L^2}^2.\no
\end{align}
Following a similar argument as in Step 1 - Step 3, we have
\begin{align}\label{htt 10}
\e^2\tnnmt{\dt\nx\phi_{b,j}}\ls \e\onnmt{a}{L^2}+\e\onnmt{c}{L^2}+\e\tnnmt{(\ik-\pk)[f]}+\tnnmt{\snn S}.
\end{align}
\ \\
Step 7: Estimate of $\dt\nx\phi_a$.\\
For fixed $t$, taking $\psi=-\m^{\frac{1}{2}}\dt\phi_a(t)$, using integration by parts, we have
\begin{align}
\e^2\iint_{\Omega\times\r^3}\dt f(t)\psi&=-\e^2\iint_{\Omega\times\r^3}\dt f(t)\m^{\frac{1}{2}}\dt\phi_a(t)=-\e^2\int_{\Omega}\dt a(t)\dt\phi_a(t)\\
&=-\e^2\int_{\Omega}\Delta_x\dt\phi_a(t)\dt\phi_a(t)=\e^2\int_{\Omega}\abs{\dt\nx\phi_a(t)}^2=\e^2\onnm{\dt\nx\phi_a(t)}{L^2}^2.\no
\end{align}
Following a similar argument as in Step 1 - Step 3, we have
\begin{align}\label{htt 11}
\e^2\tnnmt{\dt\nx\phi_a}\ls\e\onnmt{\vbb}{L^2}+\tnnmt{\snn S}.
\end{align}
\ \\
Step 8: Second Synthesis.\\
Inserting \eqref{htt 09}, \eqref{htt 10} and \eqref{htt 11} into \eqref{htt 07}, we have
\begin{align}
\e\tnnmt{\pk[f]}^2\ls&\bigg(\tnnmt{(\ik-\pk)[f]}+\tnnmt{\snn S}+\e\pnmt{(1-\pp)[f]}{\gamma_+}{2}+\e\pnmt{h}{\gamma_-}{2}\bigg)\tnnmt{\pk[f]}\\
&+\tnnmt{f}\bigg(\e\tnnmt{\pk[f]}+\e\tnnmt{(\ik-\pk)[f]}+\tnnmt{\snn S}\bigg)+\e^2\tnnm{f(t)}^2+\e^2\tnnm{z}^2.\no
\end{align}
Applying Cauchy's inequality, we have
\begin{align}
\e\tnnmt{\pk[f]}^2\ls&o(1)\e\tnnmt{\pk[f]}^2+\e^2\tnnm{f(t)}^2+\e\pnmt{(1-\pp)[f]}{\gamma_+}{2}^2+\frac{1}{\e}\tnnmt{(\ik-\pk)[f]}^2\\
&+\frac{1}{\e}\tnnmt{\snn S}^2+\e^2\tnnm{z}^2+\e\pnmt{h}{\gamma_-}{2}^2.\no
\end{align}
Hence, absorbing $o(1)\e\tnnmt{\pk[f]}^2$ into the left-hand side, we have
\begin{align}
\e\tnnmt{\pk[f]}\ls&\e^{\frac{3}{2}}\tnnm{f(t)}+\e\pnmt{(1-\pp)[f]}{\gamma_+}{2}+\tnnmt{(\ik-\pk)[f]}\\
&+\tnnmt{\snn S}+\e^{\frac{3}{2}}\tnnm{z}+\e\pnmt{h}{\gamma_-}{2}.\no
\end{align}
This completes our proof.
\end{proof}

\begin{theorem}\label{LT estimate'}
Assume \eqref{linear unsteady compatibility} and \eqref{linear unsteady normalization} hold. The solution $f(t,\vx,\vv)$ to the equation \eqref{linear unsteady} satisfies
\begin{align}\label{htt 25}
&\tnnm{f(t)}+\frac{1}{\e^{\frac{1}{2}}}\tsmt{(1-\pp)[f]}{\gamma_+}+\frac{1}{\e}\unnmt{(\ik-\pk)[f]}+\tnnmt{\pk[f]}\\
\ls& \frac{1}{\e^2}\tnnmt{\pk[S]}+\frac{1}{\e}\tnnmt{\snn(\ik-\pk)[S]}+\frac{1}{\e}\tsmt{h}{\gamma_-}+\tnnm{z}.\no
\end{align}
\end{theorem}
\begin{proof}
\ \\
Step 1: Energy Estimate.\\
Multiplying $f$ on both sides of \eqref{linear unsteady} and applying Green's identity in Lemma \ref{htt lemma 03} imply
\begin{align}\label{htt 13}
\frac{\e^2}{2}\tnnm{f(t)}^2+\frac{\e}{2}\tsmt{f}{\gamma_+}^2-\frac{\e}{2}\tsmt{\pp[f]+h}{\gamma_-}^2+\int_0^t\int_{\Omega\times\r^3}f\ll[f]
=&\frac{\e^2}{2}\tnnm{z}^2+\int_0^t\int_{\Omega\times\r^3}fS.
\end{align}
Direct computation reveals that
\begin{align}\label{htt 14}
\frac{\e}{2}\tsmt{f}{\gamma_+}^2-\frac{\e}{2}\tsmt{\pp[f]+h}{\gamma_-}^2
&=\frac{\e}{2}\tsmt{f}{\gamma_+}^2-\frac{\e}{2}\tsmt{\pp[f]}{\gamma_-}^2-\frac{\e}{2}\tsmt{h}{\gamma_-}^2+\e\int_0^t\int_{\gamma_-}h\pp[f]\ud\gamma\\
&=\frac{\e}{2}\tsmt{(1-\pp)[f]}{\gamma_+}^2-\frac{\e}{2}\tsmt{h}{\gamma_-}^2+\e\int_0^t\int_{\gamma_-}h\pp[f]\ud\gamma\no\\
&\gs \e\tsmt{(1-\pp)[f]}{\gamma_+}^2-\frac{1}{\eta}\tsmt{h}{\gamma_-}^2-\e^2\eta\tsmt{\pp[f]}{\gamma_+},\no
\end{align}
where $0<\eta<<1$ will be determined later. On the other hand, based on Lemma \ref{htt lemma 05}, we know
\begin{align}\label{htt 15}
\int_0^t\int_{\Omega\times\r^3}f\ll[f]\gs \unnmt{(\ik-\pk)[f]}^2.
\end{align}
Inserting \eqref{htt 14} and \eqref{htt 15} into \eqref{htt 13}, we have
\begin{align}\label{htt 16}
\\
\e^2\tnnm{f(t)}^2+\e\tsmt{(1-\pp)[f]}{\gamma_+}^2+\unnmt{(\ik-\pk)[f]}^2
\ls \eta\e^2\tsmt{\pp[f]}{\gamma_+}+\e^2\tnnm{z}^2+\frac{1}{\eta}\tsmt{h}{\gamma_-}^2+\int_0^t\int_{\Omega\times\r^3}fS.\no
\end{align}
\ \\
Step 2: $\tsmt{\pp[f]}{\gamma_+}$.\\
Multiplying $f$ on both sides of the equation \eqref{linear unsteady}, we have
\begin{align}\label{htt 17}
\e\dt(f^2)+\vv\cdot\nx(f^2)=\frac{2}{\e}\Big(-f\ll[f]+fS\Big).
\end{align}
Taking absolute value and integrating \eqref{htt 17} over $[0,t]\times\Omega\times\r^3$, using Lemma \ref{htt lemma 05}, we deduce
\begin{align}
\pnnmt{\e\dt(f^2)+\vv\cdot\nx(f^2)}{1}\ls&\frac{1}{\e}\bigg(\tnnmt{(\ik-\pk)[f]}^2+\abs{\int_0^t\int_{\Omega\times\r^3}fS}\bigg).
\end{align}
On the other hand, applying Lemma \ref{htt lemma 02} to $f^2$, for near grazing set $\gamma^{\d}$, we have
\begin{align}\label{htt 19}
\pnmt{\id_{\gamma\backslash\gamma^{\d}}f}{\gamma}{2}^2&=\pnmt{\id_{\gamma\backslash\gamma^{\d}}f^2}{\gamma}{1}
\leq
C(\d)\left(\e\pnnm{z^2}{1}+\pnnmt{f^2}{1}+\pnnmt{\e\dt(f^2)+\vv\cdot\nx(f^2)}{1}\right)\\
&=C(\d)\left(\e\tnnm{z}^2+\tnnmt{f}^2+\pnnmt{\e\dt(f^2)+\vv\cdot\nx(f^2)}{1}\right)\no\\
&\ls C(\d)\left(\e\tnnm{z}^2+\tnnmt{f}^2+\frac{1}{\e}\tnnmt{(\ik-\pk)[f]}^2+\frac{1}{\e}\abs{\int_0^t\int_{\Omega\times\r^3}fS}\right).\no
\end{align}
We can rewrite $\pp[f](t,\vx_0,\vv)=y(t,\vx)\mh(\vv)$. Then for $\d$ small, we deduce
\begin{align}\label{htt 18}
\tsmt{\pp[\id_{\gamma\backslash\gamma^{\d}}f]}{\gamma}^2
=&\int_0^t\int_{\p\Omega}\abs{y(t,\vx)}^2
\left(\int_{\vv\cdot\vn(\vx)\geq\d,\d\leq\abs{\vv}\leq\d^{-1}}\m(\vv)\abs{\vv\cdot\vn(\vx)}\ud{\vv}\right)\ud{\vx}\\
\geq&\half\left(\int_0^t\int_{\p\Omega}\abs{y(t,\vx)}^2\right)\left(\int_{\gamma_+}\m(\vv)\abs{\vv\cdot\vn(\vx)}\ud{\vv}\right)=\half\tsmt{\pp[f]}{\gamma_+}^2,\no
\end{align}
where we utilize the bounds that
\begin{align}
\int_{\vv\cdot\vn(\vx)\leq\d}\m(\vv)\abs{\vv\cdot\vn(\vx)}\ud{\vv}\ls&\d,\\
\int_{\abs{\vv}\leq\d\ \text{or}\ \abs{\vv}\geq\d^{-1}}\m(\vv)\abs{\vv\cdot\vn(\vx)}\ud{\vv}\ls&\d.
\end{align}
Therefore, from \eqref{htt 18} and the fact
\begin{align}
\tsmt{\pp[\id_{\gamma\backslash\gamma^{\d}}f]}{\gamma_+}\ls
\tsmt{\id_{\gamma\backslash\gamma^{\d}}f}{\gamma_+}\ls \tsmt{\id_{\gamma\backslash\gamma^{\d}}f}{\gamma},
\end{align}
we conclude
\begin{align}
\tsmt{\pp[f]}{\gamma_+}^2\ls \tsmt{\pp[\id_{\gamma\backslash\gamma^{\d}}f]}{\gamma_+}\ls\tsmt{\id_{\gamma\backslash\gamma^{\d}}f}{\gamma}.
\end{align}
Considering \eqref{htt 19}, we have
\begin{align}
\tsmt{\pp[f]}{\gamma_+}^2\ls&C(\d)\left(\e\tnnm{z}^2+\tnnmt{f}^2+\frac{1}{\e}\tnnmt{(\ik-\pk)[f]}^2+\frac{1}{\e}\abs{\int_0^t\int_{\Omega\times\r^3}fS}\right).
\end{align}
For fixed $0<\d<<1$ and using $f=\pk[f]+(\ik-\pk)[f]$, we obtain
\begin{align}\label{htt 20}
\tsmt{\pp[f]}{\gamma_+}^2\ls&\e\tnnm{z}^2+\tnnmt{\pk[f]}^2+\frac{1}{\e}\tnnmt{(\ik-\pk)[f]}^2+\frac{1}{\e}\abs{\int_0^t\int_{\Omega\times\r^3}fS}.
\end{align}
\ \\
Step 3: Synthesis.\\
Plugging \eqref{htt 20} into \eqref{htt 16} with $\e$ sufficiently small to absorb $\tnnmt{(\ik-\pk)[f]}^2$ into the left-hand side, we obtain
\begin{align}\label{htt 21}
\\
\e^2\tnnm{f(t)}^2+\e\tsmt{(1-\pp)[f]}{\gamma_+}^2+\unnmt{(\ik-\pk)[f]}^2
\ls \eta\e^2\tnnmt{\pk[f]}^2+\e^2\tnnm{z}^2+\frac{1}{\eta}\tsmt{h}{\gamma_-}^2+\abs{\int_0^t\int_{\Omega\times\r^3}fS}.\no
\end{align}
We square on both sides of \eqref{htt 12} to obtain
\begin{align}\label{htt 22}
\\
\e^2\tnnmt{\pk[f]}^2\ls&\e^{3}\tnnm{f(t)}^2+\e^2\pnmt{(1-\pp)[f]}{\gamma_+}{2}^2+\tnnmt{(\ik-\pk)[f]}^2
+\tnnmt{\snn S}^2+\e^{3}\tnnm{z}^2+\e^2\pnmt{h}{\gamma_-}{2}^2.\no
\end{align}
Multiplying a small constant on both sides of \eqref{htt 22} and adding to \eqref{htt 21} with $\eta$ sufficiently small to absorb $\e^2\pnmt{(1-\pp)[f]}{\gamma_+}{2}^2$, $\tnnmt{(\ik-\pk)[f]}^2$, $\e^{3}\tnnm{f(t)}^2$ and $\eta\e^2\tnnmt{\pk[f]}^2$ into the left-hand side, we obtain
\begin{align}\label{htt 24}
&\e^2\tnnm{f(t)}^2+\e\tsmt{(1-\pp)[f]}{\gamma_+}^2+\unnmt{(\ik-\pk)[f]}^2+\e^2\tnnmt{\pk[f]}^2\\
\ls& \tsmt{h}{\gamma_-}^2+\e^2\tnnm{z}^2+\tnnmt{\snn S}^2+\abs{\int_0^t\int_{\Omega\times\r^3}fS}.\no
\end{align}
Applying Cauchy's inequality, we have
\begin{align}\label{htt 23}
\abs{\int_0^t\int_{\Omega\times\r^3}fS}\ls &\abs{\int_0^t\int_{\Omega\times\r^3}(\ik-\pk)[f](\ik-\pk)[S]}+\abs{\int_0^t\int_{\Omega\times\r^3}\pk[f]\pk[S]}\\
\ls&o(1)\unnmt{(\ik-\pk)[f]}^2+\tnnmt{\snn(\ik-\pk)[S]}^2+o(1)\e^2\tnnmt{\pk[f]}^2+\frac{1}{\e^2}\tnnmt{\pk[S]}^2.\no
\end{align}
Inserting \eqref{htt 23} into \eqref{htt 24} to absorb $o(1)\tnnmt{(\ik-\pk)[f]}^2$ and $o(1)\e^2\tnnmt{\pk[f]}^2$ into the left-hand side, we obtain
\begin{align}\label{htt 24}
&\e^2\tnnm{f(t)}^2+\e\tsmt{(1-\pp)[f]}{\gamma_+}^2+\unnmt{(\ik-\pk)[f]}^2+\e^2\tnnmt{\pk[f]}^2\\
\ls& \frac{1}{\e^2}\tnnmt{\pk[S]}^2+\tnnmt{\snn(\ik-\pk)[S]}^2+\tsmt{h}{\gamma_-}^2+\e^2\tnnm{z}^2.\no
\end{align}
Hence, our desired result naturally follows.
\end{proof}

\begin{corollary}
Since \eqref{linear unsteady} is a linear equation, taking time derivative on both sides, we know $\dt f$ satisfies
\begin{align}\label{linear unsteady'}
\left\{
\begin{array}{l}
\e^2\dt(\dt f)+\e\vv\cdot\nx(\dt f)+\ll[\dt f]=\dt S(t,\vx,\vv)\ \ \text{in}\ \ \rp\times\Omega\times\r^3,\\\rule{0ex}{1.5em}
\dt f(0,\vx,\vv)=-\dfrac{1}{\e^2}\ll[z(\vx,\vv)]-\dfrac{1}{\e}\vv\cdot\nx z(\vx,\vv)+\dfrac{1}{\e^2}S(0,\vx,\vv)\ \ \text{in}\ \ \Omega\times\r^3,\\\rule{0ex}{1.5em}
\dt f(t,\vx_0,\vv)=\pp[\dt f](t,\vx_0,\vv)+\dt h(t,\vx_0,\vv)\ \ \text{on}\ \ \rp\times\gamma_-,
\end{array}
\right.
\end{align}
where we solve the initial data $\dt f(0,\vx,\vv)$ from \eqref{linear unsteady}. Then applying Lemma \ref{LT estimate'} to \eqref{linear unsteady'}, we obtain
\begin{align}\label{htt 26}
&\tnnm{\dt f(t)}+\frac{1}{\e^{\frac{1}{2}}}\tsmt{(1-\pp)[\dt f]}{\gamma_+}+\frac{1}{\e}\unnmt{(\ik-\pk)[\dt f]}+\tnnmt{\pk[\dt f]}\\
\ls& \frac{1}{\e^2}\tnnmt{\pk[\dt S]}+\frac{1}{\e}\tnnmt{\snn(\ik-\pk)[\dt S]}+\frac{1}{\e}\tsmt{\dt h}{\gamma_-}
+\frac{1}{\e^2}\tnnm{\nu z}+\frac{1}{\e}\tnnm{\vv\cdot\nx z}+\frac{1}{\e^2}\tnnm{S(0)}.\no
\end{align}
\end{corollary}

\subsection{$L^{2m}$ Estimates}

Throughout this section, we need $\dfrac{3}{2}<m<3$. Let $o(1)$ denote a sufficiently small constant.

\begin{lemma}\label{htt lemma 07}
Assume \eqref{linear unsteady compatibility} and \eqref{linear unsteady normalization} hold. The solution $f(t,\vx,\vv)$ to the equation \eqref{linear unsteady} satisfies
\begin{align}\label{htt 34}
\e\pnnm{\pk[f(t)]}{2m}\ls&\e\pnm{(1-\pp)[f(t)]}{\gamma_+}{\frac{4m}{3}}+\tnnm{(\ik-\pk)[f(t)]}+\e\pnnm{(\ik-\pk)[f(t)]}{2m}\\
&+\tnnm{\snn S(t)}+\e\pnm{h(t)}{\gamma_-}{\frac{4m}{3}}+\e^2\tnnm{\dt f(t)}.\no
\end{align}
\end{lemma}
\begin{proof}
This is very similar to the proof of Lemma \ref{htt lemma 04} and the stationary version in Lemma \ref{ktt lemma 1}. We apply Green's identity to the equation \eqref{linear unsteady} and choose particular test functions to control $a$, $\vbb$ and $c$. However, there is no simple way to get around the $\dt\nx\phi$ terms as in Step 5 - Step 7 of the proof of Lemma \ref{htt lemma 04}. Here, we resort to stationary techniques, i.e. to use time-independent Green's identity instead of time-dependent one.

Apply Green's identity in Lemma \ref{ktt 03} to the equation \eqref{linear unsteady}. Then for any $\psi(t)\in L^2(\Omega\times\r^3)$
satisfying $\vv\cdot\nx\psi(t)\in L^2(\Omega\times\r^3)$ and $\psi(t)\in L^2(\gamma)$, we have
\begin{align}
&\e\iint_{\gamma_+}f(t)\psi(t)\ud{\gamma}-\e\iint_{\gamma_-}f(t)\psi(t)\ud{\gamma}
-\e\iint_{\Omega\times\r^3}\Big(\vv\cdot\nx\psi(t)\Big)f(t)\\
=&-\iint_{\Omega\times\r^3}\psi(t)\ll\Big[(\ik-\pk)[f](t)\Big]+\iint_{\Omega\times\r^3}S(t)\psi(t)-\e^2\iint_{\Omega\times\r^3}\dt f(t)\psi(t).\no
\end{align}
Then except from $\ds-\e^2\iint_{\Omega\times\r^3}\dt f(t)\psi(t)$, this is exactly the same as the stationary estimates in Lemma \ref{ktt lemma 1}, so we just mimick the proof there and that of Lemma \ref{htt lemma 04}, and point out the major differences. In particular, we always use the bound
\begin{align}
\abs{\e^2\iint_{\Omega\times\r^3}\dt f(t)\psi(t)}\ls \e^2\tnnmt{\dt f(t)}\tnnmt{\psi(t)}.
\end{align}
\ \\
Step 1: Estimates of $c$.\\
We choose the test function
\begin{align}
\psi(t)=\psi_c(t)=\m^{\frac{1}{2}}(\vv)\left(\abs{\vv}^2-\beta_c\right)\Big(\vv\cdot\nx\phi_c(t,\vx)\Big),
\end{align}
where
\begin{align}
\left\{
\begin{array}{l}
-\dx\phi_c(t)=c\abs{c}^{2m-2}(t,\vx)\ \ \text{in}\ \ \Omega,\\\rule{0ex}{1.5em}
\phi_c(t)=0\ \ \text{on}\ \ \p\Omega,
\end{array}
\right.
\end{align}
and $\beta_c\in\r$ will be determined as in stationary problem. Based on the standard elliptic estimates in \cite{Krylov2008}, we have
\begin{align}
\onnm{\phi_c(t)}{W^{2,\frac{2m}{2m-1}}}\ls \onnm{\abs{c(t)}^{2m-1}}{L^{\frac{2m}{2m-1}}}\ls \onnm{c(t)}{L^{2m}}^{2m-1}.
\end{align}
Hence, by Sobolev embedding theorem, we know
\begin{align}
&\tnnm{\psi_c(t)}\ls\onnm{\phi_c}{H^1}\ls\onnm{\phi_c(t)}{W^{2,\frac{2m}{2m-1}}}\ls\onnm{c(t)}{L^{2m}}^{2m-1},\\
&\onnm{\phi_c(t)}{W^{1,\frac{2m}{2m-1}}}\ls\onnm{\phi_c(t)}{W^{2,\frac{2m}{2m-1}}}\ls\onnm{c(t)}{L^{2m}}^{2m-1}.
\end{align}
Also, for $1\leq m\leq 3$, using Sobolev embedding theorem and trace estimates, we have
\begin{align}
\onm{\nx\phi_c(t)}{L^{\frac{4m}{4m-3}}}&\ls\onm{\nx\phi_c(t)}{W^{\frac{1}{2m},\frac{2m}{2m-1}}}\ls\onnm{\nx\phi_c(t)}{W^{1,\frac{2m}{2m-1}}}\\
&\ls \onnm{\phi_c(t)}{W^{2,\frac{2m}{2m-1}}}\ls\onnm{c(t)}{L^{2m}}^{2m-1}.\no
\end{align}
Eventually, we have
\begin{align}\label{htt 29}
\e\onnm{c(t)}{L^{2m}}\ls&
\e\pnm{(1-\pp)[f(t)]}{\gamma_+}{\frac{4m}{3}}+\tnnm{(\ik-\pk)[f(t)]}+\e\pnnm{(\ik-\pk)[f(t)]}{2m}\\
&+\tnnm{\snn S(t)}+\e\pnm{h(t)}{\gamma_-}{\frac{4m}{3}}+\e^2\tnnm{\dt f(t)}.\no
\end{align}
\ \\
Step 2: Estimates of $\vbb$.\\
We further divide this step into several sub-steps:\\
\ \\
Sub-Step 2.1: Estimates of $\bigg(\p_{i}\p_j\dx^{-1}\Big(b_j\abs{b_j}^{2m-2}\Big)\bigg)b_i$ for $i,j=1,2,3$.\\
Let $\vbb=(b_1,b_2,b_3)$. We choose the test functions for $i,j=1,2,3$,
\begin{align}
\psi(t)=\psi_{b,i,j}(t)=\m^{\frac{1}{2}}(\vv)\left(v_i^2-\beta_{b,i,j}\right)\p_j\phi_{b,j},
\end{align}
where
\begin{align}
\left\{
\begin{array}{l}
-\dx\phi_{b,j}(t)=b_j\abs{b_j}^{2m-2}(t,\vx)\ \ \text{in}\ \ \Omega,\\\rule{0ex}{1.5em}
\phi_{b,j}(t)=0\ \ \text{on}\ \ \p\Omega,
\end{array}
\right.
\end{align}
and $\beta_{b,i,j}\in\r$ will be determined as in stationary problem. We can recover the elliptic estimates and trace estimates. Eventually, we have
\begin{align}\label{htt 27}
&\e\abs{\int_{\Omega}\bigg(\p_{i}\p_j\dx^{-1}\Big(b_j\abs{b_j}^{2m-2}\Big)\bigg)b_i}\\
\ls&\onnm{\vbb(t)}{L^{2m}}^{2m-1}\bigg(\e\pnm{(1-\pp)[f(t)]}{\gamma_+}{\frac{4m}{3}}+\tnnm{(\ik-\pk)[f(t)]}+\e\pnnm{(\ik-\pk)[f(t)]}{2m}\no\\
&+\tnnm{\snn S(t)}+\e\pnm{h(t)}{\gamma_-}{\frac{4m}{3}}+\e^2\tnnm{\dt f(t)}\bigg).\no
\end{align}
\ \\
Sub-Step 2.2: Estimates of $\bigg(\p_{i}\p_i\dx^{-1}\Big(b_j\abs{b_j}^{2m-2}\Big)\bigg)b_j$ for $i\neq j$.\\
Notice that the $i=j$ case is included in Sub-Step 2.1. We choose the test function
\begin{align}
\psi(t)=\tilde\psi_{b,i,j}(t)=\m^{\frac{1}{2}}(\vv)\abs{\vv}^2v_iv_j\p_i\phi_{b,j}\ \ \text{for}\ \ i\neq j.
\end{align}
Eventually, we have
\begin{align}\label{htt 28}
&\e\abs{\int_{\Omega}\bigg(\p_{i}\p_i\dx^{-1}\Big(b_j\abs{b_j}^{2m-2}\Big)\bigg)b_j}\\
\ls&\onnm{\vbb(t)}{L^{2m}}^{2m-1}\bigg(\e\pnm{(1-\pp)[f(t)]}{\gamma_+}{\frac{4m}{3}}+\tnnm{(\ik-\pk)[f(t)]}+\e\pnnm{(\ik-\pk)[f(t)]}{2m}\no\\
&+\tnnm{\snn S(t)}+\e\pnm{h(t)}{\gamma_-}{\frac{4m}{3}}+\e^2\tnnm{\dt f(t)}\bigg).\no
\end{align}
\ \\
Sub-Step 2.3: Synthesis.\\
Summarizing \eqref{htt 27} and \eqref{htt 28}, we may sum up over $j=1,2,3$ to obtain, for any $i=1,2,3$,
\begin{align}
\e\onnm{b_i(t)}{L^{2m}}^{2m}
\ls&\onnm{\vbb(t)}{L^{2m}}^{2m-1}\bigg(\e\pnm{(1-\pp)[f(t)]}{\gamma_+}{\frac{4m}{3}}+\tnnm{(\ik-\pk)[f(t)]}+\e\pnnm{(\ik-\pk)[f(t)]}{2m}\\
&+\tnnm{\snn S(t)}+\e\pnm{h(t)}{\gamma_-}{\frac{4m}{3}}+\e^2\tnnm{\dt f(t)}\bigg).\no
\end{align}
which further implies
\begin{align}\label{htt 30}
\e\onnm{\vbb(t)}{L^{2m}}
\ls&\e\pnm{(1-\pp)[f(t)]}{\gamma_+}{\frac{4m}{3}}+\tnnm{(\ik-\pk)[f(t)]}+\e\pnnm{(\ik-\pk)[f(t)]}{2m}\\
&+\tnnm{\snn S(t)}+\e\pnm{h(t)}{\gamma_-}{\frac{4m}{3}}+\e^2\tnnm{\dt f(t)}.\no
\end{align}
\ \\
Step 3: Estimates of $a$.\\
We choose the test function
\begin{align}
\psi(t)=\psi_a(t)=\m^{\frac{1}{2}}(\vv)\left(\abs{\vv}^2-\beta_a\right)\Big(\vv\cdot\nx\phi_a(t,\vx)\Big),
\end{align}
where
\begin{align}
\left\{
\begin{array}{l}
-\dx\phi_a(t)=a\abs{a}^{2m-2}(t,\vx)-\dfrac{1}{\abs{\Omega}}\ds\int_{\Omega}a\abs{a}^{2m-2}(t,\vx)\ud{\vx}\ \ \text{in}\ \ \Omega,\\\rule{0ex}{1.5em}
\dfrac{\p\phi_a(t)}{\p\vn}=0\ \ \text{on}\ \ \p\Omega,
\end{array}
\right.
\end{align}
and $\beta_a\in\r$ will be determined as in stationary problem. We can recover the elliptic estimates and trace estimates. Eventually, we have
\begin{align}\label{htt 31}
\e\onnm{a(t)}{L^{2m}}\ls&
\e\pnm{(1-\pp)[f(t)]}{\gamma_+}{\frac{4m}{3}}+\tnnm{(\ik-\pk)[f(t)]}+\e\pnnm{(\ik-\pk)[f(t)]}{2m}\\
&+\tnnm{\snn S(t)}+\e\pnm{h(t)}{\gamma_-}{\frac{4m}{3}}+\e^2\tnnm{\dt f(t)}.\no
\end{align}
\ \\
Step 4: Synthesis.\\
Collecting \eqref{htt 29}, \eqref{htt 30} and \eqref{htt 31}, we deduce
\begin{align}\label{htt 33}
\e\pnnm{\pk[f(t)]}{2m}\ls&\e\pnm{(1-\pp)[f(t)]}{\gamma_+}{\frac{4m}{3}}+\tnnm{(\ik-\pk)[f(t)]}+\e\pnnm{(\ik-\pk)[f(t)]}{2m}\\
&+\tnnm{\snn S(t)}+\e\pnm{h(t)}{\gamma_-}{\frac{4m}{3}}+\e^2\tnnm{\dt f(t)}.\no
\end{align}
\end{proof}

\begin{theorem}\label{LN estimate'.}
Assume \eqref{linear unsteady compatibility} and \eqref{linear unsteady normalization} hold. The solution $f(t,\vx,\vv)$ to the equation \eqref{linear unsteady} satisfies
\begin{align}
&\frac{1}{\e^{\frac{1}{2}}}\tsm{(1-\pp)[f(t)]}{\gamma_+}+\frac{1}{\e}\unnm{(\ik-\pk)[f(t)]}+\pnnm{\pk[f(t)]}{2m}\\
&+\frac{1}{\e^{\frac{1}{2}}}\tsmt{(1-\pp)[\dt f]}{\gamma_+}+\frac{1}{\e}\unnmt{(\ik-\pk)[\dt f]}+\tnnmt{\pk[\dt f]}\no\\
\ls&o(1)\e^{\frac{3}{2m}}\Big(\ssm{f(t)}{\gamma_+}+\snnm{f(t)}\Big)\no\\
&+\frac{1}{\e^{2}}\pnnm{\pk[S(t)]}{\frac{2m}{2m-1}}+\frac{1}{\e}\tnnm{\snn(\ik-\pk)[S(t)]}
+\frac{1}{\e^2}\tnnmt{\pk[\dt S]}+\frac{1}{\e}\tnnmt{\snn(\ik-\pk)[\dt S]}\no\\
&+\pnm{h(t)}{\gamma_-}{\frac{4m}{3}}+\frac{1}{\e}\tsm{h(t)}{\gamma_-}+\frac{1}{\e}\tsmt{\dt h}{\gamma_-}+\frac{1}{\e^2}\tnnm{\nu z}+\frac{1}{\e}\tnnm{\vv\cdot\nx z}+\frac{1}{\e^2}\tnnm{S(0)}.\no
\end{align}
\end{theorem}
\begin{proof}
\ \\
Step 1: Energy Estimate.\\
Multiplying $f$ on both sides of \eqref{linear unsteady} and use the similar estimates as in the proof of Lemma \ref{LT estimate'}, the stationary energy structure implies
\begin{align}\label{htt 35}
&\e\tsm{(1-\pp)[f(t)]}{\gamma_+}^2+\unnm{(\ik-\pk)[f(t)]}^2\\
\ls& \eta\e^2\tnnm{\pk[f(t)]}^2+\frac{1}{\eta}\tsm{h(t)}{\gamma_-}^2+\abs{\int_{\Omega\times\r^3}f(t)S(t)}+\e^2\abs{\int_{\Omega\times\r^3}f(t)\dt f(t)}.\no
\end{align}
We square on both sides of \eqref{htt 34} to obtain
\begin{align}\label{htt 36}
\e^2\pnnm{\pk[f(t)]}{2m}^2\ls&\e^2\pnm{(1-\pp)[f(t)]}{\gamma_+}{\frac{4m}{3}}^2+\tnnm{(\ik-\pk)[f(t)]}^2+\e^2\pnnm{(\ik-\pk)[f(t)]}{2m}^2\\
&+\tnnm{\snn S(t)}^2+\e^2\pnm{h(t)}{\gamma_-}{\frac{4m}{3}}^2+\e^4\tnnm{\dt f(t)}^2.\no
\end{align}
H\"older's inequality implies
\begin{align}
\tnnmt{\pk[f(t)]}\ls \pnnmt{\pk[f(t)]}{2m}.
\end{align}
Multiplying a small constant on both sides of \eqref{htt 36} and
adding to \eqref{htt 35} with $\eta$ sufficiently small to absorb $\eta\e^2\tnnmt{\pk[f(t)]}^2$, and $\tnnmt{(\ik-\pk)[f(t)]}^2$ into the left-hand side, we obtain
\begin{align}\label{htt 39}
&\e\tsm{(1-\pp)[f(t)]}{\gamma_+}^2+\unnm{(\ik-\pk)[f(t)]}^2+\e^2\pnnm{\pk[f(t)]}{2m}^2\\
\ls&\e^2\pnm{(1-\pp)[f(t)]}{\gamma_+}{\frac{4m}{3}}^2+\e^2\pnnm{(\ik-\pk)[f(t)]}{2m}^2+\e^4\tnnm{\dt f(t)}^2\no\\
&+\tnnm{\snn S(t)}^2+\e^2\pnm{h(t)}{\gamma_-}{\frac{4m}{3}}^2+\tsm{h(t)}{\gamma_-}^2+\abs{\int_{\Omega\times\r^3}f(t)S(t)}+\e^2\abs{\int_{\Omega\times\r^3}f(t)\dt f(t)}.\no
\end{align}
Now we have to handle $\e^2\pnmt{(1-\pp)[f(t)]}{\gamma_+}{\frac{4m}{3}}^2$ and $\e^2\pnnmt{(\ik-\pk)[f(t)]}{2m}^2$ on the right-hand side.\\
\ \\
Step 2: Interpolation Argument.\\
By interpolation estimate and Young's inequality, we have
\begin{align}
\pnm{(1-\pp)[f(t)]}{\gamma_+}{\frac{4m}{3}}\leq&\tsm{(1-\pp)[f(t)]}{\gamma_+}^{\frac{3}{2m}}\ssm{(1-\pp)[f(t)]}{\gamma_+}^{\frac{2m-3}{2m}}\\
=&\bigg(\frac{1}{\e^{\frac{6m-9}{4m^2}}}\tsm{(1-\pp)[f(t)]}{\gamma_+}^{\frac{3}{2m}}\bigg)
\bigg(\e^{\frac{6m-9}{4m^2}}\ssm{(1-\pp)[f(t)]}{\gamma_+}^{\frac{2m-3}{2m}}\bigg)\no\\
\ls&\bigg(\frac{1}{\e^{\frac{6m-9}{4m^2}}}\tsm{(1-\pp)[f(t)]}{\gamma_+}^{\frac{3}{2m}}\bigg)^{\frac{2m}{3}}+o(1)
\bigg(\e^{\frac{6m-9}{4m^2}}\ssm{(1-\pp)[f(t)]}{\gamma_+}^{\frac{2m-3}{2m}}\bigg)^{\frac{2m}{2m-3}}\no\\
\leq&\frac{1}{\e^{\frac{2m-3}{2m}}}\tsm{(1-\pp)[f(t)]}{\gamma_+}+o(1)\e^{\frac{3}{2m}}\ssm{(1-\pp)[f(t)]}{\gamma_+}\no\\
\leq&\frac{1}{\e^{\frac{2m-3}{2m}}}\tsm{(1-\pp)[f(t)]}{\gamma_+}+o(1)\e^{\frac{3}{2m}}\ssm{(1-\pp)[f(t)]}{\gamma_+}.\no
\end{align}
Similarly, we have
\begin{align}
\pnnm{(\ik-\pk)[f(t)]}{2m}\leq&\tnnm{(\ik-\pk)[f(t)]}^{\frac{1}{m}}\snnm{(\ik-\pk)[f(t)]}^{\frac{m-1}{m}}\\
=&\bigg(\frac{1}{\e^{\frac{3m-3}{2m^2}}}\tnnm{(\ik-\pk)[f(t)]}^{\frac{1}{m}}\bigg)
\bigg(\e^{\frac{3m-3}{2m^2}}\snnm{(\ik-\pk)[f(t)]}^{\frac{m-1}{m}}\bigg)\no\\
\ls&\bigg(\frac{1}{\e^{\frac{3m-3}{2m^2}}}\tnnm{(\ik-\pk)[f(t)]}^{\frac{1}{m}}\bigg)^{m}+o(1)
\bigg(\e^{\frac{3m-3}{2m^2}}\snnm{(\ik-\pk)[f(t)]}^{\frac{m-1}{m}}\bigg)^{\frac{m}{m-1}}\no\\
\leq&\frac{1}{\e^{\frac{3m-3}{2m}}}\tnnm{(\ik-\pk)[f(t)]}+o(1)\e^{\frac{3}{2m}}\snnm{(\ik-\pk)[f(t)]}.\no
\end{align}
We need this extra $\e^{\frac{3}{2m}}$ for the convenience of $L^{\infty}$ estimate. Then we know for sufficiently small $\e$ and $\dfrac{3}{2}< m<3$,
\begin{align}\label{htt 37}
\e^2\pnm{(1-\pp)[f(t)]}{\gamma_+}{\frac{4m}{3}}^2
\ls&\e^{2-\frac{2m-3}{m}}\tsm{(1-\pp)[f(t)]}{\gamma_+}^2+o(1)\e^{2+\frac{3}{m}}\ssm{(1-\pp)[f(t)]}{\gamma_+}^2\\
\ls&o(1)\e\tsm{(1-\pp)[f(t)]}{\gamma_+}^2+o(1)\e^{2+\frac{3}{m}}\ssm{f(t)}{\gamma_+}^2.\no
\end{align}
Similarly, we have
\begin{align}\label{htt 38}
\e^2\pnnm{(\ik-\pk)[f(t)]}{2m}^2\ls&\e^{2-\frac{3m-3}{m}}\tnnm{(\ik-\pk)[f(t)]}^2+o(1)\e^{2+\frac{3}{m}}\snnm{(\ik-\pk)[f(t)]}^2\\
\ls& o(1)\tnnm{(\ik-\pk)[f(t)]}^2+o(1)\e^{2+\frac{3}{m}}\snnm{f(t)}^2.\no
\end{align}
Inserting \eqref{htt 37} and \eqref{htt 38} into \eqref{htt 39}, and absorbing $o(1)\e\tsm{(1-\pp)[f(t)]}{\gamma_+}^2$ and $o(1)\tnnm{(\ik-\pk)[f(t)]}^2$ into the left-hand side, we obtain
\begin{align}\label{htt 39'}
&\e\tsm{(1-\pp)[f(t)]}{\gamma_+}^2+\unnm{(\ik-\pk)[f(t)]}^2+\e^2\pnnm{\pk[f(t)]}{2m}^2\\
\ls&o(1)\e^{2+\frac{3}{m}}\Big(\ssm{f(t)}{\gamma_+}^2+\snnm{f(t)}^2\Big)+\e^4\tnnm{\dt f(t)}^2\no\\
&+\tnnm{\snn S(t)}^2+\e^2\pnm{h(t)}{\gamma_-}{\frac{4m}{3}}^2+\tsm{h(t)}{\gamma_-}^2+\abs{\int_{\Omega\times\r^3}f(t)S(t)}+\e^2\abs{\int_{\Omega\times\r^3}f(t)\dt f(t)}.\no
\end{align}
\ \\
Step 3: Synthesis.\\
We can decompose
\begin{align}\label{htt 40}
\int_{\Omega\times\r^3}f(t)S(t)=\iint_{\Omega\times\r^3}\pk[f(t)]\pk[S(t)]+\iint_{\Omega\times\r^3}(\ik-\pk)[f(t)](\ik-\pk)[S(t)].
\end{align}
H\"older's inequality and Cauchy's inequality imply
\begin{align}\label{htt 41}
\iint_{\Omega\times\r^3}\pk[f(t)]\pk[S(t)]\leq\pnnm{\pk[f(t)]}{2m}\pnnm{\pk[S(t)]}{\frac{2m}{2m-1}}
\ls o(1)\e^2\pnnm{\pk[f(t)]}{2m}^2+\frac{1}{\e^{2}}\pnnm{\pk[S(t)]}{\frac{2m}{2m-1}}^2,
\end{align}
and
\begin{align}\label{htt 42}
\iint_{\Omega\times\r^3}(\ik-\pk)[f(t)](\ik-\pk)[S(t)]\ls o(1)\unnm{(\ik-\pk)[f(t)]}^2+\tnnm{\snn(\ik-\pk)[S(t)]}^2.
\end{align}
Inserting \eqref{htt 41} and \eqref{htt 42} into \eqref{htt 40} and further \eqref{htt 39}, absorbing $o(1)\e^2\pnnm{\pk[f](t)}{2m}^2$ and $o(1)\unnm{(\ik-\pk)[f(t)]}^2$ into the left-hand side, we get
\begin{align}\label{htt 42'}
&\e\tsm{(1-\pp)[f(t)]}{\gamma_+}^2+\unnm{(\ik-\pk)[f(t)]}^2+\e^2\pnnm{\pk[f(t)]}{2m}^2\\
\ls&o(1)\e^{2+\frac{3}{m}}\Big(\ssm{f(t)}{\gamma_+}^2+\snnm{f(t)}^2\Big)+\e^4\tnnm{\dt f(t)}^2\no\\
&+\frac{1}{\e^{2}}\pnnm{\pk[S(t)]}{\frac{2m}{2m-1}}^2+\tnnm{\snn(\ik-\pk)[S(t)]}^2+\e^2\pnm{h(t)}{\gamma_-}{\frac{4m}{3}}^2+\tsm{h(t)}{\gamma_-}^2
+\e^2\abs{\int_{\Omega\times\r^3}f(t)\dt f(t)}.\no
\end{align}
Now we handle the most difficult term:
\begin{align}
\e^2\abs{\int_{\Omega\times\r^3}f(t)\dt f(t)}\ls \e^2\tnnm{f(t)}\tnnm{\dt f(t)}\ls o(1)\e^2\tnnm{f(t)}^2+\e^2\tnnm{\dt f(t)}^2.
\end{align}
Here $o(1)\e^2\tnnm{f(t)}^2$ can be absorbed into the left-hand side of \eqref{htt 42'}. Then we resort to \eqref{htt 26} to tackle $\e^2\tnnm{\dt f(t)}^2$:
\begin{align}\label{htt 42''}
&\e^2\tnnm{\dt f(t)}^2+\e\tsmt{(1-\pp)[\dt f]}{\gamma_+}^2+\unnmt{(\ik-\pk)[\dt f]}^2+\e^2\tnnmt{\pk[\dt f]}^2\\
\ls& \frac{1}{\e^2}\tnnmt{\pk[\dt S]}^2+\tnnmt{\snn(\ik-\pk)[\dt S]}^2+\tsmt{\dt h}{\gamma_-}^2
+\frac{1}{\e^2}\tnnm{\nu z}^2+\tnnm{\vv\cdot\nx z}^2+\frac{1}{\e^2}\tnnm{S(0)}^2.\no
\end{align}
Multiplying a small constant on \eqref{htt 42'} and adding it to \eqref{htt 42''} to absorb $\e^2\tnnm{\dt f(t)}^2$, we have
\begin{align}
&\e\tsm{(1-\pp)[f(t)]}{\gamma_+}^2+\unnm{(\ik-\pk)[f(t)]}^2+\e^2\pnnm{\pk[f(t)]}{2m}^2\\
&+\e\tsmt{(1-\pp)[\dt f]}{\gamma_+}^2+\unnmt{(\ik-\pk)[\dt f]}^2+\e^2\tnnmt{\pk[\dt f]}^2\no\\
\ls&o(1)\e^{2+\frac{3}{m}}\Big(\ssm{f(t)}{\gamma_+}^2+\snnm{f(t)}^2\Big)\no\\
&+\frac{1}{\e^{2}}\pnnm{\pk[S(t)]}{\frac{2m}{2m-1}}^2+\tnnm{\snn(\ik-\pk)[S(t)]}^2
+\frac{1}{\e^2}\tnnmt{\pk[\dt S]}^2+\tnnmt{\snn(\ik-\pk)[\dt S]}^2\no\\
&+\e^2\pnm{h(t)}{\gamma_-}{\frac{4m}{3}}^2+\tsm{h(t)}{\gamma_-}^2+\tsmt{\dt h}{\gamma_-}^2+\frac{1}{\e^2}\tnnm{\nu z}^2+\tnnm{\vv\cdot\nx z}^2+\frac{1}{\e^2}\tnnm{S(0)}^2.\no
\end{align}
Then our desired result follows.
\end{proof}

\begin{remark}
Roughly speaking, Theorem \ref{LN estimate'.} justifies that in order to bound instantaneous $f$ in $L^{2m}$, we need the accumulative bound for $f$ and $\dt f$ in $L^2$.
\end{remark}

\subsection{$L^{\infty}$ Estimates}

Now we begin to consider the mild formulation. When tracking the solution backward along the characteristics, once it hits the in-flow boundary or initial time, it either terminates (when hitting the initial time) or is diffusively reflected (when hitting the boundary). Following this idea, we may define the backward stochastic cycles, with multiple hitting times and out-flow integrals.

\begin{definition}[Hitting Time and Position]
For any $(t,\vx,\vv)\in\rp\times\Omega\times\r^3$ with $(\vx,\vv)\notin\gamma_0$, define the backward the hitting time
\begin{align}\label{htt 43}
t_b(t,\vx,\vv):=&\inf\{s>0:\vx-\e s\vv\notin\Omega\ \ \text{or}\ \ t=\e^2s\}.
\end{align}
Also, define the hitting position
\begin{align}
\vx_b:=\vx-\e t_b(\vx,\vv)\vv.
\end{align}
\end{definition}
Note that $\vx_b\in\Omega$ means the characteristic already hit the initial time, and $\vx_b\in\p\Omega$ means the characteristic hits the boundary, so it can be reflected and continue moving.

\begin{definition}[Stochastic Cycle]
For any $(t,\vx,\vv)\in\rp\times\Omega\times\r^3$ with $(\vx,\vv)\notin\gamma_0$, let $(t_0,\vx_0,\vv_0)=(t,\vx,\vv)$. Define the first stochastic triple
\begin{align}
(t_1,\vx_1,\vv_1):=\Big(t-\e^2t_b(\vx_0,\vv_0),\vx_b(\vx_0,\vv_0),\vv_1\Big),
\end{align}
for some $\vv_1$ satisfying $\vv_1\cdot\vn(\vx_1)>0$.

Inductively, assume we know the $k^{th}$ stochastic triple $(t_k,\vx_k,\vv_k)$ with $t_k>0$ (i.e. $\vx_k\in\p\Omega$). Define the $(k+1)^{th}$ stochastic triple
\begin{align}
(t_{k+1},\vx_{k+1},\vv_{k+1}):=\Big(t_k-\e^2t_b(\vx_k,\vv_k),\vx_k(\vx_k,\vv_k),\vv_{k+1}\Big),
\end{align}
for some $\vv_{k+1}$ satisfying $\vv_{k+1}\cdot\vn(\vx_{k+1})>0$.
\end{definition}

\begin{remark}
Roughly speaking, this definition describes one characteristic line with reflection (alternatively so-called stochastic cycle), starting from $(t_k,\vx_k,\vv_k)\in\gamma_+$, tracking back to $(t_{k+1},\vx_{k+1},\vv_k)\in\{0\}\times\Omega\times\r^3$ which will terminate, or $(t_{k+1},\vx_{k+1},\vv_k)\in(0,\infty)\times\gamma_-$, diffusively reflected to $(t_{k+1},\vx_{k+1},\vv_{k+1})\in\gamma_+$, and beginning a new cycle. $t_k$ the actual time the characteristic moves backward. Note that we are free to choose any $\vv_k\cdot\vn(\vx_k)>0$, so different sequence $\ds\{\vv_k\}_{k=1}^{\infty}$ represents different stochastic cycles.
\end{remark}

\begin{definition}[Diffusive Reflection Integral]
Define $\nn_{k}=\{\vv\in \r^3:\vv\cdot\vn(\vx_{k})>0\}$, so the stochastic cycle must satisfy $\vv_k\in\nn_k$. Let the iterated integral for $k\geq2$ be defined as
\begin{align}
\int_{\prod_{j=1}^{k-1}\nn_j}\prod_{j=1}^{k-1}\ud{\sigma_j}:=\int_{\nn_1}\ldots\bigg(\int_{\nn_{k-1}}\ud{\sigma_{k-1}}\bigg)\ldots\ud{\sigma_1}
\end{align}
where $\ud{\sigma_j}:=\m(\vv_j)\abs{\vv_j\cdot\vn(\vx_j)}\ud{\vv_j}$ is a probability measure.
\end{definition}

We define a weight function scaled with parameter $\xi$, for $0\leq\vrh<\dfrac{1}{4}$ and $\vth\geq0$,
\begin{align}\label{htt 44}
\vh(\vv):=\bv,
\end{align}
and
\begin{align}\label{htt 45}
\tvh(\vv):=\frac{1}{\m^{\frac{1}{2}}(\vv)\vh(\vv)}=\sqrt{2\pi}\frac{\ue^{\left(\frac{1}{4}-\varrho\right)\abs{\vv}^2}}
{\left(1+\abs{\vv}^2\right)^{\frac{\vth}{2}}}.
\end{align}

\begin{lemma}\label{htt lemma 08}
For $T_0>0$ sufficiently large, there exists constants $C_1,C_2>0$ independent of $T_0$, such that for $k=C_1T_0^{\frac{5}{4}}$, and
$(\vx,\vv)\in\times\bar\Omega\times\r^3$,
\begin{align}
\int_{\Pi_{j=1}^{k-1}\nn_j}\id_{\{\frac{t-t_k(\vx,\vv,\vv_1,\ldots,\vv_{k-1})}{\e^2}<\frac{T_0}{\e}\}}\prod_{j=1}^{k-1}\ud{\sigma_j}\leq
\left(\frac{1}{2}\right)^{C_2T_0^{\frac{5}{4}}}.
\end{align}
\end{lemma}
\begin{proof}
This is a rescaled version of \cite[Lemma 4.1]{Esposito.Guo.Kim.Marra2013}. Since our hitting time in \eqref{htt 43} is rescaled with $\e$, we should rescale back in the statement of lemma.
\end{proof}

\begin{remark}
Roughly speaking, Lemma \ref{htt lemma 08} states that even though we have the freedom to choose $\vv_k$ in each stochastic cycle, in the long run, the accumulative time will not be too small. After enough reflections $\sim k$, most characteristics has the accumulative time that will exceed any set threshold $T_0$.
\end{remark}

\begin{theorem}\label{LI estimate'.}
Assume \eqref{linear unsteady compatibility} and \eqref{linear unsteady normalization} hold. The solution $f(t,\vx,\vv)$ to the equation \eqref{linear unsteady} satisfies for $\vth\geq0$ and $0\leq\varrho<\dfrac{1}{4}$,
\begin{align}
&\lnnmvt{f}+\lsmt{f}{\gamma_+}\\
\ls&\frac{1}{\e^{2+\frac{3}{2m}}}\pnnm{\pk[S(t)]}{\frac{2m}{2m-1}}+\frac{1}{\e^{1+\frac{3}{2m}}}\tnnm{\snn(\ik-\pk)[S(t)]}
+\frac{1}{\e^{2+\frac{3}{2m}}}\tnnmt{\pk[\dt S]}+\frac{1}{\e^{1+\frac{3}{2m}}}\tnnmt{\snn(\ik-\pk)[\dt S]}\no\\
&+\lnnmvt{\nu^{-1}S}+\frac{1}{\e^{\frac{3}{2m}}}\pnm{h(t)}{\gamma_-}{\frac{4m}{3}}+\frac{1}{\e^{1+\frac{3}{2m}}}\tsm{h(t)}{\gamma_-}
+\frac{1}{\e^{1+\frac{3}{2m}}}\tsmt{\dt h}{\gamma_-}+\lsmt{h}{\gamma_-}\no\\
&+\frac{1}{\e^{2+\frac{3}{2m}}}\tnnm{\nu z}+\frac{1}{\e^{1+\frac{3}{2m}}}\tnnm{\vv\cdot\nx z}+\lnnmv{z}+\frac{1}{\e^{2+\frac{3}{2m}}}\tnnm{S(0)}.\no
\end{align}
\end{theorem}
\begin{proof}
\ \\
Step 1: Mild formulation.\\
Denote the weighted solution
\begin{align}
g(t,\vx,\vv):=&\vh(\vv) f(t,\vx,\vv),
\end{align}
and the weighted non-local operator
\begin{align}
K_{\vh(\vv)}[g](\vv):=&\vh(\vv)K\left[\frac{g}{\vh}\right](\vv)=\int_{\r^3}k_{\vh(\vv)}(\vv,\vuu)g(\vuu)\ud{\vuu},
\end{align}
where
\begin{align}
k_{\vh(\vv)}(\vv,\vuu):=k(\vv,\vuu)\frac{\vh(\vv)}{\vh(\vuu)}.
\end{align}
Multiplying $\vh$ on both sides of \eqref{linear unsteady}, we have
\begin{align}\label{htt 46}
\left\{
\begin{array}{l}
\e^2\dt g+\e\vv\cdot\nx g+\nu g=K_{\vh}(t,\vx,\vv)+\vh(\vv) S(t,\vx,\vv)\ \ \text{in}\ \ \rp\times\Omega\times\r^3,\\\rule{0ex}{2em}
g(0,\vx,\vv)=\vh(\vv) z(\vx,\vv)\ \ \text{in}\ \ \Omega\times\r^3,\\\rule{0ex}{2em}
g(t,\vx_0,\vv)=\ds\vh(\vv)\mh(\vv)\int_{\vuu\cdot\vn>0}\tvh(\vuu)g(t,\vx_0,\vuu)\ud\vuu+\vh h(t,\vx_0,\vv)\ \ \text{for}\ \ \vx_0\in\p\Omega\ \
\text{and}\ \ \vv\cdot\vn<0,
\end{array}
\right.
\end{align}
We introduce indicator function $\id_{\{t_k=0\}}$ which implies the characteristic hits the initial time and $\id_{\{t_k>0\}}$ which implies the characteristic hits the boundary. We can rewrite the solution of the equation \eqref{linear unsteady} along the characteristics by Duhamel's principle as
\begin{align}\label{htt 47}
g(t,\vx,\vv)=&\bigg(\id_{\{t_1=0\}}\vh(\vv)z(\vx_1,\vv)\ue^{-\nu(\vv)\frac{t-t_1}{\e^2}}+\id_{\{t_1>0\}}\vh(\vv)h(t_1,\vx_1,\vv)\ue^{-\nu(\vv)\frac{t-t_1}{\e^2}}\bigg)\\
&+\int_{0}^{\frac{t-t_1}{\e^2}}\vh(\vv) S\Big(t-\e^2s,\vx-\e s\vv,\vv\Big)\ue^{-\nu(\vv)s}\ud{s}+\int_{0}^{\frac{t-t_1}{\e^2}}K_{\vh(\vv)}[g]\Big(t-\e^2s,\vx-\e s\vv,\vv\Big)\ue^{-\nu(\vv)
s}\ud{s}\no\\
&+\frac{\ue^{-\nu(\vv)\frac{t-t_1}{\e^2}}}{\tvh(\vv)}\int_{\nn_1}g(t_1,\vx_1,\vv_1)\tvh(\vv_1)\ud{\sigma_1},\no
\end{align}
where the last term refers to $\pp[f]$. We may further rewrite the last term using \eqref{htt 47} along the stochastic cycle by applying Duhamel's principle $k$ times as
\begin{align}\label{htt 48}
g(t,\vx,\vv)=& \bigg(\id_{\{t_k=0\}}\vh(\vv)z(\vx_1,\vv)\ue^{-\nu(\vv)\frac{t-t_1}{\e^2}}+\id_{\{t_k>0\}}\vh(\vv)h(t_1,\vx_1,\vv)\ue^{-\nu(\vv)\frac{t-t_1}{\e^2}}\bigg)\\
&+\int_{0}^{\frac{t-t_1}{\e^2}}\vh(\vv) S\Big(t-\e^2s,\vx-\e s\vv,\vv\Big)\ue^{-\nu(\vv)s}\ud{s}+\int_{0}^{\frac{t-t_1}{\e^2}}K_{\vh(\vv)}[g]\Big(t-\e^2s,\vx-\e s\vv,\vv\Big)\ue^{-\nu(\vv)
s}\ud{s}\no\\
&+\frac{\ue^{-\nu(\vv)\frac{t-t_1}{\e^2}}}{\tvh(\vv)}\sum_{\ell=1}^{k-1}\int_{\prod_{j=1}^{\ell}\nn_j}\Big(G_{\ell}[t,\vx,\vv]+H_{\ell}[t,\vx,\vv]\Big)\tvh(\vv_{\ell})
\bigg(\prod_{j=1}^{\ell}\ue^{-\nu(\vv_j)\frac{t_j-t_{j+1}}{\e^2}}\ud{\sigma_j}\bigg)\no\\
&+\frac{\ue^{-\nu(\vv)\frac{t-t_1}{\e^2}}}{\tvh(\vv)}\int_{\prod_{j=1}^{k}\nn_j}g(t_k,\vx_k,\vv_k)\tvh(\vv_{k})
\bigg(\prod_{j=1}^{k}\ue^{-\nu(\vv_{j})\frac{t_j-t_{j+1}}{\e^2}}\ud{\sigma_j}\bigg),\no
\end{align}
where
\begin{align}\label{htt 49}
G_{\ell}[t,\vx,\vv]:=&\id_{\{t_{\ell+1}=0\}}\vh(\vv_{\ell})z(\vx_{\ell+1},\vv_{\ell})+\id_{\{t_{\ell+1}>0\}}\vh(\vv_{\ell})h(t_{\ell+1},\vx_{\ell+1},\vv_{\ell})\\
&+\int_{0}^{\frac{t_{\ell}-t_{\ell+1}}{\e^2}}\bigg(\vh(\vv_{\ell})S\Big(t_{\ell}-\e^2s,\vx_{\ell}-\e s\vv_{\ell},\vv_{\ell}\Big)\ue^{\nu(\vv_{\ell})s}\bigg)\ud{s}\no\\
H_{\ell}[t,\vx,\vv]:=&\int_{0}^{\frac{t_{\ell}-t_{\ell+1}}{\e^2}}\bigg(K_{\vh(\vv_{\ell})}[g]\Big(t_{\ell}-\e^2s,\vx_{\ell}-\e s\vv_{\ell},\vv_{\ell}\Big)\ue^{\nu(\vv_{\ell})s}\bigg)\ud{s}.
\end{align}
\ \\
Step 2: Estimates of source terms initial terms and boundary terms.\\
We set $k=CT_0^{\frac{5}{4}}$ for $T_0$ defined in Lemma \ref{htt lemma 08}. Consider all terms in \eqref{htt 48} related to $h$ and $S$.

Since $t_1\leq t$, we have
\begin{align}\label{htt 50}
\abs{\id_{\{t_k=0\}}\vh(\vv)z(\vx_1,\vv)\ue^{-\nu(\vv)\frac{t-t_1}{\e^2}}+\id_{\{t_k>0\}}\vh(\vv)h(t_1,\vx_1,\vv)\ue^{-\nu(\vv)\frac{t-t_1}{\e^2}}}\leq \snnm{\vh z}+\ssmt{\vh h}{\gamma_-}.
\end{align}
Also,
\begin{align}\label{htt 51}
\abs{\int_{0}^{\frac{t-t_1}{\e^2}}\vh(\vv) S\Big(t-\e^2s,\vx-\e s\vv,\vv\Big)\ue^{-\nu(\vv)s}\ud{s}}
\leq& \snnmt{\nu^{-1}\vh S}\abs{\int_{0}^{\frac{t-t_1}{\e^2}}\nu(\vv)\ue^{-\nu(\vv)s}\ud{s}}
\leq \snnmt{\nu^{-1}\vh S}.
\end{align}
Then we turn to terms defined in $G_{\ell}$ of \eqref{htt 49}. Noting that $\dfrac{1}{\tvh}\ls 1$, we know
\begin{align}\label{htt 52}
&\abs{\frac{\ue^{-\nu(\vv)\frac{t-t_1}{\e^2}}}{\tvh(\vv)}\sum_{\ell=1}^{k-1}\int_{\prod_{j=1}^{\ell}\nn_j}
\id_{\{t_{\ell+1}=0\}}\vh(\vv_{\ell})z(\vx_{\ell+1},\vv_{\ell})\tvh(\vv_{\ell})
\bigg(\prod_{j=1}^{\ell}\ue^{-\nu(\vv_j)\frac{t_j-t_{j+1}}{\e^2}}\ud{\sigma_j}\bigg)}\\
\ls&\snnm{\vh z}\abs{\sum_{\ell=1}^{k-1}\int_{\prod_{j=1}^{\ell}\nn_j}\tvh(\vv_\ell)\prod_{j=1}^{\ell}\ud{\sigma_j}}
\ls\snnm{\vh z}\abs{\sum_{\ell=1}^{k-1}\int_{\nn_{\ell}}\tvh(\vv_\ell)\ud{\sigma_\ell}}\ls CT_0^{\frac{5}{4}}\snnm{\vh z},\no
\end{align}
and
\begin{align}\label{htt 53}
&\abs{\frac{\ue^{-\nu(\vv)\frac{t-t_1}{\e^2}}}{\tvh(\vv)}\sum_{\ell=1}^{k-1}\int_{\prod_{j=1}^{\ell}\nn_j}
\id_{\{t_{\ell+1}>0\}}\vh(\vv_{\ell})h(t_{\ell+1},\vx_{\ell+1},\vv_{\ell})\tvh(\vv_{\ell})
\bigg(\prod_{j=1}^{\ell}\ue^{-\nu(\vv_j)\frac{t_j-t_{j+1}}{\e^2}}\ud{\sigma_j}\bigg)}\\
\ls&\ssmt{\vh h}{\gamma_-}\abs{\sum_{\ell=1}^{k-1}\int_{\prod_{j=1}^{\ell}\nn_j}\tvh(\vv_\ell)\prod_{j=1}^{\ell}\ud{\sigma_j}}
\ls\ssmt{\vh h}{\gamma_-}\abs{\sum_{\ell=1}^{k-1}\int_{\nn_{\ell}}\tvh(\vv_\ell)\ud{\sigma_\ell}}\ls CT_0^{\frac{5}{4}}\ssmt{\vh h}{\gamma_-}.\no
\end{align}
Similarly,
\begin{align}\label{htt 54}
\\
&\abs{\frac{\ue^{-\nu(\vv)\frac{t-t_1}{\e^2}}}{\tvh(\vv)}\sum_{\ell=1}^{k-1}\int_{\prod_{j=1}^{\ell}\nn_j}
\int_{0}^{\frac{t_{\ell}-t_{\ell+1}}{\e^2}}\bigg(\vh(\vv_{\ell})S\Big(t_{\ell}-\e^2s,\vx_{\ell}-\e s\vv_{\ell},\vv_{\ell}\Big)\ue^{\nu(\vv_{\ell})s}\bigg)\ud{s}\tvh(\vv_{\ell})
\bigg(\prod_{j=1}^{\ell}\ue^{-\nu(\vv_j)\frac{t_j-t_{j+1}}{\e^2}}\ud{\sigma_j}\bigg)}\no\\
\ls&\snnmt{\nu^{-1}\vh S}\sum_{\ell=1}^{k-1}\int_{\prod_{j=1}^{\ell}\nn_j}\abs{\int_{0}^{\frac{t_{\ell}-t_{\ell+1}}{\e^2}}
\nu(\vv_{\ell})\ue^{\nu(\vv_{\ell})\left(s-\frac{t_{\ell}-t_{\ell+1}}{\e^2}\right)}\ud{s}}\tvh(\vv_\ell)
\prod_{j=1}^{\ell}\ud{\sigma_j}\bigg)\ls CT_0^{\frac{5}{4}}\snnmt{\nu^{-1}\vh S}.\no
\end{align}
Collecting all terms in \eqref{htt 50}, \eqref{htt 51}, \eqref{htt 52}, \eqref{htt 53} and \eqref{htt 54}, we have
\begin{align}\label{htt 74}
\text{Initial Term and Boundary Term Contribution}&\ls CT_0^{\frac{5}{4}}\Big(\snnm{\vh z}+\ssmt{\vh h}{\gamma_-}\Big)\\
&\ls \snnm{\vh z}+\ssmt{\vh h}{\gamma_-},\no
\end{align}
and
\begin{align}\label{htt 75}
\text{Source Term Contribution}\ls CT_0^{\frac{5}{4}}\snnmt{\nu^{-1}\vh S}\ls \snnmt{\nu^{-1}\vh S}.
\end{align}
\ \\
Step 3: Estimates of Multiple Reflection.\\
We focus on the last term in \eqref{htt 48}, which can be decomposed based on accumulative time $t_{k+1}$:
\begin{align}
&\abs{\frac{\ue^{-\nu(\vv)\frac{t-t_1}{\e^2}}}{\tvh(\vv)}\int_{\prod_{j=1}^{k}\nn_j}g(t_k,\vx_k,\vv_k)\tvh(\vv_{k})
\bigg(\prod_{j=1}^{k}\ue^{-\nu(\vv_{j})\frac{t_j-t_{j+1}}{\e^2}}\ud{\sigma_j}\bigg)}\\
\leq&\abs{\frac{\ue^{-\nu(\vv)\frac{t-t_1}{\e^2}}}{\tvh(\vv)}\int_{\prod_{j=1}^{k}\nn_j}\id_{\{\frac{t-t_k}{\e^2}\leq\frac{T_0}{\e}\}}g(t_k,\vx_k,\vv_k)\tvh(\vv_{k})
\bigg(\prod_{j=1}^{k}\ue^{-\nu(\vv_{j})\frac{t_j-t_{j+1}}{\e^2}}\ud{\sigma_j}\bigg)}\no\\
&+\abs{\frac{\ue^{-\nu(\vv)\frac{t-t_1}{\e^2}}}{\tvh(\vv)}\int_{\prod_{j=1}^{k}\nn_j}\id_{\{\frac{t-t_k}{\e^2}\geq\frac{T_0}{\e}\}}g(t_k,\vx_k,\vv_k)\tvh(\vv_{k})
\bigg(\prod_{j=1}^{k}\ue^{-\nu(\vv_{j})\frac{t_j-t_{j+1}}{\e^2}}\ud{\sigma_j}\bigg)}:=J_1+J_2.\no
\end{align}
Based on Lemma \ref{htt lemma 08}, we have
\begin{align}\label{htt 72}
J_1\ls&\snnmt{g}\abs{\int_{\Pi_{j=1}^{k-1}\nn_j}\id_{\{\frac{t-t_k}{\e^2}\leq\frac{T_0}{\e}\}}\bigg(\int_{\nn_k}\tvh(\vv_{k})\ud\sigma_k\bigg)
\bigg(\prod_{j=1}^{k-1}\ud{\sigma_j}\bigg)}\\
\ls&\snnmt{g}\abs{\int_{\Pi_{j=1}^{k-1}\nn_j}\id_{\{\frac{t-t_k}{\e^2}\leq\frac{T_0}{\e}\}}
\bigg(\prod_{j=1}^{k-1}\ud{\sigma_j}\bigg)}\ls \bigg(\frac{1}{2}\bigg)^{C_2T_0^{\frac{5}{4}}}\snnmt{g}.\no
\end{align}
On the other hand, when $t_k$ is large, the exponential terms become extremely small, so we obtain
\begin{align}\label{htt 73}
J_2\ls&\snnmt{g}\abs{\ue^{-\nu(\vv)\frac{t-t_1}{\e^2}}\int_{\Pi_{j=1}^{k-1}\nn_j}\id_{\{\frac{t-t_k}{\e^2}\geq\frac{T_0}{\e}\}}\bigg(\int_{\nn_k}\tvh(\vv_{k})\ud\sigma_k\bigg)
\bigg(\prod_{j=1}^{k-1}\ue^{-\nu(\vv_{j})\frac{t_j-t_{j+1}}{\e^2}}\ud{\sigma_j}\bigg)}\\
\ls&\snnmt{g}\abs{\ue^{-\nu(\vv)\frac{t-t_1}{\e^2}}\int_{\Pi_{j=1}^{k-1}\nn_j}\id_{\{\frac{t-t_k}{\e^2}\geq\frac{T_0}{\e}\}}
\bigg(\prod_{j=1}^{k-1}\ue^{-\nu(\vv_{j})\frac{t_j-t_{j+1}}{\e^2}}\ud{\sigma_j}\bigg)}\ls \ue^{-\frac{T_0}{\e}}\snnmt{g}.\no
\end{align}
Summarizing \eqref{htt 72} and \eqref{htt 73}, we get for $\d$ arbitrarily small
\begin{align}\label{htt 76}
\text{Multiple Reflection Term Contribution}\ls \d \snnm{g}.
\end{align}
\ \\
Step 4: Estimates of $K_{\vh}$ terms.\\
So far, the only remaining terms in \eqref{htt 48} are related to $K_{\vh}$. We focus on
\begin{align}
\abs{\int_{0}^{\frac{t-t_1}{\e^2}}K_{\vh(\vv)}[g]\Big(t-\e^2s,\vx-\e s\vv,\vv\Big)\ue^{-\nu(\vv)
s}\ud{s}}&\ls \snnmt{K_{\vh(\vv)}[g]\Big(t-\e^2s,\vx-\e s\vv,\vv\Big)}.
\end{align}
Denote $T(s;t,\vx,\vv):=t-\e^2s$ and $X(s;t,\vx,\vv):=\vx-\e(t_1-s)\vv$. Define the back-time stochastic cycle from $(T,X,\vv')$ as $(t_i',\vx_i',\vv_i')$ with $(t_0',\vx_0',\vv_0')=(T,X,\vv')$. Then we can rewrite $K_{\vh}$ along the stochastic cycle as \eqref{htt 48}
\begin{align}\label{htt 77}
&\abs{K_{\vh(\vv)}[g]\Big(t-\e^2s,\vx-\e(t_1-s)\vv,\vv\Big)}=\abs{K_{\vh(\vv)}[g](T,X,\vv)}=\abs{\int_{\r^3}k_{\vh(\vv)}(\vv,\vv')g(T,X,\vv')\ud{\vv'}}\\
\leq&\abs{\int_{\r^3}\int_{0}^{\frac{T-t_1'}{\e^2}}k_{\vh(\vv)}(\vv,\vv')K_{\vh(\vv')}[g]\Big(T-\e^2r,X-\e r\vv',\vv'\Big)\ue^{-\nu(\vv')r}\ud{r}\ud{\vv'}}\no\\
&+\abs{\int_{\r^3}\frac{\ue^{-\nu(\vv')\frac{T-t_1'}{\e^2}}}{\tvh(\vv')}\sum_{\ell=1}^{k-1}\int_{\prod_{j=1}^{\ell}\nn_j'}k_{\vh(\vv)}(\vv,\vv')H_{\ell}[T,X,\vv']\tvh(\vv_{\ell}')
\bigg(\prod_{j=1}^{\ell}\ue^{-\nu(\vv_j')\frac{t_j'-t_{j+1}'}{\e^2}}\ud{\sigma_j'}\bigg)\ud{\vv'}}\no\\
&+\abs{\int_{\r^3}k_{\vh(\vv)}(\vv,\vv')\Big(\text{initial terms + boundary terms + source terms + multiple reflection terms}\Big)\ud{\vv'}}\no\\
:=&I+II+III.\no
\end{align}
Using estimates \eqref{htt 74}, \eqref{htt 75}, \eqref{htt 76} from Step 2 and Step 3, and Lemma \ref{Regularity lemma 1'}, we can bound $III$ directly
\begin{align}\label{htt 78}
III\ls \snnm{\vh z}+\ssmt{\vh h}{\gamma_-}+\snnmt{\nu^{-1}\vh S}+\d \snnmt{g}.
\end{align}
$I$ and $II$ are much more complicated. We may further rewrite $I$ as
\begin{align}
I=&\abs{\int_{\r^3}\int_{\r^3}\int_{0}^{\frac{T-t_1'}{\e^2}}k_{\vh(\vv)}(\vv,\vv')k_{\vh(\vv')}(\vv',\vv'')g\Big(T-\e^2r,X-\e r\vv',\vv''\Big)\ue^{-\nu(\vv')
r}\ud{r}\ud{\vv'}\ud{\vv''}},
\end{align}
which will estimated in four cases:
\begin{align}
I:=I_1+I_2+I_3+I_4.
\end{align}
\ \\
Case I: $I_1:$ $\abs{\vv}\geq N$.\\
Based on Lemma \ref{Regularity lemma 1'}, we have
\begin{align}
\abs{\int_{\r^3}\int_{\r^3}k_{\vh(\vv)}(\vv,\vv')k_{\vh(\vv')}(\vv',\vv'')\ud{\vv'}\ud{\vv''}}\ls\frac{1}{1+\abs{\vv}}\ls\frac{1}{N}.
\end{align}
Hence, we get
\begin{align}\label{htt 87}
I_1\ls\frac{1}{N}\snnmt{g}.
\end{align}
\ \\
Case II: $I_2:$ $\abs{\vv}\leq N$, $\abs{\vv'}\geq2N$, or $\abs{\vv'}\leq
2N$, $\abs{\vv''}\geq3N$.\\
Notice this implies either $\abs{\vv'-\vv}\geq N$ or
$\abs{\vv'-\vv''}\geq N$. Hence, either of the following is valid
correspondingly:
\begin{align}
\abs{k_{\vh(\vv)}(\vv,\vv')}\leq& C\ue^{-\d N^2}\abs{k_{\vh(\vv)}(\vv,\vv')}\ue^{\d\abs{\vv-\vv'}^2},\\
\abs{k_{\vh(\vv')}(\vv',\vv'')}\leq& C\ue^{-\d N^2}\abs{k_{\vh(\vv')}(\vv',\vv'')}\ue^{\d\abs{\vv'-\vv''}^2}.
\end{align}
Based on Lemma \ref{Regularity lemma 1'}, we know
\begin{align}
\int_{\r^3}\abs{k_{\vh(\vv)}(\vv,\vv')}\ue^{\d\abs{\vv-\vv'}^2}\ud{\vv'}<&\infty,\\
\int_{\r^3}\abs{k_{\vh(\vv')}(\vv',\vv'')}\ue^{\d\abs{\vv'-\vv''}^2}\ud{\vv''}<&\infty.
\end{align}
Hence, we have
\begin{align}\label{htt 88}
I_2\ls \ue^{-\d N^2}\snnmt{g}.
\end{align}
\ \\
Case III: $I_3:$ $0\leq r\leq\d$ and $\abs{\vv}\leq N$, $\abs{\vv'}\leq 2N$, $\abs{\vv''}\leq 3N$.\\
In this case, since the integral with respect to $r$ is restricted in a very short interval, there is a small contribution as
\begin{align}\label{htt 89}
I_3\ls\abs{\int_{0}^{\d}\ue^{-r}\ud{r}}\snnmt{g}\ls \d\snnmt{g}.
\end{align}
\ \\
Case IV: $I_4:$ $r\geq\d$ and $\abs{\vv}\leq N$, $\abs{\vv'}\leq 2N$, $\abs{\vv''}\leq 3N$.\\
This is the most complicated case. Since $k_{\vh(\vv)}(\vv,\vv')$ has
possible integrable singularity of $\dfrac{1}{\abs{\vv-\vv'}}$, we can
introduce the truncated kernel $k_N(\vv,\vv')$ which is smooth and has compactly supported range such that
\begin{align}\label{htt 79}
\sup_{\abs{\vv}\leq 3N}\int_{\abs{\vv'}\leq
3N}\abs{k_N(\vv,\vv')-k_{\vh(\vv)}(\vv,\vv')}\ud{\vv'}\leq\frac{1}{N}.
\end{align}
Then we can split
\begin{align}\label{htt 80}
k_{\vh(\vv)}(\vv,\vv')k_{\vh(\vv')}(\vv',\vv'')=&k_N(\vv,\vv')k_N(\vv',\vv'')
+\bigg(k_{\vh(\vv)}(\vv,\vv')-k_N(\vv,\vv')\bigg)k_{\vh(\vv')}(\vv',\vv'')\\
&+\bigg(k_{\vh(\vv')}(\vv',\vv'')-k_N(\vv',\vv'')\bigg)k_N(\vv,\vv').\no
\end{align}
This means that we further split $I_4$ into
\begin{align}
I_4:=I_{4,1}+I_{4,2}+I_{4,3}.
\end{align}
Based on \eqref{htt 79}, we have
\begin{align}\label{htt 82}
I_{4,2}\ls&\frac{1}{N}\snnm{g},\quad I_{4,3}\ls\frac{1}{N}\snnm{g}.
\end{align}
Therefore, the only remaining term is $I_{4,1}$. Note that we always have $X-\e r\vv'\in\Omega$. Hence, we define the change of variable $\vv'\rt y$ as
$y=(y_1,y_2,y_3)=X-\e r\vv'$. Then the Jacobian
\begin{align}\label{htt 81}
\abs{\frac{\ud{y}}{\ud{\vv'}}}=\abs{\left\vert\begin{array}{ccc}
-\e r&0&0\\
0&-\e r&0\\
0&0&-\e r
\end{array}\right\vert}=\e^3 r^3\geq \e^3\d^3.
\end{align}
Considering $\abs{\vv},\abs{\vv'},\abs{\vv''}\leq 3N$, we know $\abs{g}\ls\abs{f}$. Also, since $k_N$ is bounded, we estimate
\begin{align}\label{htt 85}
I_{4,1}\ls&\int_{0}^{\frac{T-t_1'}{\e^2}}\int_{\abs{\vv'}\leq2N}\int_{\abs{\vv''}\leq3N}
\id_{\{X-\e r\vv'\in\Omega\}}\abs{f(T-\e^2r, X-\e r\vv',\vv'')}\ue^{-\nu(\vv')r}\ud{r}\ud{\vv'}\ud{\vv''}.
\end{align}
Using the decomposition $f=\pk[f]+(\ik-\pk)[f]$, \eqref{htt 81} and H\"older's inequality, we estimate them separately,
\begin{align}\label{htt 83}
&\int_{0}^{\frac{T-t_1'}{\e^2}}\int_{\abs{\vv'}\leq2N}\int_{\abs{\vv''}\leq3N}
\id_{\{X-\e r\vv'\in\Omega\}}\abs{\pk[f](T-\e^2r,X-\e r\vv',\vv'')}\ue^{-\nu(\vv')
r}\ud{r}\ud{\vv'}\ud{\vv''}\\
\leq&\int_{0}^{\frac{T-t_1'}{\e^2}}\bigg(\int_{\abs{\vv'}\leq2N}\int_{\abs{\vv''}\leq3N}
\id_{\{X-\e r\vv'\in\Omega\}}\ud{\vv'}\ud{\vv''}\bigg)^{\frac{2m-1}{2m}}\no\\
&\times\bigg(\int_{\abs{\vv'}\leq2N}\int_{\abs{\vv''}\leq3N}
\id_{\{X-\e r\vv'\in\Omega\}}\Big(\pk[f]\Big)^{2m}(T-\e^2r,X-\e r\vv',\vv'')\ue^{-\nu(\vv')
r}\ud{\vv'}\ud{\vv''}\bigg)^{\frac{1}{2m}}\ue^{-r}\ud r\no\\
\ls&\int_{0}^{\frac{T-t_1'}{\e^2}}\bigg(\frac{1}{\e^3\d^3}\int_{\abs{\vv''}\leq3N}\int_{\Omega}\id_{\{
y\in\Omega\}}\Big(\pk[f]\Big)^{2m}(T-\e^2r,y,\vv'')\ud{y}\ud{\vv''}\bigg)^{\frac{1}{2m}}\ue^{-r}\ud{r}\ls \frac{1}{\e^{\frac{3}{2m}}\d^{\frac{3}{2m}}}\sup_{[0,T]}\pnnm{\pk[f(t)]}{2m},\no
\end{align}
and
\begin{align}\label{htt 84}
&\int_{0}^{\frac{T-t_1'}{\e^2}}\int_{\abs{\vv'}\leq2N}\int_{\abs{\vv''}\leq3N}
\id_{\{X-\e r\vv'\in\Omega\}}\abs{(\ik-\pk)[f](T-\e^2r,X-\e r\vv',\vv'')}\ue^{-\nu(\vv')r}\ud{\vv'}\ud{\vv''}\ud{r}\\
\leq&\int_{0}^{\frac{T-t_1'}{\e^2}}\bigg(\int_{\abs{\vv'}\leq2N}\int_{\abs{\vv''}\leq3N}
\id_{\{X-\e r\vv'\in\Omega\}}\ud{\vv'}\ud{\vv''}\bigg)^{\frac{1}{2}}\no\\
&\times\bigg(\int_{\abs{\vv'}\leq2N}\int_{\abs{\vv''}\leq3N}
\id_{\{X-\e r\vv'\in\Omega\}}\Big((\ik-\pk)[f]\Big)^{2}(T-\e^2r,X-\e r\vv',\vv'')\ud{\vv'}\ud{\vv''}\bigg)^{\frac{1}{2}}\ue^{-r}\ud{r}\no\\
\ls&\int_{0}^{\frac{T-t_1'}{\e^2}}\bigg(\frac{1}{\e^3\d^3}\int_{\abs{\vv''}\leq3N}
\int_{\Omega}\id_{\{y\in\Omega\}}\Big((\ik-\pk)[f]\Big)^{2}(T-\e^2r,y,\vv'')\ud{y}\ud{\vv''}\bigg)^{\frac{1}{2}}\ue^{-r}\ud{r}
\ls \frac{1}{\e^{\frac{3}{2}}\d^{\frac{3}{2}}}\sup_{[0,T]}\tnnm{(\ik-\pk)[f(t)]}.\no
\end{align}
Inserting \eqref{htt 83} and \eqref{htt 84} into \eqref{htt 85}, we obtain
\begin{align}\label{htt 86}
I_{4,1}\ls \frac{1}{\e^{\frac{3}{2m}}\d^{\frac{3}{2m}}}\sup_{[0,T]}\pnnm{\pk[f(t)]}{2m}+\frac{1}{\e^{\frac{3}{2}}\d^{\frac{3}{2}}}\sup_{[0,T]}\tnnm{(\ik-\pk)[f(t)]}.
\end{align}
Combined with \eqref{htt 82}, we know
\begin{align}\label{htt 90}
I_4\ls \frac{1}{N}\snnmt{g}+\frac{1}{\e^{\frac{3}{2m}}\d^{\frac{3}{2m}}}\sup_{[0,T]}\pnnm{\pk[f(t)]}{2m}+\frac{1}{\e^{\frac{3}{2}}\d^{\frac{3}{2}}}\sup_{[0,T]}\tnnm{(\ik-\pk)[f(t)]}.
\end{align}
Summarizing all four cases in \eqref{htt 87}, \eqref{htt 88}, \eqref{htt 89} and \eqref{htt 90}, we obtain
\begin{align}
I\ls \bigg(\frac{1}{N}+\ue^{-\d N^2}+\d\bigg)\snnmt{g}+\frac{1}{\e^{\frac{3}{2m}}\d^{\frac{3}{2m}}}\sup_{[0,t]}\pnnm{\pk[f(t)]}{2m}+\frac{1}{\e^{\frac{3}{2}}\d^{\frac{3}{2}}}\sup_{[0,t]}\tnnm{(\ik-\pk)[f(t)]}.
\end{align}
Choosing $\d$ sufficiently small and then taking $N$ sufficiently large, we have
\begin{align}\label{htt 91}
I\ls \d\snnmt{g}+\frac{1}{\e^{\frac{3}{2m}}\d^{\frac{3}{2m}}}\sup_{[0,t]}\pnnm{\pk[f(t)]}{2m}+\frac{1}{\e^{\frac{3}{2}}\d^{\frac{3}{2}}}\sup_{[0,t]}\tnnm{(\ik-\pk)[f(t)]}.
\end{align}
By a similar but tedious computation, we arrive at
\begin{align}\label{htt 92}
II\ls \d\snnmt{g}+\frac{1}{\e^{\frac{3}{2m}}\d^{\frac{3}{2m}}}\sup_{[0,t]}\pnnm{\pk[f(t)]}{2m}+\frac{1}{\e^{\frac{3}{2}}\d^{\frac{3}{2}}}\sup_{[0,t]}\tnnm{(\ik-\pk)[f(t)]}.
\end{align}
Combined with \eqref{htt 78}, we have
\begin{align}
&\abs{\int_{0}^{\frac{t-t_1}{\e^2}}K_{\vh(\vv)}[g]\Big(t-\e^2s,\vx-\e s\vv,\vv\Big)\ue^{-\nu(\vv)s}\ud{s}}\\
\ls&\d\snnmt{g}+\frac{1}{\e^{\frac{3}{2m}}\d^{\frac{3}{2m}}}\sup_{[0,t]}\pnnm{\pk[f(t)]}{2m}+\frac{1}{\e^{\frac{3}{2}}\d^{\frac{3}{2}}}\sup_{[0,t]}\tnnm{(\ik-\pk)[f(t)]}
+\snnm{\vh z}+\ssmt{\vh h}{\gamma_-}+\snnmt{\nu^{-1}\vh S}.\no
\end{align}
All the other terms in \eqref{htt 48} related to $K_{\vh}$ can be estimated in a similar fashion. At the end of the day, we have
\begin{align}\label{htt 93}
&K_{\vh}\ \text{term contribution}\\
\ls&\d\snnmt{g}+\frac{1}{\e^{\frac{3}{2m}}\d^{\frac{3}{2m}}}\sup_{[0,t]}\pnnm{\pk[f(t)]}{2m}+\frac{1}{\e^{\frac{3}{2}}\d^{\frac{3}{2}}}\sup_{[0,t]}\tnnm{(\ik-\pk)[f(t)]}
+\snnm{\vh z}+\ssmt{\vh h}{\gamma_-}+\snnmt{\nu^{-1}\vh S}.\no
\end{align}
\ \\
Step 5: Synthesis.\\
Summarizing all above and inserting \eqref{htt 74}, \eqref{htt 75}, \eqref{htt 76} and \eqref{htt 93} into \eqref{htt 48}, we obtain for any $(t,\vx,\vv)\in\rp\times\bar\Omega\times\r^3$,
\begin{align}\label{dtt 2}
\abs{g(t,\vx,\vv)}\ls&\d\snnmt{g}+\frac{1}{\e^{\frac{3}{2m}}\d^{\frac{3}{2m}}}\sup_{[0,t]}\pnnm{\pk[f(t)]}{2m}
+\frac{1}{\e^{\frac{3}{2}}\d^{\frac{3}{2}}}\sup_{[0,t]}\tnnm{(\ik-\pk)[f(t)]}\\
&+\snnm{\vh z}+\ssmt{\vh h}{\gamma_-}+\snnmt{\nu^{-1}\vh S}.\no
\end{align}
Taking supremum over $[0,t]\times\gamma_+$ in \eqref{dtt 2}, we have
\begin{align}
\sup_{[0,t]}\ssm{g(t)}{\gamma_+}\ls&\d\snnmt{g}+\frac{1}{\e^{\frac{3}{2m}}\d^{\frac{3}{2m}}}\sup_{[0,t]}\pnnm{\pk[f(t)]}{2m}
+\frac{1}{\e^{\frac{3}{2}}\d^{\frac{3}{2}}}\sup_{[0,t]}\tnnm{(\ik-\pk)[f(t)]}\\
&+\snnm{\vh z}+\ssmt{\vh h}{\gamma_-}+\snnmt{\nu^{-1}\vh S}.\no
\end{align}
Based on Theorem \ref{LN estimate'.}, for $\dfrac{3}{2}< m<3$, we obtain
\begin{align}
\sup_{[0,t]}\ssm{g(t)}{\gamma_+}\ls&\d\snnmt{g}+o(1)\Big(\sup_{[0,t]}\ssm{f(t)}{\gamma_+}+\sup_{[0,t]}\snnm{f(t)}\Big)+E\\
\ls&\d\snnm{g}+o(1)\Big(\sup_{[0,t]}\ssm{g(t)}{\gamma_+}+\sup_{[0,t]}\snnm{g(t)}\Big)+E,\no
\end{align}
where
\begin{align}
\\
E:=&\frac{1}{\e^{2+\frac{3}{2m}}}\pnnm{\pk[S(t)]}{\frac{2m}{2m-1}}+\frac{1}{\e^{1+\frac{3}{2m}}}\tnnm{\snn(\ik-\pk)[S(t)]}
+\frac{1}{\e^{2+\frac{3}{2m}}}\tnnmt{\pk[\dt S]}+\frac{1}{\e^{1+\frac{3}{2m}}}\tnnmt{\snn(\ik-\pk)[\dt S]}\no\\
&+\snnmt{\nu^{-1}\vh S}+\frac{1}{\e^{\frac{3}{2m}}}\pnm{h(t)}{\gamma_-}{\frac{4m}{3}}+\frac{1}{\e^{1+\frac{3}{2m}}}\tsm{h(t)}{\gamma_-}
+\frac{1}{\e^{1+\frac{3}{2m}}}\tsmt{\dt h}{\gamma_-}+\ssmt{\vh h}{\gamma_-}\no\\
&+\frac{1}{\e^{2+\frac{3}{2m}}}\tnnm{\nu z}+\frac{1}{\e^{1+\frac{3}{2m}}}\tnnm{\vv\cdot\nx z}+\snnm{\vh z}+\frac{1}{\e^{2+\frac{3}{2m}}}\tnnm{S(0)}.\no
\end{align}
Absorbing $o(1)\sup_{[0,t]}\ssm{g(t)}{\gamma_+}$ into the left-hand side, we have
\begin{align}\label{dtt 1}
\sup_{[0,t]}\ssm{g(t)}{\gamma_+}\ls \d\snnmt{g}+o(1)\sup_{[0,t]}\snnm{g(t)}+E.
\end{align}
On the other hand, taking supremum over $[0,t]\times\Omega\times\r^3$ in \eqref{dtt 2}, we have
\begin{align}
\sup_{[0,t]}\snnm{g(t)}\ls& \d\snnmt{g}+\frac{1}{\e^{\frac{3}{2m}}\d^{\frac{3}{2m}}}\sup_{[0,t]}\pnnm{\pk[f(t)]}{2m}
+\frac{1}{\e^{\frac{3}{2}}\d^{\frac{3}{2}}}\sup_{[0,t]}\tnnm{(\ik-\pk)[f(t)]}\\
&+\snnm{\vh z}+\ssmt{\vh h}{\gamma_-}+\snnmt{\nu^{-1}\vh S}.\no
\end{align}
Based on Theorem \ref{LN estimate'.}, we obtain
\begin{align}
\sup_{[0,t]}\snnm{g(t)}\ls&\d\snnmt{g}+o(1)\Big(\sup_{[0,t]}\ssm{g(t)}{\gamma_+}+\sup_{[0,t]}\snnm{g(t)}\Big)+E.
\end{align}
Absorbing $\d\snnmt{g}$ and $o(1)\sup_{[0,t]}\snnm{g(t)}$ into the left-hand side, we have
\begin{align}\label{dtt 3}
\sup_{[0,t]}\snnm{g(t)}\ls&o(1)\sup_{[0,t]}\ssm{g(t)}{\gamma_+}+E.
\end{align}
Inserting \eqref{dtt 1} into \eqref{dtt 3}, and absorbing $\d\snnmt{g}$ and $o(1)\snnmt{g}$ into the left-hand side, we get
\begin{align}
\sup_{[0,t]}\snnm{g(t)}\ls E.
\end{align}
Then \eqref{dtt 1} implies
\begin{align}
\sup_{[0,t]}\ssm{g(t)}{\gamma_+}\ls E.
\end{align}
In summary, we have
\begin{align}
&\snnmt{g}+\ssmt{g}{\gamma_+}\\
\ls&\frac{1}{\e^{2+\frac{3}{2m}}}\pnnm{\pk[S(t)]}{\frac{2m}{2m-1}}+\frac{1}{\e^{1+\frac{3}{2m}}}\tnnm{\snn(\ik-\pk)[S(t)]}
+\frac{1}{\e^{2+\frac{3}{2m}}}\tnnmt{\pk[\dt S]}+\frac{1}{\e^{1+\frac{3}{2m}}}\tnnmt{\snn(\ik-\pk)[\dt S]}\no\\
&+\snnmt{\nu^{-1}\vh S}+\frac{1}{\e^{\frac{3}{2m}}}\pnm{h(t)}{\gamma_-}{\frac{4m}{3}}+\frac{1}{\e^{1+\frac{3}{2m}}}\tsm{h(t)}{\gamma_-}
+\frac{1}{\e^{1+\frac{3}{2m}}}\tsmt{\dt h}{\gamma_-}+\ssmt{\vh h}{\gamma_-}\no\\
&+\frac{1}{\e^{2+\frac{3}{2m}}}\tnnm{\nu z}+\frac{1}{\e^{1+\frac{3}{2m}}}\tnnm{\vv\cdot\nx z}+\snnm{\vh z}+\frac{1}{\e^{2+\frac{3}{2m}}}\tnnm{S(0)}.\no
\end{align}
Then our result naturally follows.
\end{proof}

\begin{remark}
In the above proof, we use the trace $\snnm{g(t)}$, $\ssm{g(t)}{\gamma_+}$ and $\ssmt{g}{\gamma_+}$ interchangeably with $\snnmt{g}$ to perform absorbing argument. Roughly speaking, we track the solution using mild formulation, so it is always continuous along the characteristics, which covers the whole domain $\rp\times\Omega\times\r^3$, so $\ssmt{g}{\gamma_+}$ will control all the rest. To be more precise, it actually relies on Ukai's trace theorem in \cite{Ukai1986}, which says that for transport operator $\dt+\vv\cdot\nx$, such traces are always well-defined and controllable.
\end{remark}

\begin{remark}[Exponential Decay]\label{exponential remark}
Define $\tilde f=\ue^{K_0t}f$. Then $\tilde f$ satisfies
\begin{align}
\left\{
\begin{array}{l}
\e^2\dt \tilde f+\e\vv\cdot\nx \tilde f+\ll[\tilde f]=\e^2K_0\tilde f+\ue^{K_0t}S(t,\vx,\vv)\ \ \text{in}\ \ \rp\times\Omega\times\r^3,\\\rule{0ex}{1.5em}
\tilde f(0,\vx,\vv)=z(\vx,\vv)\ \ \text{in}\ \ \Omega\times\r^3,\\\rule{0ex}{1.5em}
\tilde f(t,\vx_0,\vv)=\pp[\tilde f](t,\vx_0,\vv)+\ue^{K_0t}h(t,\vx_0,\vv)\ \ \text{on}\ \ \rp\times\gamma_-,
\end{array}
\right.
\end{align}
where
\begin{align}
\pp[\tilde f](t,\vx_0,\vv)=\mh(\vv)
\int_{\vuu\cdot\vn(\vx_0)>0}\m^{\frac{1}{2}}(\vuu)\tilde f(t,\vx_0,\vuu)\abs{\vuu\cdot\vn(\vx_0)}\ud{\vuu}.
\end{align}
The extra term is $\e^2K_0\tilde f$. Thanks to $\e^2$, based on $L^2$ and $L^{2m}$ energy estimates in Lemma \ref{LT estimate'} and Theorem \ref{LN estimate'.}, for $K_0$ small, we can absorb this term into the left-hand side. Therefore, we can recover all estimates as in Theorem \ref{LI estimate'.}.
\end{remark}

\newpage

\section{Hydrodynamic Limit}

\subsection{Perturbed Remainder Estimates}

We consider the perturbed evolutionary Boltzmann equation
\begin{align}\label{nonlinear unsteady}
\left\{
\begin{array}{l}
\e^2\dt f+\e\vv\cdot\nx f+\ll[f]=\Gamma[f,g]+\e^3\Gamma[f,f]+S(t,\vx,\vv)\ \ \text{in}\ \ \rp\times\Omega\times\r^3,\\\rule{0ex}{1.5em}
f(0,\vx,\vv)=z(\vx,\vv)\ \ \text{in}\ \ \Omega\times\r^3,\\\rule{0ex}{1.5em}
f(t,\vx_0,\vv)=\pp[f](t,\vx_0,\vv)+(\mb-\m)\m^{-1}\pp[f]+h(t,\vx_0,\vv)\ \ \text{for}\ \ \vx_0\in\p\Omega\ \
\text{and}\ \ \vv\cdot\vn<0.
\end{array}
\right.
\end{align}
Assume that a priori
\begin{align}\label{perturbed smallness.}
\lnnmvt{g}+\lnnmvt{\dt g}+\lnnmvt{\e^3f}=o(1)\e.
\end{align}

\begin{theorem}\label{LN estimate.}
Assume \eqref{linear unsteady compatibility} and \eqref{linear unsteady normalization} hold. The solution $f(t,\vx,\vv)$ to the equation \eqref{nonlinear unsteady} satisfies
\begin{align}
&\frac{1}{\e^{\frac{1}{2}}}\tsm{(1-\pp)[f(t)]}{\gamma_+}+\frac{1}{\e}\unnm{(\ik-\pk)[f(t)]}+\pnnm{\pk[f(t)]}{2m}\\
&\tnnm{f(t)}+\frac{1}{\e^{\frac{1}{2}}}\tsmt{(1-\pp)[f]}{\gamma_+}+\frac{1}{\e}\unnmt{(\ik-\pk)[f]}+\tnnmt{\pk[f]}\no\\
&+\frac{1}{\e^{\frac{1}{2}}}\tsmt{(1-\pp)[\dt f]}{\gamma_+}+\frac{1}{\e}\unnmt{(\ik-\pk)[\dt f]}+\tnnmt{\pk[\dt f]}\no\\
\ls&o(1)\e^{\frac{3}{2m}}\Big(\ssm{f(t)}{\gamma_+}+\snnm{f(t)}\Big)+\frac{1}{\e^{2}}\pnnm{\pk[S(t)]}{\frac{2m}{2m-1}}+\frac{1}{\e}\tnnm{\snn(\ik-\pk)[S(t)]}\no\\
&+\frac{1}{\e^2}\tnnmt{\pk[S]}+\frac{1}{\e}\tnnmt{\snn(\ik-\pk)[S]}+\frac{1}{\e^2}\tnnmt{\pk[\dt S]}+\frac{1}{\e}\tnnmt{\snn(\ik-\pk)[\dt S]}\no\\
&+\pnm{h(t)}{\gamma_-}{\frac{4m}{3}}+\frac{1}{\e}\tsm{h(t)}{\gamma_-}+\frac{1}{\e}\tsmt{h}{\gamma_-}+\frac{1}{\e}\tsmt{\dt h}{\gamma_-}+\frac{1}{\e^2}\tnnm{\nu z}+\frac{1}{\e}\tnnm{\vv\cdot\nx z}+\frac{1}{\e^2}\tnnm{S(0)}.\no
\end{align}
\end{theorem}
\begin{proof}
Since the perturbed term $\Gamma[f,g],\Gamma[f,f]\in\nk^{\perp}$, we apply Theorem \ref{LN estimate'.} to \eqref{nonlinear unsteady} to obtain
\begin{align}\label{ftt 04.}
&\frac{1}{\e^{\frac{1}{2}}}\tsm{(1-\pp)[f(t)]}{\gamma_+}+\frac{1}{\e}\unnm{(\ik-\pk)[f(t)]}+\pnnm{\pk[f(t)]}{2m}\\
&+\frac{1}{\e^{\frac{1}{2}}}\tsmt{(1-\pp)[\dt f]}{\gamma_+}+\frac{1}{\e}\unnmt{(\ik-\pk)[\dt f]}+\tnnmt{\pk[\dt f]}\no\\
\ls&o(1)\e^{\frac{3}{2m}}\Big(\ssm{f(t)}{\gamma_+}+\snnm{f(t)}\Big)\no\\
&+\frac{1}{\e^{2}}\pnnm{\pk[S(t)]}{\frac{2m}{2m-1}}+\frac{1}{\e}\tnnm{\snn(\ik-\pk)[S(t)]}
+\frac{1}{\e^2}\tnnmt{\pk[\dt S]}+\frac{1}{\e}\tnnmt{\snn(\ik-\pk)[\dt S]}\no\\
&+\pnm{h(t)}{\gamma_-}{\frac{4m}{3}}+\frac{1}{\e}\tsm{h(t)}{\gamma_-}+\frac{1}{\e}\tsmt{\dt h}{\gamma_-}+\frac{1}{\e^2}\tnnm{\nu z}+\frac{1}{\e}\tnnm{\vv\cdot\nx z}+\frac{1}{\e^2}\tnnm{S(0)}\no\\
&+\frac{1}{\e}\tnnm{\snn\Gamma[f,g](t)]}+\frac{1}{\e}\tnnmt{\snn\dt\Gamma[f,g]}+\frac{1}{\e}\tnnm{\e^3\snn\Gamma[f,f](t)]}+\frac{1}{\e}\tnnmt{\e^3\snn\dt\Gamma[f,f]}\no\\
&+\pnm{(\mb-\m)\m^{-1}\pp[f(t)]}{\gamma_-}{\frac{4m}{3}}+\frac{1}{\e}\tsm{(\mb-\m)\m^{-1}\pp[f(t)]}{\gamma_-}+\frac{1}{\e}\tsmt{ (\mb-\m)\m^{-1}\pp[\dt f]}{\gamma_-}\no\\
&+\frac{1}{\e^2}\tnnm{\Gamma[f,g](0)}+\frac{1}{\e^2}\tnnm{\e^3\Gamma[f,f](0)}.\no
\end{align}
Also, based on Lemma \ref{LT estimate'}, we have $L^2$ estimate
\begin{align}\label{ftt 24.}
&\tnnm{f(t)}+\frac{1}{\e^{\frac{1}{2}}}\tsmt{(1-\pp)[f]}{\gamma_+}+\frac{1}{\e}\unnmt{(\ik-\pk)[f]}+\tnnmt{\pk[f]}\\
\ls& \frac{1}{\e^2}\tnnmt{\pk[S]}+\frac{1}{\e}\tnnmt{\snn(\ik-\pk)[S]}+\frac{1}{\e}\tsmt{h}{\gamma_-}+\tnnm{z}\no\\
&+\frac{1}{\e}\tnnmt{\snn\Gamma[f,g]]}+\frac{1}{\e}\tnnmt{\e^3\snn\Gamma[f,f]]}++\frac{1}{\e}\tsmt{(\mb-\m)\m^{-1}\pp[f]}{\gamma_-}.\no
\end{align}
\ \\
Step 1: Bulk Perturbation Terms.\\
Using Lemma \ref{nonlinear lemma} and \eqref{perturbed smallness.}, we have
\begin{align}\label{ftt 03.}
\frac{1}{\e}\tnnm{\snn\Gamma[f,g](t)}\ls o(1)\tnnm{\sn f(t)}\ls o(1)\unnm{\pk[f(t)]}+o(1)\unnm{(\ik-\pk)[f(t)]}.
\end{align}
Note that direct computation reveals that
\begin{align}
\pnnm{\pk[f(t)]}{2m}\gs\unnm{\pk[f(t)]},
\end{align}
so inserting \eqref{ftt 03.} into \eqref{ftt 04.}, we can absorb $o(1)\unnm{\pk[f(t)]}$ and $o(1)\unnm{(\ik-\pk)[f(t)]}$ into the left-hand side. On the other hand, Using Lemma \ref{nonlinear lemma} and \eqref{perturbed smallness.}, we have
\begin{align}
\frac{1}{\e}\tnnmt{\snn\dt\Gamma[f,g]}&\ls o(1)\tnnmt{\sn f}+o(1)\tnnmt{\sn\dt f}
\end{align}
Then $o(1)\tnnm{\sn f}$ can be handled by $L^2$ estimates and $o(1)\tnnm{\sn\dt f}$ can be absorbed into LHS. Similarly,
\begin{align}
\frac{1}{\e}\tnnm{\e^3\snn\Gamma[f,f](t)}&\ls \frac{1}{\e}\lnnmv{\e^3f(t)}\tnnm{\sn f(t)}\ls o(1)\tnnm{\sn f(t)},\\
\frac{1}{\e}\tnnmt{\e^3\snn\dt\Gamma[f,f]}&\ls \frac{1}{\e}\lnnmvt{\e^3f}\tnnmt{\sn\dt f}\ls o(1)\tnnmt{\sn\dt f}.
\end{align}
Both of them can be absorbed into LHS of \eqref{ftt 04.}. A similar argument justifies the absorbing in \eqref{ftt 24.}\\
\ \\
Step 2: Boundary Perturbation Terms.\\
On the other hand, due to \eqref{smallness assumption.}, we know
\begin{align}
\pnm{(\mb-\m)\m^{-1}\pp[f(t)]}{\gamma_-}{\frac{4m}{3}}&\ls o(1)\e\ssm{f(t)}{\gamma_+},
\end{align}
which can be combined with the corresponding term on the right-hand side of \eqref{ftt 04.}. Also,
\begin{align}
\frac{1}{\e}\tsm{(\mb-\m)\m^{-1}\pp[f(t)]}{\gamma_-}&\ls o(1)\tsm{\pp[f(t)]}{\gamma_-},\\
\frac{1}{\e}\tsmt{ (\mb-\m)\m^{-1}\pp[\dt f]}{\gamma_-}&\ls o(1)\tsmt{\pp[\dt f]}{\gamma_-}.
\end{align}
Note that both of then involve $\pp[f]$, which has been controlled by the proof of Theorem \ref{LT estimate'} (Step 2). Hence, adding \eqref{ftt 24.} to \eqref{ftt 04.} and absorbing all new terms into the LHS, we can close the proof.
\end{proof}

\begin{theorem}\label{LI estimate.}
Assume \eqref{linear unsteady compatibility} and \eqref{linear unsteady normalization} hold. The solution $f(t,\vx,\vv)$ to the equation \eqref{nonlinear unsteady} satisfies for $\vth\geq0$ and $0\leq\varrho<\dfrac{1}{4}$,
\begin{align}
&\lnnmvt{f}+\lsmt{f}{\gamma_+}\\
\ls&\frac{1}{\e^{2+\frac{3}{2m}}}\pnnm{\pk[S(t)]}{\frac{2m}{2m-1}}+\frac{1}{\e^{1+\frac{3}{2m}}}\tnnm{\snn(\ik-\pk)[S(t)]}
+\frac{1}{\e^{2+\frac{3}{2m}}}\tnnmt{\pk[S]}+\frac{1}{\e^{1+\frac{3}{2m}}}\tnnmt{\snn(\ik-\pk)[S]}\no\\
&+\frac{1}{\e^{2+\frac{3}{2m}}}\tnnmt{\pk[\dt S]}+\frac{1}{\e^{1+\frac{3}{2m}}}\tnnmt{\snn(\ik-\pk)[\dt S]}+\lnnmvt{\nu^{-1}S}\no\\
&+\frac{1}{\e^{\frac{3}{2m}}}\pnm{h(t)}{\gamma_-}{\frac{4m}{3}}+\frac{1}{\e^{1+\frac{3}{2m}}}\tsm{h(t)}{\gamma_-}
+\frac{1}{\e^{1+\frac{3}{2m}}}\tsmt{h}{\gamma_-}+\frac{1}{\e^{1+\frac{3}{2m}}}\tsmt{\dt h}{\gamma_-}+\lsmt{h}{\gamma_-}\no\\
&+\frac{1}{\e^{2+\frac{3}{2m}}}\tnnm{\nu z}+\frac{1}{\e^{1+\frac{3}{2m}}}\tnnm{\vv\cdot\nx z}+\lnnmv{z}+\frac{1}{\e^{2+\frac{3}{2m}}}\tnnm{S(0)}.\no
\end{align}
\end{theorem}
\begin{proof}
Since we already have bounds for $f$ in $L^{2m}$ as in Theorem \ref{LN estimate.}, following the proof of Theorem \ref{LI estimate'.}, we obtain
\begin{align}\label{ftt 06.}
&\lnnmvt{f}+\lsmt{f}{\gamma_+}\\
\ls&\frac{1}{\e^{2+\frac{3}{2m}}}\pnnm{\pk[S(t)]}{\frac{2m}{2m-1}}+\frac{1}{\e^{1+\frac{3}{2m}}}\tnnm{\snn(\ik-\pk)[S(t)]}
+\frac{1}{\e^{2+\frac{3}{2m}}}\tnnmt{\pk[S]}+\frac{1}{\e^{1+\frac{3}{2m}}}\tnnmt{\snn(\ik-\pk)[S]}\no\\
&+\frac{1}{\e^{2+\frac{3}{2m}}}\tnnmt{\pk[\dt S]}+\frac{1}{\e^{1+\frac{3}{2m}}}\tnnmt{\snn(\ik-\pk)[\dt S]}+\lnnmvt{\nu^{-1}S}\no\\
&+\frac{1}{\e^{\frac{3}{2m}}}\pnm{h(t)}{\gamma_-}{\frac{4m}{3}}+\frac{1}{\e^{1+\frac{3}{2m}}}\tsm{h(t)}{\gamma_-}
+\frac{1}{\e^{1+\frac{3}{2m}}}\tsmt{h}{\gamma_-}+\frac{1}{\e^{1+\frac{3}{2m}}}\tsmt{\dt h}{\gamma_-}+\lsmt{h}{\gamma_-}\no\\
&+\frac{1}{\e^{2+\frac{3}{2m}}}\tnnm{\nu z}+\frac{1}{\e^{1+\frac{3}{2m}}}\tnnm{\vv\cdot\nx z}+\lnnmv{z}+\frac{1}{\e^{2+\frac{3}{2m}}}\tnnm{S(0)}\no\\
&+\lnnmvt{\nu^{-1}\Gamma[f,g]}+\lnnmvt{\e^3\nu^{-1}\Gamma[f,f]}+\lsm{(\mb-\m)\m^{-1}\pp[f]}{\gamma_-}.\no
\end{align}
Using Lemma \ref{nonlinear lemma} and \eqref{perturbed smallness.}, we have
\begin{align}\label{ftt 05.}
\lnnmvt{\nu^{-1}\Gamma[f,g]}\ls \lnnmvt{f}\lnnmvt{g}\ls o(1)\lnnmvt{f},\\
\lnnmvt{\e^3\nu^{-1}\Gamma[f,f]}\ls \lnnmvt{f}\lnnmvt{\e^3f}\ls o(1)\lnnmvt{f}.
\end{align}
Inserting \eqref{ftt 05.} into \eqref{ftt 06.}, we can absorb $o(1)\lnnmv{f}$ into the left-hand side. Also, using \eqref{smallness assumption.}, we have
\begin{align}\label{ftt 19.}
\lsmt{(\mb-\m)\m^{-1}\pp[f]}{\gamma_-}\ls o(1)\lsmt{f}{\gamma_+}.
\end{align}
Inserting \eqref{ftt 19.} into \eqref{ftt 06.} and absorbing $o(1)\lsmt{f}{\gamma_+}$ into the left-hand side, we obtain the desired result.
\end{proof}

\subsection{Analysis of Asymptotic Expansion}

\subsubsection{Analysis of Initial Layer}\label{ftt section 1.}

We first prove a theorem about well-posedness and decay of initial layer equation.

\begin{theorem}\label{initial theorem}
For equation
\begin{align}\label{initial layer..}
\left\{
\begin{array}{l}
\p_{\tau}g+\ll[g]=S(\tau,\vv)\ \ \text{in}\ \ \rp\times\r^3,\\\rule{0ex}{1.5em}
g(0,\vv)=z(\vv),
\end{array}
\right.
\end{align}
satisfying
\begin{align}\label{initial assumption}
\lnmv{z}\ls 1,\quad \lnnmv{\ue^{K_0t}S}\ls 1,
\end{align}
there exists a unique solution $g(\tau,\vvv)$ and a $g_{\infty}\in\nk$ satisfying
\begin{align}
\abs{g_{\infty}}\ls 1,\quad \lnnmv{\ue^{K_0\tau}(g-g_{\infty})}\ls 1.
\end{align}
\end{theorem}
\begin{proof}
This is very similar to the analysis of $\e$-Milne problem with geometric correction, but much simpler. We decompose $g=r+q$, where $r\in\nk^{\perp}$ and $q=\ds\sum_{k=0}^4q_k(\tau)\ne_k(\vv)\in\nk$. Then using the same $L^2-L^{\infty}$ estimates, we can get the desired result.
\end{proof}

With this theorem in hand, based on the analysis in Section \ref{att section 1.}, we know $\fl_1=0$ and $\fl_2,\fl_3,\fl_4$ are well-defined.
\begin{theorem}\label{limit theorem 2.}
For $K_0>0$ sufficiently small, the initial layer satisfies
\begin{align}
\lnnmv{\ue^{K_0\tau}\fl_2(\vx)}\ls 1,\quad\lnnmv{\ue^{K_0\tau}\fl_3(\vx)}\ls 1,\quad \lnnmv{\ue^{K_0\tau}\fl_4(\vx)}\ls 1.
\end{align}
\end{theorem}
In particular, since $\p_t=\e^{-2}\p_{\tau}$, we have the time derivative estimate
\begin{theorem}\label{limit theorem 2.}
For $K_0>0$ sufficiently small, the initial layer satisfies
\begin{align}
\lnnmv{\ue^{K_0\sigma}\frac{\p\fl_2(\vx)}{\p t}}\ls \e^{-2},\quad\lnnmv{\ue^{K_0\sigma}\frac{\p\fl_3(\vx)}{\p t}}\ls \e^{-2},\quad \lnnmv{\ue^{K_0\sigma}\frac{\p\fl_4(\vx)}{\p t}}\ls \e^{-2}.
\end{align}
\end{theorem}
Note that due to rescaling $\tau=\dfrac{t}{\e^2}$, the bound for $\dt\fl_k$ is much worse than $\fl_k$. This is the main reason that we have to expand the initial layer to more orders than interior solution and boundary layer. Also, this is why we have to enforce the compatibility condition \eqref{compatibility condition.} and let $\fl_1$ vanish.

The space derivative version follows the same fashion.
\begin{theorem}\label{limit theorem 2..}
For $K_0>0$ sufficiently small, the initial layer satisfies
\begin{align}
\lnnmv{\ue^{K_0\tau}\nx\fl_2(\vx)}\ls 1,\quad\lnnmv{\ue^{K_0\tau}\nx\fl_3(\vx)}\ls 1,\quad \lnnmv{\ue^{K_0\tau}\nx\fl_4(\vx)}\ls 1.
\end{align}
\end{theorem}

The above estimates do not involve spacial integral. Obviously, the $\vx$ integral estimates also hold.

\subsubsection{Analysis of Boundary Layer}\label{ftt section 2.}

Based on the analysis in Section \ref{att section 1.} and Section 3, we know $\fb_1=0$ and $\fb_2$ is well-defined.
\begin{theorem}\label{limit theorem 1.}
For $K_0>0$ sufficiently small, the boundary layer $\fb_2$ satisfies
\begin{align}
\lnnmv{\ue^{K_0\eta}\fb_2(t)}\ls 1,
\end{align}
and
\begin{align}
\begin{array}{ll}
\lnnmv{\ue^{K_0\eta}\va\dfrac{\p\fb_2(t)}{\p\eta}}+\lnnmv{\ue^{K_0\eta}\dfrac{\p\fb_2(t)}{\p\iota_1}}+\lnnmv{\ue^{K_0\eta}\dfrac{\p\fb_2(t)}{\p\iota_2}}\ls \abs{\ln(\e)}^8,\\\rule{0ex}{2.0em}
\lnnmv{\ue^{K_0\eta}\nu\dfrac{\p\fb_2(t)}{\p\va}}+\lnnmv{\ue^{K_0\eta}\nu\dfrac{\p\fb_2(t)}{\p\vb}}+\lnnmv{\ue^{K_0\eta}\nu\dfrac{\p\fb_2(t)}{\p\vc}}\ls\abs{\ln(\e)}^8.
\end{array}
\end{align}
\end{theorem}
However, the tricky part is the estimate of $\fb_3$, which essentially satisfies a stationary linearized Boltzmann equation
\begin{align}
\left\{
\begin{array}{ll}
\e\vv\cdot\nx\fb_3(t)+\ll[\fb_3(t)]=Z(t)\ \ \text{in}\ \ \tilde\Omega\times\r^3,\\\rule{0ex}{2.0em}
\fb_3(t)(\vx_0,\vv)=\pp[\fb_3(t)](\vx_0,\vv)+b(t)\ \ \text{for}\ \ \vx_0\in\p\Omega\ \
\text{and}\ \ \vv\cdot\vn<0,
\end{array}
\right.
\end{align}
where
\begin{align}
Z:=&2\Gamma[\fb_1,\fb_2]+2\Gamma[\f_1,\fb_2]+2\Gamma[\f_2,\fb_1]
+\dfrac{1}{P_1P_2}\Bigg(\dfrac{\p_{11}\vr\cdot\p_2\vr}{P_1(\e\kk_1\eta-1)}\vb\vc
+\dfrac{\p_{12}\vr\cdot\p_2\vr}{P_2(\e\kk_2\eta-1)}\vc^2\Bigg)\dfrac{\p\fb_2}{\p\vb}\\\rule{0ex}{2.0em}
&+\dfrac{1}{P_1P_2}\Bigg(\dfrac{\p_{22}\vr\cdot\p_1\vr}{P_2(\e\kk_2\eta-1)}\vb\vc
+\dfrac{\p_{12}\vr\cdot\p_1\vr}{P_1(\e\kk_1\eta-1)}\vb^2\Bigg)\dfrac{\p\fb_2}{\p\vc}
+\dfrac{\vb}{P_1(\e\kk_1\eta-1)}\dfrac{\p\fb_2}{\p\iota_1}+\dfrac{\vc}{P_2(\e\kk_2\eta-1)}\dfrac{\p\fb_2}{\p\iota_2},\no
\end{align}
and
\begin{align}
    b:=&\e^{-2}\Big(\mb-\m-\e\mh\m_1\Big)\m^{-1}\pp[\f_1+\fb_1]+\e^{-1}\Big(\mb-\m\Big)\m^{-1}\pp[\f_2+\fb_2]\\
&+\e^{-3}\mhh\Big(\mb-\m-\e\mh\m_1-\e^2\mh\m_2\Big)-\bigg((B_3+C_3)-\pp[B_3+C_3]\bigg).\no
\end{align}
Based on stationary $L^{2m}$ estimates in Remark \ref{LN remark}, we obtain
\begin{align}
&\frac{1}{\e^{\frac{1}{2}}}\tsm{(1-\pp)[\fb_3(t)]}{\gamma_+}+\frac{1}{\e}\unnm{(\ik-\pk)[\fb_3(t)]}+\pnnm{\pk[\fb_3(t)]}{2m}\\
\ls& \frac{1}{\e^{2}}\pnnm{\pk[Z(t)]}{\frac{2m}{2m-1}}+\frac{1}{\e}\tnnm{\snn(\ik-\pk)[Z(t)]}+\lnnmv{\nu^{-1}Z(t)}+\pnm{b(t)}{\gamma_-}{\frac{4m}{3}}+\frac{1}{\e}\tsm{b(t)}{\gamma_-}+\lsm{b(t)}{\gamma_-}\no\\
\ls& 
\frac{1}{\e^{1+\frac{1}{2m}}}\abs{\ln(\e)}^8,\no
\end{align}
where we strongly rely on the rescaling $\eta=\dfrac{\mu}{\e}$ and the exponential decay of $Z$ in $\eta$. Then using the stationary $L^{\infty}$ estimates in Theorem \ref{LI estimate}, we have
\begin{align}
&\lnnmv{\fb_3(t)}+\lsm{\fb_3(t)}{\gamma_+}\\
\ls& \frac{1}{\e^{2+\frac{3}{2m}}}\pnnm{\pk[Z(t)]}{\frac{2m}{2m-1}}+\frac{1}{\e^{1+\frac{3}{2m}}}\tnnm{\snn(\ik-\pk)[Z(t)]}+\lnnmv{\nu^{-1}Z(t)}\no\\
&+\frac{1}{\e^{\frac{3}{2m}}}\pnm{b(t)}{\gamma_-}{\frac{4m}{3}}+\frac{1}{\e^{1+\frac{3}{2m}}}\tsm{b(t)}{\gamma_-}+\lsm{b(t)}{\gamma_-}\no\\
\ls& 
\frac{1}{\e^{1+\frac{2}{m}}}\abs{\ln(\e)}^8.\no
\end{align}
Note that we lose the decay of $\fb_3$ in $\eta$.

The above is only instantaneous version. The corresponding accumulative version for both $\fb_k$ and $\dt\fb_k$ also hold when taking time decay into consideration.  

\subsubsection{Analysis of Interior Solution}\label{ftt section 3.}

Based on the analysis in matching procedure, we know $\f_k=0$ are well-defined satisfy corresponding fluid equations.
\begin{theorem}\label{limit theorem 3.}
For $K_0>0$ sufficiently small, the boundary layer satisfies
\begin{align}
\nnm{\bv\f_1}_{L^{\infty}_tH^3_xL^{\infty}_v}\ls 1,\quad\nnm{\bv\f_2}_{L^{\infty}_tH^3_xL^{\infty}_v}\ls 1,\quad\nnm{\bv\f_3}_{L^{\infty}_tH^3_xL^{\infty}_v}\ls 1.
\end{align}
\end{theorem}

\subsubsection{Analysis of Initial-Boundary Layer}

The compatibility condition \eqref{compatibility condition.} implies that at the corner points $(0,\vx_0,\vv)$, the equation \eqref{small system.} is naturally satisfied. Also, we have the simplified expansion at these points:
\begin{itemize}
\item
By our construction in Section \ref{att section 1.}, $\fb_1=0$ and $\fl_1=0$. Also,
\begin{align}
\f_1(0,\vx_0,\vv)=A_1(t,\vx_0,\vv)+B_1(t,\vx_0,\vv)+C_1(t,\vx_0,\vv)=\rh_{0,1}(\vx_0),
\end{align}
with
\begin{align}
A_1(t,\vx_0,\vv)=\rh_{0,1}(\vx_0)\mh(\vv),\quad B_1(t,\vx_0,\vv)=0,\quad C_1(t,\vx_0,\vv)=0.
\end{align}
In the end, we know
\begin{align}
\f_1(0,\vx_0,\vv)=A_1(t,\vx_0,\vv)=\rh_{0,1}(\vx_0)\mh(\vv).
\end{align}
\item
By our construction in Section \ref{att section 1.}, at $(t,\vx_0,\vv)$, $\fb_2$ and $\fl_2$ satisfy trivial equations with zero source term and zero data, so $\fb_2(0,\vx,\vv)=0$ and $\fl_2(t,\vx_0,\vv)=0$. Also,
\begin{align}
\f_2(0,\vx_0,\vv)=A_2(t,\vx_0,\vv)+B_2(t,\vx_0,\vv)+C_2(t,\vx_0,\vv)=\rh_{0,2}(\vx_0),
\end{align}
with
\begin{align}
A_2(t,\vx_0,\vv)=\rh_{0,2}(\vx_0)\mh(\vv),\quad B_2(t,\vx_0,\vv)=0,\quad C_2(t,\vx_0,\vv)=0.
\end{align}
Here the space derivative $\nx f_{0,1}(\vx_0,\vv)=0$ plays a key role. In the end, we know
\begin{align}
\f_2(0,\vx_0,\vv)=A_2(t,\vx_0,\vv)=\rh_{0,2}(\vx_0)\mh(\vv).
\end{align}
\item
Based on our construction in Section \ref{att section 1.}, we know \begin{align}
\f_3(0,\vx_0,\vv)=A_3(t,\vx_0,\vv)+B_3(t,\vx_0,\vv)+C_3(t,\vx_0,\vv).
\end{align}
In particular, have
\begin{align}
B_3(t,\vx_0,\vv)=0,\quad C_3(t,\vx_0,\vv)=0.
\end{align}
Here the space derivative $\nx f_{0,1}(\vx_0,\vv)=\nx f_{0,2}(\vx_0,\vv)=0$ and $\nx^2 f_{0,1}(\vx_0,\vv)=0$ play a key role. Also, these space derivatives accompanied with $\dt\m_1(t,\vx_0,\vv)=0$ yield $\vv\cdot\nx\fl_2=0$. Hence, we know $\fb_3$ and $\fl_3$ satisfy trivial equation with zero source term and zero data, so $\fb_3(0,\vx,\vv)=0$ and $\fl_3(t,\vx_0,\vv)=0$. In the end, we know
\begin{align}
\f_3(0,\vx_0,\vv)=A_3(t,\vx_0,\vv)=\rh_{0,3}(\vx_0)\mh(\vv).
\end{align}
\item
In summary, we have shown that at the corner point $(0,\vx_0,\vv)$, both the initial layer and boundary layer are zero up to third order. 
\end{itemize}

\subsection{Proof of Main Theorem}

Now we turn to the proof of the main result, Theorem \ref{main.}.
The asymptotic analysis already reveals that the construction of the interior solution, initial layer and boundary layer is valid. Here, we focus on the remainder estimates. We divide the proof into several steps:\\
\ \\
Step 1: Remainder definitions.\\
Define the remainder as
\begin{align}
\e^3R=&f^{\e}-\q -\qb-\ql,
\end{align}
where
\begin{align}
\q:=\sum_{k=1}^3\e^k\f_k,\quad
\qb:=\sum_{k=1}^3\e^k\fb_k,\quad
\ql:=\sum_{k=1}^4\e^k\fl_k.
\end{align}
In other words, we have
\begin{align}
f^{\e}=\q+\qb+\ql+\e^3R.
\end{align}
We write $\lll$ to denote the linearized Boltzmann operator:
\begin{align}
\lll[f]=&\e^2\dt f+\e\vv\cdot\nx f+\ll[f].
\end{align}
In studying initial layer in Section \ref{att section 3.}, we utilize the equivalent form:
\begin{align}
\lll[f]=\p_{\tau}f+\e\vv\cdot\nx u+\ll[f].
\end{align}
In studying boundary layer in Section \ref{att section 2.}, we use another equivalent form:
\begin{align}
\lll[f]=&\e^2\dt f+\va\dfrac{\p f}{\p\eta}-\dfrac{\e}{R_1-\e\eta}\bigg(\vb^2\dfrac{\p f}{\p\va}-\va\vb\dfrac{\p f}{\p\vb}\bigg)
-\dfrac{\e}{R_2-\e\eta}\bigg(\vc^2\dfrac{\p f}{\p\va}-\va\vc\dfrac{\p f}{\p\vc}\bigg)\no\\\rule{0ex}{2.0em}
&-\dfrac{\e}{P_1P_2}\Bigg(\dfrac{\p_{11}\vr\cdot\p_2\vr}{P_1(\e\kk_1\eta-1)}\vb\vc
+\dfrac{\p_{12}\vr\cdot\p_2\vr}{P_2(\e\kk_2\eta-1)}\vc^2\Bigg)\dfrac{\p f}{\p\vb}\no\\\rule{0ex}{2.0em}
&-\dfrac{\e}{P_1P_2}\Bigg(\dfrac{\p_{22}\vr\cdot\p_1\vr}{P_2(\e\kk_2\eta-1)}\vb\vc
+\dfrac{\p_{12}\vr\cdot\p_1\vr}{P_1(\e\kk_1\eta-1)}\vb^2\Bigg)\dfrac{\p f}{\p\vc}\no\\\rule{0ex}{2.0em}
&-\e\bigg(\dfrac{\vb}{P_1(\e\kk_1\eta-1)}\dfrac{\p f}{\p\iota_1}+\dfrac{\vc}{P_2(\e\kk_2\eta-1)}\dfrac{\p f}{\p\iota_2}\bigg)
+\ll[f].\no
\end{align}
\ \\
Step 2: Representation of $\lll[R]$.\\
The equation \eqref{small system.} is actually
\begin{align}
\lll[f^{\e}]=\Gamma[f^{\e},f^{\e}],
\end{align}
which means
\begin{align}\label{ett 01}
\lll[\q+\qb+\ql+\e^3R]=\Gamma[\q+\qb+\ql+\e^3R,\q+\qb+\ql+\e^3R].
\end{align}
Note that the right-hand side of \eqref{ett 01}, i.e. the nonlinear term can be decomposed as
\begin{align}\label{ett 05}
\\
\Gamma[\q+\qb+\ql+\e^3R,\q+\qb+\ql+\e^3R]=&\e^6\Gamma[R,R]+2\e^3\Gamma[R,\q+\qb+\ql]+\Gamma[\q+\qb+\ql,\q+\qb+\ql].\no
\end{align}
Then we turn to the left-hand side of \eqref{ett 01}. The interior contribution is
\begin{align}\label{ett 03}
\lll[\q]=&\e^2\dt\Big(\e\f_1+\e^2\f_2+\e^3\f_3\Big)+\e\vv\cdot\nx\Big(\e\f_1+\e^2\f_2+\e^3\f_3\Big) +\ll[\e\f_1+\e^2\f_2+\e^3\f_3]\\
=&\e^4\vv\cdot\nx\f_3+\e^4\dt\f_2+\e^5\dt\f_3+\e^2\Gamma[\f_1,\f_1]+2\e^3\Gamma[\f_1,\f_2].\no
\end{align}
On the other hand, we consider the boundary layer contribution. Since $\fb_1=0$, $\fb_2$ and $\fb_3$ terms are all included in boundary layer construction except the time derivatives, we compute
\begin{align}\label{ett 02}
\lll[\qb]=\e^4\dt\fb_2+\e^5\dt\fb_3+2\e^3\Gamma[\f_1,\fb_2].
\end{align}
Also, since $\fl_1=0$, the initial layer contribution
\begin{align}\label{ett 04}
\lll[\ql]=\e^5\vv\cdot\nx\fl_4+2\e^3\Gamma[\f_1,\fl_2]+2\e^4\Gamma[\fl_2,\fl_2]+2\e^4\Gamma[\f_2,\fl_2]+2\e^4\Gamma[\f_1,\fl_3].
\end{align}
Therefore, inserting \eqref{ett 05}, \eqref{ett 03}, \eqref{ett 02} and \eqref{ett 04} into \eqref{ett 01}, we have
\begin{align}\label{ftt 11.}
\lll[R]=&\e^3\Gamma[R,R]+2\Gamma[R,\q+\qb+\ql]+S_1+S_2,
\end{align}
where
\begin{align}
S_1=&-\e\vv\cdot\nx\f_3-\e\dt\f_2-\e^2\dt\f_3-\e\dt\fb_2-\e^2\dt\fb_3-\e^2\vv\cdot\nx\fl_4,\\
S_2=&\e\bigg(2\Gamma[\fb_2,\fl_2]+2\Gamma[\f_1,\f_3]+2\Gamma[\f_2,\f_2]+2\Gamma[\f_1,\fb_3]\bigg)+\text{higher-order $\Gamma$ terms up to $\e^4$}.
\end{align}
\ \\
Step 3: Representation of $R-\pp[R]$ and $R(0)$.\\
The boundary condition of \eqref{small system.} is essentially
\begin{align}
f^{\e}=\mb\m^{-1}\pp[f^{\e}]+\mhh(\mb-\m).
\end{align}
which means
\begin{align}
\q+\qb+\e^3R=\pp[\q+\qb+\e^3R]+(\mb-\m)\m^{-1}\pp[\q+\qb+\e^3R]+\mhh(\mb-\m).
\end{align}
Based on the boundary condition expansion in Section \ref{att section 1.}, we have
\begin{align}\label{ftt 12.}
R-\pp[R]=&H[R]+h,
\end{align}
where
\begin{align}
H[R](t,\vx_0,\vv)=(\mb-\m)\m^{-1}\pp[R],
\end{align}
and
\begin{align}
h&=-\e\fl_4.
\end{align}
In other words, the only contribution is from the initial layer $\fl_4$ at the corner point. On the other hand, for initial data
\begin{align}\label{ftt 13.}
R(0)=z=\e\fl_4(0).
\end{align}
In other words, the only contribution is from the initial data of initial layer $\fl_4$. \\
\ \\
Step 4: Remainder Estimate.\\
The equation \eqref{ftt 11.}, initial condition \eqref{ftt 13.} and boundary condition \eqref{ftt 12.} forms a system that fits into \eqref{nonlinear unsteady}:
\begin{align}
\\
\left\{
\begin{array}{l}
\e^2\dt R+\e\vv\cdot\nx R+\ll[R]=\Gamma[R,2(\q+\qb+\ql)+\e^3R]+S_1(t,\vx,\vv)+S_2(t,\vx,\vv)\ \ \text{in}\ \ \rp\times\Omega\times\r^3,\\\rule{0ex}{1.5em}
R(0,\vx,\vv)=z(\vx,\vv)\ \ \text{in}\ \ \Omega\times\r^3,\\\rule{0ex}{1.5em}
R(t,\vx_0,\vv)=\pp[R](t,\vx_0,\vv)+H[R](t,\vx_0,\vv)+h(t,\vx_0,\vv)\ \ \text{for}\ \ \vx_0\in\p\Omega\ \
\text{and}\ \ \vv\cdot\vn<0.\no
\end{array}
\right.
\end{align}
Hence, by Theorem \ref{LI estimate.}, we have
\begin{align}
&\lnnmvt{R}+\lsmt{R}{\gamma_+}\\
\ls&\frac{1}{\e^{2+\frac{3}{2m}}}\pnnm{\pk[S_1(t)]}{\frac{2m}{2m-1}}+\frac{1}{\e^{1+\frac{3}{2m}}}\tnnm{\snn(\ik-\pk)[S_1(t)]}
+\frac{1}{\e^{2+\frac{3}{2m}}}\tnnmt{\pk[S_1]}+\frac{1}{\e^{1+\frac{3}{2m}}}\tnnmt{\snn(\ik-\pk)[S_1]}\no\\
&+\frac{1}{\e^{2+\frac{3}{2m}}}\tnnmt{\pk[\dt S_1]}+\frac{1}{\e^{1+\frac{3}{2m}}}\tnnmt{\snn(\ik-\pk)[\dt S_1]}+\lnnmvt{\nu^{-1}S}+\frac{1}{\e^{2+\frac{3}{2m}}}\tnnm{S_1(0)}\no\\
&+\frac{1}{\e^{2+\frac{3}{2m}}}\pnnm{\pk[S_2(t)]}{\frac{2m}{2m-1}}+\frac{1}{\e^{1+\frac{3}{2m}}}\tnnm{\snn(\ik-\pk)[S_2(t)]}
+\frac{1}{\e^{2+\frac{3}{2m}}}\tnnmt{\pk[S_2]}+\frac{1}{\e^{1+\frac{3}{2m}}}\tnnmt{\snn(\ik-\pk)[S_2]}\no\\
&+\frac{1}{\e^{2+\frac{3}{2m}}}\tnnmt{\pk[\dt S_2]}+\frac{1}{\e^{1+\frac{3}{2m}}}\tnnmt{\snn(\ik-\pk)[\dt S_2]}+\lnnmvt{\nu^{-1}S}+\frac{1}{\e^{2+\frac{3}{2m}}}\tnnm{S_2(0)}\no\\
&+\frac{1}{\e^{\frac{3}{2m}}}\pnm{h(t)}{\gamma_-}{\frac{4m}{3}}+\frac{1}{\e^{1+\frac{3}{2m}}}\tsm{h(t)}{\gamma_-}
+\frac{1}{\e^{1+\frac{3}{2m}}}\tsmt{h}{\gamma_-}+\frac{1}{\e^{1+\frac{3}{2m}}}\tsmt{\dt h}{\gamma_-}+\lsmt{h}{\gamma_-}\no\\
&+\frac{1}{\e^{2+\frac{3}{2m}}}\tnnm{\nu z}+\frac{1}{\e^{1+\frac{3}{2m}}}\tnnm{\vv\cdot\nx z}+\lnnmv{z}.\no
\end{align}
\ \\
Step 5: Estimate of $S_1$.\\
Using results in Section \ref{ftt section 3.}, for the interior contribution $S_{IS}:=-\e\vv\cdot\nx\f_3-\e\dt\f_2-\e^2\dt\f_3$:
\begin{align}
&\pnnm{S_{IS}(t)}{\frac{2m}{2m-1}}+\tnnm{\snn S_{IS}(t)]}
+\tnnmt{S_{IS}}+\tnnmt{\snn S_{IS}}\no\\
&+\tnnmt{\dt S_{IS}}+\tnnmt{\snn\dt S_{IS}}+\lnnmvt{\nu^{-1}S_{IS}}+\tnnm{S_{IS}(0)}\ls \e.
\end{align}
Using results in Section \ref{ftt section 2.}, for the boundary layer contribution $S_{BL}:=-\e\dt\fb_2-\e^2\dt\fb_3$, note that $\nm{g(t)}_{L^p}\ls\nm{g(t)}_{L^{2m}}$ for $1\leq p\leq 2m$:
\begin{align}
&\pnnm{\pk[S_{BL}](t)}{\frac{2m}{2m-1}}\ls \e^{1-\frac{1}{2m}}\abs{\ln(\e)}^8,\quad \tnnm{\snn (\ik-\pk)[S_{BL}](t)]}\ls \e^{2-\frac{1}{2m}}\abs{\ln(\e)}^8,\\
&\tnnmt{\pk[S_{BL}]}\ls\e^{1-\frac{1}{2m}}\abs{\ln(\e)}^8,\quad \tnnmt{\snn (\ik-\pk)[S_{BL}]}\ls \e^{2-\frac{1}{2m}}\abs{\ln(\e)}^8,\no\\
&\tnnmt{\pk[\dt S_{BL}]}\ls\e^{1-\frac{1}{2m}}\abs{\ln(\e)}^8,\quad \tnnmt{\snn (\ik-\pk)[\dt S_{BL}]}\ls \e^{2-\frac{1}{2m}}\abs{\ln(\e)}^8,\no\\
&\lnnmvt{\nu^{-1}S_{BL}}\ls \e^{1-\frac{1}{m}}\abs{\ln(\e)}^8,\quad \tnnm{S_{BL}(0)}\ls \e^{1-\frac{1}{2m}}\abs{\ln(\e)}^8.
\end{align}
Using results in Section \ref{ftt section 1.}, for the initial layer contribution $S_{IL}:=-\e^2\vv\cdot\nx\fl_4$, note the rescaling $\tau=\dfrac{t}{\e^2}$:
\begin{align}
&\pnnm{\pk[S_{IL}](t)}{\frac{2m}{2m-1}}\ls \e^2,\quad \tnnm{\snn (\ik-\pk)[S_{IL}](t)]}\ls \e^{2},\\
&\tnnmt{\pk[S_{IL}]}\ls\e^{3},\quad \tnnmt{\snn (\ik-\pk)[S_{IL}]}\ls \e^{3},\no\\
&\tnnmt{\pk[\dt S_{IL}]}\ls\e,\quad \tnnmt{\snn (\ik-\pk)[\dt S_{IL}]}\ls \e,\no\\
&\lnnmvt{\nu^{-1}S_{IL}}\ls \e^{2},\quad \tnnm{S_{IL}(0)}\ls \e^{2}.
\end{align}
Hence, we have
\begin{align}
&\pnnm{\pk[S_{1}](t)}{\frac{2m}{2m-1}}\ls \e^{1-\frac{1}{2m}}\abs{\ln(\e)}^8,\quad \tnnm{\snn (\ik-\pk)[S_{1}](t)]}\ls \e^{2-\frac{1}{2m}}\abs{\ln(\e)}^8,\\
&\tnnmt{\pk[S_{1}]}\ls\e^{1-\frac{1}{2m}}\abs{\ln(\e)}^8,\quad \tnnmt{\snn (\ik-\pk)[S_{1}]}\ls \e^{2-\frac{1}{2m}}\abs{\ln(\e)}^8,\no\\
&\tnnmt{\pk[\dt S_{1}]}\ls\e^{1-\frac{1}{2m}}\abs{\ln(\e)}^8,\quad \tnnmt{\snn (\ik-\pk)[\dt S_{1}]}\ls \e^{2-\frac{1}{2m}}\abs{\ln(\e)}^8,\no\\
&\lnnmvt{\nu^{-1}S_{1}}\ls \e^{1-\frac{1}{m}}\abs{\ln(\e)}^8,\quad \tnnm{S_{1}(0)}\ls \e^{1-\frac{1}{2m}}\abs{\ln(\e)}^8.
\end{align}
\ \\
Step 6: Estimate of $S_2$.\\
It suffices to consider the leading-order term $2\e\Gamma[\fb_2,\fl_2]$ which contains the most dangerous initial layer $\fl_2$. Note that the time derivative estimate is the worst one. Using nonlinear estimates in Lemma \ref{nonlinear lemma} and rescaling $\eta=\dfrac{\mu}{\e}$ and $\tau=\dfrac{t}{\e^2}$, we have
\begin{align}
&\pnnm{\pk[S_{2}](t)}{\frac{2m}{2m-1}}=0,\quad \tnnm{\snn (\ik-\pk)[S_{2}](t)]}\ls \e^{\frac{3}{2}},\\
&\tnnmt{\pk[S_{2}]}=0,\quad \tnnmt{\snn (\ik-\pk)[S_{2}]}\ls \e^{\frac{5}{2}},\no\\
&\tnnmt{\pk[\dt S_{2}]}=0,\quad \tnnmt{\snn (\ik-\pk)[\dt S_{2}]}\ls \e^{\frac{1}{2}},\no\\
&\lnnmvt{\nu^{-1}S_{2}}\ls \e,\quad \tnnm{S_{2}(0)}\ls \e^{\frac{3}{2}}.
\end{align}
\ \\
Step 7: Estimate of $h$ and $z$.\\
For boundary data $h=-\e\fl_4$, we have
\begin{align}
&\pnm{h(t)}{\gamma_-}{\frac{4m}{3}}\ls \e,\quad \tsm{h(t)}{\gamma_-}\ls \e,\quad \tsmt{h}{\gamma_-}\ls \e^2,\\
&\tsmt{\dt h}{\gamma_-}\ls 1,\quad \lsmt{h}{\gamma_-}\ls \e.\no
\end{align}
For initial data $z=-\e\fl_4(0)$, we have
\begin{align}
\tnnm{\nu z}\ls \e,\quad\tnnm{\vv\cdot\nx z}\ls\e,\quad \lnnmv{z}\ls\e.
\end{align}
\ \\
Step 8: Synthesis.\\
Summarizing all above, we have
\begin{align}
&\lnnmvt{R}+\lsmt{R}{\gamma_+}\\
\ls&\frac{1}{\e^{2+\frac{3}{2m}}}\bigg(\e^{1-\frac{1}{2m}}\abs{\ln(\e)}^8\bigg)+\frac{1}{\e^{1+\frac{3}{2m}}}\bigg(\e^{2-\frac{1}{2m}}\abs{\ln(\e)}^8\bigg)
+\frac{1}{\e^{2+\frac{3}{2m}}}\bigg(\e^{1-\frac{1}{2m}}\abs{\ln(\e)}^8\bigg)+\frac{1}{\e^{1+\frac{3}{2m}}}\bigg(\e^{2-\frac{1}{2m}}\abs{\ln(\e)}^8\bigg)\no\\
&+\frac{1}{\e^{2+\frac{3}{2m}}}\bigg(\e^{1-\frac{1}{2m}}\abs{\ln(\e)}^8\bigg)+\frac{1}{\e^{1+\frac{3}{2m}}}\bigg(\e^{2-\frac{1}{2m}}\abs{\ln(\e)}^8\bigg)
+\bigg(\e^{1-\frac{1}{m}}\abs{\ln(\e)}^8\bigg)+\frac{1}{\e^{2+\frac{3}{2m}}}\bigg(\e^{1-\frac{1}{2m}}\abs{\ln(\e)}^8\bigg)\no\\
&+\frac{1}{\e^{2+\frac{3}{2m}}}\Big(0\Big)+\frac{1}{\e^{1+\frac{3}{2m}}}\Big(\e^{\frac{3}{2}}\Big)
+\frac{1}{\e^{2+\frac{3}{2m}}}\Big(0\Big)+\frac{1}{\e^{1+\frac{3}{2m}}}\Big(\e^{\frac{5}{2}}\Big)\no\\
&+\frac{1}{\e^{2+\frac{3}{2m}}}\Big(0\Big)+\frac{1}{\e^{1+\frac{3}{2m}}}\Big(\e^{\frac{1}{2}}\Big)+\Big(\e\Big)+\frac{1}{\e^{2+\frac{3}{2m}}}\Big(\e^{\frac{3}{2}}\Big)\no\\
&+\frac{1}{\e^{\frac{3}{2m}}}\Big(\e\Big)+\frac{1}{\e^{1+\frac{3}{2m}}}\Big(\e\Big)
+\frac{1}{\e^{1+\frac{3}{2m}}}\Big(\e^2\Big)+\frac{1}{\e^{1+\frac{3}{2m}}}\Big(1\Big)+\Big(\e\Big)\no\\
&+\frac{1}{\e^{2+\frac{3}{2m}}}\Big(\e\Big)+\frac{1}{\e^{1+\frac{3}{2m}}}\Big(\e\Big)+\Big(\e\Big)\no\\
\ls&\frac{1}{\e^{1+\frac{2}{m}}}\abs{\ln(\e)}^8.\no
\end{align}
We have shown
\begin{align}
\frac{1}{\e^3}\lnnmv{f^{\e}-\sum_{k=1}^3\e^k\f_k-\sum_{k=1}^3\e^k\fb_k-\sum_{k=1}^4\e^k\fl_k}\ls \e^{-1-\frac{2}{m}}\abs{\ln(\e)}^8.
\end{align}
Therefore, we know
\begin{align}
\lnnmv{f^{\e}-\e\f_1-\e\fb_1-\fl_1}\ls\e^{2-\frac{2}{m}}\abs{\ln(\e)}^8.
\end{align}
Since $\fb_1=\fl_1=0$, then we naturally have for $\f=\f_1$.
\begin{align}
\lnnmv{f^{\e}-\e\f}\ls\e^{2-\frac{2}{m}}\abs{\ln(\e)}^8.
\end{align}
Here $\dfrac{3}{2}< m<3$, so we may further bound
\begin{align}
\lnnmv{f^{\e}-\e\f}\ls C(\d)\e^{\frac{4}{3}-\d},
\end{align}
for any $0<\d<<1$. The exponential decay in time can be justified in a similar fashion using Remark \ref{exponential remark}.

\chapter{$\e$-Milne Problem with Geometric Correction}

\section{Well-Posedness and Decay}

In this section, we will study the well-posedness and decay of the the $\e$-Milne problem with geometric correction. Let the null space $\nk$ of the operator $\ll$ be spanned by
$\m^{\frac{1}{2}}\bigg\{1,\va,\vb,\vc,\dfrac{\abs{\vvv}^2-3}{2}\bigg\}=\{\ee_0,\ee_1,\ee_2,\ee_3,\ee_4\}$. Given data $h(\vvv)$ and $S(\eta,\vvv)$, we intend to find
\begin{align}\label{Milne transform compatibility}
\tilde h(\vvv):=\sum_{i=0}^4\tilde D_i\ee_i\in\nk,
\end{align}
with $\tilde D_1=0$ and the other $\tilde D_i$ are constants such that the $\e$-Milne problem with geometric correction for
$\gg(\eta,\vvv)$ in the domain
$(\eta,\vvv)\in[0,L]\times\r^3$ as
\begin{align}\label{Milne transform}
\left\{
\begin{array}{l}\displaystyle
\va\dfrac{\p\gg }{\p\eta}-\dfrac{\e}{R_1-\e\eta}\bigg(\vb^2\dfrac{\p\gg}{\p\va}-\va\vb\dfrac{\p\gg}{\p\vb}\bigg)
-\dfrac{\e}{R_2-\e\eta}\bigg(\vc^2\dfrac{\p\gg}{\p\va}-\va\vc\dfrac{\p\gg}{\p\vc}\bigg)+\ll[\gg]
=\ss,\\\rule{0ex}{1.5em}
\gg (0,\vvv)=h (\vvv)-\tilde h(\vvv)\ \ \text{for}\ \
\va>0,\\\rule{0ex}{1.5em}
\displaystyle\gg (L,\vvv)=\gg (L,\rr[\vvv]),
\end{array}
\right.
\end{align}
is well-posed, and $\gg$ decays exponentially fast to zero as $\eta$ becomes larger and larger. Here $\rr[\vvv]=(-\va,\vb,\vc)$ and $L=\e^{-\frac{1}{2}}$. For simplicity, we temporarily ignore the dependence of $\iota_1,\iota_2$, but our estimates are uniform in these variables. Also, the estimates and decaying rate should be uniform in $\e$.

To achieve this goal, we will focus on studying a variation of the above equation, i.e. the $\e$-Milne problem with geometric correction for $\g(\eta,\vvv)$ in
the domain
$(\eta,\vvv)\in[0,L]\times\r^3$
\begin{align}\label{Milne}
\left\{
\begin{array}{l}\displaystyle
\va\dfrac{\p\g }{\p\eta}-\dfrac{\e}{R_1-\e\eta}\bigg(\vb^2\dfrac{\p\g}{\p\va}-\va\vb\dfrac{\p\g}{\p\vb}\bigg)
-\dfrac{\e}{R_2-\e\eta}\bigg(\vc^2\dfrac{\p\g}{\p\va}-\va\vc\dfrac{\p\g}{\p\vc}\bigg)+\ll[\g ]
=\ss,\\\rule{0ex}{1.5em}
\g (0,\vvv)=h (\vvv)\ \ \text{for}\ \
\va>0,\\\rule{0ex}{1.5em}
\displaystyle\g (L,\vvv)=\g (L,\rr[\vvv]).
\end{array}
\right.
\end{align}

In the following, we will show that $\g$ converges to $\g_L(\vvv)\in\nk$ as $\eta\rt\infty$, where $\g_L$ is a function independent of $\eta$ and is completely determined by $h$ and $S$. Next we will prove that the desired $\tilde h$ is roughly $\g_L$ with some minor modifications.

We can decompose
\begin{align}
\g:=w_g+q_g,
\end{align}
where
\begin{align}
q_g&=\m^{\frac{1}{2}}\bigg(q_{g,0}+q_{g,1}\va+q_{g,2}\vb+q_{g,3}\vc+q_{g,4}\dfrac{\abs{\vvv}^2-3}{2}\bigg)\\
&=q_{g,0}\ee_0+q_{g,1}\ee_1+q_{g,2}\ee_2+q_{g,3}\ee_3+q_{g,4}\ee_4\in\nk,\no
\end{align}
and
\begin{align}
w_g\in\nk^{\perp},
\end{align}
where $\nk^{\perp}$ is the orthogonal complement of $\nk$ in $L^2_{\vvv}$. When there is no confusion, we will simply write $\g=w+q$.

In this section, we introduce some special notation to describe the
norms for $(\eta,\vvv)\in[0,L]\times\r^3$. Define
the $L^2$ norms as follows:
\begin{align}
\tnm{f(\eta)}:=&\bigg(\int_{\r^3}\abs{f(\eta,\vvv)}^2\ud{\vvv}\bigg)^{\frac{1}{2}},\\
\tnnm{f}:=&\bigg(\int_0^{L}\int_{\r^3}\abs{f(\eta,\vvv)}^2\ud{\vvv}\ud{\eta}\bigg)^{\frac{1}{2}}.
\end{align}
Define the inner product in $\vvv$
\begin{align}
\br{f,g}(\eta):=\int_{\r^3}
f(\eta,\vvv)g(\eta,\vvv)\ud{\vvv}.
\end{align}
Define the weighted $L^{\infty}$ norms as follows:
\begin{align}
\lnm{f(\eta)}{\vth,\varrho}:=&\esssup_{\vvv\in\r^3}\bigg(\bvv\abs{f(\eta,\vvv)}\bigg),\\
\lnnm{f}{\vth,\varrho}:=&\esssup_{(\eta,\vvv)\in[0,L]\times\r^3}\bigg(\bvv\abs{f(\eta,\vvv)}\bigg),
\end{align}
Define the mixed $L^2$ and weighted $L^{\infty}$ norm as follows:
\begin{align}
\ltnm{f}{\varrho}:=&\esssup_{\eta\in[0,L]}\bigg(\int_{\r^3}\abs{\ue^{2\varrho\abs{\vvv}^2}f(\eta,\vvv)}^2\ud{\vvv}\bigg)^{\frac{1}{2}}.
\end{align}
Here, we require $0\leq\varrho<\dfrac{1}{4}$ and $\vth>3$.

Since the boundary data $h(\vvv)$ is only defined on $\va>0$, we
naturally extend above definitions to this half-domain as follows:
\begin{align}
\tnmh{h}:=&\bigg(\int_{\va>0}\abs{h(\vvv)}^2\ud{\vvv}\bigg)^{\frac{1}{2}},\\
\lnmh{h}:=&\sup_{\va>0}\bigg(\bvv\abs{h(\vvv)}\bigg).
\end{align}
Throughout this section, we assume
\begin{align}\label{Milne bound}
\lnmh{h}\ls 1,\quad\lnnm{\ue^{K\eta}\ss}{\vth,\varrho}\ls 1,
\end{align}
for some constant $K>0$ uniform in $\e$.
%

\begin{lemma}\label{linearized operator lemma}
For $f\in\nk^{\perp}$, we have
\begin{align}
\tnm{\ll[f]}\ls \tnm{\nu f}.
\end{align}
\end{lemma}
\begin{proof}
Based on \cite[Section 3]{Glassey1996}, we know $\ll=\nu I-K$, where
\begin{align}
\tnm{K[f]}\ls \tnm{f},
\end{align}
so $\ll$ estimate naturally follows. 
\end{proof}

The existence of uniqueness of $\g$ and $\gg$ follow from a standard iteration argument as in \cite{AA004} and \cite{AA003}, so we will omit the proof here and focus on the a priori estimates.

\subsection{$L^2$ Estimates}

\subsubsection{$\ss\in\nk^{\perp}$ Case}

Denote
\begin{align}
G_1(\eta):=-\dfrac{\e}{R_1-\e\eta},\quad G_2(\eta):=-\dfrac{\e}{R_2-\e\eta},
\end{align}
and
\begin{align}
G(\eta):=G_1(\eta)+G_2(\eta).
\end{align}
Let $W_i(\eta)$ satisfy
\begin{align}\label{mtt 10}
G_i=-\dfrac{\ud W_i}{\ud\eta},\quad W_i(0)=0\ \ \text{for}\ \ i=1,2.
\end{align}
Hence, it is easy to check that
\begin{align}
W_i(\eta)=\ln\left(\dfrac{R_i}{R_i-\e\eta}\right).
\end{align}
Denote
\begin{align}
W(\eta):=W_1(\eta)+W_2(\eta).
\end{align}

\begin{remark}\label{Milne power}
We know for $\e<<1$, $\e\eta\leq\e L=\e^{\frac{1}{2}}<<1$, which implies $W(\eta)\sim 0$ and further $\ue^{W(\eta)}\sim 1$.
\end{remark}

We will estimate $\g=w+q$ separately and divide it into several steps.

\begin{lemma}[orthogonality estimate]\label{Milne lemma 3}
Assume $\ss\in\nk^{\perp}$. We have
\begin{align}
\br{\va\ee_j,g}(\eta)=&0\ \ \text{for}\ \ j=0,2,3,4\ \ \text{and}\ \ \eta\in[0,L].
\end{align}
\end{lemma}
\begin{proof}
Multiplying $\ee_j$ for $j=0,2,3,4$ on both sides of \eqref{Milne} and
integrating over $\vvv\in\r^3$, we have
\begin{align}
\frac{\ud{}}{\ud{\eta}}\br{\va\ee_j,\g}+
G_1\br{\vb^2\dfrac{\p
\g}{\p\va}-\va\vb\dfrac{\p \g}{\p\vb},\ee_j}+
G_2\br{\vc^2\dfrac{\p
\g}{\p\va}-\va\vc\dfrac{\p \g}{\p\vc},\ee_j}=-\br{\ll[\g],\ee_j}+\br{\ss,\ee_j}.
\end{align}
Since $\ll$ is self-adjoint and $\ee_j\in\nk$ as well as $\ss\in\nk^{\perp}$, we have
\begin{align}
\br{\ll[\g],\ee_j}=\br{\ll[\ee_j],\g}=0,\quad \br{\ss,\ee_j}=0.
\end{align}
An integration by parts implies
\begin{align}
G_1\br{\vb^2\dfrac{\p
\g}{\p\va}-\va\vb\dfrac{\p \g}{\p\vb},\ee_j}&=-G_1\br{\frac{\p}{\p\va}(\ee_j\vb^2)-\frac{\p}{\p\vb}(\ee_j\va\vb),\g}=C_1G_1\br{\ee_j\va,\g},\\
G_2\br{\vc^2\dfrac{\p
\g}{\p\va}-\va\vc\dfrac{\p \g}{\p\vc},\ee_j}&=-G_2\br{\frac{\p}{\p\va}(\ee_j\vc^2)-\frac{\p}{\p\vc}(\ee_j\va\vc),\g}=C_1G_2\br{\ee_j\va,\g}.
\end{align}
where $C_1$ and $C_2$ are constants. Summarizing all above, we know that \eqref{mtt 33} is
\begin{align}
\frac{\ud{}}{\ud{\eta}}\br{\ee_j\va,\g}=(C_1G_1+C_2G_2)\br{\ee_j\va,\g}.
\end{align}
Considering the reflexive boundary which implies $\br{\ee_j\va,\g}(L)=0$, we have for any $\eta\in[0,L]$,
\begin{align}
\br{\ee_j\va,\g}(\eta)=0.
\end{align}
\end{proof}

\begin{remark}
Note that $\br{\va\ee_1,\g}(\eta)$ is not necessarily zero.
\end{remark}

\begin{lemma}[$L^2$ estimates of $L^2$]\label{Milne lemma 1}
Assume \eqref{Milne bound} holds and $\ss\in\nk^{\perp}$. We have
\begin{align}
\tnnm{\sn{w}}\ls 1.
\end{align}
\end{lemma}
\begin{proof}
Multiplying $\g$ on both sides of \eqref{Milne} and
integrating over $\vvv\in\r^3$, we have
\begin{align}\label{mtt 00}
\half\frac{\ud{}}{\ud{\eta}}\br{\va\g,\g}+
G_1\br{\vb^2\dfrac{\p
\g}{\p\va}-\va\vb\dfrac{\p \g}{\p\vb},\g}+
G_2\br{\vc^2\dfrac{\p
\g}{\p\va}-\va\vc\dfrac{\p \g}{\p\vc},\g}=-\br{\g,\ll[\g]}+\br{\g,\ss}.
\end{align}
An integration by parts implies
\begin{align}
\br{\vb^2\dfrac{\p
\g}{\p\va}-\va\vb\dfrac{\p \g}{\p\vb},\g}=\half\br{\vb^2,\dfrac{\p
(\g^2)}{\p\va}}-\half\br{\va\vb,\dfrac{\p (\g^2)}{\p\vb}}=\half\br{\va\g,\g},\\
\br{\vc^2\dfrac{\p
\g}{\p\va}-\va\vc\dfrac{\p \g}{\p\vc},\g}=\half\br{\vc^2,\dfrac{\p
(\g^2)}{\p\va}}-\half\br{\va\vc,\dfrac{\p (\g^2)}{\p\vc}}=\half\br{\va\g,\g}.
\end{align}
Also, since $\ll$ is a self-adjoint operator with null space $\nk$, we get
\begin{align}\label{mtt 49}
\br{\g,\ll[\g]}=\br{q,\ll[q]}+\br{w,\ll[q]}+\br{q,\ll[w]}+\br{w,\ll[w]}=\br{w,\ll[w]}.
\end{align}
Therefore, we simplify \eqref{mtt 00} to obtain
\begin{align}\label{mtt 01}
\half\frac{\ud{}}{\ud{\eta}}\br{\va\g,\g}+\half
G\br{\va\g,\g}=-\br{w,\ll[w]}+\br{w,\ss}.
\end{align}
Define
\begin{align}\label{mtt 38}
\alpha(\eta)=\half\br{\va\g,\g}(\eta).
\end{align}
Then \eqref{mtt 01} may be rewritten as
\begin{align}
\frac{\ud{\alpha}}{\ud{\eta}}+G\alpha=-\br{w,\ll[w]}+\br{w,\ss}.
\end{align}
Then regarding the above as an ODE and solve it in $[\eta,L]$, we have
\begin{align}
\alpha(\eta)=&\alpha(L)\exp\bigg(\int_{\eta}^LG(y)\ud{y}\bigg)+
\int_{\eta}^L\exp\bigg(-\int_{\eta}^yG(z)\ud{z}\bigg)\bigg(\br{w,\ll[w]}(y)+\br{w,\ss}(y)\bigg)\ud{y}.\label{mtt 02}
\end{align}
Note the fact that $\alpha(L)=0$ due to the reflexive boundary condition. Also, $\br{w,\ll[w]}(\eta)\geq\tnm{\sn{w(\eta)}}^2$ due to coercivity. Hence, \eqref{mtt 02} implies that
\begin{align}
\alpha(\eta)\geq\int_{\eta}^L\exp\bigg(-\int_{\eta}^yG(z)\ud{z}\bigg)\bigg(\tnm{\sn{w(y)}}^2+\br{w,\ss}(y)\bigg)\ud{y}.
\end{align}
In particular, taking $\eta=0$, we have
\begin{align}\label{mtt 32}
\alpha(0)&\geq\int_{0}^L\exp\bigg(-\int_{0}^yG(z)\ud{z}\bigg)\bigg(\tnm{\sn{w(y)}}^2+\br{w,\ss}(y)\bigg)\ud{y}\\
&=\int_{0}^L\ue^{W(y)}\bigg(\tnm{\sn{w(y)}}^2+\br{w,\ss}(y)\bigg)\ud{y},\no
\end{align}
which yields
\begin{align}\label{mtt 03}
\int_{0}^{L}\ue^{W(y)}\tnm{\sn{w(y)}}^2\ud{y}&\leq \alpha(0)-\int_{0}^L\ue^{W(y)}\br{w,\ss}(y)\ud{y}.
\end{align}
On the other hand, \eqref{Milne bound} implies
\begin{align}\label{mtt 31}
\alpha(0)=&\half\br{\va\g,\g}(0)=\half\int_{\va>0}\va\g^2(0,\vvv)\ud{\vvv}+\half\int_{\va<0}\va\g^2(0,\vvv)\ud{\vvv}\leq\half\int_{\va>0}\va\g^2(0,\vvv)\ud{\vvv}\\
=&\half\int_{\va>0}\va h^2(\vvv)\ud{\vvv}\ls 1.\no
\end{align}
Combined \eqref{mtt 03} and \eqref{mtt 31}, we obtain
\begin{align}
\int_{0}^{L}\ue^{W(y)}\tnm{\sn{w(y)}}^2\ud{y}\ls 1+\int_{0}^L\ue^{W(y)}\br{w,\ss}(y)\ud{y}.
\end{align}
Using \eqref{mtt 10} and Remark \ref{Milne power}, as well as H\"older's inequality and Cauchy's inequality, we get
\begin{align}
\int_0^L\tnm{\sn{w}(\eta)}^2\ud{\eta}&\ls 1+\int_{0}^L\abs{\br{w,\ss}(\eta)}\ud{\eta}\ls 1+\int_{0}^L\tnm{\sn{w}(\eta)}\tnm{\snn{\ss}(\eta)}\ud\eta\\
&\ls 1+\d\int_0^L\tnm{\sn{w}(\eta)}^2\ud{\eta}+\d^{-1}\int_0^L\tnm{\snn{\ss}(\eta)}^2\ud{\eta}.\no
\end{align}
Therefore, for sufficiently small $\d$, we absorb $\ds\d\int_0^L\tnm{\sn{w}(\eta)}^2\ud{\eta}$ into LHS and use \eqref{Milne bound} to obtain
\begin{align}
\int_0^L\tnm{\sn{w}(\eta)}^2\ud{\eta}\ls 1+\int_0^L\tnm{\snn{\ss}(\eta)}^2\ud{\eta}\ls 1.
\end{align}
\end{proof}

\begin{remark}
Based on the proof of Lemma \ref{Milne lemma 1}, \eqref{mtt 32} actually implies
\begin{align}
\alpha(0)\gs \int_{0}^L\ue^{W(y)}\br{w,\ss}(y)\ud{y}\gs \int_{0}^L\abs{\br{w,\ss}(y)}\ud{y}.
\end{align}
Hence, with \eqref{mtt 31} holds, $\alpha(0)$ actually has both upper and lower bounds, i.e.
\begin{align}\label{mtt 04}
\int_{0}^L\abs{\br{w,\ss}(y)}\ud{y}\ls\alpha(0)\ls 1.
\end{align}
\end{remark}

\begin{lemma}[point-wise estimate of $q$]\label{Milne lemma 2}
Assume \eqref{Milne bound} holds and $\ss\in\nk^{\perp}$. We have $q_1(\eta)=0$ and
\begin{align}
\abs{q_j(\eta)}\ls 1+\eta+\tnm{\sn{w}(\eta)}\ \ \text{for}\ \ j=0,2,3,4.
\end{align}
\end{lemma}
\begin{proof}
Multiplying $\ee_0$ on both sides of \eqref{Milne} and
integrating over $\vvv\in\r^3$, we have
\begin{align}\label{mtt 33}
\frac{\ud{}}{\ud{\eta}}\br{\va\ee_0,\g}+
G_1\br{\vb^2\dfrac{\p
\g}{\p\va}-\va\vb\dfrac{\p \g}{\p\vb},\ee_0}+
G_2\br{\vc^2\dfrac{\p
\g}{\p\va}-\va\vc\dfrac{\p \g}{\p\vc},\ee_0}=-\br{\ll[\g],\ee_0}+\br{\ss,\ee_0}.
\end{align}
Since $\ll$ is self-adjoint and $\ee_0\in\nk$ as well as $\ss\in\nk^{\perp}$, we have
\begin{align}
\br{\ll[\g],\ee_0}=\br{\ll[\ee_0],\g}=0,\quad \br{\ss,\ee_0}=0.
\end{align}
An integration by parts implies
\begin{align}
G_1\br{\vb^2\dfrac{\p
\g}{\p\va}-\va\vb\dfrac{\p \g}{\p\vb},\ee_0}&=-G_1\br{\frac{\p}{\p\va}(\ee_0\vb^2)-\frac{\p}{\p\vb}(\ee_0\va\vb),\g}=G_1\br{\ee_0\va,\g},\\
G_2\br{\vc^2\dfrac{\p
\g}{\p\va}-\va\vc\dfrac{\p \g}{\p\vc},\ee_0}&=-G_2\br{\frac{\p}{\p\va}(\ee_0\vc^2)-\frac{\p}{\p\vc}(\ee_0\va\vc),\g}=G_2\br{\ee_0\va,\g}.
\end{align}
Summarizing all above, we know that \eqref{mtt 33} is
\begin{align}\label{mtt 12}
\frac{\ud{}}{\ud{\eta}}\br{\ee_0\va,\g}=-G\br{\ee_0\va,\g}.
\end{align}
Since $\br{\ee_0\va,\g}=\br{\ee_1,\g}=q_1$, \eqref{mtt 12} is actually
\begin{align}
\frac{\ud{q_1}}{\ud{\eta}}=-Gq_1 .
\end{align}
Considering the reflexive boundary which implies $q_1 (L)=0$, we have for any $\eta\in[0,L]$,
\begin{align}\label{mt 11}
q_1(\eta)=0.
\end{align}
Multiplying $\va\ee_j$ with $j=0,2,3,4$ on both sides of \eqref{Milne} and integrating over
$\vvv\in\r^3$, we obtain
\begin{align}\label{mtt 34}
\\
\frac{\ud{}}{\ud{\eta}}\br{\va^2\ee_j,\g}+G_1\br{\va\ee_j,
\vb^2\dfrac{\p \g}{\p\va}-\va\vb\dfrac{\p
\g}{\p\vb}}+G_2\br{\va\ee_j,
\vc^2\dfrac{\p \g}{\p\va}-\va\vc\dfrac{\p
\g}{\p\vc}}=-\br{\va\ee_j,\ll[\g]}+\br{\va\ee_j,\ss}.\no
\end{align}
Define
\begin{align}
\beta_j(\eta)=&\br{\va^2\ee_j,q}(\eta),\\
\beta(\eta)=&\bigg(\beta_0(\eta),\beta_2(\eta),\beta_3(\eta),\beta_4(\eta)\bigg)^T.
\end{align}
For $j=0,2,3,4$,
\begin{align}
\br{\va^2\ee_j,\g}&=\br{\va^2\ee_j,q}+\br{\va^2\ee_j,w}=\beta_j+\br{\va^2\ee_j,w}.
\end{align}
Using integration by parts, we have
\begin{align}
G_1\br{\va\ee_j,\vb^2\dfrac{\p \g}{\p\va}-\va\vb\dfrac{\p\g}{\p\vb}}&=-G_1\br{\frac{\p}{\p\va}(\va\vb^2\ee_j)-\frac{\p}{\p\vb}(\va^2\vb\ee_j),\g},\\
G_2\br{\va\ee_j,\vc^2\dfrac{\p \g}{\p\va}-\va\vc\dfrac{\p\g}{\p\vc}}&=-G_2\br{\frac{\p}{\p\va}(\va\vc^2\ee_j)-\frac{\p}{\p\vc}(\va^2\vc\ee_j),\g}.
\end{align}
Considering $g=w+q$ and summarizing the above, we can simplify \eqref{mtt 34} into
\begin{align}
\frac{\ud{}}{\ud{\eta}}\br{\va^2\ee_j,\g}=&G_1\br{\frac{\p}{\p\va}(\va\vb^2\ee_j)-\frac{\p}{\p\vb}(\va^2\vb\ee_j),
\g}+G_2\br{\frac{\p}{\p\va}(\va\vc^2\ee_j)-\frac{\p}{\p\vc}(\va^2\vc\ee_j),
\g}\\
&-\br{\va\ee_j,\ll[w]}+\br{\va\ee_j,\ss},\no
\end{align}
which further implies
\begin{align}\label{mtt 05}
\frac{\ud{\beta_j}}{\ud{\eta}}=&G_1\br{\frac{\p}{\p\va}(\va\vb^2\ee_j)-\frac{\p}{\p\vb}(\va^2\vb\ee_j),
q +w}+G_2\br{\frac{\p}{\p\va}(\va\vc^2\ee_j)-\frac{\p}{\p\vc}(\va^2\vc\ee_j),
q +w}\\
&-\br{\va\ee_j,\ll[w]}+\br{\va\ee_j,\ss}-\frac{\ud{}}{\ud{\eta}}\br{\va^2\ee_j,w}.\no
\end{align}
Then we can write
\begin{align}
\br{\frac{\p}{\p\va}(\va\vb^2\ee_j)-\frac{\p}{\p\vb}(\va^2\vb\ee_j),
q }(\eta)=\sum_{i}B_{ji}^{(1)} q_i (\eta),\\
\br{\frac{\p}{\p\va}(\va\vc^2\ee_j)-\frac{\p}{\p\vc}(\va^2\vc\ee_j),
q }(\eta)=\sum_{i}B_{ji}^{(2)} q_i (\eta),
\end{align}
for $i,j=0,2,3,4$, where there is no $q_1$ contribution since $q_1=0$. Here $B^{(1)}$ and $B^{(2)}$ are $4\times4$ constant matrices defined by
\begin{align}
B_{ji}^{(1)}=\br{\frac{\p}{\p\va}(\va\vb^2\ee_j)-\frac{\p}{\p\vb}(\va^2\vb\ee_j),
\ee_i},\\
B_{ji}^{(2)}=\br{\frac{\p}{\p\va}(\va\vc^2\ee_j)-\frac{\p}{\p\vc}(\va^2\vc\ee_j),
\ee_i}.
\end{align}
Moreover, we may rewrite
\begin{align}
\beta_j(\eta)=\sum_{k}A_{jk} q_k (\eta),
\end{align}
for $k,j=0,2,3,4$, where $A$ is an invertible $4\times4$ constant matrix defined by
\begin{align}
A_{jk}=\br{\va^2\ee_j,\ee_k}.
\end{align}
Thus, we can express back
\begin{align}\label{mtt 43}
(q_0,q_2,q_3,q_4)^T=A^{-1}(\beta_0,\beta_2,\beta_3,\beta_4)^T.
\end{align}
Hence, \eqref{mtt 05} can be rewritten in vector form as
\begin{align}\label{mtt 35}
\frac{\ud{\beta}}{\ud{\eta}}=\Big((G_1B^{(1)}+G_2B^{(2)})A^{-1}\Big)\beta+D+E-\frac{\ud F}{\ud\eta},
\end{align}
where the four-vector $D$, $E$ and $F$ are defined for $j=0,2,3,4$
\begin{align}
D_j=&G_1\br{\frac{\p}{\p\va}(\va\vb^2\ee_j)-\frac{\p}{\p\vb}(\va^2\vb\ee_j),
w}
+G_2\br{\frac{\p}{\p\va}(\va\vc^2\ee_j)-\frac{\p}{\p\vc}(\va^2\vc\ee_j),
w},\label{mtt 57}\\
E_j=&-\br{\va\ee_j,\ll[w]}+\br{\va\ee_j,\ss},\quad
F_j=\br{\va^2\ee_j,w}.\no
\end{align}
\eqref{mtt 35} is an ODE system. Using \eqref{mtt 10}, we can solve for $\beta$ as
\begin{align}\label{mtt 06}
\beta(\eta)=&\exp\bigg(-\Big(W_1(\eta)B^{(1)}+W_2(\eta)B^{(2)}\Big)A^{-1}\bigg)\beta(0)\\
&+\int_0^{\eta}\exp\bigg(\Big(W_1(y)-W_1(\eta)\Big)B^{(1)}A^{-1}+\Big(W_2(y)-W_2(\eta)\Big)B^{(2)}A^{-1}\bigg)\bigg(D(y)+E(y)-\frac{\ud F}{\ud y}(y)\bigg)\ud{y}.\no
\end{align}
Again using \eqref{mtt 10}, we may directly integrate by parts for the $F$ term to obtain
\begin{align}\label{mtt 36}
&\int_0^{\eta}\exp\bigg(\Big(W_1(y)-W_1(\eta)\Big)B^{(1)}A^{-1}+\Big(W_2(y)-W_2(\eta)\Big)B^{(2)}A^{-1}\bigg)\frac{\ud F}{\ud y}(y)\ud{y}\\
=&F(\eta)-\exp\bigg(-W_1(\eta)B^{(1)}A^{-1}-W_2(\eta)B^{(2)}A^{-1}\bigg)F(0)\no\\
&+\int_0^{\eta}\exp\bigg(\Big(W_1(y)-W_1(\eta)\Big)B^{(1)}A^{-1}+\Big(W_2(y)-W_2(\eta)\Big)B^{(2)}A^{-1}\bigg)\bigg(G_1(y)B^{(1)}A^{-1}+G_2(y)B^{(2)}A^{-1}\bigg)F(y)\ud{y}.\no
\end{align}
Hence, inserting \eqref{mtt 36} into \eqref{mtt 05}, we have
\begin{align}\label{mtt 07}
\beta(\eta)=&\exp\bigg(-\Big(W_1(\eta)B^{(1)}+W_2(\eta)B^{(2)}\Big)A^{-1}\bigg)\theta-F(\eta)\\
&+\int_0^{\eta}\exp\bigg(\Big(W_1(y)-W_1(\eta)\Big)B^{(1)}A^{-1}+\Big(W_2(y)-W_2(\eta)\Big)B^{(2)}A^{-1}\bigg)Z(y)\ud{y},\no
\end{align}
where $\theta$ is a four-vector satisfying
\begin{align}
\theta_j=\beta_j(0)+F_j(0)=\br{\va^2\ee_j, \g}(0),\ \ j=0,2,3,4,
\end{align}
and $Z$ is a four-vector satisfying
\begin{align}\label{mtt 37}
Z=D+E-(G_1B^{(1)}+G_2B^{(2)})A^{-1}F.
\end{align}
Hence, considering $A$, $B^{(1)}$, $B^{(2)}$ are all constant matrices and Remark \ref{Milne power}, we can directly estimate \eqref{mtt 07} to get
\begin{align}\label{mtt 08}
\abs{\beta_j(\eta)}\ls
\abs{\theta_j}+\abs{F_j(\eta)}+\int_0^{\eta}\abs{Z_j(y)}\ud{y}\ \ \text{for}\ \ j=0,2,3,4.
\end{align}
Using H\"older's inequality and Lemma \ref{linearized operator lemma}, we have
\begin{align}
\abs{D_j(\eta)}\ls& \abs{G_1}\abs{\br{\frac{\p}{\p\va}(\va\vb^2\ee_j)-\frac{\p}{\p\vb}(\va^2\vb\ee_j),w}(\eta)}
+\abs{G_2}\abs{\br{\frac{\p}{\p\va}(\va\vc^2\ee_j)-\frac{\p}{\p\vc}(\va^2\vc\ee_j),w}(\eta)}\label{mtt 13}\\
\ls& \e \tnm{\snn\bigg(\frac{\p}{\p\va}(\va\vb^2\ee_j)-\frac{\p}{\p\vb}(\va^2\vb\ee_j)\bigg)}\tnm{\sn w(\eta)}\no\\
&+\e\tnm{\snn\bigg(\frac{\p}{\p\va}(\va\vc^2\ee_j)-\frac{\p}{\p\vc}(\va^2\vc\ee_j)\bigg)}\tnm{\sn w(\eta)}\no\\
\ls& \e\tnm{\sn{w}(\eta)},\no
\end{align}
and
\begin{align}
\abs{E_j(\eta)}\ls&\abs{\br{\va\ee_j,\ll[w]}(\eta)}+\abs{\br{\va\ee_j,\ss}(\eta)}=\abs{\br{\ll[\va\ee_j],w}(\eta)}+\abs{\br{\va\ee_j,\ss}(\eta)}\label{mtt 11}\\
\ls&\tnm{\snn\ll[\va\ee_j]}\tnm{\sn w(\eta)}+\tnm{\sn\va\ee_j}\tnm{\snn\ss(\eta)}\ls \tnm{\sn w(\eta)}+\tnm{\snn\ss(\eta)},\no
\end{align}
as well as
\begin{align}
\abs{F_j(\eta)}\ls&\abs{\br{\va^2\ee_j,w}(\eta)}\ls \tnm{\snn\va^2\ee_j}\tnm{\sn w}\ls \tnm{\sn{w}(\eta)}.\label{mtt 09}
\end{align}
Inserting \eqref{mtt 13}, \eqref{mtt 11} and \eqref{mtt 09} into \eqref{mtt 37}, we obtain
\begin{align}\label{mtt 42}
\abs{Z_j}\ls \tnm{\sn w(\eta)}+\tnm{\snn\ss(\eta)}\ls 1+\tnm{\sn w(\eta)}.
\end{align}
On the other hand, for $\theta_j$, using H\"older's inequality, we have
\begin{align}\label{mtt 39}
\abs{\theta_j}=\abs{\br{\va^2\ee_j, \g}(0)}\ls \tnm{\abs{\va}^{\frac{3}{2}}\ee_j}\tnm{\abs{\va}^{\frac{1}{2}}\g(0)}
\ls \tnm{\abs{\va}^{\frac{1}{2}}\g(0)},
\end{align}
where
\begin{align}\label{mtt 40}
\tnm{\abs{\va}^{\frac{1}{2}}\g(0)}=\int_{\va>0}\va
h^2(\vvv)\ud{\vvv}-\int_{\va<0}\va\g^2(0,\vvv)\ud{\vvv}.
\end{align}
Using \eqref{mtt 38} and \eqref{mtt 04}, we have
\begin{align}
\int_{\va>0}\va h^2(\vvv)\ud{\vvv}+\int_{\va<0}\va\g^2(0,\vvv)\ud{\vvv}=2\alpha(0)\gs \int_{0}^L\abs{\br{w,\ss}(y)}\ud{y},
\end{align}
which implies
\begin{align}\label{mtt 41}
-\int_{\va<0}\va\g^2(0,\vvv)\ud{\vvv}\ls \int_{\va>0}\va h^2(\vvv)\ud{\vvv}-\int_{0}^L\abs{\br{w,\ss}(y)}\ud{y}.
\end{align}
Hence, inserting \eqref{mtt 41} into \eqref{mtt 40} and further \eqref{mtt 39}, applying H\"older's inequality and Cauchy's inequality, we have
\begin{align}\label{mtt 14}
\abs{\theta_j}&\ls \int_{0}^L\abs{\br{w,\ss}(y)}\ud{y}+\int_{\va>0}\va h^2(\vvv)\ud{\vvv}
\ls 1+\tnnm{\sn w}\tnnm{\snn\ss}\ls 1+\tnnm{\sn w}.
\end{align}
In conclusion, inserting \eqref{mtt 42}, \eqref{mtt 13} and \eqref{mtt 14} into \eqref{mtt 08},  we have
\begin{align}
\abs{\beta_j(\eta)}\ls 1+\tnnm{\sn{w}}+\tnm{\sn w(\eta)}+\int_0^{\eta}\Big(1+\tnm{\sn{w}(y)}\Big)\ud{y}\ \ \text{for}\ \ j=0,2,3,4,
\end{align}
which, using \eqref{mtt 43} and Lemma \ref{Milne lemma 1}, further implies
\begin{align}\label{mtt 51}
\abs{{q}_j(\eta)}\ls
1+\eta+\tnm{\sn{w}(\eta)}+\int_0^{\eta}\tnm{\sn{w}(y)}\ud{y}\ \ \text{for}\ \ j=0,2,3,4.
\end{align}
An application of H\"older's inequality, Cauchy's inequality and Lemma \ref{Milne lemma 1} lead to
\begin{align}
\abs{{q}_j(\eta)}\ls
1+\eta+\tnm{\sn{w}(\eta)}+\eta^{\frac{1}{2}}\tnnm{\sn w}\ls 1+\eta+\tnm{\sn{w}(\eta)}\ \ \text{for}\ \ j=0,2,3,4.
\end{align}
\end{proof}

\begin{remark}
Using a standard iteration argument, Lemma \ref{Milne lemma 1} and Lemma \ref{Milne lemma 2} justify the well-posedness of solution $\g=w+q$. However, the estimates in Lemma \ref{Milne lemma 2} are not uniform in $\eta$, so we need a stronger version.
\end{remark}

\begin{lemma}[$L^2$ decay of $w$]\label{Milne lemma 4}
Assume \eqref{Milne bound} holds and $\ss\in\nk^{\perp}$. There exists $0<K_0<K$ such that
\begin{align}
\tnnm{\ue^{K_0\eta}\sn w}\ls 1.
\end{align}
\end{lemma}
\begin{proof}
Multiplying $\ue^{2K_0\eta}\g$ on both sides of \eqref{Milne} and integrating over $\vvv\in\r^3$, we obtain
\begin{align}\label{mtt 18}
\half\frac{\ud{}}{\ud{\eta}}\br{\va\g,\ue^{2K_0\eta}\g}+
G_1\br{\vb^2\dfrac{\p
\g}{\p\va}-\va\vb\dfrac{\p \g}{\p\vb},\ue^{2K_0\eta}\g}\\
+
G_2\br{\vc^2\dfrac{\p
\g}{\p\va}-\va\vc\dfrac{\p \g}{\p\vc},\ue^{2K_0\eta}\g}&=K_0\ue^{2K_0\eta}\br{\va{\g},{\g}}-\ue^{2K_0\eta}\br{\g,\ll[\g]}+\ue^{2K_0\eta}\br{\ss,\g}.\no
\end{align}
We simplify each term here. The orthogonal properties in Lemma \ref{Milne lemma 3} implies
\begin{align}\label{mtt 46}
\br{\va\ee_j,\g}(\eta)=\br{\va\ee_j,w}(\eta)+\br{\va\ee_j,q}(\eta)=0\ \ \text{for}\ \ j=0,2,3,4.
\end{align}
Based on Lemma \ref{Milne lemma 2}, $q_1=0$. Combined with oddness, we know
\begin{align}\label{mtt 47}
\br{\va\ee_j,q}(\eta)=\sum_{k=0}^4q_k\br{\va\ee_j,\ee_k}(\eta)=0.
\end{align}
Inserting \eqref{mtt 47} into \eqref{mtt 46}, we obtain
\begin{align}
\br{\va\ee_j,w}(\eta)=0.
\end{align}
Still by $q_1=0$, we have
\begin{align}
\br{\va q,w}(\eta)=\sum_{j=0}^4q_j\br{\va\ee_j,w}=0,
\end{align}
and also by oddness
\begin{align}
\br{\va q,q}(\eta)=0.
\end{align}
Therefore, we deduce that
\begin{align}\label{mtt 48}
\br{\va\g,\g}(\eta)=\br{\va w, w}(\eta)+\br{\va q,q}(\eta)+2\br{\va w,q}(\eta)=\br{\va{w},{w}}(\eta).
\end{align}
On the other hand, \eqref{mtt 49} yields
\begin{align}
\br{\va{\g},{\g}}(\eta)=\br{\va w,w}(\eta).
\end{align}
Similar to the proof of Lemma \ref{Milne lemma 1}, an integration by parts and \eqref{mtt 48} imply
\begin{align}
\br{\vb^2\dfrac{\p
\g}{\p\va}-\va\vb\dfrac{\p \g}{\p\vb},\ue^{2K_0\eta}\g}&=\half\br{\vb^2,\dfrac{\p(\ue^{2K_0\eta}\g^2)}{\p\va}}-\half\br{\va\vb,\dfrac{\p (\ue^{2K_0\eta}\g^2)}{\p\vb}}\\
&=\half\br{\va\g,\ue^{2K_0\eta}\g}=\half\br{\va w,\ue^{2K_0\eta}w},\no\\
\br{\vc^2\dfrac{\p\g}{\p\va}-\va\vc\dfrac{\p \g}{\p\vc},\ue^{2K_0\eta}\gg}&=\half\br{\vc^2,\dfrac{\p(\ue^{2K_0\eta}\g^2)}{\p\va}}-\half\br{\va\vc,\dfrac{\p (\ue^{2K_0\eta}\g^2)}{\p\vc}}\\
&=\half\br{\va\g,\ue^{2K_0\eta}\g}=\half\br{\va w,\ue^{2K_0\eta}w}.\no
\end{align}
Also,
\begin{align}
\br{\ss,\g}=\br{\ss,w}.
\end{align}
Summarizing all above, \eqref{mtt 18} is actually
\begin{align}
\half\frac{\ud{}}{\ud{\eta}}\br{\va w,\ue^{2K_0\eta}w}+\half
G(\eta)\br{\va w,\ue^{2K_0\eta}w}=K_0\ue^{2K_0\eta}\br{\va{w},{w}}-\ue^{2K_0\eta}\br{w,\ll[w]}+\ue^{2K_0\eta}\br{\ss,w}.
\end{align}
Since
\begin{align}
\br{\ll[{w}],{w}}\gs \tnm{\sn{w}}^2,
\end{align}
for $K_0$ sufficiently small, we have
\begin{align}
\br{\ll[{w}],{w}}-K_0\br{\va{w},{w}}\gs\tnm{\sn{w}}^2.
\end{align}
Then by a similar argument as in the proof of Lemma \ref{Milne lemma 1}, we can show that
\begin{align}\label{mtt 29}
\int_0^{L}\ue^{2K_0\eta}\tnm{\sn{w}(\eta)}^2\ud{\eta}\ls 1.
\end{align}
\end{proof}

\begin{lemma}[$q-q_L$ estimate]\label{Milne lemma 5}
Assume \eqref{Milne bound} holds and $\ss\in\nk^{\perp}$. There exists
\begin{align}
q_L=\sum_{k=0}^4q_{k,L}\ee_k\in\nk,
\end{align}
satisfying
\begin{align}
\abs{q_{k,L}}\ls 1\ \ \text{for}\ \ k=0,1,2,3,4,
\end{align}
and
\begin{align}
\tnnm{q-q_L}\ls 1.
\end{align}
\end{lemma}
\begin{proof}
Recall \eqref{mtt 07}
\begin{align}\label{mtt 50}
\beta(\eta)=&\exp\bigg(-\Big(W_1(\eta)B^{(1)}+W_2(\eta)B^{(2)}\Big)A^{-1}\bigg)\theta-F(\eta)\\
&+\int_0^{\eta}\exp\bigg(\Big(W_1(y)-W_1(\eta)\Big)B^{(1)}A^{-1}+\Big(W_2(y)-W_2(\eta)\Big)B^{(2)}A^{-1}\bigg)Z(y)\ud{y},\no
\end{align}
Define
\begin{align}\label{mtt 56}
\beta_L:=&\exp\bigg(-\Big(W_1(L)B^{(1)}+W_2(L)B^{(2)}\Big)A^{-1}\bigg)\theta\\
&+\int_0^{L}\exp\bigg(\Big(W_1(y)-W_1(L)\Big)B^{(1)}A^{-1}+\Big(W_2(y)-W_2(L)\Big)B^{(2)}A^{-1}\bigg)Z(y)\ud{y}.\no
\end{align}
Here
\begin{align}
\beta_L=\Big(\beta_{0,L},\ \beta_{2,L},\ \beta_{3,L},\ \beta_{4,L}\Big)^T,
\end{align}
is a four-vector. Based on \eqref{mtt 08},
\begin{align}\label{mtt 52}
\abs{\beta_{j,L}}\ls
\abs{\theta_j}+\int_0^{L}\abs{Z_j(y)}\ud{y}\ \ \text{for}\ \ j=0,2,3,4.
\end{align}
Inserting \eqref{mtt 42}, \eqref{mtt 13} and \eqref{mtt 14} into \eqref{mtt 52},  we have
\begin{align}\label{mtt 53}
\abs{\beta_{j,L}}\ls 1+\tnnm{\sn{w}}+\int_0^{L}\Big(\tnm{\snn\ss(y)}+\tnm{\sn{w}(y)}\Big)\ud{y}\ \ \text{for}\ \ j=0,2,3,4.
\end{align}
Using \eqref{Milne bound}, we know
\begin{align}\label{mtt 54}
\int_0^{L}\tnm{\snn\ss(y)}\ud{y}\ls 1.
\end{align}
Applying H\"older's inequality and Lemma \ref{Milne lemma 4}, we obtain
\begin{align}\label{mtt 55}
\int_0^{L}\tnm{\sn{w}(y)}\ud{y}&\ls \tnnm{\ue^{-K_0\eta}}\tnnm{\ue^{K_0\eta}\sn w}\ls 1.
\end{align}
Inserting \eqref{mtt 54} and \eqref{mtt 55} into \eqref{mtt 53} and using Lemma \ref{Milne lemma 1}, we have
\begin{align}
\abs{\beta_{j,L}}\ls 1\ \ \text{for}\ \ j=0,2,3,4.
\end{align}
Then by \eqref{mtt 43}, define
\begin{align}\label{mtt 66}
\Big(q_{0,L},\ q_{2,L},\ q_{3,L},\ q_{4,L}\Big)^T:=A^{-1}\Big(\beta_{0,L},\ \beta_{2,L},\ \beta_{3,L},\ \beta_{4,L}\Big)^T.
\end{align}
We have
\begin{align}
\abs{q_{j,L}}\ls 1.
\end{align}
Let $q_{1,L}=0$. Then we know $q_L$ is always well-defined and
\begin{align}
\abs{q_{k,L}}\ls 1\ \ \text{for}\ \ k=0,1,2,3,4.
\end{align}
Then we investigate the estimate $q-q_L$. Denote
\begin{align}
\Xi(\eta)=\Big(W_1(\eta)B^{(1)}+W_2(\eta)B^{(2)}\Big)A^{-1}.
\end{align}
Considering \eqref{mtt 50} and \eqref{mtt 56}, we have
\begin{align}\label{mtt 65}
\beta(\eta)-\beta_L&=\ue^{-\Xi(\eta)}\Big(\ue^{\Xi(\eta)}\beta(\eta)\Big)-\ue^{-\Xi(L)}\Big(\ue^{\Xi(L)}\beta_L\Big)\\
&=\ue^{-\Xi(\eta)}\Big(\ue^{\Xi(\eta)}\beta(\eta)-\ue^{\Xi(L)}\beta_L\Big)+(\ue^{-\Xi(\eta)}-\ue^{-\Xi(L)})\Big(\ue^{\Xi(L)}\beta_L\Big):=\Delta_1(\eta)+\Delta_2(\eta).\no
\end{align}
Here, using Remark \ref{Milne power}, we have
\begin{align}
\int_0^L\Delta_1^2(\eta)\ud\eta\ls \int_0^L\bigg(\ue^{\Xi(\eta)}\beta(\eta)-\ue^{\Xi(L)}\beta_L\bigg)^2\ud\eta,
\end{align}
where
\begin{align}
\ue^{\Xi(\eta)}\Omega(\eta)-\ue^{\Xi(L)}\Omega_L
=&-\ue^{\Xi(\eta)}F(\eta)+\int_{\eta}^L\ue^{\Xi(y)}Z(y)\ud y.
\end{align}
Note that
\begin{align}
Z=D+E-(G_1B^{(1)}+G_2B^{(2)})A^{-1}F,
\end{align}
where $D,E,F$ are defined in \eqref{mtt 57}, and due to Remark \ref{Milne power},
\begin{align}\label{mtt 62}
\\
\int_0^L\Delta_1^2(\eta)\ud\eta\ls\int_0^LF^2(\eta)\ud\eta+\int_0^L\bigg(\int_{\eta}^LD(y)\bigg)^2\ud\eta
+\int_0^L\bigg(\int_{\eta}^LE(y)\bigg)^2\ud\eta+\e\int_0^L\bigg(\int_{\eta}^LF(y)\bigg)^2\ud\eta.\no
\end{align}
We need to estimate each term. In the following, let $j=0,2,3,4$, using H\"older's inequality and Lemma \ref{Milne lemma 4}, we have
\begin{align}\label{mtt 58}
\int_0^LF_j^2(\eta)\ud\eta&\ls \int_0^L\tnm{\va^2\ee_j}^2\tnm{\sn w(\eta)}^2\ud\eta\ls \tnnm{\sn w}^2\ls 1.
\end{align}
Similarly, using H\"older's inequality and Lemma \ref{Milne lemma 4}, we have
\begin{align}\label{mtt 59}
\int_0^L\bigg(\int_{\eta}^LD(y)\bigg)^2\ud\eta&\ls \e^2\int_0^L\bigg(\int_{\eta}^L\tnm{\sn w(y)}\bigg)^2\ud\eta\\
&\ls \e^2\int_0^L\bigg(\int_{\eta}^L\ue^{2K_0y}\tnm{\sn w(y)}^2\bigg)\bigg(\int_{\eta}^L\ue^{-2K_0y}\bigg)\ud\eta\no\\
&\ls\e^2\tnnm{\ue^{K_0\eta}\sn w}^2\int_0^L\ue^{-2K_0\eta}\ud\eta\ls \e^2.\no
\end{align}
Similarly, using H\"older's inequality, Lemma \ref{Milne lemma 4} and Lemma \ref{linearized operator lemma}, we have
\begin{align}
\int_0^L\bigg(\int_{\eta}^LE(y)\bigg)^2\ud\eta&\ls\tnnm{\ue^{K_0\eta}\sn w}^2+\tnnm{\ue^{K_0\eta}\snn S}^2\ls 1,\label{mtt 60}\\
\e\int_0^L\bigg(\int_{\eta}^LF(y)\bigg)^2\ud\eta&\ls \tnnm{\ue^{K_0\eta}\sn w}^2.\label{mtt 61}
\end{align}
Inserting \eqref{mtt 58}, \eqref{mtt 59}, \eqref{mtt 60} and \eqref{mtt 61} into \eqref{mtt 62}, we have
\begin{align}\label{mtt 63}
\int_0^L\Delta_1^2(\eta)\ud\eta\ls 1.
\end{align}
On the other hand, using Remark \ref{Milne power}, since $\abs{\ue^{s}-1}\ls \abs{s}$ and $\abs{\ln(1+s)}\ls \abs{s}$ for $\abs{s}<<1$, we have
\begin{align}\label{mtt 64}
\int_0^L\Delta_2^2(\eta)\ud\eta&\ls\int_0^L\bigg(\ue^{-\Xi(\eta)}-\ue^{-\Xi(L)}\bigg)^2\ud\eta\ls\int_0^L\bigg(\ue^{\Xi(L)-\Xi(\eta)}-1\bigg)^2\ud\eta\\
&\ls\int_0^L\bigg(\ue^{W_1(\eta)-W_1(L)+W_2(\eta)-W_2(L)}-1\bigg)^2\ud\eta\no\\
&\ls\int_0^L\bigg(W_1(\eta)-W_1(L)\bigg)^2\ud\eta+\int_0^L\bigg(W_2(\eta)-W_2(L)\bigg)^2\ud\eta\no\\
&\ls\int_0^L\ln^2\left(\frac{R_1-\e L}{R_1-\e\eta}\right)\ud\eta+\int_0^L\ln^2\left(\frac{R_2-\e L}{R_2-\e\eta}\right)\ud\eta\no\\
&\ls\int_0^L\Big(\e(\eta-L)\Big)^2\ud\eta\ls \e^2 L^{3}\ls \e^{\frac{1}{2}}.\no
\end{align}
Inserting \eqref{mtt 63} and \eqref{mtt 64} into \eqref{mtt 65}, we obtain
\begin{align}
\int_0^L\bigg(\beta(\eta)-\beta_L\bigg)^2\ud\eta\ls 1.
\end{align}
Since $A$ is invertible, by \eqref{mtt 66}, we know
\begin{align}
\int_0^L\bigg(q_j-q_{j,L}\bigg)^2\ls 1\ \ \text{for}\ \ j=0,2,3,4.
\end{align}
It is easy to see that $q_1(\eta)=q_{1,L}=0$. Therefore, we prove that
\begin{align}
\tnnm{{q}-{q}_{L}}\leq C.
\end{align}
\end{proof}

\begin{remark}
This proof highly depends on the fact that $G\sim\e$ and $L\sim \e^{-\frac{1}{2}}$. Also, the $L^2$ decay of $w$ in Lemma \ref{Milne lemma 4} is indispensable.
\end{remark}

\begin{lemma}[$L^2$ estimate of $\g-\g_L$]\label{Milne lemma 6}
Assume \eqref{Milne bound} holds and $\ss\in\nk^{\perp}$. There exists a unique solution $\g(\eta,\vvv)$ to the $\e$-Milne problem with geometric correction
\eqref{Milne} satisfying
\begin{align}
\tnnm{\g-\g_{L}}\ls 1,
\end{align}
for some $g_L=\ds\sum_{k=0}^4g_{k,L}\ee_k\in\nk$ satisfying $\abs{g_{k,L}}\ls 1$.
\end{lemma}
\begin{proof}
Taking $\g_{L}={q}_{L}$ in Lemma \ref{Milne lemma 5}, combined with Lemma \ref{Milne lemma 1}, we can naturally obtain the desired result.
\end{proof}

\subsubsection{$\ss\notin\nk^{\perp}$ Case}

\begin{lemma}[$L^2$ well-posedness of $\g$]\label{Milne lemma 7}
Assume \eqref{Milne bound} holds. There exists a unique solution $\g(\eta,\vvv)$ to the $\e$-Milne problem with geometric correction
\eqref{Milne} satisfying
\begin{align}
\tnnm{\g-\g_{L}}\ls 1,
\end{align}
for some $g_L=\ds\sum_{k=0}^4g_{k,L}\ee_k\in\nk$ satisfying $\abs{g_{k,L}}\ls 1$.
\end{lemma}
\begin{proof}
We decompose the source term as
\begin{align}
S=S_Q+S_W,
\end{align}
where $S_Q\in\nk$ is the kernel part and $S_W=S-S_Q\in\nk^{\perp}$. In the following, we will construct a few auxiliary functions $\g_i$ to handle $S_Q$ and $S_W$ separately.\\
\ \\
Step 1: Construction of $\g_1$.\\
We first solve the problem for auxiliary function $\g_1$ with source term $S_W$ as
\begin{align}
\left\{
\begin{array}{l}\displaystyle
\va\dfrac{\p\g_1 }{\p\eta}+G_1\bigg(\vb^2\dfrac{\p\g_1}{\p\va}-\va\vb\dfrac{\p\g_1}{\p\vb}\bigg)
+G_2\bigg(\vc^2\dfrac{\p\g_1}{\p\va}-\va\vc\dfrac{\p\g_1}{\p\vc}\bigg)+\ll[\g_1 ]
=\ss_W,\\\rule{0ex}{1.5em}
\g_1 (0,\vvv)=h (\vvv)\ \ \text{for}\ \
\va>0,\\\rule{0ex}{1.5em}
\displaystyle\g_1 (L,\vvv)=\g_1 (L,\rr[\vvv]).
\end{array}
\right.
\end{align}
Applying Lemma \ref{Milne lemma 6}, we know $\g_1$ is well-posed.\\
\ \\
Step 2: Construction of $\g_2$.\\
There is no way to apply Lemma \ref{Milne lemma 6} to $S_Q$ part, so we resort to explicit formula and analyze it in the following two steps. First, we try to find an auxiliary function $\g_2$ such that
\begin{align}
\va\frac{\p \g_2}{\p\eta}+G_1\bigg(\vb^2\dfrac{\p\g_2}{\p\va}-\va\vb\dfrac{\p\g_2}{\p\vb}\bigg)
+G_2\bigg(\vc^2\dfrac{\p\g_2}{\p\va}-\va\vc\dfrac{\p\g_2}{\p\vc}\bigg)+S_Q\in\nk^{\perp},
\end{align}
which further means
\begin{align}\label{mtt 23}
\int_{\r^3}\ee_j\Bigg(\va\frac{\p \g_2}{\p\eta}+G_1\bigg(\vb^2\dfrac{\p\g_2}{\p\va}-\va\vb\dfrac{\p\g_2}{\p\vb}\bigg)
+G_2\bigg(\vc^2\dfrac{\p\g_2}{\p\va}-\va\vc\dfrac{\p\g_2}{\p\vc}\bigg)+S_Q\Bigg)\ud{\vvv}=0.
\end{align}
for $j=0,1,2,3,4$. Denote
\begin{align}
S_Q=\sum_{k=0}^4S_{Q,k}\ee_k.
\end{align}
We make an ansatz that
\begin{align}
\g_2:=\m^{\frac{1}{2}}\bigg(A(\eta)\va+B_1(\eta)+B_2(\eta)\va\vb+B_3(\eta)\va\vc+C(\eta)\va\abs{\vvv}^2\bigg).
\end{align}
Hence, we can directly compute
\begin{align}
\frac{\p\g_2}{\p\va}&=-\va\g_2+\m^{\frac{1}{2}}\bigg(A+B_2\vb+B_3\vc+C\abs{\vvv}^2+2C\va^2\bigg),\\
\frac{\p\g_2}{\p\vb}&=-\vb\g_2+\m^{\frac{1}{2}}\bigg(B_2\va+2C\va\vb\bigg),\\
\frac{\p\g_2}{\p\vc}&=-\vc\g_2+\m^{\frac{1}{2}}\bigg(B_3\va+2C\va\vc\bigg),
\end{align}
and further
\begin{align}
\vb^2\dfrac{\p\g_2}{\p\va}-\va\vb\dfrac{\p\g_2}{\p\vb}&=\m^{\frac{1}{2}}\bigg(A\vb^2+B_2\vb(\vb^2-\va^2)+B_3\vb^2\vc+C\vb^2\abs{\vvv}^2\bigg),\\
\vc^2\dfrac{\p\g_2}{\p\va}-\va\vc\dfrac{\p\g_2}{\p\vc}&=\m^{\frac{1}{2}}\bigg(A\vc^2+B_2\vb\vc^2+B_3\vc(\vc^2-\va^2)+C\vc^2\abs{\vvv}^2\bigg).
\end{align}
Also, note the Gaussian integral
\begin{align}
\int_{\r^3}\m\ud\vvv=1,\quad\int_{\r^3}\abs{\vvv}^2\m\ud\vvv=3,\quad\int_{\r^3}\abs{\vvv}^4\m\ud\vvv=15,\quad\int_{\r^3}\abs{\vvv}^6\m\ud\vvv=105.
\end{align}
Plugging this ansatz into the equation \eqref{mtt 23}, we obtain a system of linear ordinary differential equations
\begin{align}
\\
\frac{\ud}{\ud\eta}\left(\begin{array}{c}A+5C\\\rule{0ex}{1.5em}
B_1\\\rule{0ex}{1.5em}
B_2\\\rule{0ex}{1.5em}
B_3\\\rule{0ex}{1.5em}
A+10C\end{array}\right)
+\left(\begin{array}{ccccc}G_1+G_2&0&0&0&5G_1+5G_2\\\rule{0ex}{1.5em}
0&0&0&0&0\\\rule{0ex}{1.5em}
0&0&2G_1+G_2&0&0\\\rule{0ex}{1.5em}
0&0&0&G_1+2G_2&0\\\rule{0ex}{1.5em}
G_1+G_2&0&0&0&10G_1+10G_2\end{array}\right)
\left(\begin{array}{c}A\\B_1\\B_2\\B_3\\C\end{array}\right)
=-\left(\begin{array}{c}S_{Q,0}\\S_{Q,1}\\S_{Q,2}\\S_{Q,3}\\S_{Q,4}\end{array}\right).\no
\end{align}
It is easy to check that all five variables $A, B_1,B_2,B_3,C$ are well-defined as long as $S_{Q,k}$ decays exponentially (by solving them explicitly). Furthermore, $\g_2$ decays exponentially to $g_2(L)=0$ as long as the the boundary data are taken properly.\\
\ \\
Step 3: Construction of $\g_3$.\\
Let
\begin{align}
\bar\ss:=\va\frac{\p \g_2}{\p\eta}+G_1\bigg(\vb^2\dfrac{\p\g_2}{\p\va}-\va\vb\dfrac{\p\g_2}{\p\vb}\bigg)
+G_2\bigg(\vc^2\dfrac{\p\g_2}{\p\va}-\va\vc\dfrac{\p\g_2}{\p\vc}\bigg)+\ll[\g_2]+S_Q.
\end{align}
We know $\bar\ss\in\nk^{\perp}$ due to analysis in Step 2 and $\ll[\g_2]\in\nk^{\perp}$. Then we may define an auxiliary function $\g_3$ as the solution of the equation
\begin{align}
\left\{
\begin{array}{l}\displaystyle
\va\dfrac{\p\g_3 }{\p\eta}+G_1\bigg(\vb^2\dfrac{\p\g_3}{\p\va}-\va\vb\dfrac{\p\g_3}{\p\vb}\bigg)
+G_2\bigg(\vc^2\dfrac{\p\g_3}{\p\va}-\va\vc\dfrac{\p\g_3}{\p\vc}\bigg)+\ll[\g_3 ]
=\bar\ss,\\\rule{0ex}{1.5em}
\g_3 (0,\vvv)=-\g_2(0,\vvv)\ \ \text{for}\ \
\va>0,\\\rule{0ex}{1.5em}
\displaystyle\g_3 (L,\vvv)=\g_3 (L,\rr[\vvv]).
\end{array}
\right.
\end{align}
Applying Lemma \ref{Milne lemma 6}, we know $\g_3$ is well-posed.\\
\ \\
Step 4: Construction of $\g_4$.\\
We may directly verify that $\g_4=\g_2+\g_3$ satisfies the equation
\begin{align}
\left\{
\begin{array}{l}\displaystyle
\va\dfrac{\p\g_4 }{\p\eta}+G_1\bigg(\vb^2\dfrac{\p\g_4}{\p\va}-\va\vb\dfrac{\p\g_4}{\p\vb}\bigg)
+G_2\bigg(\vc^2\dfrac{\p\g_4}{\p\va}-\va\vc\dfrac{\p\g_4}{\p\vc}\bigg)+\ll[\g_4 ]
=\ss_Q,\\\rule{0ex}{1.5em}
\g_4 (0,\vvv)=0\ \ \text{for}\ \
\va>0,\\\rule{0ex}{1.5em}
\displaystyle\g_4 (L,\vvv)=\g_4 (L,\rr[\vvv]).
\end{array}
\right.
\end{align}
\ \\
In summary, by superposition, we know $\g=\g_1+\g_4$ satisfies the equation \eqref{Milne} and is well-posed.
\end{proof}

\subsubsection{$L^2$ Boundedness}

Then we turn to the construction of $\tilde h$ and the well-posedness of the equation \eqref{Milne transform}.
\begin{theorem}[$L^2$ well-posedness of $\gg$]\label{Milne theorem 1}
Assume \eqref{Milne bound} holds. Then there exists $\tilde h\in\nk$ such that there exists a unique solution
$\gg(\eta,\vvv)$ to the $\e$-Milne problem with geometric correction \eqref{Milne transform}
satisfying
\begin{align}\label{mtt 67}
\tnnm{\gg}\ls 1.
\end{align}
\end{theorem}
\begin{proof}
Given $h$ and $\ss$, Lemma \ref{Milne lemma 6} tells us that the equation \eqref{Milne} for $\g$ is well-posed and $\g_L$ is well-defined. By a similar argument, we know for any $\tilde h$, $\gg$ must also be well-posed. Hence, our main concern here is to delicately choose $\tilde h$ such that $\gg_L=0$, and then Lemma \ref{Milne lemma 6} implies that \eqref{mtt 67} holds.\\
\ \\
Step 1:\\
Let $\tilde\g=\g-\gg$, which satisfies the equation
\begin{align}\label{mtt 70}
\left\{
\begin{array}{l}\displaystyle
\va\frac{\p \tilde\g}{\p\eta}+G_1\bigg(\vb^2\dfrac{\p\tilde\g}{\p\va}-\va\vb\dfrac{\p\tilde\g}{\p\vb}\bigg)
+G_2\bigg(\vc^2\dfrac{\p\tilde\g}{\p\va}-\va\vc\dfrac{\p\tilde\g}{\p\vc}\bigg)+\ll[\tilde\g]
=0,\\\rule{0ex}{1.5em}
\tilde\g(0,\vvv)=\tilde h(\vvv)\ \ \text{for}\ \ \va>0,\\\rule{0ex}{1.5em}
\tilde\g(L,\vvv)=\tilde\g(L,\rr[\vvv]).
\end{array}
\right.
\end{align}
In order for $\gg_L=0$, we must choose proper $\tilde h$ such that
\begin{align}
\tilde\g_{L}(\vvv)=\g_{L}(\vvv)={q}_{0,L}\ee_0+{q}_{1,L}\ee_1+{q}_{2,L}\ee_2+{q}_{3,L}\ee_3+{q}_{4,L}\ee_4,
\end{align}
where $\tilde\g_L$ is defined as in Lemma \ref{Milne lemma 6}.

In other words, $\g$ and $\tilde\g$ may have different in-flow boundary ($h$ or $\tilde h$) and source terms ($S$ or $0$), but they share the same $\tilde\g_{L}(\vvv)=\g_{L}(\vvv)\in\nk$.\\
\ \\
Step 2:\\
Note that
\begin{align}
\tilde h(\vvv):=\tilde D_0\ee_0+\tilde D_1\ee_1+\tilde
D_2\ee_2+\tilde D_3\ee_3+\tilde D_4\ee_4.
\end{align}
Hence, in \eqref{mtt 70}, we actually need to build a mapping between $\tilde h\in\nk$ and $\tilde\g_L\in\nk$. We can take $\tilde D_1=q_{1,L}=0$. Then we consider the endomorphism $\mathcal{M}$ in a four-dimensional space $\tilde\nk=\text{span}\{\ee_0,\ee_2,\ee_3,\ee_4\}$ defined as $\mathcal{M}:\tilde h\rt \mathcal{M}[\tilde h]=\tilde g_{L}$. Therefore, we only need to study the matrix of $\mathcal{M}$ at the basis $\{\ee_0,\ee_2,\ee_3,\ee_4\}$. It suffices to show that $\mathcal{M}$ is invertible.\\
\ \\
Step 3:\\
It is easy to check when $\tilde h=\ee_0$ and $\tilde h=\ee_4$, $\mathcal{M}$ is an identity mapping, i.e.
\begin{align}
\mathcal{M}[\ee_0]=\ee_0,\quad
\mathcal{M}[\ee_4]=\ee_4.
\end{align}
The main obstacle is when $\tilde h=\ee_2,\ee_3$. Actually, $\mathcal{M}[\ee_2]$ is almost $\ee_2$, so we only need to estimate the difference. For $\tilde h=\ee_2$ in \eqref{mtt 70}, define $\tilde\g'=\tilde\g-\ee_2$. Then $\tilde\g'$ satisfies the equation
\begin{align}\label{mtt 69}
\left\{
\begin{array}{l}\displaystyle
\va\frac{\p \tilde\g'}{\p\eta}+G_1(\eta)\bigg(\vb^2\dfrac{\p \tilde\g'}{\p\va}-\va\vb\dfrac{\p \tilde\g'}{\p\vb}\bigg)+G_2(\eta)\bigg(\vc^2\dfrac{\p \tilde\g'}{\p\va}-\va\vc\dfrac{\p \tilde\g'}{\p\vc}\bigg)+\ll[\tilde\g']
=G(\eta)\m^{\frac{1}{2}}\va\vb,\\\rule{0ex}{1.5em}
\tilde\g'(0,\vvv)=0\ \ \text{for}\ \ \va>0,\\\rule{0ex}{1.5em}
\tilde\g'(L,\vvv)=\tilde\g'(L,\rr[\vvv]).
\end{array}
\right.
\end{align}
Here we cannot directly apply Lemma \ref{Milne lemma 1} to Lemma \ref{Milne lemma 6} with $S=G(\eta)\m^{\frac{1}{2}}\va\vb\in\nk^{\perp}$ and $h=0$, since $G(\eta)\m^{\frac{1}{2}}\va\vb$ does not decay exponentially. At best, we only have $\lnnm{S}{\vrh,\vth}\ls\e$ and have to modify the proof accordingly. In Lemma \ref{Milne lemma 1}, we can show that
\begin{align}\label{mtt 68}
\tnnm{\sn w}\ls \e^{\frac{3}{4}}.
\end{align}
Lemma \ref{Milne lemma 3} remains the same. Lemma \ref{Milne lemma 4} does not hold any more, so we need to use the smallness of $S$ and \eqref{mtt 68} in proving Lemma \ref{Milne lemma 5} instead of exponential decay. We focus on the derivation of $q_L$. Here the estimates of $D,E,F,\theta$ remains the same. Then we have
\begin{align}
\abs{q_{j,L}}\ls \tnnm{\sn w}+\int_0^L\tnm{\sn w(\eta)}\ud\eta+\int_0^L\tnm{S(\eta)}\ud\eta\ls \e^{\frac{3}{4}}L^{\frac{1}{2}}\ls\e^{\frac{1}{2}}.
\end{align}
In other other, the limit $\tilde q'_L$ to \eqref{mtt 69} is at the order $\e^{\frac{1}{2}}$ and is very small, i.e.
\begin{align}
\mathcal{M}[\ee_2]=\ee_2+\tilde q'_L\sim \ee_2+\e^{\frac{1}{2}}\ee_j.
\end{align}
A similar argument can justify $\tilde\g''=\tilde\g-\ee_3$ case, i.e.
\begin{align}
\mathcal{M}[\ee_3]=\ee_3+\tilde q''_L\sim \ee_3+\e^{\frac{1}{2}}\ee_j.
\end{align}
\ \\
Step 4: \\
In summary, we know the matrix of $\mathcal{M}$ is just a small perturbation of identity matrix
\begin{align}
\mathcal{M}\left[\begin{array}{c}\ee_0\\\ee_2\\\ee_3\\\ee_4\end{array}\right]=\left(
\begin{array}{cccc}
1&0&0&0\\
0&1+\tilde{q}'_{2,L}&\tilde{q}''_{2,L}&0\\
0&\tilde{q}'_{3,L}&1+\tilde{q}''_{3,L}&0\\
0&0&0&1
\end{array}
\right)\left[\begin{array}{c}\ee_0\\\ee_2\\\ee_3\\\ee_4\end{array}\right].
\end{align}
Here $\tilde q'_{k,L}$ and $\tilde q''_{k,L}$ are defined as in Step 3 and are of order $\e^{\frac{1}{2}}$. For $\e$ sufficiently small, this matrix is invertible, which means $\mathcal{M}$ is bijective. Therefore, we can always find $\tilde h$ such that
$\tilde\g_{L}=\g_{L}$, which is desired.
\end{proof}

\subsubsection{$L^2$ Decay}

\begin{theorem}[$L^2$ decay]\label{Milne theorem 3}
Assume \eqref{Milne bound} holds. Then there exists $0<K_0<K$ such that the solution $\g(\eta,\vvv)$ to \eqref{Milne transform} satisfying
\begin{align}
\tnnm{\ue^{K_0\eta}\gg}\ls1.
\end{align}
\end{theorem}
\begin{proof}
We decompose $\gg=w+q$ with $\gg_L=q_L=0$. Lemma \ref{Milne lemma 4} already justifies the decay of $w$
\begin{align}
\tnnm{\ue^{K_0\eta}w}\ls 1.
\end{align}
Hence, we focus on $q$ decay. Here, we use the same notation as in the proof of Lemma \ref{Milne lemma 2} and Lemma \ref{Milne lemma 5}.
Recall \eqref{mtt 56}. $\beta_L=q_L=0$ implies
\begin{align}\label{mtt 71}
\theta=-\int_0^{L}\exp\bigg(W_1(y)B^{(1)}A^{-1}+W_2(y)B^{(2)}A^{-1}\bigg)Z(y)\ud{y}
\end{align}
Inserting \eqref{mtt 71} into \eqref{mtt 50}, we obtain
\begin{align}
\beta(\eta)&=-F(\eta)-\int_{\eta}^L\exp\bigg(\Big(W_1(y)-W_1(\eta)\Big)B^{(1)}A^{-1}+\Big(W_2(y)-W_2(\eta)\Big)B^{(2)}A^{-1}\bigg)Z(y)\ud{y}
\end{align}
Note that
\begin{align}
Z=D+E-(G_1B^{(1)}+G_2B^{(2)})A^{-1}F,
\end{align}
where $D,E,F$ are defined in \eqref{mtt 57}, and due to Remark \ref{Milne power},
\begin{align}
\tnnm{\ue^{K_0'\eta}q}^2\ls\tnnm{\ue^{K_0'\eta}\beta}^2\ls \int_{0}^L\ue^{2K_0'\eta}F^2(\eta)\ud\eta+\int_0^L\ue^{2K_0'\eta}\bigg(\int_{\eta}^LZ(y)\ud y\bigg)^2\ud\eta.
\end{align}
Then the proof is similar to that of Lemma \ref{Milne lemma 5}, so we omit it here. Here, we take $K_0'\leq \dfrac{K_0}{2}$.
\end{proof}

\begin{remark}
In \eqref{Milne}, $\g-\g_L$ does not necessarily decay exponentially. This is the main reason we have to introduce Theorem \ref{Milne theorem 1} to design the boundary data such that $\gg_L=0$.
\end{remark}

\subsection{$L^{\infty}$ Estimates}\label{mtt section 01}

\subsubsection{Characteristic Formulation}

We rewrite \eqref{Milne} as the following $\e$-transport problem for $\g(\eta,\vvv)$
\begin{align}\label{transport}
\left\{
\begin{array}{l}\displaystyle
\va\frac{\p \g}{\p\eta}+G_1(\eta)\bigg(\vb^2\dfrac{\p
\g}{\p\va}-\va\vb\dfrac{\p \g}{\p\vb}\bigg)+G_2(\eta)\bigg(\vc^2\dfrac{\p
\g}{\p\va}-\va\vc\dfrac{\p \g}{\p\vc}\bigg)+\nu\g
=Q(\eta,\vvv),\\\rule{0ex}{1.5em} \g(0,\vvv)=h(\vvv)\ \
\text{for}\ \ \va>0,\\\rule{0ex}{1.5em}
\g(L,\vvv)=\g(L,R[\vvv]),
\end{array}
\right.
\end{align}
Here $Q:=K[\g]+\ss$.

Define the characteristics starting from
$\Big(\eta(0),\va(0),\vb(0),\vc(0)\Big)$ as $\Big(\eta(s),\va(s),\vb(s),\vc(s)\Big)$ for some $s\in\r$
satisfying
\begin{align}
\frac{\ud{\eta}}{\ud{s}}=\va,\ \
\frac{\ud{\va}}{\ud{s}}=G_1(\eta)\vb^2+G_2(\eta)\vc^2,\ \
\frac{\ud{\vb}}{\ud{s}}=-G_1(\eta)\va\vb,\ \
\frac{\ud{\vc}}{\ud{s}}=-G_2(\eta)\va\vc,
\end{align}
which leads to
\begin{align}\label{mtt 81}
\va^2(s)+\vb^2(s)+\vc^2(s):=E_1,\ \
\vb(s)\ue^{-W_1(\eta(s))}:=E_2,\ \
\vc(s)\ue^{-W_2(\eta(s))}:=E_3,
\end{align}
where the energy $E_i$ are constants depending on the starting
point.

Therefore, along the characteristics, $\va^2+\vb^2+\vc^2$, $\vb\ue^{-W_1(\eta)}$ and $\vc\ue^{-W_2(\eta)}$ are conserved quantities and the equation \eqref{transport} can be
rewritten as
\begin{align}
\frac{\ud\g}{\ud s}+\nu\g=Q,
\end{align}
or equivalently,
\begin{align}\label{mtt 72}
\va\frac{\ud \g}{\ud\eta}+\nu\g=&Q.
\end{align}

Let
\begin{align}
\vb'(\eta,\vvv;\eta'):=\vb\ue^{W_1(\eta')-W_1(\eta)},\ \ \vc'(\eta,\vvv;\eta'):=\vc\ue^{W_2(\eta')-W_2(\eta)}.
\end{align}
On the characteristics, we should always have $E_1\geq\vb'^2+\vc'^2$. Define
\begin{align}\label{mtt 83}
\va'(\eta,\vvv;\eta'):=&\sqrt{E_1-\vb'^2(\eta,\vvv;\eta')-\vc'^2(\eta,\vvv;\eta')},\\
\vvv'(\eta,\vvv;\eta'):=&\Big(\va'(\eta,\vvv;\eta'),\vb'(\eta,\vvv;\eta'),\vc'(\eta,\vvv;\eta')\Big),\\
\rr[\vvv'(\eta,\vvv;\eta')]:=&\Big(-\va'(\eta,\vvv;\eta'),\vb'(\eta,\vvv;\eta'),\vc'(\eta,\vvv;\eta')\Big).
\end{align}
Basically, this means $(\eta,\va,\vb,\vc)$ and $(\eta',\va',\vb',\vc')$, $(\eta',-\va',\vb',\vc')$ are on the same characteristics. In particular, this implies $\va'\geq0$.

We can rewrite the solution to the equation \eqref{transport} along the characteristics using \eqref{mtt 72} as
\begin{align}
\g(\eta,\vvv)=\k[h](\eta,\vvv)+\t[Q](\eta,\vvv),
\end{align}
where the operators $\k$ and $\t$ are defined as follows:\\
\ \\
Case I: $\va>0$:\\
The characteristics directly tracks back to the in-flow boundary $\eta=0$ and $\va>0$, i.e.
\begin{align}\label{mtt 26}
\k[h](\eta,\vvv):=&h\Big(\vvv'(\eta,\vvv; 0)\Big)\exp(-H_{\eta,0}),\\
\t[Q](\eta,\vvv):=&\int_0^{\eta}\frac{Q\Big(\eta',\vvv'(\eta,\vvv;\eta')\Big)}{\va'(\eta,\vvv;\eta')}\exp(-H_{\eta,\eta'})\ud{\eta'}.
\end{align}
Here
\begin{align}\label{mtt 73}
H_{\eta,\eta'}:=&\int_{\eta'}^{\eta}\frac{\nu\Big(\vvv'(\eta,\vvv;y)\Big)}{\va'(\eta,\vvv;y)}\ud{y}.
\end{align}
\ \\
Case II: $\va<0$ and $\va^2+\vb^2+\vc^2\geq \vb'^2(\eta,\vvv;L)+\vc'^2(\eta,\vvv;L)$:\\
The characteristics first goes a bit farther to the boundary $\eta=L$, then gets reflected and tracks back to the in-flow boundary, i.e.
\begin{align}\label{mtt 27}
\k[h](\eta,\vvv):=&h\Big(\vvv'(\eta,\vvv; 0)\Big)\exp(-H_{L,0}-\rr[H_{L,\eta}]),\\
\t[Q](\eta,\vvv):=&\bigg(\int_0^{L}\frac{Q\Big(\eta',\vvv'(\eta,\vvv;\eta')\Big)}{\va'(\eta,\vvv;\eta')}
\exp(-H_{L,\eta'}-\rr[H_{L,\eta}])\ud{\eta'}\\
&+\int_{\eta}^{L}\frac{Q\Big(\eta',\rr[\vvv'(\eta,\vvv;\eta')]\Big)}{\va'(\eta,\vvv;\eta')}\exp(\rr[H_{\eta,\eta'}])\ud{\eta'}\bigg).\no
\end{align}
Here
\begin{align}
\rr[H_{\eta,\eta'}]:=&\int_{\eta'}^{\eta}\frac{\nu\Big(\rr[\vvv'(\eta,\vvv;y)]\Big)}{\va'(\eta,\vvv;y)}\ud{y}.
\end{align}
Actually, since $\nu$ only depends on $\abs{\vvv}$, we must have $H_{\eta,\eta'}=\rr[H_{\eta,\eta'}]$. This distinction is purely for clarification and does not play a role in the estimates.\\
\ \\
Case III: $\va<0$ and $\va^2+\vb^2+\vc^2\leq \vb'^2(\eta,\vvv;L)+\vc'^2(\eta,\vvv;L)$\\
The characteristics reaches the line $\va=0$ before reaching the boundary $\eta=L$, and then directly tracks back to the in-flow boundary, i.e.
\begin{align}\label{mtt 28}
\k[h](\eta,\vvv):=&h\Big(\vvv'(\eta,\vvv; 0)\Big)\exp(-H_{\eta^+,0}-\rr[H_{\eta^+,\eta}]),\\
\t[Q](\eta,\vvv):=&\bigg(\int_0^{\eta^+}\frac{Q\Big(\eta',\vvv'(\eta,\vvv;\eta')\Big)}{\va'(\eta,\vvv;\eta')}
\exp(-H_{\eta^+,\eta'}-\rr[H_{\eta^+,\eta}])\ud{\eta'}\\
&+\int_{\eta}^{\eta^+}\frac{Q\Big(\eta',\rr[\vvv'(\eta,\vvv;\eta')]\Big)}{\va'(\eta,\vvv;\eta')}\exp(\rr[H_{\eta,\eta'}])\ud{\eta'}\bigg).\no
\end{align}
Here $\eta^{+}(\eta,\vvv)$ is defined by
\begin{align}\label{mtt 84}
E_1(\eta,\vvv)=\vb'^2(\eta,\vvv;\eta^+)+\vc'^2(\eta,\vvv;\eta^+).
\end{align}
locates the position that the characteristics touch $\va=0$ line, i.e. $(\eta^+,0,\vb',\vc')$ is on the same characteristics as $(\eta,\va,\vb,\vc)$.

In order to achieve the estimate of $\g$, we need to control $\k[h]$ and $\t[Q]$. Since we always assume that $(\eta,\vvv)$ and $(\eta',\vvv')$ are on the same characteristics, in the following, we will simply write $\vvv'(\eta')$ or even $\vvv'$ instead of $\vvv'(\eta,\vvv;\eta')$ when there is no confusion.

\subsubsection{$L^{\infty}$ Boundedness}

We first prove some important lemmas characterizing the operators $\k$ and $\t$.

\begin{lemma}[estimate of boundary term]\label{Milne lemma 8}
There is a positive $0<\beta<\nu_0$ such that for any $\vth\geq0$ and
$\varrho\geq0$,
\begin{align}
\lnnm{\ue^{\beta\eta}\k[h]}{\vth,\varrho}\ls\lnmh{h}.
\end{align}
\end{lemma}
\begin{proof}
Consider \eqref{mtt 73}, we know
\begin{align}
\frac{\nu(\vvv')}{\va'}\geq\nu_0,\ \
\frac{\nu(\rr[\vvv'])}{\va'}\geq\nu_0.
\end{align}
It follows that
\begin{align}
\exp(-H_{\eta,0})\leq&\ue^{-\beta\eta},\\
\exp(-H_{L,0}-\rr[H_{L,\eta}])\leq&\ue^{-\beta\eta},\\
\exp(-H_{\eta^+,0}-\rr[H_{\eta^+,\eta}])\leq&\ue^{-\beta\eta}.
\end{align}
Then our results are obvious.
\end{proof}

\begin{lemma}[estimate of bulk term]\label{Milne lemma 9}
For any $\vth\geq0$, $\varrho\geq0$ and
$0\leq\beta\leq\dfrac{\nu_0}{2}$, there is a constant $C$ such that
\begin{align}
\lnnm{\ue^{\beta\eta}\t[Q]}{\vth,\varrho}\ls\lnnm{\nu^{-1}\ue^{\beta\eta}Q}{\vth,\varrho}.
\end{align}
\end{lemma}
\begin{proof}
For $\va>0$ case, we have
\begin{align}
\beta(\eta-\eta')-H_{\eta,\eta'}\leq\beta(\eta-\eta')-\frac{\nu_0(\eta-\eta')}{2}-\frac{H_{\eta,\eta'}}{2}\leq-\frac{H_{\eta,\eta'}}{2}.
\end{align}
It is natural that
\begin{align}
\int_0^{\eta}\frac{\nu\Big(\vvv'(\eta')\Big)}{\va'(\eta')}
\exp\Big(\beta(\eta-\eta')-H_{\eta,\eta'}\Big)\ud{\eta'}\leq\int_0^{\infty}\exp\bigg(-\frac{z}{2}\bigg)\ud{z}=2,
\end{align}
for $z=H_{\eta,\eta'}$. Notice that $\abs{\vvv}=\abs{\vvv'}$. Then we estimate
\begin{align}
\abs{\bvv\ue^{\beta\eta}\t[Q]}\leq& \ue^{\beta\eta}\int_0^{\eta}\bvv\frac{\abs{Q\Big(\eta',\vvv'(\eta')\Big)}}{\va'(\eta')}\exp(-H_{\eta,\eta'})\ud{\eta'}\\
\leq&\lnnm{\nu^{-1}\ue^{\beta\eta}Q}{\vth,\varrho}\int_0^{\eta}\frac{\nu\Big(\vvv'(\eta')\Big)}{\va'(\eta')}
\exp\Big(\beta(\eta-\eta')-H_{\eta,\eta'}\Big)\ud{\eta'}\no\\
\ls&\lnnm{\nu^{-1}\ue^{\beta\eta}Q}{\vth,\varrho}.\no
\end{align}
The $\va<0$ case can be proved in a similar fashion, so we omit it here.
\end{proof}

\begin{lemma}[further estimate of bulk term]\label{Milne lemma 10}
For any $\d>0$, an integer $\vth>3$ and $\varrho\geq0$, there is a
constant $C(\d)$ such that
\begin{align}
\ltnm{\t[Q]}{\varrho}\leq C(\d)\tnnm{\nu^{-\frac{1}{2}}Q}+\d\lnnm{Q}{\vth,\varrho}.
\end{align}
\end{lemma}
\begin{proof}
In the following, we will repeatedly use the fact that $\abs{\vvv}=\abs{\vvv'}$.\\
\ \\
Case I: For $\va>0$, $\t[Q]$ is defined in \eqref{mtt 26}.
We need to estimate
\begin{align}
\int_{\r^3}\ue^{2\varrho\abs{\vvv}^2}\bigg(\int_0^{\eta}\frac{Q\Big(\eta',\vvv(\eta')\Big)}{\va'(\eta')}\exp(-H_{\eta,\eta'})\ud{\eta'}
\bigg)^2\ud{\vvv}.
\end{align}
Assume $m>0$ is sufficiently small, $M>0$ is sufficiently large and $\sigma>0$ is sufficiently small, which will be determined in the following. We can split the above integral into four parts
\begin{align}
I:=I_1+I_2+I_3+I_4.
\end{align}
In the following, we use $\chi_i$ for $i=1,2,3,4$ to represent the indicator function of each type.\\
\ \\
Case I - Type I: $\chi_1$: $M\leq\va'(\eta')$ or $M\leq\vb'(\eta')$ or $M\leq\vc'(\eta')$.\\
We have
\begin{align}
\abs{\vvv(\eta')}+1\ls\nu\Big(\vvv(\eta')\Big).
\end{align}
Then for $\vth>3$, since $\abs{\vvv}$ is conserved along the characteristics,  we have
\begin{align}
I_1\ls&\lnnm{Q}{\vth,\varrho}^2\int_{\r^3}\chi_1\bigg(\int_0^{\eta}\frac{1}{\br{\vvv'}^{\vth}}\frac{\exp(-H_{\eta,\eta'})}{\va'(\eta')}\ud{\eta'}
\bigg)^2\ud{\vvv}\\
\ls&\frac{1}{M^{\vth}}\lnnm{Q}{\vth,\varrho}^2\int_{\r^3}\frac{1}{\br{\vvv}^{\vth}}\bigg(\int_0^{\eta}\frac{\exp(-H_{\eta,\eta'})}{\va'(\eta')}\ud{\eta'}
\bigg)^2\ud{\vvv}\no\\
\ls&\frac{1}{M^{\vth}}\lnnm{Q}{\vth,\varrho}^2\int_{\r^3}\frac{1}{\br{\vvv}^{\vth}}\ud{\vvv}\no\\
\ls&\frac{1}{M^{\vth}}\lnnm{Q}{\vth,\varrho}^2,\no
\end{align}
since for $y=H_{\eta,\eta'}$,
\begin{align}
\abs{\int_0^{\eta}\frac{\exp(-H_{\eta,\eta'})}{\va'(\eta')}\ud{\eta'}}\ls& \abs{\int_0^{\eta}\frac{\nu\Big(\vvv'(\eta')\Big)\exp(-H_{\eta,\eta'})}{\va'(\eta')}\ud{\eta'}}
\ls\int_0^{\infty}\ue^{-y}\ud{y}=1.
\end{align}
\ \\
Case I - Type II: $\chi_2$: $\va\geq\sigma$, $m\leq\va'(\eta')\leq M$ and $\vb'(\eta')\leq M$ and $\vc'(\eta')\leq M$.\\
Since along the characteristics, $\abs{\vvv}^2$ can be bounded by
$3M^2$ and the integral domain for $\vvv$ is finite, by Cauchy's inequality, we have
\begin{align}
I_2\ls&\ue^{6\varrho M^2}\int_{\r^3}\bigg(\int_0^{\eta}\frac{Q^2}{\nu}\Big(\eta',\vvv'(\eta')\Big)\ud{\eta'}\bigg)
\bigg(\int_0^{\eta}\frac{\nu\Big(\vvv'(\eta')\Big)\exp(-2H_{\eta,\eta'})}{\va'^2(\eta')}\ud{\eta'}\bigg)\ud{\vvv}\\
\ls&\frac{\ue^{6\varrho M^2}}{m}\int_{\r^3}\bigg(\int_0^{\eta}\frac{Q^2}{\nu}\Big(\eta',\vvv'(\eta')\Big)\ud{\eta'}\bigg)
\bigg(\int_0^{\eta}\frac{\nu\Big(\vvv'(\eta')\Big)\exp(-2H_{\eta,\eta'})}{\va'(\eta')}\ud{\eta'}\bigg)\ud{\vvv}\no\\
\ls&\frac{\ue^{6\varrho M^2}}{m}\bigg(\int_{\r^3}\int_0^{\eta}\frac{Q^2}{\nu}\Big(\eta',\vvv'(\eta')\Big)\ud{\eta'}\ud{\vvv}\bigg)\no\\
\ls&\frac{M\ue^{6\varrho M^2}}{m\sigma}\bigg(\int_{\r^3}\int_0^{\eta}\frac{Q}{\nu}\Big(\eta',\vvv'\Big)\ud{\eta'}\ud{\vvv'}\bigg)\no\\
\ls&\frac{M\ue^{6\varrho M^2}}{m\sigma}\tnnm{\nu^{-\frac{1}{2}}Q}^2,\no
\end{align}
where for $y=H_{\eta,\eta'}$,
\begin{align}
\int_0^{\eta}\frac{\nu\Big(\vvv'(\eta')\Big)}{\va'(\eta')}\exp(-2H_{\eta,\eta'})\ud{\eta'}
\ud{\vvv}\ls&\int_0^{\infty}\ue^{-2y}\ud{y}=\half,
\end{align}
and the Jacobian
\begin{align}
\abs{\dfrac{\ud{\vvv}}{\ud{\vvv'}}}=\abs{\dfrac{R_1-\e\eta}{R_1-\e\eta'}\dfrac{R_2-\e\eta}{R_2-\e\eta'}\dfrac{\va'}{\va}}\ls\dfrac{\va'}{\va}\ls \dfrac{M}{\sigma}.
\end{align}
\ \\
Case I - Type III: $\chi_3$: $\va\geq\sigma$, $0\leq\va'(\eta')\leq m$ and $\vb'(\eta')\leq M$ and $\vc'(\eta')\leq M$.\\
We can directly verify the fact that
\begin{align}
0\leq\va\leq\va'(\eta'),
\end{align}
for $\eta'\leq\eta$. Then we know the integral of $\va$ is always in a small domain.
We have for $y=H_{\eta,\eta'}$,
\begin{align}
I_3\ls&\ue^{6\varrho M^2}\lnnm{Q}{\vth,\varrho}^2\int_{\r^3}\frac{\chi_3}{\br{\vvv}^{\vth}}\bigg(\int_0^{\eta}\frac{\exp(-H_{\eta,\eta'})}{\va'(\eta')}\ud{\eta'}\bigg)^2\ud{\vvv}\\
\ls&\ue^{6\varrho M^2}\lnnm{Q}{\vth,\varrho}^2\int_{\r^3}\frac{\chi_3}{\br{\vvv}^{\vth}}\bigg(\int_0^{\infty}\ue^{-y}\ud{y}
\bigg)^2\ud{\vvv}\no\\
\ls&\ue^{6\varrho M^2}\lnnm{Q}{\vth,\varrho}^2\int_{\r^3}\frac{\chi_3}{\br{\vvv}^{\vth}}\ud{\vvv}\no\\
\ls&\ue^{6\varrho M^2}m\lnnm{Q}{\vth,\varrho}^2.\no
\end{align}
\ \\
Case I - Type IV: $\chi_4$: $\va\leq\sigma$, $\va'(\eta')\leq M$ and $\vb'(\eta')\leq M$ and $\vc'(\eta')\leq M$.\\
Similar to Case I - Type III , we know the integral of $\va$ is always in a small domain.
We have for $y=H_{\eta,\eta'}$,
\begin{align}
I_4\ls&\ue^{6\varrho M^2}\lnnm{Q}{\vth,\varrho}^2\int_{\r^3}\frac{\chi_4}{\br{\vvv}^{\vth}}\bigg(\int_0^{\eta}\frac{\exp(-H_{\eta,\eta'})}{\va'(\eta')}\ud{\eta'}
\bigg)^2\ud{\vvv}\\
\ls&\ue^{6\varrho M^2}\lnnm{Q}{\vth,\varrho}^2\int_{\r^3}\frac{\chi_4}{\br{\vvv}^{\vth}}\bigg(\int_0^{\infty}\ue^{-y}\ud{y}
\bigg)^2\ud{\vvv}\no\\
\ls&\ue^{6\varrho M^2}\lnnm{Q}{\vth,\varrho}^2\int_{\r^3}\frac{\chi_4}{\br{\vvv}^{\vth}}\ud{\vvv}\no\\
\ls&\ue^{6\varrho M^2}\sigma\lnnm{Q}{\vth,\varrho}^2.\no
\end{align}
\ \\
Collecting all four types, we have
\begin{align}
I\ls\frac{M\ue^{6\varrho
M^2}}{m\sigma}\tnnm{\nu^{-\frac{1}{2}}Q}^2+\bigg(\frac{1}{M^{\vth}}+\ue^{6\varrho M^2}(m+\sigma)\bigg)\lnnm{Q}{\vth,\varrho}^2.
\end{align}
Taking $M$ sufficiently large, $m<<\ue^{-6\varrho M^2}$ and
$\sigma<<\ue^{-6\varrho M^2}$ sufficiently small, we obtain
the desired result.\\
\ \\
Case II: \\
For $\va<0$ and $\va^2+\vb^2+\vc^2\geq \vb'^2(L)+\vc'^2(L)$, $\t[Q]$ is defined in \eqref{mtt 27}.
We first estimate
\begin{align}
\int_{\r^3}\ue^{2\varrho\abs{\vvv}^2}\bigg(\int_{\eta}^{L}\frac{Q\Big(\eta',\rr[\vvv(\eta')]\Big)}{\va'(\eta')}
\exp(\rr[H_{\eta,\eta'}])\ud{\eta'}\bigg)^2\ud{\vvv}.
\end{align}
We can split the above integral into four parts:
\begin{align}
II:=II_1+II_2+II_3+II_4.
\end{align}
\ \\
Case II - Type I: $\chi_1$: $M\leq\va'(\eta')$ or $M\leq\vb'(\eta')$ or $M\leq\vc'(\eta')$.\\
Similar to Case I - Type I, we have
\begin{align}
II_1\ls&\lnnm{Q}{\vth,\varrho}^2\int_{\r^3}\chi_1\bigg(\int_{\eta}^{L}\frac{1}{\br{\vvv'}^{\vth}}\frac{\exp(\rr[H_{\eta,\eta'}])}{\va'(\eta')}\ud{\eta'}
\bigg)^2\ud{\vvv}\\
\ls&\frac{1}{M^{\vth}}\lnnm{Q}{\vth,\varrho}^2\int_{\r^3}\frac{1}{\br{\vvv}^{\vth}}\bigg(\int_{\eta}^{L}\frac{\exp(\rr[H_{\eta,\eta'}])}{\va'(\eta')}\ud{\eta'}
\bigg)^2\ud{\vvv}\no\\
\ls&\frac{1}{M^{\vth}}\lnnm{Q}{\vth,\varrho}^2\int_{\r^3}\frac{1}{\br{\vvv}^{\vth}}\ud{\vvv}\no\\
\ls&\frac{1}{M^{\vth}}\lnnm{Q}{\vth,\varrho}^2,\no
\end{align}
since for $y=H_{\eta,\eta'}$,
\begin{align}
\abs{\int_{\eta}^{L}\frac{\exp(\rr[H_{\eta,\eta'}])}{\va'(\eta')}\ud{\eta'}}\ls& \abs{\int_{\eta}^{L}\frac{\nu\Big(\vvv'(\eta')\Big)\exp(\rr[H_{\eta,\eta'}])}{\va'(\eta')}\ud{\eta'}}
\ls\int^0_{-\infty}\ue^{y}\ud{y}=1.
\end{align}
\ \\
Case II - Type II: $\chi_2$: $m\leq\va'(\eta')\leq M$ and $\vb'(\eta')\leq M$ and $\vc'(\eta')\leq M$.\\
We can directly verify the fact that
\begin{align}
0\leq\va'(\eta')\leq\abs{\va},
\end{align}
for $\eta'\geq\eta$. Similar to Case I - Type II, by Cauchy's inequality, we have
\begin{align}
II_2\ls&\ue^{6\varrho M^2}\int_{\r^3}\bigg(\int_0^{\eta}\frac{Q^2}{\nu}\Big(\eta',\vvv(\eta')\Big)\ud{\eta'}\bigg)
\bigg(\int_{\eta}^{L}\frac{\nu\Big(\vvv'(\eta')\Big)\exp(2\rr[H_{\eta,\eta'}])}{\va'^2(\eta')}\ud{\eta'}\bigg)\ud{\vvv}\\
\ls&\frac{\ue^{6\varrho M^2}}{m}\int_{\r^3}\bigg(\int_{\eta}^{L}\frac{Q^2}{\nu}\Big(\eta',\vvv(\eta')\Big)\ud{\eta'}\bigg)
\bigg(\int_{\eta}^{L}\frac{\nu\Big(\vvv'(\eta')\Big)\exp(2\rr[H_{\eta,\eta'}])}{\va'(\eta')}\ud{\eta'}\bigg)\ud{\vvv}\no\\
\ls&\frac{\ue^{6\varrho M^2}}{m}\bigg(\int_{\r^3}\int_{\eta}^{L}\frac{Q^2}{\nu}\Big(\eta',\vvv(\eta')\Big)\ud{\eta'}\ud{\vvv}\bigg)\no\\
\ls&\frac{\ue^{6\varrho M^2}}{m}\bigg(\int_{\r^3}\int_{\eta}^{L}\frac{Q^2}{\nu}\Big(\eta',\vvv(\eta')\Big)\ud{\eta'}\ud{\vvv'}\bigg)\no\\
\ls&\frac{\ue^{6\varrho M^2}}{m}\tnnm{\nu^{-\frac{1}{2}}Q}^2,\no
\end{align}
where for $y=H_{\eta,\eta'}$,
\begin{align}
\int_0^{\eta}\frac{\nu\Big(\vvv'(\eta')\Big)}{\va'(\eta')}\exp(-2H_{\eta,\eta'})\ud{\eta'}
\ud{\vvv}\ls&\int_0^{\infty}\ue^{-2y}\ud{y}=\half,
\end{align}
and the Jacobian
\begin{align}
\abs{\dfrac{\ud{\vvv}}{\ud{\vvv'}}}=\abs{\dfrac{R_1-\e\eta}{R_1-\e\eta'}\dfrac{R_2-\e\eta}{R_2-\e\eta'}\dfrac{\va'}{\va}}\ls\abs{\dfrac{\va'}{\va}}\ls 1.
\end{align}
\ \\
Case II - Type III: $\chi_3$: $0\leq\va'(\eta')\leq m$, $\vb'(\eta')\leq M$, $\vc'(\eta')\leq M$ and $\eta'-\eta\geq\sigma$.\\
We know
\begin{align}
H_{\eta,\eta'}\leq-\frac{\sigma}{m}.
\end{align}
Then after substitution $y=H_{\eta,\eta'}$, the integral is not from zero, but from
$-\dfrac{\sigma}{m}$. In detail, we have
\begin{align}
II_3\ls&\ue^{6\varrho M^2}\lnnm{Q}{\vth,\varrho}^2\int_{\r^3}\frac{\chi_3}{\br{\vvv}^{\vth}}\bigg(\int_{\eta}^{L}\frac{\exp(\rr[H_{\eta,\eta'}])}{\va'(\eta')}\ud{\eta'}
\bigg)^2\ud{\vvv}\\
\ls&\ue^{6\varrho M^2}\lnnm{Q}{\vth,\varrho}^2\int_{\r^3}\frac{\chi_3}{\br{\vvv}^{\vth}}\bigg(\int_{\eta}^{L}\frac{\nu\Big(\vvv'(\eta')\Big)
\exp(\rr[H_{\eta,\eta'}])}{\va'(\eta')}\ud{\eta'}
\bigg)^2\ud{\vvv}\no\\
\ls&\ue^{6\varrho M^2}\lnnm{Q}{\vth,\varrho}^2\int_{\r^3}\frac{\chi_3}{\br{\vvv}^{\vth}}\bigg(\int^{-\frac{\sigma}{m}}_{-\infty}\ue^{y}\ud{y}\bigg)^2\ud{\vvv}\no\\
\ls&\ue^{6\varrho M^2}\ue^{-\frac{2\sigma}{m}}\lnnm{Q}{\vth,\varrho}^2\int_{\r^3}\frac{\chi_3}{\br{\vvv}^{\vth}}\ud{\vvv}\no\\
\ls&\ue^{6\varrho M^2}\ue^{-\frac{2\sigma}{m}}\lnnm{Q}{\vth,\varrho}^2.\no
\end{align}
\ \\
Case II - Type IV: $\chi_4$: $0\leq\va'(\eta')\leq m$, $\vb'(\eta')\leq M$, $\vc'(\eta')\leq M$ and $\eta'-\eta\leq\sigma$.\\
For $\eta'\leq\eta$ and $\eta'-\eta\leq\sigma$, we have
\begin{align}
\va=&\sqrt{\va'^2(\eta')+\vb'^2(\eta')+\vc'^2(\eta')-\vb^2-\vc^2}\\
=&\sqrt{\va'^2(\eta')+\vb'^2(\eta')+\vc'^2(\eta')-\vb'^2(\eta')\ue^{2W_1(\eta)-2W_1(\eta')}-\vc'^2(\eta')\ue^{2W_2(\eta)-2W_2(\eta')}}\no\\
=&\sqrt{\va'^2(\eta')+\vb'^2(\eta')+\vc'^2(\eta')
-\vb'^2(\eta')\left(\frac{R_1-\e\eta'}{R_1-\e\eta}\right)^2-\vc'^2(\eta')\left(\frac{R_2-\e\eta'}{R_2-\e\eta}\right)^2}\no\\
\ls&\sqrt{\va'^2(\eta')+2(R_1+R_2) M^2\e(\eta'-\eta)}\no\\
\ls&
\sqrt{m^2+\e M^2\sigma}\leq C(m+M\sqrt{\e\sigma}).\no
\end{align}
Therefore, the integral domain for $\va$ is very small. We have the estimate for $y=H_{\eta,\eta'}$
\begin{align}
II_4\ls&\ue^{6\varrho M^2}\lnnm{Q}{\vth,\varrho}^2\int_{\r^3}\frac{\chi_4}{\br{\vvv}^{\vth}}\bigg(\int_{\eta}^{L}\frac{\exp(\rr[H_{\eta,\eta'}])}{\va'(\eta')}\ud{\eta'}
\bigg)^2\ud{\vvv}\\
\ls&\ue^{6\varrho M^2}\lnnm{Q}{\vth,\varrho}^2\int_{\r^3}\frac{\chi_4}{\br{\vvv}^{\vth}}\bigg(\int^0_{-\infty}\ue^{y}\ud{y}
\bigg)^2\ud{\vvv}\no\\
\ls&\ue^{6\varrho M^2}\lnnm{Q}{\vth,\varrho}^2\int_{\r^3}\frac{\chi_4}{\br{\vvv}^{\vth}}\ud{\vvv}\no\\
\ls&\ue^{6\varrho M^2}(m+\sqrt{\e\sigma})\lnnm{Q}{\vth,\varrho}^2.\no
\end{align}
\ \\
Collecting all four types, we have
\begin{align}
II\leq C\frac{\ue^{6\varrho
M^2}}{m}\tnnm{\nu^{-\frac{1}{2}}Q}^2+C\bigg(\frac{1}{M^{\vth}}+\ue^{6\varrho M^2}\Big(\ue^{-\frac{2\sigma}{m}}+m+\sqrt{\e\sigma}\Big)\bigg)\lnnm{Q}{\vth,\varrho}^2.
\end{align}
Taking $M$ sufficiently large, $\sigma<<\ue^{-6\varrho M^2}$ sufficiently small and
$m<<\min\{\sigma,\ue^{-4\varrho M^2}\}$ sufficiently small, we obtain
the desired result.

Note that we have the decomposition
\begin{align}
&\int_0^{L}\frac{Q\Big(\eta',\vvv(\eta')\Big)}{\va'(\eta')}
\exp(-H_{L,\eta'}-\rr[H_{L,\eta}])\ud{\eta'}\\
=&\int_0^{\eta}\frac{Q\Big(\eta',\vvv(\eta')\Big)}{\va'(\eta')}
\exp(-H_{L,\eta'}-\rr[H_{L,\eta}])\ud{\eta'}+\int_{\eta}^{L}\frac{Q\Big(\eta',\vvv(\eta')\Big)}{\va'(\eta')}
\exp(-H_{L,\eta'}-\rr[H_{L,\eta}])\ud{\eta'}.\no
\end{align}
Then this term can actually be bounded using the techniques in Case I and Case II.\\
\ \\
Case III: \\
For $\va<0$ and $\va^2+\vb^2+\vc^2\leq \vb'^2(L)+\vc'^2(L)$, $\t[Q]$ is defined in \eqref{mtt 28}.
This is a combination of Case I and Case II, so we omit the proof here.
\end{proof}

\begin{lemma}[$L^{\infty}$ estimate of $\g-\g_L$]\label{Milne lemma 11}
Assume \eqref{Milne bound} holds. Then the solution $\g(\eta,\vvv)$ to the $\e$-Milne problem with geometric correction \eqref{Milne}
satisfies for $\varrho\geq0$ and $\vth>3$,
\begin{align}
\lnnm{\g-\g_{L}}{\vth,\varrho}\ls 1+\tnnm{\g-\g_{L}}.
\end{align}
\end{lemma}
\begin{proof}
Define $u=\g-\g_{L}$. Then $u$ satisfies the equation
\begin{align}
\left\{
\begin{array}{l}\displaystyle
\va\frac{\p u}{\p\eta}+G_1\bigg(\vb^2\dfrac{\p
u}{\p\va}-\va\vb\dfrac{\p u}{\p\vb}\bigg)+G_2\bigg(\vc^2\dfrac{\p
u}{\p\va}-\va\vc\dfrac{\p u}{\p\vc}\bigg)+\nu u-K[u]
=\tilde S,\\\rule{0ex}{1.5em} u(0,\vvv)=p(\vvv)\ \
\text{for}\ \ \va>0,\\\rule{0ex}{1.5em}
u(L,\vvv)=u(L,\rr[\vvv]),
\end{array}
\right.
\end{align}
where
\begin{align}
\tilde S&=S+\g_{2,L}G_1\m^{\frac{1}{2}}\va\vb+\g_{3,L}G_2\m^{\frac{1}{2}}\va\vc,\\
p&=h(\vvv)-\g_{L}(\vvv).
\end{align}
The using the operators $\k$ and $\t$ defined in \eqref{mtt 26}, \eqref{mtt 27} and \eqref{mtt 28}, we can write $u=\k[p]+\t\Big[K[u]\Big]+\t[\tilde S]$.
Based on Lemma \ref{Milne lemma 8}, Lemma \ref{Milne lemma 9} and Lemma \ref{Milne lemma 10}, we have
\begin{align}\label{mtt 77}
\ltnm{u}{\varrho}\ls&\ltnm{\k[p]}{\varrho}+\ltnm{\t\Big[K[u]\Big]}{\varrho}+\ltnm{\t[\tilde S]}{\varrho}\\
\ls&\lnnmv{\k[p]}+\ltnm{\t\Big[K[u]\Big]}{\varrho}+\lnnmv{\t[\tilde S]}\no\\
\ls&\lnmh{p}+
C(\d)\tnnm{\nu^{-\frac{1}{2}}K[u]}+\d\lnnm{K[u]}{\vth,\varrho}+\lnnm{\nu^{-1}\tilde S}{\vth,\varrho}\no\\
\ls&
\lnmh{p}+
C(\d)\tnnm{u}+\d\lnnm{K[u]}{\vth,\varrho}+\lnnm{\nu^{-1}\tilde S}{\vth,\varrho},\no
\end{align}
where \cite[Section 3.5]{Glassey1996} verifies
\begin{align}
\tnnm{\nu^{-\frac{1}{2}}K[u]}\ls \tnnm{K[u]}\ls&\tnnm{u}.
\end{align}
In \cite[Lemma 3.3.1]{Glassey1996}, it is shown that
\begin{align}
\lnnm{K[u]}{\vth,\varrho}\leq&\lnnm{u}{\vth-1,\varrho},\label{mtt 74}\\
\lnnm{K[u]}{0,\varrho}\leq&\ltnm{u}{\varrho}.\label{mtt 75}
\end{align}
Since $u=\k[p]+\t\Big[K[u]\Big]+\t[\tilde S]$, by Lemma \ref{Milne lemma 8} and Lemma \ref{Milne lemma 9}, using \eqref{mtt 74}, we can estimate
\begin{align}\label{mtt 76}
\lnnm{u}{\vth,\varrho}\ls&\lnnm{\k[p]}{\vth,\varrho}+\lnnm{\t\Big[K[u]\Big]}{\vth,\varrho}+\lnnm{\t[\tilde S]}{\vth,\varrho}\\
\ls&\lnmh{p}+\lnnm{K[u]}{\vth,\varrho}+\lnnm{\nu^{-1}\tilde S}{\vth,\varrho}\no\\
\ls&\lnmh{p}+\lnnm{u}{\vth-1,\varrho}+\lnnm{\nu^{-1}\tilde S}{\vth,\varrho}.\no
\end{align}
Note that now we have $\lnnm{u}{\vth-1,\varrho}$. Hence, it is available to redo the above estimate \eqref{mtt 76} for $\vth-1$. This procedure can keep going until the zeroth order. Then using \eqref{mtt 75} and \eqref{mtt 77}, we obtain
\begin{align}\label{mtt 78}
\lnnm{u}{\vth,\varrho}\ls& \lnmh{p}+\lnnm{K[u]}{\vth,\varrho}+\lnnm{\nu^{-1}\tilde S}{\vth,\varrho}\\
\ls&\lnmh{p}+\lnnm{K[u]}{0,\varrho}+\lnnm{\nu^{-1}\tilde S}{\vth,\varrho}\no\\
\ls&\lnmh{p}+\ltnm{u}{\varrho}+\lnnm{\nu^{-1}\tilde S}{\vth,\varrho}\no\\
\ls&\lnmh{p}+C(\d)\tnnm{u}+\d\lnnm{K[u]}{\vth,\varrho}+\lnnm{\nu^{-1}\tilde S}{\vth,\varrho}.\no
\end{align}
Therefore, for $\d$ sufficiently small, absorbing $\d\lnnm{K[u]}{\vth,\varrho}$ into the right-hand side of the first inequality of \eqref{mtt 78}, we get
\begin{align}\label{mtt 79}
\lnnm{K[u]}{\vth,\varrho}\ls\lnmh{p}+\tnnm{u}+\lnnm{\nu^{-1}\tilde S}{\vth,\varrho}.
\end{align}
Therefore, inserting \eqref{mtt 79} into the first inequality of \eqref{mtt 78}, we have
\begin{align}
\lnnm{u}{\vth,\varrho}\ls\lnmh{p}+\tnnm{u}+\lnnm{\nu^{-1}\tilde S}{\vth,\varrho}.
\end{align}
In particular, using Lemma \ref{Milne lemma 6} and \eqref{Milne bound}, we know
\begin{align}
\lnmh{p}&\ls \lnmh{h}+\lnm{\g_L}{\vth,\varrho}\ls1,\\
\lnnm{\nu^{-1}\tilde S}{\vth,\varrho}&\ls  \lnnm{\nu^{-1}S}{\vth,\varrho}+\e\Big(\abs{\g_{2,L}}+\abs{\g_{3,L}}\Big)\ls 1.
\end{align}
Then our result naturally follows.
\end{proof}

\begin{lemma}[$L^{\infty}$ well-posedness of $\g$]\label{Milne lemma 12}
Assume \eqref{Milne bound} holds. Then there exists a unique solution $\g(\eta,\vvv)$ to the $\e$-Milne problem with geometric correction
\eqref{Milne} satisfying for $\varrho\geq0$ and $\vth>3$,
\begin{align}
\lnnm{\g-\g_{L}}{\vth,\varrho}\ls 1.
\end{align}
\end{lemma}
\begin{proof}
Based on Lemma \ref{Milne lemma 6} and Lemma \ref{Milne lemma 11}, this is obvious.
\end{proof}

\begin{theorem}[$L^{\infty}$ well-posedness of $\gg$]\label{Milne theorem 2}
Assume \eqref{Milne bound} holds. Then there exists a unique solution $\gg(\eta,\vvv)$ to the $\e$-Milne problem with geometric correction \eqref{Milne transform} satisfying for
$\varrho\geq0$ and integer $\vth\geq3$,
\begin{align}
\lnnm{\gg}{\vth,\varrho}\ls 1.
\end{align}
\end{theorem}
\begin{proof}
Based on Theorem \ref{Milne theorem 1} and Lemma \ref{Milne lemma 12}, this is obvious.
\end{proof}

\subsubsection{$L^{\infty}$ Decay}

Now we intend to show the $L^{\infty}$ decay of solution to the equation \eqref{Milne transform}. Define $U=\ue^{K_0\eta}\gg$. Then $U$ satisfies the equation
\begin{align}\label{exponential}
\left\{
\begin{array}{l}\displaystyle
\va\frac{\p U}{\p\eta}+G_1\bigg(\vb^2\dfrac{\p
U}{\p\va}-\va\vb\dfrac{\p U}{\p\vb}\bigg)+G_2\bigg(\vc^2\dfrac{\p
U}{\p\va}-\va\vc\dfrac{\p U}{\p\vc}\bigg)+\ll[U]
=K_0\va U+\ue^{K_0\eta}\ss,\\\rule{0ex}{1.5em}
U(0,\vvv)=h(\vvv)-\tilde h(\vvv)\ \ \text{for}\ \
\va>0,\\\rule{0ex}{1.5em}
U(L,\vvv)=U(L,\rr[\vvv])
\end{array}
\right.
\end{align}

\begin{lemma}[decay estimate]\label{Milne lemma 14}
Assume \eqref{Milne bound} holds. Then there exists $0<K_0<K$ such that for
$\varrho\geq0$ and $\vth>3$
\begin{align}
\lnnm{U}{\vth,\varrho}\ls 1+\tnnm{U}.
\end{align}
\end{lemma}
\begin{proof}
Since $U=\k[p]+\t\Big[K[U]\Big]+\t[K_0\va U]+\t[\ue^{K_0\eta}\ss]$, similar to the proof of Lemma \ref{Milne lemma 11}, we have
\begin{align}
\lnnm{U}{\vth,\varrho}\ls& \lnmh{p}+\tnnm{U}+\lnnm{\nu^{-1}K_0\va U}{\vth,\varrho}+\lnnm{\nu^{-1}\ue^{K_0\eta}\ss}{\vth,\varrho}\\
\ls&\lnmh{p}+\tnnm{U}+K_0\lnnm{U}{\vth,\varrho}+\lnnm{\nu^{-1}\ue^{K_0\eta}\ss}{\vth,\varrho}.\no
\end{align}
When $K_0>0$ is sufficiently small, we may absorb $K_0\lnnm{U}{\vth,\varrho}$ into the left-hand side to obtain
\begin{align}
\lnnm{U}{\vth,\varrho}\ls\lnmh{p}+\tnnm{U}+\lnnm{\nu^{-1}\ue^{K_0\eta}\ss}{\vth,\varrho}.
\end{align}
Then \eqref{Milne bound} leads to the desired result.
\end{proof}

\begin{theorem}[$L^{\infty}$ decay]\label{Milne theorem 4}
Assume \eqref{Milne bound} holds. Then there exists $0<K_0<K$ such that the solution $\g(\eta,\vvv)$ to \eqref{Milne transform} satisfying for
$\varrho\geq0$ and $\vth>3$,
\begin{align}
\lnnm{\ue^{K_0\eta}\gg}{\vth,\varrho}\ls 1.
\end{align}
\end{theorem}
\begin{proof}
Based on Theorem \ref{Milne theorem 3} and Lemma \ref{Milne lemma 14}, this is obvious.
\end{proof}

\newpage

\section{Regularity}

Now we begin to study the regularity of the solution $\gg$ to \eqref{Milne transform}.
In this section, denote the boundary data $p=h-\tilde h$. Besides \eqref{Milne bound}, throughout this section, we further require the regularity bound that for $\varrho\geq0$ and $\vth>3$
\begin{align}\label{Regularity bound}
\lnmh{\nabla_{\vvv}p}\ls 1,\quad\lnnm{\ue^{K\eta}\p_{\eta}\ss}{\vth,\varrho}+\lnnm{\ue^{K\eta}\nabla_{\vvv}\ss}{\vth,\varrho}\ls 1.
\end{align}

\subsection{Preliminaries}

\subsubsection{Weight Function}

Define a weight function
\begin{align}\label{weight}
\zeta(\eta;\vvv)=\left(\left(\va^2+\vb^2+\vc^2\right)-\left(\frac{R_1-\e\eta}{R_1}\right)^2\vb^2-\left(\frac{R_2-\e\eta}{R_2}\right)^2\vc^2\right)^{\frac{1}{2}}.
\end{align}
It is easy to see that the closer a point $(\eta;\va,\vb,\vc)$ is to the grazing set $(\eta;\va,\vb,\vc)=(0;0,\vb,\vc)$, the smaller $\zeta$ is. In particular, at the grazing set, $\zeta(0;0,\vb,\vc)=0$.
\begin{lemma}[weight function in $\e$ Milne problem]\label{weight lemma}
Let $\zeta$ be defined as in \eqref{weight}. We have
\begin{align}
\va\dfrac{\p\zeta}{\p\eta}-\dfrac{\e}{R_1-\e\eta}\bigg(\vb^2\dfrac{\p\zeta}{\p\va}-\va\vb\dfrac{\p\zeta}{\p\vb}\bigg)
-\dfrac{\e}{R_2-\e\eta}\bigg(\vc^2\dfrac{\p\zeta}{\p\va}-\va\vc\dfrac{\p\zeta}{\p\vc}\bigg)=0.
\end{align}
\end{lemma}
\begin{proof}
We may directly compute
\begin{align}
&\frac{\p\zeta}{\p\eta}=\frac{1}{\zeta}\left(\frac{R_1-\e\eta}{R_1^2}\e\vb^2+\frac{R_2-\e\eta}{R_2^2}\e\vc^2\right),\\
&\frac{\p\zeta}{\p\va}=\frac{1}{\zeta}\va,\ \
\frac{\p\zeta}{\p\vb}=\frac{1}{\zeta}\left(\vb-\left(\frac{R_1-\e\eta}{R_1}\right)^2\vb\right),\ \
\frac{\p\zeta}{\p\vc}=\frac{1}{\zeta}\left(\vc-\left(\frac{R_2-\e\eta}{R_2}\right)^2\vc\right).
\end{align}
Then we know
\begin{align}
&\va\dfrac{\p\zeta}{\p\eta}-\dfrac{\e}{R_1-\e\eta}\bigg(\vb^2\dfrac{\p\zeta}{\p\va}-\va\vb\dfrac{\p\zeta}{\p\vb}\bigg)
-\dfrac{\e}{R_2-\e\eta}\bigg(\vc^2\dfrac{\p\zeta}{\p\va}-\va\vc\dfrac{\p\zeta}{\p\vc}\bigg)\\
=&\frac{1}{\zeta}\Bigg(\frac{R_1-\e\eta}{R_1^2}\e\va\vb^2+\frac{R_2-\e\eta}{R_2^2}\e\va\vc^2\no\\
&-\frac{\e}{R_1-\e\eta}\left(\va\vb^2-\va\vb^2+\va\vb^2\left(\frac{R_1-\e\eta}{R_1}\right)^2\right)
-\frac{\e}{R_2-\e\eta}\left(\va\vc^2-\va\vc^2+\va\vc^2\left(\frac{R_2-\e\eta}{R_2}\right)^2\right)\Bigg)=0.\no
\end{align}
\end{proof}

\begin{remark}\label{weight remark}
With this lemma in hand, we know for any function $f$, we can put the weight $\zeta$ inside the $\e$-Milne operator, i.e.
\begin{align}
&\va\dfrac{\p(\zeta f)}{\p\eta}-\dfrac{\e}{R_1-\e\eta}\bigg(\vb^2\dfrac{\p(\zeta f)}{\p\va}-\va\vb\dfrac{\p(\zeta f)}{\p\vb}\bigg)
-\dfrac{\e}{R_2-\e\eta}\bigg(\vc^2\dfrac{\p(\zeta f)}{\p\va}-\va\vc\dfrac{\p(\zeta f)}{\p\vc}\bigg)\\
=&\zeta\Bigg(\va\dfrac{\p f}{\p\eta}-\dfrac{\e}{R_1-\e\eta}\bigg(\vb^2\dfrac{\p f}{\p\va}-\va\vb\dfrac{\p f}{\p\vb}\bigg)
-\dfrac{\e}{R_2-\e\eta}\bigg(\vc^2\dfrac{\p f}{\p\va}-\va\vc\dfrac{\p f}{\p\vc}\bigg)\Bigg).\no
\end{align}
\end{remark}

\subsubsection{Important Lemmas}

\begin{lemma}\label{Regularity lemma 0}
For Boltzmann collision frequency $\nu=\nu(\abs{\vvv})$, we have
\begin{align}
\abs{\frac{\ud\nu}{\ud\abs{\vvv}}}\ls 1.
\end{align}
\end{lemma}
\begin{proof}
Based on \cite[Chapter 3]{Glassey1996}, we know
\begin{align}
\nu(\abs{\vvv})\sim \left(2\abs{\vvv}+\frac{1}{\abs{\vvv}}\right)\int_0^{\abs{\vvv}}\ue^{-z^2}\ud z+\ue^{-\abs{\vvv}^2}.
\end{align}
Then for $\abs{\vvv}\geq1$, we have
\begin{align}
\abs{\frac{\ud\nu}{\ud\abs{\vvv}}}&\ls\left(1+\frac{1}{\abs{\vvv}^2}\right)\int_0^{\abs{\vvv}}\ue^{-z^2}\ud z+\left(\abs{\vvv}+\frac{1}{\abs{\vvv}}\right)\ue^{-\abs{\vvv}^2}\ls 1.
\end{align}
For $\abs{\vvv}\leq1$, the key difficulty is the fractional term. Taylor expansion implies
\begin{align}
\frac{1}{\abs{\vvv}}\int_0^{\abs{\vvv}}\ue^{-z^2}\ud z\sim \frac{1}{\abs{\vvv}}\sum_{k=0}^{\infty}\frac{(-1)^{k}}{(2k+1)k!}\abs{\vvv}^{2k+1}=\sum_{k=0}^{\infty}\frac{(-1)^{k}}{(2k+1)k!}\abs{\vvv}^{2k}\ls 1.
\end{align}
Hence, the desired result naturally follows.
\end{proof}

%

\begin{lemma}\label{Regularity lemma 2}
Let $0\leq\varrho< \dfrac{1}{4}$ and $\vth\geq0$. Then for $\d>0$ sufficiently small and any $\vvv\in\r^3$,
\begin{align}
\int_{\r^3}\ue^{\d\abs{\vuu-\vvv}^2}\frac{1}{\abs{\vuu}}\abs{k(\vuu,\vvv)}
\frac{\br{\vvv}^{\vth}\ue^{\vrh\abs{\vvv}^2}}{\br{\vuu}^{\vth}\ue^{\vrh\abs{\vuu}^2}}\ud{\vuu}
\ls 1.
\end{align}
\end{lemma}
\begin{proof}
This proof is mainly motivated by \cite[Lemma 3]{Guo2010}. Notice that
\begin{align}\label{rtt 01}
\abs{\frac{\br{\vvv}^{\vth}\ue^{\vrh\abs{\vvv}^2}}{\br{\vuu}^{\vth}\ue^{\vrh\abs{\vuu}^2}}}\ls \left(1+\abs{\vuu-\vvv}^2\right)^{\frac{\vth}{2}}\ue^{-\varrho\left(\abs{\vuu}^2-\abs{\vvv}^2\right)}.
\end{align}
Combining Lemma \ref{Regularity lemma 1} and \eqref{rtt 01}, we have
\begin{align}\label{rtt 04}
\abs{k(\vuu,\vvv)}
\frac{\br{\vvv}^{\vth}\ue^{\vrh\abs{\vvv}^2}}{\br{\vuu}^{\vth}\ue^{\vrh\abs{\vuu}^2}}
\ls \left(1+\abs{\vuu-\vvv}^2\right)^{\frac{\vth}{2}}
\left(\abs{\vuu-\vvv}+\frac{1}{\abs{\vuu-\vvv}}\right)\ue^{-\frac{1}{8}\abs{\vuu-\vvv}^2-\frac{1}{8}\frac{\abs{\abs{\vuu}^2-\abs{\vvv}^2}^2}{\abs{\vuu-\vvv}^2}
-\varrho\left(\abs{\vuu}^2-\abs{\vvv}^2\right)}.
\end{align}
We first handle the exponential term in \eqref{rtt 04}. Let $\vsi=\vuu-\vvv$, so $\vuu=\vsi+\vvv$. Then we have
\begin{align}\label{rtt 02}
&-\frac{1}{8}\abs{\vuu-\vvv}^2-\frac{1}{8}\frac{\abs{\abs{\vuu}^2-\abs{\vvv}^2}^2}{\abs{\vuu-\vvv}^2}-\varrho\left(\abs{\vuu}^2-\abs{\vvv}^2\right)\\
=&-\frac{1}{8}\abs{\vsi}^2-\frac{1}{8}\frac{\abs{\abs{\vsi+\vvv}^2-\abs{\vvv}^2}^2}{\abs{\vsi}^2}-\varrho\left(\abs{\vsi+\vvv}^2-\abs{\vvv}^2\right)\no\\
=&-\frac{1}{8}\abs{\vsi}^2-\frac{1}{8}\frac{\abs{\abs{\vsi}^2-2\vsi\cdot\vvv}^2}{\abs{\vsi}^2}-\varrho\left(\abs{\vsi}^2-2\vsi\cdot\vvv\right)\no\\
=&-\frac{1}{4}\abs{\vsi}^2+\frac{1}{2}\vsi\cdot\vvv-\frac{1}{2}\frac{\abs{\vsi\cdot\vvv}^2}{\abs{\vsi}^2}-\varrho\left(\abs{\vsi}^2-2\vsi\cdot\vvv\right)\no\\
=&\left(-\frac{1}{4}-\varrho\right)\abs{\vsi}^2+\left(\frac{1}{2}+2\varrho\right)\vsi\cdot\vvv-\frac{1}{2}\frac{\abs{\vsi\cdot\vvv}^2}{\abs{\vsi}^2}.\no
\end{align}
For $0\leq\varrho\leq \dfrac{1}{4}$, the discriminant
\begin{align}
\Delta=\left(\frac{1}{2}+2\varrho\right)^2+2\left(-\frac{1}{4}-\varrho\right)=4\varrho^2-\frac{1}{4}<0,
\end{align}
so the above quadratic form for $\abs{\vsi}$ and $\dfrac{\vsi\cdot\vvv}{\abs{\vsi}}$ is negative definite. This implies
\begin{align}
-\frac{1}{8}\abs{\vuu-\vvv}^2-\frac{1}{8}\frac{\abs{\abs{\vuu}^2-\abs{\vvv}^2}^2}{\abs{\vuu-\vvv}^2}-\varrho\left(\abs{\vuu}^2-\abs{\vvv}^2\right)\ls -\left(\abs{\vsi}^2+\frac{\abs{\vsi\cdot\vvv}^2}{\abs{\vsi}^2}\right)\ls -\abs{\vuu-\vvv}^2.
\end{align}
In particular, for $\d$ small, the perturbed form is still negative definite, i.e.
\begin{align}\label{rtt 05}
&-\left(\frac{1}{8}-\d\right)\abs{\vuu-\vvv}^2-\frac{1}{8}\frac{\abs{\abs{\vuu}^2-\abs{\vvv}^2}^2}{\abs{\vuu-\vvv}^2}-\varrho\left(\abs{\vuu}^2-\abs{\vvv}^2\right)
\ls -\left(\abs{\vsi}^2+\frac{\abs{\vsi\cdot\vvv}^2}{\abs{\vsi}^2}\right)\ls -\abs{\vuu-\vvv}^2.
\end{align}
Hence, using H\"{o}lder's inequality, \eqref{rtt 04} and \eqref{rtt 05}, we may bound
\begin{align}\label{rtt 08}
&\int_{\r^3}\ue^{\d\abs{\vuu-\vvv}^2}\frac{1}{\abs{\vuu}}\abs{k(\vuu,\vvv)}
\frac{\br{\vvv}^{\vth}\ue^{\vrh\abs{\vvv}^2}}{\br{\vuu}^{\vth}\ue^{\vrh\abs{\vuu}^2}}\ud{\vuu}\\
\ls& \int_{\r^3}\left(1+\abs{\vuu-\vvv}^2\right)^{\frac{\vth}{2}}\frac{1}{\abs{\vuu}}\left(\abs{\vuu-\vvv}+\frac{1}{\abs{\vuu-\vvv}}\right)
\ue^{-\abs{\vuu-\vvv}^2}\ud{\vuu}\no\\
\ls&\int_{\r^3}\frac{1}{\abs{\vuu}}\left(\abs{\vuu-\vvv}+\frac{1}{\abs{\vuu-\vvv}}\right)
\ue^{-\abs{\vuu-\vvv}^2}\ud{\vuu}\no\\
\ls&\bigg(\int_{\r^3}\frac{1}{\abs{\vuu}^2}
\ue^{-\abs{\vuu-\vvv}^2}\ud{\vuu}\bigg)^{\frac{1}{2}}
\bigg(\int_{\r^3}\Big(\abs{\vuu-\vvv}^2+\frac{1}{\abs{\vuu-\vvv}^2}\Big)\ue^{-\abs{\vuu-\vvv}^2}\ud\vuu\bigg)^{\frac{1}{2}}:=I\times II.\no
\end{align}
Here, the second inequality is due to the fact that exponential term decays much faster than polynomial term. Then we need to bound $I$ and $II$ separately. Using spherical coordinates and substitution $\vuu\rt\vsi=\vuu-\vvv$, we have
\begin{align}\label{rtt 06}
I\ls&\bigg(\int_{\abs{\vuu}\leq 1}\frac{1}{\abs{\vuu}^2}
\ue^{-\abs{\vuu-\vvv}^2}\ud{\vuu}\bigg)^{\frac{1}{2}}+\bigg(\int_{\abs{\vuu}\geq 1}\frac{1}{\abs{\vuu}^2}
\ue^{-\abs{\vuu-\vvv}^2}\ud{\vuu}\bigg)^{\frac{1}{2}}\\
\ls&\bigg(\int_{\abs{\vuu}\leq 1}\frac{1}{\abs{\vuu}^2}\ud{\vuu}\bigg)^{\frac{1}{2}}+\bigg(\int_{\abs{\vuu}\geq 1}
\ue^{-\abs{\vuu-\vvv}^2}\ud{\vuu}\bigg)^{\frac{1}{2}}
\ls 1+\bigg(\int_{\r^3}\ue^{-\abs{\vsi}^2}\ud{\vsi}\bigg)^{\frac{1}{2}}\ls 1.\no
\end{align}
Similarly, using spherical coordinates and substitution $\vuu\rt\vsi=\vuu-\vvv$, we have
\begin{align}\label{rtt 07}
II\ls&\bigg(\int_{\r^3}\Big(\abs{\vsi}^2+\frac{1}{\abs{\vsi}^2}\Big)\ue^{-\abs{\vsi}^2}\ud{\vsi}\bigg)^{\frac{1}{2}}\ls 1.
\end{align}
In summary, inserting \eqref{rtt 06} and \eqref{rtt 07} into \eqref{rtt 08}, we obtain the desired result.
\end{proof}

\begin{lemma}\label{Regularity lemma 2'}
Let $0\leq\varrho< \dfrac{1}{4}$ and $\vth\geq0$. We have
\begin{align}
\int_{\r^3}\abs{\nabla_{\vvv}k(\vuu,\vvv)}\frac{\br{\vvv}^{\vth}\ue^{\vrh\abs{\vvv}^2}}{\br{\vuu}^{\vth}\ue^{\vrh\abs{\vuu}^2}}\ud\vuu\ls 1.
\end{align}
\end{lemma}
\begin{proof}
Based on \cite[Chapter 3]{Glassey1996}, for hard-sphere gas, $k=k_1+k_2$, where
\begin{align}
k_1(\vuu,\vvv)&\sim\abs{\vuu-\vvv}\ue^{-\frac{1}{2}\abs{\vuu}^2-\frac{1}{2}\abs{\vvv}^2},\\
k_2(\vuu,\vvv)&\sim\frac{1}{\abs{\vuu-\vvv}}\ue^{-\frac{1}{4}\abs{\vuu-\vvv}^2-\frac{1}{4}\frac{\abs{\abs{\vuu}^2-\abs{\vvv}^2}^2}{\abs{\vuu-\vvv}^2}}.
\end{align}
Following the similar argument as in Lemma \ref{Regularity lemma 1'} and \eqref{rtt 04} in Lemma \ref{Regularity lemma 1}, we have
\begin{align}
\abs{\nabla_{\vvv}k(\vuu,\vvv)}\frac{\br{\vvv}^{\vth}\ue^{\vrh\abs{\vvv}^2}}{\br{\vuu}^{\vth}\ue^{\vrh\abs{\vuu}^2}}
\ls \left(1+\abs{\vuu-\vvv}^2\right)^{\frac{\vth}{2}}
\abs{\nabla_{\vvv}k(\vuu,\vvv)}\ue^{-\varrho\left(\abs{\vuu}^2-\abs{\vvv}^2\right)}
\end{align}
Here, the key is to bound $\abs{\nabla_{\vvv}k(\vuu,\vvv)}$. Substituting $\vuu\rt\vsi=\vuu-\vvv$, we get
\begin{align}
k_1(\vsi,\vvv)&=\abs{\vsi}\ue^{-\abs{\vvv}^2-\vsi\cdot\vvv-\frac{1}{2}\abs{\vsi}^2},\label{rtt 31}\\
k_2(\vsi,\vvv)&=\frac{1}{\abs{\vsi}}\ue^{-\frac{1}{4}\abs{\vsi}^2-\frac{1}{4}\frac{\abs{\abs{\vsi}^2-2\vsi\cdot\vvv}^2}{\abs{\vsi}^2}}.\label{rtt 32}
\end{align}
Then we compute
\begin{align}
\nabla_{\vvv}k_1(\vsi,\vvv)&=\abs{\vsi}\Big(-2\vvv-\vsi\Big)\ue^{-\abs{\vvv}^2-\vsi\cdot\vvv-\frac{1}{2}\abs{\vsi}^2},
\end{align}
which implies
\begin{align}
\abs{\nabla_{\vvv}k_1(\vsi,\vvv)}&\ls \abs{\vsi}^2\ue^{-\abs{\vvv}^2-\vsi\cdot\vvv-\frac{1}{2}\abs{\vsi}^2}+\abs{\vsi}\abs{\vvv}\ue^{-\abs{\vvv}^2-\vsi\cdot\vvv-\frac{1}{2}\abs{\vsi}^2}:=I_1+I_2.
\end{align}
Here, $I_1$ is covered by similar techniques as in the proof of Lemma \ref{Regularity lemma 2}, $I_2$ is covered in Lemma \ref{Regularity lemma 1'}. We obtain
\begin{align}
I_1\ls 1,\quad I_2\ls \frac{\abs{\vvv}}{1+\abs{\vvv}}\ls 1,
\end{align}
which implies
\begin{align}\label{rtt 21}
\int_{\r^3}\nabla_{\vvv}k_1(\vuu,\vvv)\frac{\br{\vvv}^{\vth}\ue^{\vrh\abs{\vvv}^2}}{\br{\vuu}^{\vth}\ue^{\vrh\abs{\vuu}^2}}\ud\vuu\ls 1.
\end{align}
On the other hand, we compute
\begin{align}
\abs{\nabla_{\vvv}k_2(\vsi,\vvv)}&= \frac{1}{\abs{\vsi}}\bigg(\vsi-\frac{2\vsi\cdot\vvv}{\abs{\vsi}^2}\vsi\bigg)\ue^{-\frac{1}{4}\abs{\vsi}^2-\frac{1}{4}\frac{\abs{\abs{\vsi}^2-2\vsi\cdot\vvv}^2}{\abs{\vsi}^2}},
\end{align}
which implies
\begin{align}
\abs{\nabla_{\vvv}k_2(\vsi,\vvv)}&\ls \ue^{-\frac{1}{4}\abs{\vsi}^2-\frac{1}{4}\frac{\abs{\abs{\vsi}^2-2\vsi\cdot\vvv}^2}{\abs{\vsi}^2}}+\frac{\abs{\vvv}}{\abs{\vsi}}
\ue^{-\frac{1}{4}\abs{\vsi}^2-\frac{1}{4}\frac{\abs{\abs{\vsi}^2-2\vsi\cdot\vvv}^2}{\abs{\vsi}^2}}:=II_1+II_2.
\end{align}
Still, $II_1$ is covered by similar techniques as in the proof of Lemma \ref{Regularity lemma 2}, $II_2$ is covered in Lemma \ref{Regularity lemma 1'}. We obtain
\begin{align}
II_1\ls 1,\quad II_2\ls \frac{\abs{\vvv}}{1+\abs{\vvv}}\ls 1,
\end{align}
which implies
\begin{align}\label{rtt 22}
\int_{\r^3}\nabla_{\vvv}k_2(\vuu,\vvv)\frac{\br{\vvv}^{\vth}\ue^{\vrh\abs{\vvv}^2}}{\br{\vuu}^{\vth}\ue^{\vrh\abs{\vuu}^2}}\ud\vuu\ls 1.
\end{align}
Then the desired results follow from \eqref{rtt 21} and \eqref{rtt 22}.
\end{proof}

\begin{lemma}\label{Regularity lemma 2''}
Let $0\leq\varrho< \dfrac{1}{4}$ and $\vth\geq0$. We have
\begin{align}
\int_{\r^3}\abs{\nabla_{\vuu}k(\vuu,\vvv)}\frac{\br{\vvv}^{\vth}\ue^{\vrh\abs{\vvv}^2}}{\br{\vuu}^{\vth}\ue^{\vrh\abs{\vuu}^2}}\ud\vuu\ls \br{\vvv}^2.
\end{align}
\end{lemma}
\begin{proof}
This is very similar to the proof of Lemma \ref{Regularity lemma 2'}. Following a similar argument, we have
\begin{align}
\abs{\nabla_{\vuu}k(\vuu,\vvv)}\frac{\br{\vvv}^{\vth}\ue^{\vrh\abs{\vvv}^2}}{\br{\vuu}^{\vth}\ue^{\vrh\abs{\vuu}^2}}
\ls \left(1+\abs{\vuu-\vvv}^2\right)^{\frac{\vth}{2}}
\abs{\nabla_{\vuu}k(\vuu,\vvv)}\ue^{-\varrho\left(\abs{\vuu}^2-\abs{\vvv}^2\right)}
\end{align}
Here, the key is to bound $\abs{\nabla_{\vuu}k(\vuu,\vvv)}$. Substituting $\vuu\rt \vsi=\vuu-\vvv=(\sigma_{\eta},\sigma_{\phi},\sigma_{\phi})$, we get \eqref{rtt 31} and \eqref{rtt 32}. Also, note that $\nabla_{\vuu}=\nabla_{\vsi}$.
Then we compute
\begin{align}
\nabla_{\vsi}k_1(\vsi,\vvv)&=\abs{\vsi}\Big(-\vvv-\vsi\Big)\ue^{-\abs{\vvv}^2-\vsi\cdot\vvv-\frac{1}{2}\abs{\vsi}^2}
+\frac{\vsi}{\abs{\vsi}}\ue^{-\abs{\vvv}^2-\vsi\cdot\vvv-\frac{1}{2}\abs{\vsi}^2},
\end{align}
which implies
\begin{align}
\abs{\nabla_{\vsi}k_1(\vsi,\vvv)}&\ls \Big(\abs{\vsi}^2+1\Big)\ue^{-\abs{\vvv}^2-\vsi\cdot\vvv-\frac{1}{2}\abs{\vsi}^2}+\abs{\vsi}\abs{\vvv}\ue^{-\abs{\vvv}^2-\vsi\cdot\vvv-\frac{1}{2}\abs{\vsi}^2}:=I_1+I_2.
\end{align}
Here, using similar techniques as in the proof of Lemma \ref{Regularity lemma 2}, we obtain
\begin{align}
I_1\ls 1,\quad I_2\ls \br{\vvv},
\end{align}
which implies
\begin{align}\label{rtt 21'}
\int_{\r^3}\nabla_{\vuu}k_1(\vuu,\vvv)\frac{\br{\vvv}^{\vth}\ue^{\vrh\abs{\vvv}^2}}{\br{\vuu}^{\vth}\ue^{\vrh\abs{\vuu}^2}}\ud\vuu\ls \abs{\vvv}.
\end{align}
On the other hand, we compute
\begin{align}
\abs{\nabla_{\vsi}k_2(\vsi,\vvv)}&= \frac{1}{\abs{\vsi}}\bigg(-\vsi+\vvv-\frac{2\vsi\cdot\vvv}{\abs{\vsi}^2}(\vvv\cdot \mathscr{T})\bigg)\ue^{-\frac{1}{4}\abs{\vsi}^2-\frac{1}{4}\frac{\abs{\abs{\vsi}^2-2\vsi\cdot\vvv}^2}{\abs{\vsi}^2}}
-\frac{\vsi}{\abs{\vsi}^3}\ue^{-\frac{1}{4}\abs{\vsi}^2-\frac{1}{4}\frac{\abs{\abs{\vsi}^2-2\vsi\cdot\vvv}^2}{\abs{\vsi}^2}},
\end{align}
for tensor
\begin{align}
\mathscr{T}:=\frac{1}{\abs{\vsi}^3}\left(\begin{array}{lll}\sigma_{\phi}^2+\sigma_{\psi}^2&-\sigma_{\eta}\sigma_{\phi}&-\sigma_{\eta}\sigma_{\psi}\\
-\sigma_{\eta}\sigma_{\phi}&\sigma_{\eta}^2+\sigma_{\psi}^2&-\sigma_{\phi}\sigma_{\psi}\\
-\sigma_{\eta}\sigma_{\psi}&-\sigma_{\phi}\sigma_{\psi}&\sigma_{\eta}^2+\sigma_{\phi}^2
\end{array}\right),
\end{align}
which implies
\begin{align}
\abs{\nabla_{\vsi}k_2(\vsi,\vvv)}&\ls \bigg(1+\frac{1}{\abs{\vsi}^2}\bigg)\ue^{-\frac{1}{4}\abs{\vsi}^2-\frac{1}{4}\frac{\abs{\abs{\vsi}^2-2\vsi\cdot\vvv}^2}{\abs{\vsi}^2}}
+\bigg(\frac{\abs{\vvv}}{\abs{\vsi}}+\frac{\abs{\vvv}^2}{\abs{\vsi}^2}\bigg)
\ue^{-\frac{1}{4}\abs{\vsi}^2-\frac{1}{4}\frac{\abs{\abs{\vsi}^2-2\vsi\cdot\vvv}^2}{\abs{\vsi}^2}}:=II_1+II_2.
\end{align}
Still, using similar techniques as in the proof of Lemma \ref{Regularity lemma 2}, we obtain
\begin{align}
II_1\ls 1,\quad II_2\ls \br{\vvv}^2,
\end{align}
which implies
\begin{align}\label{rtt 22'}
\int_{\r^3}\nabla_{\vuu}k_2(\vuu,\vvv)\frac{\br{\vvv}^{\vth}\ue^{\vrh\abs{\vvv}^2}}{\br{\vuu}^{\vth}\ue^{\vrh\abs{\vuu}^2}}\ud\vuu\ls \br{\vvv}^2.
\end{align}
Then the desired results follow from \eqref{rtt 21'} and \eqref{rtt 22'}.
\end{proof}

\begin{lemma}\label{Regularity lemma 3}
For any $\vvv\in\r^3$, we have
\begin{align}
\int_{\r^3}\frac{1}{\ze(\eta;\vuu)}\abs{k(\vuu,\vvv)}\frac{\br{\vvv}^{\vth}\ue^{\vrh\abs{\vvv}^2}}{\br{\vuu}^{\vth}\ue^{\vrh\abs{\vuu}^2}}\ud\vuu\ls 1+\abs{\ln(\e\eta)}.
\end{align}
\end{lemma}
\begin{proof}
Based on \eqref{rtt 04} and \eqref{rtt 05}, we know
\begin{align}\label{rtt 09}
\abs{k(\vuu,\vvv)}\frac{\br{\vvv}^{\vth}\ue^{\vrh\abs{\vvv}^2}}{\br{\vuu}^{\vth}\ue^{\vrh\abs{\vuu}^2}}\ls \left(\abs{\vuu-\vvv}+\frac{1}{\abs{\vuu-\vvv}}\right)\ue^{-\d\abs{\vuu-\vvv}^2}.
\end{align}
Based on \eqref{weight}, letting $\vuu=(\ua,\ub,\uc)$, we directly obtain
\begin{align}\label{rtt 10}
\ze(\eta;\vuu)\gs \sqrt{\ua^2+(\e\eta)\ub^2+(\e\eta)\uc^2}.
\end{align}
Hence, using \eqref{rtt 09} and \eqref{rtt 10}, we bound
\begin{align}\label{rtt 19}
\int_{\r^3}\frac{1}{\ze(\eta;\vuu)}\abs{k(\vuu,\vvv)}\ud\vuu&\ls \int_{\r^3}\frac{1}{\sqrt{\ua^2+(\e\eta)\ub^2+(\e\eta)\uc^2}}\abs{\vuu-\vvv}\ue^{-\d\abs{\vuu-\vvv}^2}\ud\vuu\\
&+\int_{\r^3}\frac{1}{\sqrt{\ua^2+(\e\eta)\ub^2+(\e\eta)\uc^2}}\frac{1}{\abs{\vuu-\vvv}}\ue^{-\d\abs{\vuu-\vvv}^2}\ud\vuu:=I+II.\no
\end{align}
We need to estimate $I$ and $II$ separately. Since exponential term decays much faster than polynomial term, we have
\begin{align}\label{rtt 11}
I&\ls \int_{\r^3}\frac{1}{\sqrt{\ua^2+(\e\eta)\ub^2+(\e\eta)\uc^2}}\ue^{-\abs{\vuu-\vvv}^2}\ud\vuu\\
&\ls \int_{\abs{\vuu}\leq 1}\frac{1}{\sqrt{\ua^2+(\e\eta)\ub^2+(\e\eta)\uc^2}}\ue^{-\abs{\vuu-\vvv}^2}\ud\vuu
+\int_{\abs{\vuu}\geq 1}\frac{1}{\sqrt{\ua^2+(\e\eta)\ub^2+(\e\eta)\uc^2}}\ue^{-\abs{\vuu-\vvv}^2}\ud\vuu\no\\
&\ls \int_{\abs{\ub}\leq 1,\abs{\uc}\leq 1}\Bigg(\int_{\abs{\ua}\leq 1}\frac{1}{\sqrt{\ua^2+(\e\eta)\ub^2+(\e\eta)\uc^2}}\ud\ua\Bigg)\ud\ub\ud\uc
+\int_{\abs{\vuu}\geq 1}\ue^{-\abs{\vuu-\vvv}^2}\ud\vuu\no\\
&\ls 1+\int_{\abs{\ub}\leq 1,\abs{\uc}\leq 1}\Bigg(\int_{\abs{\ua}\leq 1}\frac{1}{\sqrt{\ua^2+(\e\eta)\ub^2+(\e\eta)\uc^2}}\ud\ua\Bigg)\ud\ub\ud\uc.\no
\end{align}
The key is to bound the inner integral for $\abs{\ub}\leq 1,\abs{\uc}\leq 1$, $0<\eta\leq L=\e^{-\frac{1}{2}}$,
\begin{align}\label{rtt 12}
J:&=\int_{\abs{\ua}\leq 1}\frac{1}{\sqrt{\ua^2+(\e\eta)\ub^2+(\e\eta)\uc^2}}\ud\ua\\
&=2\ln\bigg(1+\sqrt{1+(\e\eta)\ub^2+(\e\eta)\uc^2}\bigg)-2\ln\bigg(\sqrt{(\e\eta)\ub^2+(\e\eta)\uc^2}\bigg)\no\\
&\ls \sqrt{1+(\e\eta)\ub^2+(\e\eta)\uc^2}+\abs{\ln\Big((\e\eta)\ub^2+(\e\eta)\uc^2\Big)}\no\\
&\ls 1+\abs{\ln\Big((\e\eta)\ub^2\Big)}+\abs{\ln\Big((\e\eta)\uc^2\Big)}\ls 1+\abs{\ln(\e\eta)}+\abs{\ln\abs{\ub}}+\abs{\ln\abs{\uc}}.\no
\end{align}
Inserting \eqref{rtt 12} into \eqref{rtt 11}, we obtain
\begin{align}\label{rtt 13}
I&\ls  1+\int_{\abs{\ub}\leq 1,\abs{\uc}\leq 1}\Big(1+\abs{\ln(\e\eta)}+\abs{\ln\abs{\ub}}+\abs{\ln\abs{\uc}}\Big)\ud\ub\ud\uc\\
&\ls 1+\abs{\ln(\e\eta)}+\int_{\abs{\ub}\leq 1}\abs{\ln\abs{\ub}}\ud\ub+\int_{\abs{\uc}\leq 1}\abs{\ln\abs{\uc}}\ud\uc\ls 1+\abs{\ln(\e\eta)}.\no
\end{align}
On the other hand, similar to \eqref{rtt 11}, we have
\begin{align}\label{rtt 14}
II\ls&\int_{\abs{\vuu}\leq 1}\frac{1}{\sqrt{\ua^2+(\e\eta)\ub^2+(\e\eta)\uc^2}}\frac{1}{\abs{\vuu-\vvv}^2}\ue^{-\abs{\vuu-\vvv}^2}\ud\vuu\\
&+\int_{\abs{\vuu}\geq 1}\frac{1}{\sqrt{\ua^2+(\e\eta)\ub^2+(\e\eta)\uc^2}}\frac{1}{\abs{\vuu-\vvv}^2}\ue^{-\abs{\vuu-\vvv}^2}\ud\vuu\no\\
\ls &\int_{\abs{\ub}\leq 1,\abs{\uc}\leq 1}\Bigg(\int_{\abs{\ua}\leq 1}\frac{1}{\sqrt{\ua^2+(\e\eta)\ub^2+(\e\eta)\uc^2}}\ud\ua\Bigg)\frac{1}{\sqrt{(\ub-\vb)^2+(\uc-\vc)^2}}\ud\ub\ud\uc\no\\
&+\int_{\abs{\vuu}\geq 1}\frac{1}{\abs{\vuu-\vvv}^2}\ue^{-\abs{\vuu-\vvv}^2}\ud\vuu\no\\
\ls&1+\int_{\abs{\ub}\leq 1,\abs{\uc}\leq 1}\Bigg(\int_{\abs{\ua}\leq 1}\frac{1}{\sqrt{\ua^2+(\e\eta)\ub^2+(\e\eta)\uc^2}}\ud\ua\Bigg)\frac{1}{\sqrt{(\ub-\vb)^2+(\uc-\vc)^2}}\ud\ub\ud\uc.\no
\end{align}
Inserting \eqref{rtt 12} into \eqref{rtt 14}, and applying H\"older's inequality, we obtain
\begin{align}\label{rtt 15}
II\ls&1+\int_{\abs{\ub}\leq 1,\abs{\uc}\leq 1}\frac{1+\abs{\ln(\e\eta)}+\abs{\ln\abs{\ub}}+\abs{\ln\abs{\uc}}}{\sqrt{(\ub-\vb)^2+(\uc-\vc)^2}}\ud\ub\ud\uc\\
\ls&1+\abs{\ln(\e\eta)}+\int_{\abs{\ub}\leq 1,\abs{\uc}\leq 1}\frac{\abs{\ln\abs{\ub}}+\abs{\ln\abs{\uc}}}{\sqrt{(\ub-\vb)^2+(\uc-\vc)^2}}\ud\ub\ud\uc\no\\
\ls&1+\abs{\ln(\e\eta)}+\Bigg(\int_{\abs{\ub}\leq 1,\abs{\uc}\leq 1}\frac{1}{\Big((\ub-\vb)^2+(\uc-\vc)^2\Big)^{\frac{3}{4}}}\ud\ub\ud\uc\Bigg)^{\frac{2}{3}}\no\\
&\times\Bigg(\int_{\abs{\ub}\leq 1,\abs{\uc}\leq 1}\Big(\abs{\ln\abs{\ub}}+\abs{\ln\abs{\uc}}\Big)^3\ud\ub\ud\uc\Bigg)^{\frac{1}{3}}.\no
\end{align}
Note that using polar coordinates, we have
\begin{align}
\int_{\abs{\ub}\leq 1,\abs{\uc}\leq 1}\frac{1}{\Big((\ub-\vb)^2+(\uc-\vc)^2\Big)^{\frac{3}{4}}}\ud\ub\ud\uc&\ls 1,\label{rtt 16}\\
\int_{\abs{\ub}\leq 1,\abs{\uc}\leq 1}\Big(\abs{\ln\abs{\ub}}+\abs{\ln\abs{\uc}}\Big)^3\ud\ub\ud\uc&\ls 1.\label{rtt 17}
\end{align}
Hence, inserting \eqref{rtt 16} and \eqref{rtt 17} into \eqref{rtt 15}, we get
\begin{align}\label{rtt 18}
II\ls 1+\abs{\ln(\e\eta)}.
\end{align}
Inserting \eqref{rtt 13} and \eqref{rtt 18} into \eqref{rtt 19}, we obtain the desired result.
\end{proof}

\begin{remark}
Lemma \ref{Regularity lemma 2} and Lemma \ref{Regularity lemma 3} are valid uniformly in $\vvv\in\r^3$.
\end{remark}

\subsection{Mild Formulation}

Taking $\eta$ derivative in \eqref{Milne transform} and multiplying $\zeta$ defined in \eqref{weight} on both sides, we obtain the $\e$-transport problem for $\a:=\zeta\dfrac{\p\gg}{\p\eta}$ as
\begin{align}\label{normal derivative equation}
\left\{
\begin{array}{l}\displaystyle
\va\frac{\p\a}{\p\eta}+G_1(\eta)\bigg(\vb^2\dfrac{\p\a }{\p\va}-\va\vb\dfrac{\p\a }{\p\vb}\bigg)+G_2(\eta)\bigg(\vc^2\dfrac{\p\a }{\p\va}-\va\vc\dfrac{\p\a }{\p\vc}\bigg)+\nu\a=\tilde\a+S_{\a},\\\rule{0ex}{1.5em}
\a(0,\vvv)=p_{\a}(\vvv)\ \ \text{for}\ \ \va>0,\\\rule{0ex}{1.5em}
\a(L,\vvv)=\a(L,\rr[\vvv]),
\end{array}
\right.
\end{align}
where the crucial non-local term
\begin{align}\label{rtt 77}
\tilde\a(\eta,\vvv)=\int_{\r^3}\frac{\zeta(\eta,\vvv)}{\zeta(\eta,\vuu)}k(\vuu,\vvv)\a(\eta,\vuu)\ud{\vuu}.
\end{align}
Here we utilize Lemma \ref{weight lemma} to move $\zeta$ inside the derivative. $p_{\a}$ and $S_{\a}$ will be specified later. We need to derive the a priori estimate of $\a$. Note that $\tilde\a$ is different from $K[\a]$ since the denominator $\zeta(\eta,\vuu)$ is possibly zero. Thus, this creates a strong singularity and becomes the major difficulty in this section.

Here we use the notation as in $L^{\infty}$ estimates of Section \ref{mtt section 01}.
We can easily check that the weight function satisfies $\zeta=\sqrt{E_1-E_2^2-E_3^2}$. Along the characteristics, where $E_1$, $E_2$, $E_3$ and $\zeta$ are constants, the equation \eqref{normal derivative equation} can be rewritten as:
\begin{align}
\va\frac{\ud{\a}}{\ud{\eta}}+\a=\tilde\a+S_{\a}.
\end{align}
We can define the solution to \eqref{normal derivative equation} along the characteristics as follows:
\begin{align}\label{rtt 73}
\a(\eta,\vvv)=\k[p_{\a}]+\t[\tilde\a+S_{\a}],
\end{align}
where the operators $\k$ and $\t$ are defined in \eqref{mtt 26} to \eqref{mtt 28}. Based on Lemma \ref{Milne lemma 8} and Lemma \ref{Milne lemma 9},
we can directly obtain
\begin{align}
\lnnmv{\k[p_{\a}]}\ls&\lnmh{p_{\a}},\label{rtt 71}\\
\lnnmv{\t[S_{\a}]}\ls&\lnnmv{\nu^{-1}S_{\a}}\label{rtt 72}.
\end{align}
The next three sections will be devoted to the estimate of $\t[\tilde\a]$.

Similar to Section \ref{mtt section 01}, since we always assume that $(\eta,\vvv)$ and $(\eta',\vvv')$ are on the same characteristics, in the following, we will simply write $\vvv'(\eta')$ or even $\vvv'$ instead of $\vvv'(\eta,\vvv;\eta')$ when there is no confusion. In addition, we will use $\d$ or $\d_0$ to represent small quantities. They are not necessarily constants, but may depend on $\e$ and need to be chosen later.

In the analysis below, we will repeatedly use the following packages of simple facts (PSF):
\begin{itemize}
\item
Based on Theorem \ref{Milne theorem 2} and Theorem \ref{Milne theorem 4}, we know $\lnnmv{\ue^{K_0\eta}\gg}\ls 1$.
\item
Based on Lemma \ref{Regularity lemma 1'}, we have for $0\leq\vrh<\dfrac{1}{4}$ and $\vth>3$
\begin{align}
\lnnmv{\ue^{K_0\eta}K[\gg]}\ls \lnnmv{\ue^{K_0\eta}\nu^{-1}\gg}\ls 1.
\end{align}
\item
Based on Lemma \ref{Regularity lemma 2'}, we know
\begin{align}
\lnnmv{\ue^{K_0\eta}\nabla_vK[\gg]}\ls\lnnmv{\ue^{K_0\eta}\gg}\ls 1.
\end{align}
\item
Since $E_1$ is conserved along the characteristics, we must have $\abs{\vvv}=\abs{\vvv'}$ and further $\bvv=\bvvp$.
\end{itemize}

\subsubsection{Region I: $\va>0$}

Based on \eqref{mtt 26}, we need to bound
\begin{align}\label{rtt 20}
I=\t[\tilde\a]=\int_0^{\eta}\frac{\tilde\a\Big(\eta',\vvv'(\eta,\vvv;\eta')\Big)}{\va'(\eta,\vvv;\eta')}\exp(-H_{\eta,\eta'})\ud{\eta'}.
\end{align}
\ \\
Step 0: Preliminaries.\\
Based on \eqref{mtt 81} and \eqref{mtt 10}, we have
\begin{align}
E_2(\eta',\vb')=\frac{R_1-\e\eta'}{R_1}\vb',\ \ E_3(\eta',\vc')=\frac{R_2-\e\eta'}{R_2}\vc'.
\end{align}
Then we can directly obtain for $0<\eta'<L=\e^{-\frac{1}{2}}$,
\begin{align}\label{ptt 01}
\zeta(\eta',\vvv')=&\sqrt{(\va'^2+\vb'^2+\vc'^2)-\left(\frac{R_1-\e\eta'}{R_1}\right)^2\vb'^2-\left(\frac{R_2-\e\eta'}{R_2}\right)^2\vc'^2}\\
=&\sqrt{\va'^2+\frac{R_1^2-(R_1-\e\eta')^2}{R_1^2}\vb'^2+\frac{R_2^2-(R_2-\e\eta')^2}{R_2^2}\vc'^2}\no\\
\leq& \sqrt{\va'^2}+\frac{1}{R_1}\sqrt{\Big(R_1^2-(R_1-\e\eta')^2\Big)\vb'^2}+\frac{1}{R_2}\sqrt{\Big(R_2^2-(R_2-\e\eta')^2\Big)\vc'^2}\no\\
\ls& \abs{\va'}+\sqrt{\e\eta'}\abs{\vb'}+\sqrt{\e\eta'}\abs{\vc'}\ls \abs{\vvv'},\no
\end{align}
and
\begin{align}\label{ptt 02}
\zeta(\eta',\vvv')\geq&\frac{1}{2}\left(\sqrt{\va'^2}+\frac{1}{R_1}\sqrt{\Big(R_1^2-(R_1-\e\eta')^2\Big)\vb'^2}+\frac{1}{R_2}\sqrt{\Big(R_2^2-(R_2-\e\eta')^2\Big)\vc'^2}\right)\\
\gs& \abs{\va'}+\sqrt{\e\eta'}\abs{\vb'}+\sqrt{\e\eta'}\abs{\vc'}\gs\sqrt{\e\eta'}\abs{\vvv'}.\no
\end{align}
Also, considering \eqref{mtt 81} and \eqref{mtt 10}, we know for $0\leq\eta'\leq\eta$,
\begin{align}
\va'=&\sqrt{\va^2+\vb^2+\vc^2-\vb'^2-\vc'^2}=\sqrt{\va^2+\vb^2+\vc^2-\vb^2\bigg(\frac{R_1-\e\eta}{R_1-\e\eta'}\bigg)^2-\vc^2\bigg(\frac{R_2-\e\eta}{R_2-\e\eta'}\bigg)^2}\\
=&\sqrt{\va^2+\frac{(2R_1-\e\eta-\e\eta')(\e\eta-\e\eta')}{R_1-\e\eta'}\vb^2+\frac{(2R_2-\e\eta-\e\eta')(\e\eta-\e\eta')}{R_2-\e\eta'}\vc^2}.\no
\end{align}
Since for $i=1,2$
\begin{align}
0&\leq(2R_i-\e\eta-\e\eta')(\e\eta-\e\eta')\ls \e(\eta-\eta'),\\
1&\ls R_i-\e\eta',
\end{align}
we have
\begin{align}
\va\leq\va'
\ls \sqrt{\va^2+\e(\eta-\eta')\vb^2+\e(\eta-\eta')\vc^2},
\end{align}
which means
\begin{align}
\frac{1}{2\sqrt{\va^2+\e(\eta-\eta')\vb^2+\e(\eta-\eta')\vc^2}}\ls \frac{1}{\va'}
\leq\frac{1}{\va}.
\end{align}
Therefore,
\begin{align}\label{ptt 03}
-\int_{\eta'}^{\eta}\frac{1}{\va'(y)}\ud{y}\ls& -\int_{\eta'}^{\eta}\frac{1}{2\sqrt{\va^2+\e(\eta-y)\vb^2+\e(\eta-y)\vc^2}}\ud{y}\\
=&\frac{1}{\e(\vb^2+\vc^2)}\bigg(\va-\sqrt{\va^2+\e(\eta-\eta')\vb^2+\e(\eta-\eta')\vc^2}\bigg)\no\\
=&-\frac{\eta-\eta'}{\va+\sqrt{\va^2+\e(\eta-\eta')\vb^2+\e(\eta-\eta')\vc^2}}\ls-\frac{\eta-\eta'}{\sqrt{\va^2+\e(\eta-\eta')\vb^2+\e(\eta-\eta')\vc^2}}.\no
\end{align}
Define a $C^{\infty}$ cut-off function $\chi\in C^{\infty}[0,\infty)$ satisfying
\begin{align}\label{rtt 28}
\chi(\va)=\left\{
\begin{array}{ll}
1&\text{for}\ \ \abs{\va}\leq\d,\\
0&\text{for}\ \ \abs{\va}\geq2\d.
\end{array}
\right.
\end{align}
We use $\chi$ to avoid discontinuous cut-off for the convenience of integration by parts. In the following, we will divide the estimate of $I$ in \eqref{rtt 20} into several cases based on the value of $\va$, $\va'$, $\e\eta'$ and $\e(\eta-\eta')$. Assume the dummy variable $\vuu=(\ua,\ub,\uc)=(\ua,\tu)$. The similar notation also applies to $\vvv=(\va,\vb,\vc)=(\va,\tv)$. \\
\ \\
Step 1: Estimate of $I_1: \va\geq\d_0$.\\
In this step, we will not resort to $\a$ equation \eqref{normal derivative equation}, but rather directly bound
\begin{align}
\abs{\bvv I_1} \ls&\abs{\zeta}\abs{\bvv\frac{\p\gg}{\p\eta}}\ls \abs{\bw{\vrh}{\vth+1}\frac{\p\gg}{\p\eta}}.
\end{align}
Hence, the key is to estimate $\dfrac{\p\gg}{\p\eta}$. As in \eqref{mtt 26}, we rewrite the equation \eqref{Milne transform} along the characteristics as
\begin{align}\label{rtt 23}
\gg(\eta,\vvv)=&\exp\left(-H_{\eta,0}\right)\Bigg(p\Big(\vvv'(0)\Big)
+\int_0^{\eta}\frac{(K[\gg]+S)\Big(\eta',\vvv'(\eta')\Big)}{\va'(\eta')}
\exp\left(H_{\eta',0}\right)\ud{\eta'}\Bigg).
\end{align}
Taking $\eta$ derivative on both sides of \eqref{rtt 23}, we have
\begin{align}
\frac{\p\gg}{\p\eta}:=X_1+X_2+X_3+X_4+X_5+X_6,
\end{align}
where
\begin{align}
X_1=&-\exp\left(-H_{\eta,0}\right)\frac{\p{H_{\eta,0}}}{\p{\eta}}\Bigg(p\Big(\vvv'(0)\Big)
+\int_0^{\eta}\frac{K[\gg]\Big(\eta',\vvv'(\eta')\Big)}{\va'(\eta')}
\exp\left(H_{\eta',0}\right)\ud{\eta'}\Bigg),\\
X_2=&\exp\left(-H_{\eta,0}\right)\frac{\p p\Big(\vvv'(0)\Big)}{\p\eta},\\
X_3=&\frac{(K[\gg]+S)(\eta,\vvv)}{\va},\\
X_4=&-\exp\left(-H_{\eta,0}\right)\int_0^{\eta}\Bigg((K[\gg]+S)\Big(\eta',\vvv'(\eta')\Big)
\exp\left(H_{\eta',0}\right)\frac{1}{\va'^2(\eta')}\frac{\p\va'(\eta')}{\p\eta}\ud{\eta'}\bigg)\\
X_5=&\exp\left(-H_{\eta,0}\right)\int_0^{\eta}\frac{(K[\gg]+S)\Big(\eta',\vvv'(\eta')\Big)}{\va'(\eta')}
\exp\left(H_{\eta',0}\right)\frac{\p{H_{\eta',0}}}{\p{\eta}}\ud{\eta'},\\
X_6=&\exp\left(-H_{\eta,0}\right)
\int_0^{\eta}\frac{1}{\va'(\eta')}\bigg(\nabla_{\vvv'} (K[\gg]+S)\Big(\eta',\vvv'(\eta')\Big)\frac{\p\vvv'(\eta')}{\p\eta}\bigg)
\exp\left(H_{\eta',0}\right)\ud{\eta'}.
\end{align}
We need to estimate each term. Below are some preliminary results:
\begin{itemize}
\item
For $\eta'\leq\eta$, we must have $\va'\geq\va\geq\d_0$, which means $\dfrac{1}{\va'}\leq\dfrac{1}{\va}\leq\dfrac{1}{\d_0}$.
\item
Using substitution $y=H_{\eta,\eta'}$, we know
\begin{align}\label{rtt 24}
\abs{\int_0^{\eta}\frac{\nu\Big(\vvv'(\eta')\Big)}{\va'(\eta')}\exp(-H_{\eta,\eta'})\ud{\eta'}}\leq\abs{\int_0^{\infty}\ue^{-y}\ud{y}}=1.
\end{align}
\item
For $t,s\in[0,\eta]$, based on (PSF), we have
\begin{align}\label{rtt 27}
\abs{{H_{t,s}}}\ls\abs{\int_{s}^{t}\frac{\nu(\vvv'(y))}{\va'(y)}\ud{y}}\ls \frac{\abs{\vvv}}{\d_0}\abs{t-s}.
\end{align}
\item
Considering
\begin{align}
\vb'(\eta')=&\vb\ue^{W_1(\eta')-W_1(\eta)}=\vb\frac{R_1-\e\eta}{R_1-\e\eta'},\\
\vc'(\eta')=&\vc\ue^{W_2(\eta')-W_2(\eta)}=\vc\frac{R_2-\e\eta}{R_2-\e\eta'},\\
\va'(\eta')=&\sqrt{\va^2+\vb^2+\vc^2-\vb'^2-\vc'^2}=\sqrt{\va^2+\vb^2+\vc^2
-\vb^2\bigg(\frac{R_1-\e\eta}{R_1-\e\eta'}\bigg)^2-\vc^2\bigg(\frac{R_2-\e\eta}{R_2-\e\eta'}\bigg)^2},
\end{align}
we know
\begin{align}
&\frac{\p\vb'(\eta')}{\p\eta}=-\frac{\e\vb}{R_1-\e\eta'},\quad\frac{\p\vc'(\eta')}{\p\eta}=-\frac{\e\vc}{R_2-\e\eta'},\\
&\frac{\p\va'(\eta')}{\p\eta}=\frac{2\e}{\va'(\eta)}\left(\vb^2\frac{R_1-\e\eta}{R_1-\e\eta'}+\vc^2\frac{R_2-\e\eta}{R_2-\e\eta'}\right).\no
\end{align}
This implies
\begin{align}\label{rtt 25}
\abs{\frac{\p\vb'(\eta')}{\p\eta}}\ls \e\abs{\vvv},\quad\abs{\frac{\p\vc'(\eta')}{\p\eta}}\ls\e\abs{\vvv},\quad\abs{\frac{\p\va'(\eta')}{\p\eta}}\ls \frac{\e\abs{\vvv}^2}{\va'(\eta')}\ls \frac{\e\abs{\vvv}^2}{\d_0}.
\end{align}
\item
For $t,s\in[0,\eta]$, note that
\begin{align}
&\frac{\p{H_{t,s}}}{\p{\eta}}=\int_{s}^{t}\frac{\p{}}{\p{\eta}}\bigg(\frac{\nu(\vvv'(y))}{\va'(y)}\bigg)\ud{y}\\
=&\int_{s}^{t}\frac{1}{\va'(y)}\frac{\p\nu(\abs{\vvv'})}{\p\abs{\vvv'}}(y)\frac{1}{\abs{\vvv'(y)}}
\bigg(\va'(y)\frac{\p\va'(y)}{\p\eta}+\vb'(y)\frac{\p\vb'(y)}{\p\eta}+\vc'(y)\frac{\p\vc'(y)}{\p\eta}\bigg)\ud{y}
-\int_{s}^{t}\frac{\nu(\abs{\vvv'}(y))}{\va'^2(y)}\frac{\p\va'(y)}{\p\eta}\ud{y}.\no
\end{align}
Based on \eqref{rtt 27}, Lemma \ref{Regularity lemma 0} and (PSF), we obtain
\begin{align}\label{rtt 26}
\abs{\frac{\p{H_{t,s}}}{\p{\eta}}}&\ls\abs{\int_s^t\frac{\nu(\vvv'(y))}{\va'(y)}\bigg(\e+\frac{\e\abs{\vvv}}{\d_0}\bigg)\ud y}+\abs{\int_{s}^{t}\frac{\nu(\abs{\vvv'}(y))}{\va'(y)}\frac{\e\abs{\vvv}^2}{\d_0^2}\ud{y}}\\
&\ls\frac{\e\br{\vvv}^2}{\d_0^2}\abs{H_{t,s}}\ls\frac{\e\br{\vvv}^3}{\d_0^3}\abs{t-s}\ls\frac{\e\eta\br{\vvv}^3}{\d_0^3}\ls\frac{\br{\vvv}^3}{\d_0^3}.\no
\end{align}
\end{itemize}
The estimate of $X_i$ is standard based on (PSF) and the above preliminaries. Using \eqref{rtt 23} and \eqref{rtt 26}, we have
\begin{align}
\abs{\bvv X_1}&\ls\abs{\frac{\p{H_{\eta,0}}}{\p{\eta}}}\abs{\bvv\gg}\ls \Bigg(\abs{\frac{\nu(\vvv)}{\va}}+\abs{\int_{0}^{\eta}\frac{\p{}}{\p{\eta}}\bigg(\frac{\nu(\vvv'(y))}{\va'(y)}\bigg)\ud{y}}\Bigg)\abs{\bvv\gg}\\
&\ls\bigg(\frac{\abs{\vvv}}{\d_0}+\frac{\br{\vvv}^3}{\d_0^3}\bigg)\abs{\bvv\gg}\ls \frac{1}{\d_0^3}\lnnm{\gg}{\vth+3,\vrh}\ls \frac{1}{\d_0^3}.\no
\end{align}
Based on \eqref{rtt 25} and \eqref{Regularity bound}, we know
\begin{align}
\abs{\bvv X_2}&\ls\abs{\exp\left(-H_{\eta,0}\right)}\abs{\bvv\nabla_{\vvv}p}\abs{\frac{\p \vvv'(0)}{\p\eta}}\ls\bigg(\e\abs{\vvv}+\frac{\e\abs{\vvv}^2}{\d_0}\bigg)\abs{\bvv\nabla_{\vvv}p}\\
&\ls \frac{\e}{\d_0^2}\lnm{\nabla_{\vvv}p}{\vth+2,\vrh}\ls \frac{\e}{\d_0^2}.\no
\end{align}
Also, using \eqref{Milne bound} and Lemma \ref{Regularity lemma 1'}, we have
\begin{align}
\abs{\bvv X_3}&\ls\abs{\frac{1}{\va}}\bigg(\abs{\bvv K[\gg]}+\abs{\bvv S}\bigg)\ls \frac{1}{\d_0}\bigg(1+\abs{\bvv\nu^{-1}\gg}\bigg)\ls \frac{1}{\d_0}.
\end{align}
On the other hand, using \eqref{rtt 25}, \eqref{rtt 24} and \eqref{Milne bound}, we obtain
\begin{align}
\abs{\bvv X_4}&\ls\int_0^{\eta}\bigg(\abs{\bvvp K[\gg]}+\abs{\bvvp S}\bigg)\exp\left(-H_{\eta,\eta'}\right)\frac{1}{\d_0^2}\frac{\e\abs{\vvv}^2}{\d_0}\ud\eta'\\
&\ls\frac{\e}{\d_0^3}\bigg(\lnnm{\nu^{-1}\gg}{\vth+2,\vrh}+\lnnm{S}{\vth+2,\vrh}\bigg)\bigg(\int_0^{\eta}\exp\left(-H_{\eta,\eta'}\right)\ud\eta'\bigg)\ls\frac{\e}{\d_0^3}.\no
\end{align}
Using \eqref{rtt 26}, \eqref{rtt 24} and \eqref{Milne bound}, we know
\begin{align}
\abs{\bvv X_5}&\ls\int_0^{\eta}\frac{1}{\d_0}\bigg(\abs{\bvvp K[\gg]}+\abs{\bvvp S}\bigg)\exp\left(-H_{\eta,\eta'}\right)\frac{\br{\vvv}^3}{\d_0^3}\ud\eta'\\
&\ls\frac{1}{\d_0^4}\bigg(\lnnm{\nu^{-1}\gg}{\vth+3,\vrh}+\lnnm{S}{\vth+3,\vrh}\bigg)\bigg(\int_0^{\eta}\exp\left(-H_{\eta,\eta'}\right)\ud\eta'\bigg)\ls\frac{1}{\d_0^4}.\no
\end{align}
Finally, using \eqref{rtt 25}, \eqref{rtt 24} and \eqref{Regularity bound}, we have
\begin{align}
\\
\abs{\bvv X_6}&\ls\int_0^{\eta}\frac{1}{\d_0}\bigg(\abs{\bvvp \nabla_{\vvv'}K[\gg]}+\abs{\bvvp \nabla_{\vvv'}S}\bigg)\exp\left(-H_{\eta,\eta'}\right)\bigg(\e\abs{\vvv}+\frac{\e\abs{\vvv}^2}{\d_0}\bigg)\ud\eta'\no\\
&\ls\frac{\e}{\d_0^3}\bigg(\lnnm{\gg}{\vth+2,\vrh}+\lnnm{S}{\vth+2,\vrh}\bigg)\bigg(\int_0^{\eta}\exp\left(-H_{\eta,\eta'}\right)\ud\eta'\bigg)\ls\frac{\e}{\d_0^3}.\no
\end{align}
Collecting all $X_i$ estimates, we have
\begin{align}\label{rtt 47}
\abs{\bvv I_1} \ls& \frac{\e}{\d_0^3}+\frac{1}{\d_0^4}.
\end{align}
\ \\
Step 2: Estimate of $I_2$: $0\leq\va\leq\d_0$ with $1-\chi(\ua)$ term.\\
We naturally decompose $1=\Big(1-\chi(\ua)\Big)+\chi(\ua)$. In this step, we focus on $1-\chi(\ua)$ part, while $\chi(\ua)$ part will handled in following steps involving $I_3,I_4,I_5$. Based on \eqref{rtt 28}, the cut-off $1-\chi(\ua)$ is nonzero only when $\abs{\ua}\geq\d$. We have
\begin{align}\label{rtt 29}
I_2:=&\int_0^{\eta}\bigg(\int_{\r^3}\frac{\zeta(\eta',\vvv')}{\zeta(\eta',\vuu)}\Big(1-\chi(\ua)\Big)
k(\vuu,\vvv')\a(\eta',\vuu)\ud{\vuu}\bigg)
\frac{1}{\va'}\exp(-H_{\eta,\eta'})\ud{\eta'}\\
=&\int_0^{\eta}\bigg(\int_{\r^3}\Big(1-\chi(\ua)\Big)
k(\vuu,\vvv')\frac{\gg(\eta',\vuu)}{\p\eta'}\ud{\vuu}\bigg)
\frac{\zeta(\eta',\vvv')}{\va'}\exp(-H_{\eta,\eta'})\ud{\eta'}.\no
\end{align}
We first handle the inner integral. Based on \eqref{Milne transform}, $\gg(\eta',\vuu)$ satisfies
\begin{align}
\ua\frac{\p\gg(\eta',\vuu)}{\p\eta'}
+G_1(\eta')\bigg(\ub^2\frac{\p\gg(\eta',\vuu)}{\p\ua}-\ua\ub\frac{\p\gg(\eta',\vuu)}{\p\ub}\bigg)\\
+G_2(\eta')\bigg(\uc^2\frac{\p\gg(\eta',\vuu)}{\p\ua}-\ua\uc\frac{\p\gg(\eta',\vuu)}{\p\uc}\bigg)&
+\nu\gg(\eta',\vuu)-K[\gg](\eta',\vuu)=S(\eta',\vuu),\no
\end{align}
which implies
\begin{align}\label{rtt 30}
\frac{\p\gg(\eta',\vuu)}{\p\eta'}=-\frac{1}{\ua}
\Bigg(G_1(\eta')\bigg(\ub^2\frac{\p\gg(\eta',\vuu)}{\p\ua}-\ua\ub\frac{\p\gg(\eta',\vuu)}{\p\ub}\bigg)&\\
+G_2(\eta')\bigg(\uc^2\frac{\p\gg(\eta',\vuu)}{\p\ua}-\ua\uc\frac{\p\gg(\eta',\vuu)}{\p\uc}\bigg)&
+\nu\gg(\eta',\vuu)-K[\gg](\eta',\vuu)-S(\eta',\vuu)\Bigg).\no
\end{align}
Hence, inserting \eqref{rtt 30} into the inner integral in \eqref{rtt 29}, we have
\begin{align}
J:=&\int_{\r^3}\Big(1-\chi(\ua)\Big)
k(\vuu,\vvv')\frac{\gg(\eta',\vuu)}{\p\eta'}\ud{\vuu}\\
=&-\int_{\r^3}\Big(1-\chi(\ua)\Big)
k(\vuu,\vvv')\frac{1}{\ua}\bigg(\nu\gg(\eta',\vuu)-K[\gg](\eta',\vuu)-S(\eta',\vuu)\bigg)\ud{\vuu}\no\\
&-\int_{\r^3}\Big(1-\chi(\ua)\Big)
k(\vuu,\vvv')\frac{1}{\ua}
G_1(\eta')\bigg(\ub^2\frac{\p\gg(\eta',\vuu)}{\p\ua}-\ua\ub\frac{\p\gg(\eta',\vuu)}{\p\ub}\bigg)\ud{\vuu}\no\\
&-\int_{\r^3}\Big(1-\chi(\ua)\Big)
k(\vuu,\vvv')\frac{1}{\ua}
G_2(\eta')\bigg(\uc^2\frac{\p\gg(\eta',\vuu)}{\p\ua}-\ua\uc\frac{\p\gg(\eta',\vuu)}{\p\uc}\bigg)\ud{\vuu}\no\\
:=&J_1+J_2+J_3.\no
\end{align}
Since $\abs{\ua}\geq\d$, using Lemma \ref{Regularity lemma 1'}, \eqref{Milne bound} and (PSF), we obtain
\begin{align}
\abs{\bvvp J_1}
&\ls\abs{\bvvp\int_{\r^3}\Big(1-\chi(\ua)\Big)
k(\vuu,\vvv')\frac{1}{\ua}\bigg(\nu\gg(\eta',\vuu)-K[\gg](\eta',\vuu)-S(\eta',\vuu)\bigg)\ud{\vuu}}\\
&\ls\frac{1}{\d}\bigg(\lnnm{\gg}{\vth+1,\vrh}+\lnnm{S}{\vth,\vrh}\bigg)\abs{\int_{\r^3}k(\vuu,\vvv')\frac{\bvvp}{\br{\vuu}^{\vth}\ue^{\vrh\abs{\vuu}^2}}\ud\vuu}\ls \frac{1}{\d}.\no
\end{align}
On the other hand, an integration by parts yields
\begin{align}
J_2=&\int_{\r^3}\frac{\p}{\p\ua}\bigg(\frac{\ub^2}{\ua}
G_1(\eta')\Big(1-\chi(\ua)\Big)
k(\vuu,\vvv')\bigg)\gg(\eta',\vuu)\ud{\vuu}\\
&-\int_{\r^3}\frac{\p}{\p\ub}\bigg(\ub
G_1(\eta')\Big(1-\chi(\ua)\Big)
k(\vuu,\vvv')\bigg)\gg(\eta',\vuu)\ud{\vuu},\no\\
=&\int_{\r^3}\bigg(-\frac{\ub^2}{\ua^2}\Big(1-\chi(\ua)\Big)-\frac{\ub^2}{\ua}\chi'(\ua)-\Big(1-\chi(\ua)\Big)\bigg)G_1(\eta')k(\vuu,\vvv')\gg(\eta',\vuu)\ud{\vuu}\no\\
&+\int_{\r^3}G_1(\eta')\Big(1-\chi(\ua)\Big)\bigg(-\frac{\ub^2}{\ua}\frac{\p k(\vuu,\vvv')}{\p\ua}+\ub\frac{\p k(\vuu,\vvv')}{\p\ub}\bigg)\gg(\eta',\vuu)\ud\vuu:=J_{2,1}+J_{2,2}.\no
\end{align}
Since $\abs{\ua}\geq\d$, using Lemma \ref{Regularity lemma 2} and Lemma \ref{Regularity lemma 1'}, we have
\begin{align}\label{rtt 33}
\abs{\bvvp J_{2,1}}&\ls \frac{\e}{\d^2}\lnnm{\gg}{\vth+2,\vrh}\ls\frac{\e}{\d^2}.
\end{align}
Also, using Lemma \ref{Regularity lemma 2} and Lemma \ref{Regularity lemma 2''}, we have
\begin{align}\label{rtt 34}
\abs{\bvvp J_{2,2}}\ls \frac{\e}{\d}\lnnm{\gg}{\vth+4,\vrh}\ls\frac{\e}{\d}.
\end{align}
\eqref{rtt 33} and \eqref{rtt 34} yield
\begin{align}
\abs{\bvvp J_2}\ls \frac{\e}{\d^2}.
\end{align}
Similarly,
\begin{align}
\abs{\bvvp J_3}\ls \frac{\e}{\d^2}.
\end{align}
In summary, the inner integral in \eqref{rtt 29}
\begin{align}
\abs{\bvvp J}\ls \abs{\bvvp J_1}+\abs{\bvvp J_2}+\abs{\bvvp J_3}\ls \frac{1}{\d}+\frac{\e}{\d^2}.
\end{align}
Then for the outer integral in \eqref{rtt 29}, we can use \eqref{ptt 01} and \eqref{rtt 24} to show that
\begin{align}\label{rtt 38}
\abs{\int_0^{\eta}\frac{\zeta(\eta',\vvv')}{\va'(\eta')}\exp(-H_{\eta,\eta'})\ud{\eta'}}\leq \abs{\int_0^{\eta}\frac{\nu(\vvv')}{\va'(\eta')}\exp(-H_{\eta,\eta'})\ud{\eta'}}\ls 1.
\end{align}
Then we have
\begin{align}\label{rtt 48}
\abs{\bvv I_2}\ls&\abs{\bvvp J}\ls \frac{1}{\d}+\frac{\e}{\d^2}.
\end{align}
\ \\
Step 3: Estimate of $I_3$: $0\leq\va\leq\d_0$, with $\chi(\ua)$ term, and $\sqrt{\e\eta'}\abs{\tv'}\geq\va'$.\\
Based on \eqref{rtt 28} and \eqref{rtt 29}, we are left with $\chi(\ua)$ part, which is nonzero only when $\abs{\ua}\leq2\d$, i.e.
\begin{align}\label{rtt 35}
\int_0^{\eta}\bigg(\int_{\r^3}\frac{\zeta(\eta',\vvv')}{\zeta(\eta',\vuu)}\chi(\ua)
k(\vuu,\vvv')\a(\eta',\vuu)\ud{\vuu}\bigg)
\frac{1}{\va'}\exp(-H_{\eta,\eta'})\ud{\eta'}.
\end{align}
We will further decompose this integral into $I_3,I_4,I_5$. In this step, based on \eqref{ptt 01}, $\sqrt{\e\eta'}\abs{\tv'}\geq\va'$ implies
\begin{align}\label{rtt 36}
\zeta(\eta',\vvv')\ls \abs{\va'}+\sqrt{\e\eta'}\abs{\tv'}\ls \sqrt{\e\eta'}\abs{\tv'}.
\end{align}
On the other hand, \eqref{ptt 02} implies
\begin{align}\label{rtt 37}
\zeta(\eta',\vuu)\gs \sqrt{\e\eta'}\abs{\vuu}.
\end{align}
Then considering \eqref{rtt 36} and \eqref{rtt 37}, the inner integral in \eqref{rtt 35}
\begin{align}\label{rtt 39}
M:=&\abs{\int_{\r^3}\frac{\zeta(\eta',\vvv')}{\zeta(\eta',\vuu)}\chi(\ua)k(\vuu,\vvv')\a(\eta',\vuu)\ud{\vuu}}\ls \abs{\tv'}\abs{\int_{\r^3}\frac{1}{\abs{\vuu}}\chi(\ua)k(\vuu,\vvv')\a(\eta',\vuu)\ud{\vuu}}.
\end{align}
Using Lemma \ref{Regularity lemma 2}, we know
\begin{align}
\abs{\bvvp M}&\ls \abs{\tv'}\lnnm{\a}{\vth,\vrh}.
\end{align}
This bound is too weak since we have not used the smallness $\abs{\ua}\leq2\d$, which means the integral is actually over a very small domain. We naturally modify the proof of Lemma \ref{Regularity lemma 2}. The key step is \eqref{rtt 06}. Here for either $\abs{\vuu}\leq 1$ or $\abs{\vuu}\geq1$, the small domain of $\ua$ produces an extra smallness in integral. In precise,
\begin{align}\label{rtt 40}
\abs{\bvvp M}&\ls \d\abs{\tv'}\lnnm{\a}{\vth,\vrh}.
\end{align}
Here, this $\abs{\tv'}$ will be handled in outer integral of \eqref{rtt 35} as in \eqref{rtt 38},
\begin{align}
\int_0^{\eta}\frac{\abs{\tv'}}{\va'}\exp(-H_{\eta,\eta'})\ud{\eta'}\ls \int_0^{\eta}\frac{\nu(\vvv')}{\va'}\exp(-H_{\eta,\eta'})\ud{\eta'}\ls 1.
\end{align}
In total, we have
\begin{align}\label{rtt 49}
\abs{\bvv I_3}\ls \d\lnnmv{\a}.
\end{align}
\ \\
Step 4: Estimate of $I_4$: $0\leq\va\leq\d_0$, with $\chi(\ua)$ term, $\sqrt{\e\eta'}\abs{\tv'}\leq\va'$ and $\va^2\leq\e(\eta-\eta')\abs{\tv}^2$.\\
$I_4$ is defined similar as \eqref{rtt 35}. Based on \eqref{ptt 01}, $\sqrt{\e\eta'}\abs{\tv'}\leq\va'$ implies
\begin{align}\label{rtt 41}
\zeta(\eta',\vvv')\ls \abs{\va'}+\sqrt{\e\eta'}\abs{\tv'}\ls \va'.
\end{align}
Hence, similar to the derivation for $I_3$ in \eqref{rtt 39} and \eqref{rtt 40}, using \eqref{rtt 37} and \eqref{rtt 41}, we have
\begin{align}
M\ls \frac{\va'}{\sqrt{\e\eta'}}\abs{\int_{\r^3}\frac{1}{\abs{\vuu}}\chi(\ua)k(\vuu,\vvv')\a(\eta',\vuu)\ud{\vuu}},
\end{align}
and
\begin{align}
\abs{\bvvp M}\ls \d\frac{\va'}{\sqrt{\e\eta'}}\lnnm{\a}{\vth,\vrh}.
\end{align}
Hence, we must handle $\dfrac{\va'}{\sqrt{\e\eta'}}$ with the outer integral in \eqref{rtt 35}. Based on \eqref{ptt 03}, $\va^2\leq\e(\eta-\eta')\abs{\tv}^2$ leads to
\begin{align}
-H_{\eta,\eta'}=-\int_{\eta'}^{\eta}\frac{\nu(\vvv)}{\va'(y)}\ud{y}\ls&-\frac{\nu(\vvv)(\eta-\eta')}{\abs{\tv}\sqrt{\e(\eta-\eta')}}
\ls-\frac{\nu(\vvv)}{\abs{\tv}}\sqrt{\frac{\eta-\eta'}{\e}}.
\end{align}
Therefore, we know
\begin{align}\label{rtt 44}
&\int_0^{\eta}\dfrac{\va'}{\sqrt{\e\eta'}}\frac{1}{\va'}\exp(-H_{\eta,\eta'})\ud{\eta'}=\int_0^{\eta}\dfrac{1}{\sqrt{\e\eta'}}\exp(-H_{\eta,\eta'})\ud{\eta'}\\
\ls& \int_0^{\eta}\frac{1}{\sqrt{\e\eta'}}\exp\bigg(-\frac{\nu(\vvv)}{\abs{\tv}}\sqrt{\frac{\eta-\eta'}{\e}}\bigg)\ud{\eta'}
=\int_0^{\frac{\eta}{\e}}\frac{1}{\sqrt{z}}\exp\bigg(-\frac{\nu(\vvv)}{\abs{\tv}}\sqrt{\frac{\eta}{\e}-z}\bigg)\ud{z}\no\\
=&\int_0^{1}\frac{1}{\sqrt{z}}\exp\bigg(-\frac{\nu(\vvv)}{\abs{\tv}}\sqrt{\frac{\eta}{\e}-z}\bigg)\ud{z}
+\int_1^{\frac{\eta}{\e}}\frac{1}{\sqrt{z}}\exp\bigg(-\frac{\nu(\vvv)}{\abs{\tv}}\sqrt{\frac{\eta}{\e}-z}\bigg)\ud{z},
\end{align}
where we define substitution $\eta'\rt z=\dfrac{\eta'}{\e}$, which implies $\ud{\eta'}=\e\ud{z}$. We can estimate these two terms separately.
\begin{align}\label{rtt 42}
\int_0^{1}\frac{1}{\sqrt{z}}\exp\bigg(-\frac{\nu(\vvv)}{\abs{\tv}}\sqrt{\frac{\eta}{\e}-z}\bigg)\ud{z}\leq&\int_0^{1}\frac{1}{\sqrt{z}}\ud{z}=2.
\end{align}
\begin{align}\label{rtt 43}
\int_1^{\frac{\eta}{\e}}\frac{1}{\sqrt{z}}\exp\bigg(-\frac{\nu(\vvv)}{\abs{\tv}}\sqrt{\frac{\eta}{\e}-z}\bigg)\ud{z}
&\leq\int_1^{\frac{\eta}{\e}}\exp\bigg(-\frac{\nu(\vvv)}{\abs{\tv}}\sqrt{\frac{\eta}{\e}-z}\bigg)\ud{z}\\
&\overset{t^2=\frac{\eta}{\e}-z}{\ls}\int_0^{\infty}t\ue^{-\frac{\nu(\vvv)}{\abs{\tv}}t}\ud{t}\ls \left(\frac{\abs{\tv}}{\nu(\vvv)}\right)^2\ls 1.\no
\end{align}
Inserting \eqref{rtt 42} and \eqref{rtt 43} into \eqref{rtt 44}, we know the outer integral in \eqref{rtt 35} is bounded. Therefore, we have
\begin{align}\label{rtt 50}
\abs{\bvv I_4}\ls \d\lnnmv{\a}.
\end{align}
\ \\
Step 5: Estimate of $I_5$: $0\leq\va\leq\d_0$, with $\chi(\ua)$ term, $\sqrt{\e\eta'}\abs{\tv'}\leq\va'$, $\va^2\geq\e(\eta-\eta')\abs{\tv}^2$.\\
$I_5$ is defined similar as \eqref{rtt 35}. Using \eqref{rtt 41}, we have
\begin{align}
M\ls \abs{\int_{\r^3}\frac{\va'}{\zeta(\eta',\vuu)}\chi(\ua)k(\vuu,\vvv')\a(\eta',\vuu)\ud{\vuu}}.
\end{align}
Using Lemma \ref{Regularity lemma 3}, we may bound
\begin{align}
\abs{\bvvp M}&\ls \va'\Big(1+\abs{\ln(\e\eta')}\Big)\lnnmv{\a}.
\end{align}
Hence, we must handle $\va'\Big(1+\abs{\ln(\e\eta')}\Big)$ with the outer integral in \eqref{rtt 35}. Based on \eqref{ptt 03}, $\va^2\geq\e(\eta-\eta')\abs{\tv}^2$ implies
\begin{align}
-H_{\eta,\eta'}=-\int_{\eta'}^{\eta}\frac{\nu(\vvv')}{\va'(y)}\ud{y}\ls&-\frac{\nu(\vvv')(\eta-\eta')}{\va}.
\end{align}
Therefore, we know
\begin{align}\label{rtt 45}
\int_0^{\eta}\va'\Big(1+\abs{\ln(\e\eta')}\Big)\frac{1}{\va'}\exp(-H_{\eta,\eta'})\ud{\eta'}&=\int_0^{\eta}\Big(1+\abs{\ln(\e\eta')}\Big)\exp(-H_{\eta,\eta'})\ud{\eta'}\\
&\ls \int_0^{\eta}\Big(1+\abs{\ln(\e\eta')}\Big)\exp\bigg(-\frac{\nu(\vvv')(\eta-\eta')}{\va}\bigg)\ud{\eta'}.\no
\end{align}
Naturally,
\begin{align}
1+\abs{\ln(\e\eta')}\ls \Big(1+\abs{\ln(\e)}\Big)+\abs{\ln(\eta')}.
\end{align}
Since $0\leq\va\leq\d_0$, direct computation reveals that
\begin{align}\label{rtt 46}
\int_0^{\eta}\Big(1+\abs{\ln(\e)}\Big)\exp\bigg(-\frac{\nu(\vvv')(\eta-\eta')}{\va}\bigg)\ud{\eta'}\ls \Big(1+\abs{\ln(\e)}\Big)\frac{\va}{\nu(\vvv)}
\ls \d_0\Big(1+\abs{\ln(\e)}\Big).
\end{align}
Hence, it suffices to consider
\begin{align}
Q=\int_0^{\eta}\abs{\ln(\eta')}\exp\bigg(-\frac{\nu(\vvv')(\eta-\eta')}{\va}\bigg)\ud{\eta'}.
\end{align}
If $0\leq\eta\leq 2$, applying H\"older's inequality, we have
\begin{align}
Q&\ls \bigg(\int_0^2\abs{\ln(\eta')}^2\ud\eta'\bigg)^{\frac{1}{2}}\bigg(\int_0^2\exp\bigg(-\frac{2\nu(\vvv')(\eta-\eta')}{\va}\bigg)\ud{\eta'}\bigg)^{\frac{1}{2}}
\ls \sqrt{\frac{\va}{\nu(\vvv)}}\ls \sqrt{\d_0}.
\end{align}
If $2\leq\eta\leq L=\e^{-\frac{1}{2}}$, we decompose and apply H\"older's inequality to obtain
\begin{align}
Q&\ls \int_0^{2}\abs{\ln(\eta')}\exp\bigg(-\frac{\nu(\vvv')(\eta-\eta')}{\va}\bigg)\ud{\eta'}+\int_2^{\eta}\abs{\ln(\eta')}\exp\bigg(-\frac{\nu(\vvv')(\eta-\eta')}{\va}\bigg)\ud{\eta'}\\
&\ls \bigg(\int_0^2\abs{\ln(\eta')}^2\ud\eta'\bigg)^{\frac{1}{2}}\bigg(\int_0^2\exp\bigg(-\frac{2\nu(\vvv')(\eta-\eta')}{\va}\bigg)\ud{\eta'}\bigg)^{\frac{1}{2}}
+\ln(L)\int_2^{\eta}\exp\bigg(-\frac{\nu(\vvv')(\eta-\eta')}{\va}\bigg)\ud{\eta'}\no\\
&\ls \sqrt{\d_0}\Big(1+\abs{\ln(\e)}\Big).\no
\end{align}
In summary, we have
\begin{align}
\int_0^{\eta}\abs{\ln(\eta')}\exp\bigg(-\frac{\nu(\vvv')(\eta-\eta')}{\va}\bigg)\ud{\eta'}\ls \sqrt{\d_0}\Big(1+\abs{\ln(\e)}\Big).
\end{align}
This completes the bound of outer integral of \eqref{rtt 35}. Hence, we know
\begin{align}\label{rtt 51}
\abs{\bvv I_5}\ls \sqrt{\d_0}\Big(1+\abs{\ln(\e)}\Big)\lnnmv{\a}.
\end{align}
\ \\
Step 6: Synthesis.\\
Collecting all estimates related to $I_i$ in \eqref{rtt 47}, \eqref{rtt 48}, \eqref{rtt 49}, \eqref{rtt 50} and \eqref{rtt 51}, we have proved
\begin{align}\label{rtt 74}
\abs{\bvv I}\ls& \bigg(\d+\sqrt{\d_0}\Big(1+\abs{\ln(\e)}\Big)\bigg)\lnnmv{\a}+\bigg(\frac{\e}{\d_0^3}+\frac{1}{\d_0^4}+\frac{\e}{\d^2}+\frac{1}{\d}\bigg).
\end{align}

\subsubsection{Region II: $\va<0$ and $\va^2+\vb^2+\vc^2\geq \vb'^2(L)+\vc'^2(L)$}

Based on \eqref{mtt 27}, we only need to estimate
\begin{align}
\t[\tilde\a]=&\int_0^{L}\frac{\tilde\a\Big(\eta',\vvv'(\eta,\vvv;\eta')\Big)}{\va'(\eta,\vvv;\eta')}
\exp(-H_{L,\eta'}-\rr[H_{L,\eta}])\ud{\eta'}\\
&+\int_{\eta}^{L}\frac{\tilde\a\Big(\eta',\rr[\vvv'(\eta,\vvv;\eta')]\Big)}{\va'(\eta,\vvv;\eta')}\exp(\rr[H_{\eta,\eta'}])\ud{\eta'}.\no
\end{align}
Here $\rr[H]=H$ just for clarification. Notice that
\begin{align}\label{rtt 52}
\exp(-H_{L,\eta'}-\rr[H_{L,\eta}])\ls \exp(-\rr[H_{\eta',\eta}]).
\end{align}
Also, we can decompose
\begin{align}\label{rtt 53}
\t[\tilde\a]=&\int_0^{\eta}\frac{\tilde\a\Big(\eta',\vvv'(\eta')\Big)}{\va'(\eta')}\exp(-H_{L,\eta'}-\rr[H_{L,\eta}])\ud{\eta'}\\
&+\int_{\eta}^{L}\frac{\tilde\a\Big(\eta',\vvv'(\eta')\Big)}{\va'(\eta')}\exp(-H_{L,\eta'}-\rr[H_{L,\eta}])\ud{\eta'}
+\int_{\eta}^{L}\frac{\tilde\a\Big(\eta',\rr[\vvv'(\eta')]\Big)}{\va'(\eta')}\exp(\rr[H_{\eta,\eta'}])\ud{\eta'}.\no
\end{align}
The integral $\displaystyle\int_0^{\eta}$ part can be estimated as in Region I due to \eqref{rtt 52}, so we focus on the integral $\displaystyle\int_{\eta}^L$ part.
Also, due to \eqref{rtt 52}, it suffices to estimate
\begin{align}
II=\int_{\eta}^{L}\frac{\tilde\a\Big(\eta',\vvv'(\eta')\Big)}{\va'(\eta')}\exp(-H_{\eta',\eta})\ud{\eta'}.
\end{align}
Here the proof is almost identical to that in Region I, so we only point out the key differences.\\
\ \\
Step 0: Preliminaries.\\
\eqref{ptt 01} and \eqref{ptt 02} still holds, but the key result \eqref{ptt 03} needs to be updated. For $0\leq\eta\leq\eta'$,
\begin{align}
\va'=&\sqrt{E_1-\vb'^2-\vc'^2}=\sqrt{\va^2+\vb^2+\vc^2-\bigg(\frac{R_1-\e\eta}{R_1-\e\eta'}\bigg)^2\vb^2-\bigg(\frac{R_2-\e\eta}{R_2-\e\eta'}\bigg)^2\vc^2}\leq\abs{\va}.
\end{align}
Then we have
\begin{align}\label{ptt 08}
-\int_{\eta}^{\eta'}\frac{1}{\va'(y)}\ud{y}\leq &-\int_{\eta}^{\eta'}\frac{1}{\abs{\va}}\ud{y}=-\frac{\eta'-\eta}{\abs{\va}}.
\end{align}
Here, note that $\va<0$ but $\va'\geq0$ defined in \eqref{mtt 83}.\\
\ \\
Step 1: Estimate of $II_1$: $\va\leq-\d_0$ and $\va'\geq\dfrac{\d_0}{2}$ for all $\eta'\in[0,L]$.\\
Since $\eta'\geq\eta$, we must have $\va'\leq \abs{\va}$, so it is unclear whether $\abs{\va'}\geq\dfrac{\d_0}{2}$ directly from $\va\leq\d_0$. Hence, we must put this as an additional requirement. If there exists some $\va'\leq\dfrac{\d_0}{2}$, it will be handled in $II_5$ estimate later. As for the estimate, this is very similar to the estimate of $I_1$. We will use the mild formulation of $\gg$ in \eqref{mtt 27} instead of $\a$ in \eqref{normal derivative equation}.
\begin{align}
\abs{\bvv II_1} \ls&\abs{\zeta}\abs{\bvv\frac{\p\gg}{\p\eta}}\ls \abs{\bw{\vrh}{\vth+1}\frac{\p\gg}{\p\eta}}.
\end{align}
Hence, the key is to estimate $\dfrac{\p\gg}{\p\eta}$. As in \eqref{mtt 27}, we rewrite the equation \eqref{Milne transform} along the characteristics as
It suffices to consider
\begin{align}\label{rtt 54}
\gg(\eta,\vvv)\sim &\exp\left(H_{\eta,0}\right)\int_{\eta}^L\frac{(K[\gg]+S)\Big(\eta',\vvv'(\eta')\Big)}{\va'(\eta')}
\exp\left(-H_{\eta',0}\right)\ud{\eta'}.
\end{align}
where $\sim$ denotes that we only focus on $\ds\int_{\eta}^L$ part due to the decomposition as in \eqref{rtt 53}. The justification of $\ds\int_0^{\eta}$ and boundary data $p$ part is covered by the estimate of $I_1$ and \eqref{rtt 52}.

Taking $\eta$ derivative on both sides of \eqref{rtt 54}, we have
\begin{align}
\frac{\p\gg}{\p\eta}:=Y_1+Y_2+Y_3+Y_4+Y_5,
\end{align}
where
\begin{align}
Y_1=&\exp\left(H_{\eta,0}\right)\frac{\p{H_{\eta,0}}}{\p{\eta}}\int_{\eta}^L\frac{K[\gg]\Big(\eta',\vvv'(\eta')\Big)}{\va'(\eta')}
\exp\left(-H_{\eta',0}\right)\ud{\eta'},\\
Y_2=&\frac{(K[\gg]+S)(\eta,\vvv)}{\va},\\
Y_3=&-\exp\left(H_{\eta,0}\right)\int_{\eta}^L\Bigg((K[\gg]+S)\Big(\eta',\vvv'(\eta')\Big)
\exp\left(-H_{\eta',0}\right)\frac{1}{\va'^2(\eta')}\frac{\p\va'(\eta')}{\p\eta}\ud{\eta'}\bigg)\\
Y_4=&-\exp\left(H_{\eta,0}\right)\int_{\eta}^L\frac{(K[\gg]+S)\Big(\eta',\vvv'(\eta')\Big)}{\va'(\eta')}
\exp\left(-H_{\eta',0}\right)\frac{\p{H_{\eta',0}}}{\p{\eta}}\ud{\eta'},\\
Y_5=&\exp\left(H_{\eta,0}\right)
\int_{\eta}^L\frac{1}{\va'(\eta')}\bigg(\nabla_{\vvv'} (K[\gg]+S)\Big(\eta',\vvv'(\eta')\Big)\frac{\p\vvv'(\eta')}{\p\eta}\bigg)
\exp\left(-H_{\eta',0}\right)\ud{\eta'}.
\end{align}
We need to estimate each term. Below are some preliminary results, which are direct extension of \eqref{rtt 24}, \eqref{rtt 27}, \eqref{rtt 25} and \eqref{rtt 26}:
\begin{itemize}
\item
For $\eta'\geq\eta$, we must have $\dfrac{1}{\va'}\ls\dfrac{1}{\d_0}$.
\item
Using substitution $y=H_{\eta,\eta'}$, we know
\begin{align}\label{rtt 24'}
\abs{\int_{\eta}^L\frac{\nu\Big(\vvv'(\eta')\Big)}{\va'(\eta')}\exp(H_{\eta,\eta'})\ud{\eta'}}\leq\abs{\int_{-\infty}^0\ue^{y}\ud{y}}=1.
\end{align}
\item
For $t,s\in[\eta,L]$, based on (PSF), we have
\begin{align}\label{rtt 27'}
\abs{{H_{t,s}}}\ls\frac{\abs{\vvv}}{\d_0}\abs{t-s}.
\end{align}
\item
We have
\begin{align}\label{rtt 25'}
\abs{\frac{\p\vb'(\eta')}{\p\eta}}\ls \e\abs{\vvv},\quad\abs{\frac{\p\vc'(\eta')}{\p\eta}}\ls\e\abs{\vvv},\quad\abs{\frac{\p\va'(\eta')}{\p\eta}}\ls \frac{\e\abs{\vvv}^2}{\va'(\eta')}\ls \frac{\e\abs{\vvv}^2}{\d_0}.
\end{align}
\item
For $t,s\in[\eta,L]$, we obtain
\begin{align}\label{rtt 26'}
\abs{\frac{\p{H_{t,s}}}{\p{\eta}}}&\ls\frac{\e\br{\vvv}^3}{\d_0^3}\abs{t-s}\ls\frac{\e L\br{\vvv}^3}{\d_0^3}\ls\frac{\br{\vvv}^3}{\d_0^3}.
\end{align}
\end{itemize}
The estimate of $Y_i$ is standard based on (PSF) and the above preliminaries. Using Lemma \ref{Regularity lemma 1'} and \eqref{rtt 26'}, we have
\begin{align}
\abs{\bvv Y_1}&\ls\abs{\frac{\p{H_{\eta,0}}}{\p{\eta}}}\abs{\int_{\eta}^L\frac{\nu\Big(\vvv'(\eta')\Big)}{\va'(\eta')}\exp(H_{\eta,\eta'})\ud{\eta'}}
\abs{\bvv\nu^{-1}\gg}\\
&\ls\frac{\br{\vvv}^3}{\d_0^3}\abs{\bvv\nu^{-1}\gg}\ls \frac{1}{\d_0^3}\lnnm{\gg}{\vth+2,\vrh}\ls \frac{1}{\d_0^3}.\no
\end{align}
Also, using \eqref{Milne bound} and Lemma \ref{Regularity lemma 1'}, we have
\begin{align}
\abs{\bvv Y_2}&\ls\abs{\frac{1}{\va}}\bigg(\abs{\bvv K[\gg]}+\abs{\bvv S}\bigg)\ls \frac{1}{\d_0}\bigg(1+\abs{\bvv\nu^{-1}\gg}\bigg)\ls \frac{1}{\d_0}.
\end{align}
On the other hand, using \eqref{rtt 25'}, \eqref{rtt 24'} and \eqref{Milne bound}, we obtain
\begin{align}
\abs{\bvv Y_3}&\ls\int_{\eta}^L\bigg(\abs{\bvvp K[\gg]}+\abs{\bvvp S}\bigg)\exp\left(H_{\eta,\eta'}\right)\frac{1}{\d_0^2}\frac{\e\abs{\vvv}^2}{\d_0}\ud\eta'\\
&\ls\frac{\e}{\d_0^3}\bigg(\lnnm{\nu^{-1}\gg}{\vth+2,\vrh}+\lnnm{S}{\vth+2,\vrh}\bigg)\bigg(\int_{\eta}^L\exp\left(H_{\eta,\eta'}\right)\ud\eta'\bigg)\ls\frac{\e}{\d_0^3}.\no
\end{align}
Using \eqref{rtt 26'}, \eqref{rtt 24'} and \eqref{Milne bound}, we know
\begin{align}
\abs{\bvv Y_4}&\ls\int_{\eta}^L\frac{1}{\d_0}\bigg(\abs{\bvvp K[\gg]}+\abs{\bvvp S}\bigg)\exp\left(H_{\eta,\eta'}\right)\frac{\br{\vvv}^3}{\d_0^3}\ud\eta'\\
&\ls\frac{1}{\d_0^4}\bigg(\lnnm{\nu^{-1}\gg}{\vth+3,\vrh}+\lnnm{S}{\vth+3,\vrh}\bigg)\bigg(\int_{\eta}^L\exp\left(H_{\eta,\eta'}\right)\ud\eta'\bigg)\ls\frac{1}{\d_0^4}.\no
\end{align}
Finally, using \eqref{rtt 25'}, \eqref{rtt 24'} and \eqref{Regularity bound}, we have
\begin{align}
\\
\abs{\bvv Y_5}&\ls\int_{\eta}^L\frac{1}{\d_0}\bigg(\abs{\bvvp \nabla_{\vvv'}K[\gg]}+\abs{\bvvp \nabla_{\vvv'}S}\bigg)\exp\left(H_{\eta,\eta'}\right)\bigg(\e\abs{\vvv}+\frac{\e\abs{\vvv}^2}{\d_0}\bigg)\ud\eta'\no\\
&\ls\frac{\e}{\d_0^3}\bigg(\lnnm{\gg}{\vth+2,\vrh}+\lnnm{S}{\vth+2,\vrh}\bigg)\bigg(\int_{\eta}^L\exp\left(H_{\eta,\eta'}\right)\ud\eta'\bigg)\ls\frac{\e}{\d_0^3}.\no
\end{align}
Collecting all $Y_i$ estimates, we have
\begin{align}\label{rtt 66}
\abs{\bvv II_1} \ls& \frac{\e}{\d_0^3}+\frac{1}{\d_0^4}.
\end{align}
\ \\
Step 2: Estimate of $II_2$: $-\d_0\leq\va\leq0$ with $1-\chi(\ua)$ term.\\
We decompose $1=\Big(1-\chi(\ua)\Big)+\chi(\ua)$.
\begin{align}
II_2:=&\int_{\eta}^L\bigg(\int_{\r^3}\frac{\zeta(\eta',\vvv')}{\zeta(\eta',\vuu)}\Big(1-\chi(\ua)\Big)
k(\vuu,\vvv')\a(\eta',\vuu)\ud{\vuu}\bigg)
\frac{1}{\va'}\exp(H_{\eta,\eta'})\ud{\eta'}.
\end{align}
Then by a similar argument as estimating $I_2$, we have
\begin{align}\label{rtt 67}
\abs{\bvv II_2}\ls&\frac{1}{\d}+\frac{\e}{\d^2}.
\end{align}
\ \\
Step 3: Estimate of $II_3$: $-\d_0\leq\va\leq0$, with $\chi(\ua)$ term and $\sqrt{\e\eta'}\vb'\geq\va'$.\\
This is similar to the estimate of $I_3$, we have
\begin{align}\label{rtt 68}
\abs{\bvv II_3}\ls \d\lnnmv{\a}.
\end{align}
\ \\
Step 4: Estimate of $II_4$: $-\d_0\leq\va\leq0$, with $\chi(\ua)$ term, and $\sqrt{\e\eta'}\vb'\leq\va'$.\\
This step is different. We do not need to further decompose the cases like $I_4$ and $I_5$.
Based on \eqref{ptt 08}, we have,
\begin{align}
-H_{\eta,\eta'}\ls &-\frac{\nu(\vvv)(\eta'-\eta)}{\va}.
\end{align}
Then following the same argument in estimating $I_5$, we know
\begin{align}\label{rtt 69}
\abs{\bvv II_4}\ls \sqrt{\d_0}\Big(1+\abs{\ln(\e)}\Big)\lnnmv{\a}.
\end{align}
\ \\
Step 5: Estimate of $II_5$: $\va\leq-\d_0$ and $\va'\leq\dfrac{\d_0}{2}$ for some $\eta'\in[0,L]$.\\
Now we come back to study the leftover in Step 1, i.e. though the characteristic starts from a point with $\abs{\va}\geq\d_0$, as it goes, we finally arrive at the region that $\va'\leq\dfrac{\d_0}{2}$.

Let $\left(\eta^{\ast},-\dfrac{\d_0}{2},\vb^{\ast},\vc^{\ast}\right)$ be on the same characteristic as $(\eta,\vvv)$, i.e. this is the first time that the characteristic enters the region $\va'\leq\dfrac{\d_0}{2}$. In detail, we have
\begin{align}
&\vb^{\ast}=\frac{R_1-\e\eta}{R_1-\e\eta^{\ast}}\vb,\quad \vc^{\ast}=\frac{R_2-\e\eta}{R_2-\e\eta^{\ast}}\vc,\label{rtt 55}\\
&\va^2+\vb^2+\vc^2=\dfrac{\d_0^2}{4}+\left(\frac{R_1-\e\eta}{R_1-\e\eta^{\ast}}\right)^2\vb^2+\left(\frac{R_2-\e\eta}{R_2-\e\eta^{\ast}}\right)^2\vc^2.\label{rtt 56}
\end{align}
Taking $\eta$ derivative in \eqref{rtt 56}, we obtain
\begin{align}\label{rtt 57}
\frac{\p\eta^{\ast}}{\p\eta}&=\frac{\dfrac{R_1-\e\eta}{(R_1-\e\eta^{\ast})^2}\vb^2+\dfrac{R_2-\e\eta}{(R_2-\e\eta^{\ast})^2}\vc^2}
{\dfrac{(R_1-\e\eta)^2}{(R_1-\e\eta^{\ast})^3}\vb^2+\dfrac{(R_2-\e\eta)^2}{(R_2-\e\eta^{\ast})^3}\vc^2},
\end{align}
Here we do not need to compute $\eta^{\ast}$ explicitly. Since $\eta<\eta^{\ast}\leq L$, we know $0\leq\e\eta<\e\eta^{\ast}\leq \e L=\e^{\frac{1}{2}}$, which implies
\begin{align}\label{rtt 58}
\frac{R_1}{2}\leq R_1-\e\eta^{\ast}<R_1-\e\eta\leq R_1,\quad \frac{R_2}{2}\leq R_2-\e\eta^{\ast}<R_2-\e\eta\leq R_2.
\end{align}
Inserting \eqref{rtt 58} into \eqref{rtt 57}, we have
\begin{align}\label{rtt 59}
\abs{\frac{\p\eta^{\ast}}{\p\eta}}\ls 1.
\end{align}
Taking $\eta$ derivative in \eqref{rtt 55} and using \eqref{rtt 59} and \eqref{rtt 58}, we obtain
\begin{align}\label{rtt 60}
\abs{\frac{\p\vb^{\ast}}{\p\eta}}=\e\abs{\vb}\abs{\frac{R_1-\e\eta}{(R_1-\e\eta^{\ast})^2}\frac{\p\eta^{\ast}}{\p\eta}-\frac{1}{R_1-\e\eta^{\ast}}}\ls \e\nu(\vvv),\\
\abs{\frac{\p\vc^{\ast}}{\p\eta}}=\e\abs{\vc}\abs{\frac{R_2-\e\eta}{(R_2-\e\eta^{\ast})^2}\frac{\p\eta^{\ast}}{\p\eta}-\frac{1}{R_2-\e\eta^{\ast}}}\ls \e\nu(\vvv).\label{rtt 61}
\end{align}
Then we have the mild formulation between $\eta$ and $\eta^{\ast}$ as
\begin{align}\label{rtt 62}
\gg(\eta,\vvv)=\gg\left(\eta^{\ast},-\frac{\d_0}{2},\vb^{\ast},\vc^{\ast}\right)\exp(-H_{\eta^{\ast},\eta})+\int_{\eta}^{\eta^{\ast}}
\frac{(K[\gg]+S)\Big(\eta',\vvv'(\eta,\vvv;\eta')\Big)}{\va'(\eta,\vvv;\eta')}\exp(H_{\eta',\eta})\ud{\eta'}.
\end{align}
Similar to the estimate of $II_1$, taking $\eta$ derivative in \eqref{rtt 62} and multiplying $\zeta$ on both sides, we obtain
\begin{align}\label{rtt 65}
\abs{\bvv II_5}=\abs{\bvv\zeta(\eta,\vvv)\frac{\p\gg}{\p\eta}}\ls \abs{\bvv\zeta(P_1+P_2)},
\end{align}
where
\begin{align}
P_1&=\frac{\p\gg\left(\eta^{\ast},-\frac{\d_0}{2},\vb^{\ast},\vc^{\ast}\right)}{\p\eta}\exp(-H_{\eta^{\ast},\eta}),\\
P_2&=-\gg\left(\eta^{\ast},-\frac{\d_0}{2},\vb^{\ast},\vc^{\ast}\right)\exp(-H_{\eta^{\ast},\eta})\frac{\p H_{\eta^{\ast},\eta}}{\p\eta}\\
&+\frac{\p}{\p\eta}\Bigg(\int_{\eta}^{\eta^{\ast}}
\frac{(K[\gg]+S)\Big(\eta',\vvv'(\eta,\vvv;\eta')\Big)}{\va'(\eta,\vvv;\eta')}\exp(H_{\eta',\eta})\ud{\eta'}\Bigg).\no
\end{align}
Since for $\eta'\in[\eta,\eta^{\ast}]$, we always have $\va'\geq\dfrac{\d_0}{2}$, mimicking Step 1 to estimate $II_1$ and using \eqref{rtt 59}, we may bound
\begin{align}\label{rtt 63}
\abs{\bvv\zeta P_2}\ls \frac{\e}{\d_0^3}+\frac{1}{\d_0^4}.
\end{align}
The key is the estimate of $P_1$: considering $\abs{\exp(-H_{\eta^{\ast},\eta})}\ls 1$ and using \eqref{rtt 59}, \eqref{rtt 60} and \eqref{rtt 61}, we have
\begin{align}
&\abs{\bvv\zeta P_1}\ls\abs{\bv\zeta\frac{\p\gg\left(\eta^{\ast},-\frac{\d_0}{2},\vb^{\ast},\vc^{\ast}\right)}{\p\eta}}\\
\leq&\abs{\bv\zeta\frac{\p\gg\left(\eta^{\ast},-\frac{\d_0}{2},\vb^{\ast},\vc^{\ast}\right)}{\p\eta^{\ast}}\frac{\p\eta^{\ast}}{\p\eta}}+
\abs{\bv\zeta\frac{\p\gg\left(\eta^{\ast},-\frac{\d_0}{2},\vb^{\ast},\vc^{\ast}\right)}{\p \vb^{\ast}}\frac{\p\vb^{\ast}}{\p\eta}}\no\\
&+
\abs{\bv\zeta\frac{\p\gg\left(\eta^{\ast},-\frac{\d_0}{2},\vb^{\ast},\vc^{\ast}\right)}{\p \vc^{\ast}}\frac{\p\vc^{\ast}}{\p\eta}}\no\\
\ls&\abs{\bv\a\left(\eta^{\ast},-\frac{\d_0}{2},\vb^{\ast},\vc^{\ast}\right)}
+\e\lnnmv{\nu\left(\zeta\frac{\p\gg}{\p\vb}\right)}+\e\lnnmv{\nu\left(\zeta\frac{\p\gg}{\p\vc}\right)}.\no
\end{align}
The estimate of $\abs{\bv\a\left(\eta^{\ast},-\frac{\d_0}{2},\vb^{\ast},\vc^{\ast}\right)}$ is achieved as in $II_2,II_3,II_4$ since now $\abs{\va^{\ast}}\leq\dfrac{\d_0}{2}$. However, we have to preserve the latter two terms related to $\dfrac{\p\gg}{\p\vb}$ and $\dfrac{\p\gg}{\p\vc}$. Hence, we have
\begin{align}\label{rtt 64}
\abs{\bvv\zeta P_2}\ls & \bigg(\d+\sqrt{\d_0}\Big(1+\abs{\ln(\e)}\Big)\bigg)\lnnmv{\a}+\bigg(\frac{\e}{\d^2}+\frac{1}{\d}\bigg).\\
&+\e\lnnmv{\nu\left(\zeta\frac{\p\gg}{\p\vb}\right)}+\e\lnnmv{\nu\left(\zeta\frac{\p\gg}{\p\vc}\right)}.\no
\end{align}
Inserting \eqref{rtt 63} and \eqref{rtt 64} into \eqref{rtt 65}, we obtain
\begin{align}\label{rtt 70}
\abs{\bvv II_5}\ls& \bigg(\d+\sqrt{\d_0}\Big(1+\abs{\ln(\e)}\Big)\bigg)\lnnmv{\a}+\bigg(\frac{\e}{\d_0^3}+\frac{1}{\d_0^4}+\frac{\e}{\d^2}+\frac{1}{\d}\bigg)\\
&+\e\lnnmv{\nu\left(\zeta\frac{\p\gg}{\p\vb}\right)}+\e\lnnmv{\nu\left(\zeta\frac{\p\gg}{\p\vc}\right)}.\no
\end{align}
\ \\
Step 6: Synthesis.\\
Collecting all estimates related to $II_i$ in \eqref{rtt 66}, \eqref{rtt 67}, \eqref{rtt 68}, \eqref{rtt 69} and \eqref{rtt 70}, we have proved
\begin{align}\label{rtt 75}
\abs{\bvv II}\ls& \bigg(\d+\sqrt{\d_0}\Big(1+\abs{\ln(\e)}\Big)\bigg)\lnnmv{\a}+\bigg(\frac{\e}{\d_0^3}+\frac{1}{\d_0^4}+\frac{\e}{\d^2}+\frac{1}{\d}\bigg)\\
&+\e\lnnmv{\nu\left(\zeta\frac{\p\gg}{\p\vb}\right)}+\e\lnnmv{\nu\left(\zeta\frac{\p\gg}{\p\vc}\right)}.\no
\end{align}

\subsubsection{Region III: $\va<0$ and $\va^2+\vb^2+\vc^2\leq \vb'^2(L)+\vc'^2(L)$}

Based on \eqref{mtt 28}, we only need to estimate
\begin{align}
III=\t[\tilde\a]=&\int_0^{\eta^+}\frac{\tilde\a\Big(\eta',\vvv'(\eta,\vvv;\eta')\Big)}{\va'(\eta,\vvv;\eta')}
\exp(-H_{\eta^+,\eta'}-\rr[H_{\eta^+,\eta}])\ud{\eta'}\\
&+\int_{\eta}^{\eta^+}\frac{\tilde\a\Big(\eta',\rr[\vvv'(\eta,\vvv;\eta')]\Big)}{\va'(\eta,\vvv;\eta')}\exp(\rr[H_{\eta,\eta'}])\ud{\eta'}.\no
\end{align}
Here $\eta^+$ is defined in \eqref{mtt 84} and $\rr[H]=H$. Notice that
\begin{align}\label{rtt 52'}
\exp(-H_{\eta^+,\eta'}-\rr[H_{\eta^+,\eta}])\ls \exp(-\rr[H_{\eta',\eta}]).
\end{align}
Also, we can decompose
\begin{align}
\t[\tilde\a]=&\int_0^{\eta}\frac{\tilde\a\Big(\eta',\vvv'(\eta')\Big)}{\va'(\eta')}\exp(-H_{\eta^+,\eta'}-\rr[H_{\eta^+,\eta}])\ud{\eta'}\\
&+\int_{\eta}^{\eta^+}\frac{\tilde\a\Big(\eta',\vvv'(\eta')\Big)}{\va'(\eta')}\exp(-H_{\eta^+,\eta'}-\rr[H_{\eta^+,\eta}])\ud{\eta'}\no\\
&+\int_{\eta}^{\eta^+}\frac{\tilde\a\Big(\eta',\rr[\vvv'(\eta')]\Big)}{\va'(\eta')}\exp(\rr[H_{\eta,\eta'}])\ud{\eta'}.\no
\end{align}
Due to \eqref{rtt 52'}, the integral $\displaystyle\int_0^{\eta}$ part can be estimated as in Region I and the integral $\ds\int_{\eta}^{\eta^+}$ part can be estimated as in Region II, so we omit the details here. At the end of the day, we have
\begin{align}\label{rtt 76}
\abs{\bvv III}\ls& \bigg(\d+\sqrt{\d_0}\Big(1+\abs{\ln(\e)}\Big)\bigg)\lnnmv{\a}+\bigg(\frac{\e}{\d_0^3}+\frac{1}{\d_0^4}+\frac{\e}{\d^2}+\frac{1}{\d}\bigg)\\
&+\e\lnnmv{\nu\left(\zeta\frac{\p\gg}{\p\vb}\right)}+\e\lnnmv{\nu\left(\zeta\frac{\p\gg}{\p\vc}\right)}.\no
\end{align}

\subsection{Regularity Estimates}

\subsubsection{Estimates of Normal Derivative}

Collecting estimates \eqref{rtt 74}, \eqref{rtt 75}, \eqref{rtt 76} in these three regions, and inserting \eqref{rtt 71} and \eqref{rtt 72} into \eqref{rtt 73}, we have
\begin{align}\label{ptt 10}
\lnnmv{\a}\ls& \bigg(\d+\sqrt{\d_0}\Big(1+\abs{\ln(\e)}\Big)\bigg)\lnnmv{\a}+\bigg(\frac{\e}{\d_0^3}+\frac{1}{\d_0^4}+\frac{\e}{\d^2}+\frac{1}{\d}\bigg)\\
&+\e\lnnmv{\nu\left(\zeta\frac{\p\gg}{\p\vb}\right)}+\e\lnnmv{\nu\left(\zeta\frac{\p\gg}{\p\vc}\right)}+\lnmh{p_{\a}}+\lnnmv{\nu^{-1}S_{\a}}.\no
\end{align}
Then we choose these constants to perform absorbing argument. First we choose $0<\d<<1$ sufficiently small such that
\begin{align}
C\d\leq\frac{1}{4}.
\end{align}
Then we take $\d_0=\d^2(1+\abs{\ln(\e)})^{-2}$ such that
\begin{align}
C(1+\abs{\ln(\e)})\sqrt{\d_0}\leq C\d\leq\frac{1}{4}.
\end{align}
for $\e$ sufficiently small. Hence, we can absorb all the term related to $\lnnmv{\a}$ on the right-hand side of \eqref{ptt 10} to the left-hand side to obtain the desired result.
\begin{lemma}\label{Regularity lemma 4}
Assume \eqref{Milne bound} and \eqref{Regularity bound} holds. We have
\begin{align}
\lnnmv{\a}\ls&\abs{\ln(\e)}^{8}+\lnmh{p_{\a}}+\lnnmv{\nu^{-1}S_{\a}}
+\e\lnnmv{\nu\left(\zeta\frac{\p\gg}{\p\vb}\right)}+\e\lnnmv{\nu\left(\zeta\frac{\p\gg}{\p\vc}\right)}.
\end{align}
\end{lemma}

\subsubsection{Estimates of Velocity Derivatives}

Taking $\va$ derivative in \eqref{Milne transform} and multiplying $\zeta$ defined in \eqref{weight} on both sides, we obtain the $\e$-transport problem for $\b:=\zeta\dfrac{\p\gg}{\p\va}$ as
\begin{align}
\left\{
\begin{array}{l}\displaystyle
\va\frac{\p\b}{\p\eta}+G_1(\eta)\bigg(\vb^2\dfrac{\p\b }{\p\va}-\va\vb\dfrac{\p\b }{\p\vb}\bigg)+G_2(\eta)\bigg(\vc^2\dfrac{\p\b }{\p\va}-\va\vc\dfrac{\p\b }{\p\vc}\bigg)+\nu\b=\tilde\b+S_{\b},\\\rule{0ex}{1.5em}
\b(0,\vvv)=p_{\b}(\vvv)\ \ \text{for}\ \ \va>0,\\\rule{0ex}{1.5em}
\b(L,\vvv)=-\b(L,\rr[\vvv]),
\end{array}
\right.
\end{align}
where the crucial non-local term
\begin{align}
\tilde\b(\eta,\vvv)=\int_{\r^3}\zeta(\vvv)\p_{\va}k(\vuu,\vvv)\gg(\eta,\vuu)\ud{\vuu}.
\end{align}
Here we utilize Lemma \ref{weight lemma} to move $\zeta$ inside the derivative. $p_{\b}$ and $S_{\b}$ will be specified later. We need to derive the a priori estimate of $\b$.
Compared with $\tilde\a$ defined in \eqref{rtt 77}, the key difference is that $\tilde\b$ does not contain $\b$ directly but rather $\gg$. Hence, we no longer need the analysis in previous sections to tackle the strong singularities. Then directly tracking along the characteristics, by a similar but much simpler argument using Theorem \ref{Milne theorem 4}, Lemma \ref{Regularity lemma 2'} and \eqref{Milne bound}, \eqref{Regularity bound}, we obtain the desired result.
\begin{lemma}\label{Regularity lemma 5}
Assume \eqref{Milne bound} and \eqref{Regularity bound} holds. We have
\begin{align}
\lnnmv{\b}\ls 1+\lnmh{p_{\b}}+\lnnmv{\nu^{-1}S_{\b}}.
\end{align}
\end{lemma}
In a similar fashion, $\c:=\zeta\dfrac{\p\gg}{\p\vb}$ and $\dd:=\zeta\dfrac{\p\gg}{\p\vc}$ can be estimated.
\begin{lemma}\label{Regularity lemma 6}
Assume \eqref{Milne bound} and \eqref{Regularity bound} holds. We have
\begin{align}
\lnnmv{\c}\ls 1+\lnmh{p_{\c}}+\lnnmv{\nu^{-1}S_{\c}}\\
\lnnmv{\dd}\ls 1+\lnmh{p_{\dd}}+\lnnmv{\nu^{-1}S_{\dd}}.
\end{align}
\end{lemma}

\subsubsection{A Priori Estimates}

In this subsection, we combine above a priori estimates of normal and velocity derivatives.
\begin{theorem}\label{Regularity theorem 1}
Assume \eqref{Milne bound} and \eqref{Regularity bound} holds. We have
\begin{align}
\lnnmv{\zeta\frac{\p\gg}{\p\eta}}+\lnnmv{\nu\zeta\frac{\p\gg}{\p\va}}&\ls \abs{\ln(\e)}^8,\\
\lnnmv{\nu\zeta\frac{\p\gg}{\p\vb}}+\lnnmv{\nu\zeta\frac{\p\gg}{\p\vc}}&\ls 1.
\end{align}
\end{theorem}
\begin{proof}
Collecting the estimates for $\a$, $\b$, $\c$ and $\dd$ in Lemma \ref{Regularity lemma 4}, Lemma \ref{Regularity lemma 5}, and Lemma \ref{Regularity lemma 6}, we have
\begin{align}
\lnnmv{\a}\ls&\abs{\ln(\e)}^8+\lnmh{p_{\a}}+\lnnmv{\nu^{-1}S_{\a}}+\e\bigg(\lnnmv{\nu \c}+\lnnmv{\nu\dd}\bigg),\label{rtt 84}\\
\lnnmv{\b}\ls&1+\lnmv{p_{\b}}+\lnnmv{\nu^{-1}S_{\b}},\label{rtt 85}\\
\lnnmv{\c}\ls&1+\lnmv{p_{\c}}+\lnnmv{\nu^{-1}S_{\c}},\label{rtt 86}\\
\lnnmv{\dd}\ls&1+\lnmv{p_{\dd}}+\lnnmv{\nu^{-1}S_{\dd}}.\label{rtt 87}
\end{align}
Now we clear up these boundary terms and source terms. At $\eta=0$, we know $\zeta=\va$. Hence, we may solve from \eqref{Milne transform} to get
\begin{align}
p_{\a}&=\va\frac{\p\gg}{\p\eta}(0,\vvv)=-\frac{\e}{R_1}\bigg(\vb^2\frac{\p p}{\p\va}-\va\vb\frac{\p p}{\p\vb}\bigg)-\frac{\e}{R_2}\bigg(\vc^2\frac{\p p}{\p\va}-\va\vc\frac{\p p}{\p\vc}\bigg)+\nu p-K[\gg](0,\vvv)
\end{align}
Therefore, using Theorem \ref{Milne theorem 4}, Lemma \ref{Regularity lemma 1'}, \eqref{Milne bound} and \eqref{Regularity bound}, we have
\begin{align}\label{rtt 78}
\lnmh{p_{\a}}&\ls \e\lnm{\nabla_{\vvv}p}{\vth+2,\vrh}+\lnm{p}{\vth+1,\vrh}+\lnnm{\nu^{-1}\gg}{\vth,\vrh}\ls 1.
\end{align}
On the other hand, we can directly take derivative in the boundary data $p$ to get
\begin{align}
p_{\b}=\va\frac{\p p}{\p\va},\quad p_{\c}=\va\frac{\p p}{\p\vb},\quad p_{\dd}=\va\frac{\p p}{\p\vc},
\end{align}
which, using \eqref{Regularity bound}, yield
\begin{align}\label{rtt 79}
\lnmh{p_{\b}}+\lnmh{p_{\c}}+\lnmh{p_{\dd}}\ls \lnm{\nabla_{\vvv}p}{\vth+1,\vrh}\ls 1.
\end{align}
Directly Taking $\eta$ and $\vvv$ derivatives on both sides of \eqref{Milne transform} and multiplying $\zeta$, we obtain
\begin{align}
S_{\a}=&\frac{\ud{G_1}}{\ud{\eta}}\bigg(\vb^2\b-\va\vb\c\bigg)+\frac{\ud{G_2}}{\ud{\eta}}\bigg(\vc^2\b-\va\vc\dd\bigg),\\
S_{\b}=&\a-G_1\vb\c-G_2\vc\dd,\quad
S_{\c}=G_1\bigg(2\vb\b-\va\c\bigg),\quad
S_{\dd}=G_2\bigg(2\vc\b-\va\dd\bigg).
\end{align}
Note that fact that $\abs{G_1}+\abs{G_2}\ls\e$ and $\abs{\dfrac{\ud G_1}{\ud\eta}}+\abs{\dfrac{\ud G_2}{\ud\eta}}\ls \e^2$. We have
\begin{align}
\lnnmv{\nu^{-1}S_{\a}}&\ls \e^2\bigg(\lnnmv{\nu\b}+\lnnmv{\nu\c}+\lnnmv{\nu\dd}\bigg),\label{rtt 80}\\
\lnnmv{\nu^{-1}S_{\b}}&\ls\lnnmv{\nu^{-1}\a}+\e\bigg(\lnnmv{\c}+\lnnmv{\dd}\bigg),\label{rtt 81}\\
\lnnmv{\nu^{-1}S_{\c}}&\ls \e\bigg(\lnnmv{\b}+\lnnmv{\c}\bigg),\label{rtt 82}\\
\lnnmv{\nu^{-1}S_{\dd}}&\ls \e\bigg(\lnnmv{\b}+\lnnmv{\dd}\bigg).\label{rtt 83}
\end{align}
Inserting \eqref{rtt 79} and \eqref{rtt 82} into \eqref{rtt 86}, and absorbing $\e\lnnmv{\c}$ into the left-hand side, we get
\begin{align}\label{rtt 88}
\lnnmv{\c}\ls&1+\e\lnnmv{\b}.
\end{align}
Similarly, inserting \eqref{rtt 79} and \eqref{rtt 83} into \eqref{rtt 87}, and absorbing $\e\lnnmv{\dd}$ into the left-hand side, we get
\begin{align}\label{rtt 89}
\lnnmv{\dd}\ls&1+\e\lnnmv{\b}.
\end{align}
Inserting \eqref{rtt 88} and \eqref{rtt 89} into \eqref{rtt 81}, and further with \eqref{rtt 79} into \eqref{rtt 85}, after absorbing $\e^2\lnnmv{\b}$ into the left-hand side, we have
\begin{align}\label{rtt 90}
\lnnmv{\b}\ls&1+\lnnmv{\nu^{-1}\a}.
\end{align}
Then inserting \eqref{rtt 90} into \eqref{rtt 88} and \eqref{rtt 89}, we obtain
\begin{align}\label{rtt 91}
\lnnmv{\c}\ls&1+\e\lnnmv{\nu^{-1}\a},\quad \lnnmv{\dd}\ls1+\e\lnnmv{\nu^{-1}\a}.
\end{align}
Finally, inserting \eqref{rtt 90} and \eqref{rtt 91} into \eqref{rtt 80}, and further with \eqref{rtt 78} into \eqref{rtt 84}, after absorbing $\e^2\lnnmv{\a}$ into the left-hand side, we obtain
\begin{align}\label{rtt 92}
\lnnmv{\a}\ls&\abs{\ln(\e)}^8.
\end{align}
Hence, inserting \eqref{rtt 92} into \eqref{rtt 90} and \eqref{rtt 91}, we get the desired result.
\end{proof}

\begin{remark}
The estimates of weighted velocity derivatives $\zeta\frac{\p\gg}{\p\va}$, $\zeta\frac{\p\gg}{\p\vb}$ and $\zeta\frac{\p\gg}{\p\vc}$ have an extra $\nu$ in the estimates. This is crucial for the tangential derivative estimates.
\end{remark}

\begin{theorem}\label{Regularity theorem 2}
Assume \eqref{Milne bound} and \eqref{Regularity bound} holds. For $K_0>0$ sufficiently small, we have
\begin{align}
\lnnmv{\ue^{K_0\eta}\zeta\frac{\p\gg}{\p\eta}}+\lnnmv{\ue^{K_0\eta}\nu\zeta\frac{\p\gg}{\p\va}}&\ls\abs{\ln(\e)}^8,\\
\lnnmv{\ue^{K_0\eta}\nu\zeta\frac{\p\gg}{\p\vb}}+\lnnmv{\ue^{K_0\eta}\nu\zeta\frac{\p\gg}{\p\vc}}&\ls 1.
\end{align}
\end{theorem}
\begin{proof}
This proof is almost identical to that of Theorem \ref{Regularity theorem 1}. In each step, we need to multiple $\ue^{K_0\eta}$ on both sides (sometimes inside the integral). When $K_0$ is sufficiently small, we can close the proof.
\end{proof}

\begin{corollary}\label{Regularity corollary}
Assume \eqref{Milne bound} and \eqref{Regularity bound} holds. We have
\begin{align}
\e\lnnmv{\ue^{K_0\eta}\vb^2\frac{\p\gg}{\p\va}}+\e\lnnmv{\ue^{K_0\eta}\vc^2\frac{\p\gg}{\p\va}}\ls \abs{\ln(\e)}^8.
\end{align}
\end{corollary}
\begin{proof}
We rearrange the terms in \eqref{Milne transform} to obtain
\begin{align}
\Big(G_1\vb^2+G_2\vc^2\Big)\frac{\p\gg}{\p\va}=\Big(S-\nu\gg+K[\gg]\Big)-\va\frac{\p\gg}{\p\eta}+G_1\va\vb\frac{\p\gg}{\p\vb}+G_2\va\vc\frac{\p\gg}{\p\vc}.
\end{align}
Recall $\zeta$ definition in \eqref{weight}, we know $\abs{\va}\leq\zeta$. Therefore, using \eqref{Milne bound} and Theorem \ref{Regularity theorem 1}, we know
\begin{align}\label{rtt 93}
&\lnnmv{\Big(G_1\vb^2+G_2\vc^2\Big)\frac{\p\gg}{\p\va}}\\
&\ls \lnnmv{S-\nu\gg+K[\gg]}+\lnnmv{\zeta\frac{\p\gg}{\p\eta}}+\lnnmv{G_1\nu\zeta\frac{\p\gg}{\p\vb}}+\lnnmv{G_2\nu\zeta\frac{\p\gg}{\p\vc}}\no\\
&\ls \lnnmv{S}+\lnnm{\gg}{\vth+2,\vrh}+\lnnmv{\zeta\frac{\p\gg}{\p\eta}}+\e\lnnmv{\nu\zeta\frac{\p\gg}{\p\vb}}+\e\lnnmv{\nu\zeta\frac{\p\gg}{\p\vc}}\ls \abs{\ln(\e)}^8.\no
\end{align}
Since $G_1$ and $G_2$ have the same sign and $\e\ls \abs{G_1}\ls\e$, $\e\ls \abs{G_2}\ls\e$, we can separate the two terms in the left-hand side of \eqref{rtt 93} to obtain
\begin{align}
\e\lnnmv{\vb^2\frac{\p\gg}{\p\va}}+\e\lnnmv{\vc^2\frac{\p\gg}{\p\va}}\ls \abs{\ln(\e)}^8.
\end{align}
We can perform the same analysis with an extra $\ue^{K_0\eta}$ term. Hence, our result naturally follows.
\end{proof}

\subsubsection{Estimates of Tangential Derivative}

Now we pull the tangential variables $\iota_1$ and $\iota_2$ dependence back and study the tangential derivatives.
\begin{theorem}\label{Milne tangential}
Assume \eqref{Milne bound} and \eqref{Regularity bound} holds. We have
\begin{align}
\lnnmv{\ue^{K_0\eta}\frac{\p\gg}{\p\iota_1}}\ls \abs{\ln(\e)}^8,\quad \lnnmv{\ue^{K_0\eta}\frac{\p\gg}{\p\iota_2}}\ls \abs{\ln(\e)}^8.
\end{align}
\end{theorem}
\begin{proof}
Let $\w:=\dfrac{\p\gg}{\p\iota_i}$ for $i=1,2$. Taking $\iota_i$ derivative on both sides of \eqref{Milne transform}, we know that $\w$ satisfies the equation
\begin{align}\label{Milne tangential problem}
\left\{
\begin{array}{l}\displaystyle
\va\frac{\p \w}{\p\eta}+G_1(\eta)\bigg(\vb^2\dfrac{\p
\w}{\p\va}-\va\vb\dfrac{\p
\w}{\p\vb}\bigg)+G_2(\eta)\bigg(\vc^2\dfrac{\p
\w}{\p\va}-\va\vc\dfrac{\p
\w}{\p\vc}\bigg)+\nu\w-K[\w]=S_{\w},\\\rule{0ex}{2.0em}
\w(0,\iota_1,\iota_2,\vvv)=\dfrac{\p p}{\p\iota_i}(\iota_1,\iota_2,\vvv)\ \ \text{for}\ \ \sin\phi>0,\\\rule{0ex}{2.0em}
\w(L,\iota_1,\iota_2,\vvv)=\w(L,\iota_1,\iota_2,\rr[\vvv]),
\end{array}
\right.
\end{align}
where
\begin{align}
S_{\w}=\frac{\p\ss}{\p\iota_i}+\dfrac{\p_{\iota_i}R_1}{R_1-\e\eta}G_1(\eta)\bigg(\vb^2\dfrac{\p\gg}{\p\va}-\va\vb\dfrac{\p\gg}{\p\vb}\bigg)
+\dfrac{\p_{\iota_i}R_2}{R_2-\e\eta}G_2(\eta)\bigg(\vc^2\dfrac{\p\gg}{\p\va}-\va\vc\dfrac{\p\gg}{\p\vc}\bigg).
\end{align}
For $\eta\in[0,L]$, we have
\begin{align}
\dfrac{\p_{\iota_i}R_j}{R_j-\e\eta}\ls \max_{i,j=1,2}\p_{\iota_i}R_j\ls 1.
\end{align}
Therefore, noting that $\abs{\va}\leq\zeta$, using \eqref{Regularity bound}, Theorem \ref{Regularity theorem 2} and Corollary \ref{Regularity corollary}, we have
\begin{align}
\lnnmv{S_{\w}}&\ls\lnnmv{\frac{\p\ss}{\p\iota_i}}+ \lnnmv{G_1(\eta)\bigg(\vb^2\dfrac{\p\gg}{\p\va}-\va\vb\dfrac{\p\gg}{\p\vb}\bigg)}
+\lnnmv{G_2(\eta)\bigg(\vc^2\dfrac{\p\gg}{\p\va}-\va\vc\dfrac{\p\gg}{\p\vc}\bigg)}\\
&\ls 1+\e\lnnmv{\vb^2\dfrac{\p\gg}{\p\va}}+\e\lnnmv{\vc^2\dfrac{\p\gg}{\p\va}}+\e\lnnmv{\nu\zeta\dfrac{\p\gg}{\p\vb}}+\e\lnnmv{\nu\zeta\dfrac{\p\gg}{\p\vc}}
\ls \abs{\ln(\e)}^8.\no
\end{align}
By a similar argument, we can add $\ue^{K_0\eta}$ contribution to obtain
\begin{align}
\lnnmv{\ue^{K_0\eta}S_{\w}}\ls \abs{\ln(\e)}^8.
\end{align}
Therefore, applying Theorem \ref{Milne theorem 4} to \eqref{Milne tangential problem}, we have that
\begin{align}
\lnnmv{\ue^{K_0\eta}\w(\eta,\iota_1,\iota_2,\vvv)}\ls \abs{\ln(\e)}^8.
\end{align}
\end{proof}

\begin{theorem}\label{Milne velocity}
Assume \eqref{Milne bound} and \eqref{Regularity bound} holds. We have
\begin{align}
\lnnmv{\ue^{K_0\eta}\nu\frac{\p\gg}{\p\vb}}\ls \abs{\ln(\e)}^8,\quad \lnnmv{\ue^{K_0\eta}\nu\frac{\p\gg}{\p\vc}}\ls \abs{\ln(\e)}^8.
\end{align}
\end{theorem}
\begin{proof}
Let $\v:=\vb\dfrac{\p\gg}{\p\vb}$. Taking $\vb$ derivative on both sides of \eqref{Milne transform} and multiplying $\vb$, we know that $\v$ satisfies the equation
\begin{align}\label{Milne tangential problem}
\left\{
\begin{array}{l}\displaystyle
\va\frac{\p \v}{\p\eta}+G_1(\eta)\bigg(\vb^2\dfrac{\p
\v}{\p\va}-\va\vb\dfrac{\p
\v}{\p\vb}\bigg)+G_2(\eta)\bigg(\vc^2\dfrac{\p
\v}{\p\va}-\va\vc\dfrac{\p
\v}{\p\vc}\bigg)+\nu\v=S_{\v},\\\rule{0ex}{2.0em}
\v(0,\iota_1,\iota_2,\vvv)=\vb\dfrac{\p p}{\p\vb}(\iota_1,\iota_2,\vvv)\ \ \text{for}\ \ \sin\phi>0,\\\rule{0ex}{2.0em}
\v(L,\iota_1,\iota_2,\vvv)=\v(L,\iota_1,\iota_2,\rr[\vvv]),
\end{array}
\right.
\end{align}
where
\begin{align}
S_{\v}=\int_{\r^3}\vb\p_{\vb}k(\vuu,\vvv)\ud\vuu+\vb\frac{\p\ss}{\p\vb}+2G_1\vb^2\frac{\p\gg}{\p\va}-2G_1\va\vb\frac{\p\gg}{\p\vb}.
\end{align}
Based on \eqref{Regularity bound}, Lemma \ref{Regularity lemma 2'} and Theorem \ref{Milne theorem 4}, we have
\begin{align}
\lnnmv{\int_{\r^3}\vb\p_{\vb}k(\vuu,\vvv)\ud\vuu}+\abs{\vb\frac{\p\ss}{\p\vb}}\ls 1.
\end{align}
Using Corollary \ref{Regularity corollary}, we get
\begin{align}
\lnnmv{2G_1\vb^2\frac{\p\gg}{\p\va}}\ls \e\lnnmv{\vb^2\frac{\p\gg}{\p\va}}\ls \abs{\ln(\e)}^8.
\end{align}
Using Theorem \ref{Regularity theorem 2}, we obtain
\begin{align}
\lnnmv{2G_1\va\vb\frac{\p\gg}{\p\vb}}\ls \e\lnnmv{\nu\zeta\frac{\p\gg}{\p\vb}}\ls 1.
\end{align}
Hence, collecting all above, we have proved that
\begin{align}
\lnnmv{S_{\v}}\ls \abs{\ln(\e)}^8.
\end{align}
Based on the analysis in Section \ref{mtt section 01}, we have
\begin{align}
\lnnmv{\v}\ls \lnmv{\vb\dfrac{\p p}{\p\vb}}+\lnnmv{\nu^{-1}S_{\v}}\ls \abs{\ln(\e)}^8.
\end{align}
By a similar argument, we can add $\ue^{K_0\eta}$ contribution to obtain
\begin{align}
\lnnmv{\ue^{K_0\eta}\vb\frac{\p\gg}{\p\vb}}\ls \abs{\ln(\e)}^8.
\end{align}
Similarly, we can show
\begin{align}
\lnnmv{\ue^{K_0\eta}\vc\frac{\p\gg}{\p\vb}}\ls \abs{\ln(\e)}^8.
\end{align}
Since $\abs{\va}\ls\zeta$, Theorem \ref{Regularity theorem 2} implies
\begin{align}
\lnnmv{\ue^{K_0\eta}\va\frac{\p\gg}{\p\vb}}\ls \abs{\ln(\e)}^8.
\end{align}
Then our result naturally follows. The $\dfrac{\p\gg}{\p\vc}$ bounds can be shown in a similar fashion.
\end{proof}

\begin{remark}
Theorem \ref{Regularity theorem 2}, Corollary \ref{Regularity corollary}, Theorem \ref{Milne tangential} and Theorem \ref{Milne velocity} provide bounds of all kinds of normal and velocity derivatives. However, note that $\dfrac{\p\gg}{\p\eta}$ estimate must be accompanied by the weight $\zeta$ since it may have singularity near the grazing set. Similarly, $\dfrac{\p\gg}{\p\va}$ estimate should be with either $\zeta$ or $\e$. On the other hand, $\dfrac{\p\gg}{\p\iota_i}$, $\dfrac{\p\gg}{\p\vb}$ and $\dfrac{\p\gg}{\p\vc}$ can avoid the introduction of $\zeta$ or $\e$, since they do not directly interact with grazing set.
\end{remark}

\section*{Acknowledgement}

L. Wu is supported by NSF grant DMS-1853002.


\bibliographystyle{siam}
\bibliography{Reference}

\end{document}